\documentclass[11pt,reqno]{amsart}
\usepackage{amsmath,amsthm,amssymb}
\usepackage{yhmath}
\usepackage{mathtools}
\usepackage{enumerate}
\usepackage{amsfonts}
\usepackage{eucal}
\usepackage{tikz-cd}
\usepackage[all]{xy}
\CompileMatrices
\usepackage{hyperref}
\definecolor{webgreen}{rgb}{0,.5,0}
\definecolor{webbrown}{rgb}{.6,0,0}
\definecolor{ocre}{RGB}{52,177,201}
	\definecolor{royalblue(web)}{rgb}{0.25, 0.41, 0.88}
\hypersetup{
 colorlinks=true,linktocpage=true,pdfstartpage=1,
 pdfstartview=FitV,breaklinks=true,pdfpagemode=UseNone,
 pageanchor=true,pdfpagemode=UseOutlines,
 plainpages=false,bookmarksnumbered,
 bookmarksopen=true,bookmarksopenlevel=1,
 hypertexnames=true,pdfhighlight=/O,
 urlcolor=webbrown,linkcolor=royalblue(web),
 citecolor=webgreen,
 hyperfootnotes=false,pdfpagelabels,
 pdfcreator={pdfLaTeX},
 pdfproducer={LaTeXwithArsClassica}
 }
\usepackage{enumitem}
\usepackage[normalem]{ulem}
\usepackage{slashed}
\usepackage{scalerel}
\usepackage{galois}
\usepackage{esvect}
\usepackage{imakeidx,etoolbox}
\usepackage[usestackEOL]{stackengine}
\makeindex[title=Index,columns=2]

\newcommand{\defoperator}[1]{\expandafter\def\csname#1\endcsname{\operatorname{#1}}}
\newcommand{\deffrak}[1]{\expandafter\def\csname#1\endcsname{\mathfrak{#1}}}
\newcommand{\defbb}[1]{\expandafter\def\csname#1#1\endcsname{\mathbb{#1}}}
\newcommand{\dbar}{\bar{\partial}}

\newcommand{\sA}{{\mathcal A}}
\newcommand{\sO}{{\mathcal O}}

\defbb{Z}
\defbb{R}
\defbb{C}
\defbb{H}
\defbb{V}
\deffrak{g}
\deffrak{t}
\defoperator{Im}
\defoperator{Re}
\defoperator{Id}
\defoperator{id}
\defoperator{End}
\defoperator{Hom}
\defoperator{Aut}
\defoperator{Iso}
\defoperator{Diff}
\defoperator{ad}
\defoperator{pr}
\defoperator{Ad}
\defoperator{At}
\defoperator{rank}
\defoperator{ker}
\defoperator{coker}
\defoperator{im}
\defoperator{tr}
\defoperator{Str}
\defoperator{ind}
\defoperator{diag}
\defoperator{ord}
\defoperator{Spec}
\defoperator{Proj}
\defoperator{Ext}
\defoperator{Gal}
\defoperator{Jac}
\defoperator{Gr}
\defoperator{Ric}
\defoperator{Isom}
\defoperator{DF}

\defoperator{GL}
\defoperator{SL}
\defoperator{PSL}
\defoperator{U}
\defoperator{SU}
\defoperator{PSU}
\defoperator{SO}
\defoperator{Spin}
\defoperator{grad}
\defoperator{Alb}
\defoperator{Stab}

\DeclareMathOperator*{\midoplus}{\text{\raisebox{0.25ex}{\scalebox{0.8}{$\bigoplus$}}}}

\makeatletter
\newcommand*\rel@kern[1]{\kern#1\dimexpr\macc@kerna}
\newcommand*\widebar[1]{%
  \begingroup
  \def\mathaccent##1##2{%
    \rel@kern{0.8}%
    \overline{\rel@kern{-0.8}\macc@nucleus\rel@kern{0.2}}%
    \rel@kern{-0.2}%
  }%
  \macc@depth\@ne
  \let\math@bgroup\@empty \let\math@egroup\macc@set@skewchar
  \mathsurround\z@ \frozen@everymath{\mathgroup\macc@group\relax}%
  \macc@set@skewchar\relax
  \let\mathaccentV\macc@nested@a
  \macc@nested@a\relax111{#1}%
  \endgroup
}
\makeatother

\newcommand{\I}{\mathrm{i}}
\newcommand{\E}{\mathrm{e}}

\DeclareFontFamily{OT1}{rsfs}{}
\DeclareFontShape{OT1}{rsfs}{n}{it}{<-> rsfs10}{}
\DeclareMathAlphabet{\mathscr}{OT1}{rsfs}{n}{it}
\renewcommand*{\setminus}{-}
\DeclareEmphSequence{\bfseries}

\makeatletter
\newcommand*{\Rom}[1]{\expandafter\@slowromancap\romannumeral #1@}
\makeatother

\theoremstyle{plain}
  \newtheorem{theorem}{Theorem}
  \numberwithin{theorem}{section}
	
  \newtheorem{proposition}[theorem]{Proposition}
  \newtheorem{lemma}[theorem]{Lemma}
  \newtheorem{corollary}[theorem]{Corollary}
\theoremstyle{definition}

  \newtheorem{definition}[theorem]{Definition}
  
  \newtheorem{assumption}[theorem]{Assumption}

  \newtheorem{example}[theorem]{Example}
  \newtheorem{remark}[theorem]{Remark}

\numberwithin{equation}{section}

{
  \begin{enumerate}[label=(#1*),leftmargin=1.7\parindent]
}%
{
  \end{enumerate}
}

\newcommand{\Dol}{\mathrm{Dol}}
\newcommand{\dR}{\mathrm{dR}}
\newcommand{\B}{\mathrm{B}}
\newcommand{\D}{\mathrm{d}}
\newcommand{\G}{\mathbb{G}}
\newcommand{\calH}{\mathcal{H}}

\newcommand{\calE}{\mathcal{E}}
\newcommand{\calA}{\mathcal{A}}

\usepackage{graphicx}
\newcommand\sbullet[1][.5]{\mathbin{\ThisStyle{\vcenter{\hbox{%
  \scalebox{#1}{$\SavedStyle\bullet$}}}}}%
}

\title{Connections, metrics and Higgs fields on complex fiber bundles}
\author{Nianzi Li}
\email{lnz@mail.tsinghua.edu.cn}
\address{Yau Mathematical Sciences Center, Tsinghua University, Beijing 100084, China}

\author{Mao Sheng}
\email{msheng@tsinghua.edu.cn}
\address{Yau Mathematical Sciences Center, Tsinghua University, Beijing 100084, China \& Yanqi
Lake Beijing Institute of Mathematical Sciences and Applications, Beijing, 101408, China}
\begin{document}
\begin{abstract}
We give a differential-geometric representation of the extension class associated to a holomorphic fibration, generalizing the work of Atiyah on holomorphic principal bundles in a natural way. As an application, we obtain a nonlinear analogue of the classical result of Weil on characterizing the existence of flat connections on holomorphic vector bundles over compact Riemann surfaces. We further establish a faithful functor from the category of nonlinear flat holomorphic bundles reductive of K\"ahler type to the category of nonlinear Higgs bundles over the same connected compact K\"ahler base. Finally, we establish a notion of nonlinear harmonic bundle and prove that the variation of nonabelian Hodge structure is a nonlinear harmonic bundle in the rank one case and in the semisimple case.
\end{abstract}

\maketitle
\tableofcontents
\section{Introduction}
Weil \cite{weil1938generalisation} showed that an indecomposable holomorphic vector bundle over a compact Riemann surface admits a flat holomorphic connection if and only if its degree is zero. Atiyah \cite{atiyah57} gave an alternative proof of this result, and showed that the obvious generalization (namely holomorphic vector bundles with vanishing Chern classes) does not hold over a higher-dimensional base. Indeed, stability comes into play. As an important consequence of the existence theorem of Uhlenbeck-Yau \cite{uhlenbeck_yau1986existence}, stable bundles with vanishing Chern classes over a compact K\"ahler manifold admit flat connections. In fact, they constructed a canonical flat structure out of it, namely the unitary flat connection. This leads to the Donaldson-Uhlenbeck-Yau (DUY) correspondence, generalizing the fundamental work of Narasimhan-Seshadri \cite{narasimhan_seshadri1965stable} on the correspondence between irreducible unitary representations of the topological fundamental group of a compact Riemann surface and stable vector bundles of degree zero over it. However, even in the one-dimensional case, a stable bundle of degree zero may admit more than one flat structure, and an indecomposable holomorphic vector bundle may well be unstable. These phenomena can be beautifully explained by the nonabelian Hodge correspondence \cite{donaldson87twisted, corlette88, hitchin87selfdual, simpson1988construct}, extending the DUY correspondence by introducing an additional structure on holomorphic vector bundles, the so-called Higgs fields \cite{hitchin87selfdual, Si1}.

In a recent work \cite{She25b} of the second-named author, a nonlinear Hodge correspondence in characteristic $p$ was established, extending the work of Ogus-Vologodsky \cite{OV07} on nonabelian Hodge correspondence in characteristic $p$. It is called nonlinear because one considers there connections or Higgs fields on arbitrary smooth morphisms, which extend linear connections or linear Higgs fields on vector bundles in a natural way. It is natural to ask for a nonlinear Hodge correspondence over the field of complex numbers. In \cite{She25a}, the first attempt towards this goal has been made. In this work, we shall continue to explore the same circle of ideas, enrich the notion of a nonlinear harmonic bundle introduced therein, and establish a half nonlinear Hodge correspondence in the finite-dimensional case.

Let us begin by recalling the Atiyah class that was originally introduced by Atiyah in \cite{atiyah57}. Let $G$ be a connected complex Lie group, and let $\pi:P\to S$ be a holomorphic principal $G$-bundle over a complex manifold $S$. The infinitesimal structure of $P$ is encoded in the so-called Atiyah sequence
\begin{equation}\label{AtiyahSeq_eq}
    0 \to \ad (P)\to \At (P) \to TS\to 0,
\end{equation}
where $\ad (P)$ is the adjoint bundle, and $\At(P)=TP/G$ is the Atiyah bundle. A holomorphic connection on $P$ is defined as a holomorphic splitting of \eqref{AtiyahSeq_eq}. The obstruction to the existence of such a connection is the Atiyah class $A(P)\in H^1(S,\ad P\otimes \Omega_S)$, which is the extension class of \eqref{AtiyahSeq_eq}. Therefore, $P$ admits a holomorphic connection if and only if $A(P)=0$ (\cite[Theorem 2]{atiyah57}). An important step towards a differential geometrical understanding of the Atiyah class is that $A(P)$ coincides with the Dolbeault cohomology class $-[F_{A}^{1,1}]\in H^{1,1}(S,\ad P)$, where $F_{A}$ is the curvature of \emph{any} $C^\infty$ connection $A$ of type $(1,0)$ (\cite[Proposition 4]{atiyah57}, with the comparison sign adjusted to our convention).

A holomorphic fibration $f: X\to S$ is a holomorphic submersion between complex manifolds $X$ and $S$. In the above, $\pi$ is a special kind of holomorphic fibration: It is a holomorphic fiber bundle, namely, locally over open subsets $U\subset S$, $\pi$ is holomorphically isomorphic to the projection $U\times G\to U$. Our first group of results is that Atiyah's idea can be generalized to an arbitrary holomorphic fibration. Indeed, canonically associated to $f$ is the following short exact sequence of holomorphic tangent bundles
\begin{equation}\label{TangentSeqFibration_eq}
    0 \to T_{X/S}\to TX\to f^\ast TS\to 0.
\end{equation}
The Atiyah sequence \eqref{AtiyahSeq_eq} is nothing but the quotient of the above sequence for $\pi$ by the $G$-action. A holomorphic connection on $P$ is a $G$-equivariant holomorphic splitting of the above sequence.

Throughout the paper, for vector-valued one-forms $\alpha,\beta$ whose values carry a Lie bracket, we use the convention
\[
[\alpha,\beta](u,v):=[\alpha(u),\beta(v)]-[\alpha(v),\beta(u)],
\]
where $u,v$ are tangent vectors. In particular, $\frac12[\alpha,\alpha](u,v)=[\alpha(u),\alpha(v)]$. We use the \v{C}ech differential $(\delta u)_{ab}=u_b-u_a$ and the total differential $\D_{\mathrm{tot}}=\delta+(-1)^p\dbar$ in \v{C}ech degree $p$. We identify $(0,q)$-forms with values in holomorphic one-forms with $(1,q)$-forms by $\beta\otimes\alpha\mapsto\beta\wedge\alpha$; the same convention applies with vector bundle coefficients.

\begin{definition}
A holomorphic connection on $f$ is a holomorphic splitting of the sequence \eqref{TangentSeqFibration_eq}. That is, it is a holomorphic bundle morphism
\[
\nabla: f^*TS\to TX
\]
such that its composition with the projection $TX\to f^*TS$ is the identity. $\nabla$ is said to be integrable or flat if its curvature, which is a holomorphic bundle morphism defined by
\[
F_{\nabla}: f^*\wedge^2TS\longrightarrow T_{X/S},\qquad F_{\nabla}(u,v)=[\nabla(u),\nabla(v)]-\nabla([u,v]),
\]
is zero. Here $u,v$ are local holomorphic vector fields on $S$.
\end{definition}
Clearly, the integrability condition is vacuous when $S$ is one-dimensional. An integrable connection on a fibration $f: X\to S$ is nothing but a foliation on the total space $X$ that is transversal to the fibers of $f$ everywhere. With this point of view, Riccati foliations (resp. turbulent foliations) on complex surfaces (\cite[Ch.~4]{Bru15}, \cite{LN02}) can be regarded as integrable connections on rational fibrations (resp. elliptic fibrations) after restricting to the locus where the fibration is smooth and the foliation is transverse to its fibers. By its very definition, the vanishing of the extension class $A(X)\in H^1(X,f^\ast \Omega_S\otimes T_{X/S})$ is equivalent to the existence of a holomorphic connection on $f$. In order to give a differential geometrical interpretation of $A(X)$, it is necessary to introduce a wider class of connections.
\begin{definition}
A complex connection $\nabla^{\mathbb{C}}$ on $f$ is a smooth splitting of the natural projection
$$
TX^{\mathbb{C}}=TX\oplus \overline{TX}\to f^*TS^{\mathbb{C}}=f^*TS\oplus f^*\overline{TS}.
$$
It is said to be pure if $\nabla^{\mathbb{C}}$ maps $f^*TS$ into $TX$, and $f^*\overline{TS}$ into $\overline{TX}$. A $(1,0)$-connection $\nabla^{1,0}$ on $f$ is a smooth splitting of the composite of projections $TX^{\mathbb C}\to f^*TS^{\mathbb C}\to f^*TS$ such that the projection of its image under $TX^{\mathbb C}\to f^*TS^{\mathbb C}$ is contained in $f^* TS$. It is said to be relatively holomorphic if $[\overline{T_{X/S}},\textrm{im}(\nabla^{1,0})]\subset \overline{T_{X/S}}\oplus \textrm{im}(\nabla^{1,0})$ holds.
\end{definition}
We may then summarize Proposition \ref{prop:dbar_complex_structure}, Corollary \ref{cor:rel_hol_ks_vanish}, Proposition \ref{prop:dolb_class} and Theorem \ref{thm:Ext_class_curv_class} into the following statement, which is a generalization of \cite[Proposition 4]{atiyah57}.
\begin{theorem}\label{thm: main result 1}
There is a $\bar{\partial}_X$-closed tensor $\mathcal{R}\in A^{0,1}(X,f^*\Omega_S\otimes T_{X/S})$ canonically associated to a pure complex connection $\nabla^{\mathbb{C}}=\nabla^{1,0}+ \nabla^{0,1}$, whose class $[\mathcal{R}]\in H^1(X,f^*\Omega_S\otimes T_{X/S})$ equals $A(X)$. When $\nabla^{1,0}$ is relatively holomorphic, the Kodaira-Spencer map of $f$ vanishes everywhere. In this case, $\mathcal{R}$ equals the curvature $F^{1,1}_{D}$ associated to $D=\nabla^{1,0}+\bar \partial_{f}$, where $\bar \partial_f$ is the canonical $\bar{\partial}$-operator associated to $f$.
\end{theorem}
\begin{remark}
A holomorphic fibration $f: X\to S$ whose underlying smooth map is a fiber bundle can be regarded as a complex fiber bundle $(f,T_{X/S})$ equipped with an integrable $\bar\partial$-operator (\S \ref{sub: complex structure}). This point of view is basic throughout the paper. The notion of relatively holomorphic $(1,0)$-connections is also meaningful for a complex fiber bundle equipped with a (not necessarily integrable) $\bar\partial$-operator, and the curvature $F^{1,1}_D$ is also well-defined in this generalization (Lemma \ref{lem:curv_def_rel_hol}).
\end{remark}
For a holomorphic fiber bundle, there always exists a relatively holomorphic $(1,0)$-connection (see Proposition \ref{prop:exist_fiberwise_holo_connection}). Combining Theorem \ref{thm: main result 1} with Weil's original theorem for vector bundles and a theorem of Azad-Biswas for principal bundles \cite{hassan_biswas03}, we obtain a nonlinear analogue of Weil's theorem.
\begin{corollary}[Corollary \ref{cor:Weil}]
Let $S$ be a compact connected Riemann surface and $G$ be a connected complex reductive group. Let $f:X\to S$ be a holomorphic fiber bundle with typical fiber $Y$ and structure group $G\leq \Aut(Y)$ (in other words, $f$ is the associated bundle of a holomorphic principal $G$-bundle $P\to S$). Suppose $H^0(Y,TY)$ is finite-dimensional (e.g. $Y$ is proper). Then $f$ admits a holomorphic (automatically flat) connection if and only if each summand in the Remak decomposition of $\ad P$ has degree zero and $c(P)\in \pi_1(G)$ is torsion.
\end{corollary}
For a proper holomorphic fiber bundle $f$ over a Riemann surface (structure group may be non-reductive), we have a characterization of the existence of a flat structure (representation of $\pi_1(S)$) on $f$ in terms of the vanishing of the curvature class $[F^{1,1}_D]_S\in H^{1,1}(S,f^{\mathrm{hol}}_*T_{X/S})$ (see Corollary \ref{cor:RH_RiemannSurface}), where $f^{\mathrm{hol}}_*T_{X/S}$ is the direct image sheaf. Relatively holomorphic connections continue to play an important role in our next investigation, harmonic metrics on (nonlinear) flat holomorphic bundles. The following definition is the obvious analytic analogue of the one introduced in \cite{She25b}.
\begin{definition}\label{def:holomorphic_Higgs_field}
Let $S$ be a complex manifold. A flat holomorphic bundle over $S$ is a pair $(f,\nabla)$, where $f: X\to S$ is a holomorphic fibration, and $\nabla$ is an integrable holomorphic connection on $f$. A nonlinear Higgs bundle is also a pair $(f,\theta)$, where $\theta$ is a holomorphic Higgs field on $f$: It is an $\mathcal{O}_S$-linear morphism $\theta: TS\to f^{\mathrm{hol}}_*T_{X/S}$ satisfying the integrability condition $[\theta,\theta]=0$. For two flat holomorphic bundles $(f_1: X_1\to S,\nabla_1)$ and $(f_2: X_2\to S,\nabla_2)$ over $S$, we define a morphism $F: (f_1,\nabla_1)\to (f_2,\nabla_2)$ to be a holomorphic map $F: X_1\to X_2$ such that $f_2\comp F=f_1$ and $\D F\comp \nabla_1=F^*\nabla_2$ where $\D F: TX_1\to F^*TX_2$ is the tangent map of $F$. Morphisms between nonlinear Higgs bundles over $S$ are similarly defined.
\end{definition}
When the context is clear, we shall call a flat holomorphic bundle (resp. nonlinear Higgs bundle) simply a flat bundle (resp. Higgs bundle). Let $f: X\to S$ be a holomorphic fibration. Note that when $f$ is non-proper, not every flat holomorphic connection on $f$ gives rise to a flat structure. This motivates us to introduce the following
\begin{definition}
A holomorphic connection on $f$ is said to be complete if the horizontal lift of any smooth path $\gamma: [0,1]\to S$ starting at any point $x_0\in X_{\gamma(0)}$ is defined for all $t\in [0,1]$. A flat holomorphic bundle $(f,\nabla)$ is said to be complete if $\nabla$ is complete.
\end{definition}
We obtain a nonlinear generalization of the classical Riemann-Hilbert correspondence.
\begin{theorem}[Theorem \ref{thm:HolomorphicRH}]\label{RH}
Let $S$ be a connected complex manifold. Then there is an equivalence between the category of representations $\pi_1(S,s_0)\to \Aut(Y)$, where $Y$ is a (non-fixed) complex manifold, and the category of complete flat holomorphic bundles over $S$.
\end{theorem}
Let $G\leq \Aut(Y)$ be a complex Lie subgroup. A complete flat holomorphic bundle $(f,\nabla)$ over $S$ with typical fiber $Y$ is said to have holonomy contained in $G$ if it corresponds via Theorem \ref{RH} to a representation $\pi_1(S,s_0)\to G\leq \Aut(Y)$. Now we suppose that $Y$ is of K\"ahler type, and consider a group $G$ satisfying the following
\begin{assumption}\label{as: model metric}
There exists a K\"ahler metric $\omega_Y$ on $Y$ such that $K:=\Stab_G(\omega_Y)$ is a compact real form of $G\leq \Aut(Y)$, where $G$ is a connected complex reductive closed Lie subgroup.
\end{assumption}
A typical example to keep in mind is a polarized compact constant scalar curvature K\"ahler (cscK) manifold $(Y,H_Y,\omega_Y)$, where $(Y,H_Y)$ is a polarized compact complex manifold and $\omega_Y\in c_1(H_Y)$ is a cscK metric. Then for $G=\Aut_0(Y,H_Y)$, $K=\Stab_G(\omega_Y)$ satisfies the assumption. See Example \ref{ex: mabuchi's result} for more information.
\begin{definition}
A flat holomorphic bundle over $S$ is of K\"ahler type if it is complete and has holonomy contained in a group $G$ satisfying Assumption \ref{as: model metric}. Furthermore, it is said to be reductive if the corresponding representation $\pi_1(S,s_0)\to G$ is reductive, meaning that the Zariski closure of its image is reductive.
\end{definition}
The collection of flat holomorphic bundles reductive of K\"ahler type over $S$ and the admissible morphisms of Definition \ref{def:admissible_flat_morphism} between them constitutes a subcategory of the category of flat holomorphic bundles over $S$. We obtain one half of a nonlinear Hodge correspondence over $\CC$.
\begin{theorem}[Proposition \ref{prop:faithful_functor}]\label{thm: half correspondence in proper case}
Let $S$ be a connected  compact complex manifold of K\"ahler type. Then there is a faithful functor from the category of flat holomorphic bundles reductive of K\"ahler type over $S$ to the category of nonlinear Higgs bundles over $S$.
\end{theorem}
The functor is independent of the choice of a K\"ahler metric $\omega_Y$ on $Y$ satisfying Assumption \ref{as: model metric}. Let us describe our approach to the construction of the functor in Theorem \ref{thm: half correspondence in proper case}. It utilizes a generalized Simpson mechanism of Section \ref{sec:simpson_mechanism}, which seeks harmonic fiberwise K\"ahler metrics on fibrations. Let $(f: X\to S, T_{X/S})$ be a complex fiber bundle (Definition \ref{def: cpx fiber bundle}). A fiberwise K\"ahler metric $\omega_{X/S}$ on it is a smooth section of $\wedge^{1,1}T^*_{X/S}$ such that it restricts to a K\"ahler metric on each fiber. It is said to be modeled on $(Y,\omega_Y)$ if there exists a unitary atlas for $(f,\omega_{X/S})$ (see Lemma \ref{lem:unitary atlas}).

\begin{proposition}[Proposition \ref{prop:Chern_conn}]
Let $(f,T_{X/S})$ be a complex fiber bundle over $S$. Let $\bar \partial_f$ be a $\bar \partial$-operator satisfying the lifting condition (Lemma \ref{lem:integrability_dbar}) and $\omega_{X/S}$ be a fiberwise K\"ahler metric on $(f,T_{X/S})$ modeled on $(Y,\omega_Y)$. Let $\mathfrak{k}\subset \mathfrak{aut}(Y,\omega_Y)$ be a real subspace satisfying $\mathfrak{k}\cap \I\mathfrak{k}=\{0\}$. Choose and then fix a $\mathfrak{k}$-compatible unitary atlas $\{(U_a,\Phi_a)\}_a$ in the sense of Definition \ref{def:k_compatible_unitary_atlas}. Suppose $\bar \partial_f-\bar \partial_a\in A^{0,1}(U_a,\mathfrak{k}^{\CC})$ for each $a$, where $\{\bar \partial_a\}_a$ are the $\bar \partial$-operators attached to the unitary atlas (Proposition \ref{prop: kahler connection}). Then this atlas determines a relatively holomorphic $(1,0)$-connection $\nabla^{1,0}_{\omega_{X/S}}: f^*TS\to TX^{\mathbb{C}}$ satisfying the following two properties:
\begin{itemize}
    \item [(i)]
    $\nabla^{1,0}_{\omega_{X/S}}$ is pure with respect to $\bar \partial_f$. That is, the composition
    $$
f^*\overline{TS}\xrightarrow{\overline{\nabla_{\omega_{X/S}}^{1,0}}} TX^{\mathbb{C}}\to \frac{TX^{\mathbb{C}}}{\overline{T_{X/S}}}
    $$
    is $\bar \partial_f$;
    \item [(ii)]  $\nabla^{1,0}_{\omega_{X/S}}$ preserves $\omega_{X/S}$. That is, for any real horizontal vector field $H_{v}$ with respect to $\nabla^{\RR}_{\omega_{X/S}}$, $\mathcal{L}_{H_v}\omega_{X/S}=0$. Here $\nabla^{\RR}_{\omega_{X/S}}$ is the real connection whose complexification is $\nabla^{1,0}_{\omega_{X/S}}+\overline{\nabla^{1,0}_{\omega_{X/S}}}$.
\end{itemize}
The connection is independent of the choice of an equivalent $\mathfrak{k}$-compatible unitary atlas. We shall call $\nabla^{1,0}_{\omega_{X/S}}$ the Chern connection associated to $(\bar \partial_f, \omega_{X/S})$ relative to this atlas class.
\end{proposition}
Clearly, $\mathfrak{k}=\textrm{Lie}(K)$ in Assumption \ref{as: model metric} satisfies the conditions $\mathfrak{k}\subset\mathfrak{aut}(Y,\omega_Y)$ and $\mathfrak{k}\cap\I\mathfrak{k}=\{0\}$. $\nabla^{1,0}_{\omega_{X/S}}$ is said to be the Chern connection, because it generalizes the classical Chern connection attached to a Hermitian vector bundle equipped with an ordinary $\bar \partial$-operator. This notion is crucial in the Simpson mechanism: Let $(f,\nabla)$ be a flat holomorphic bundle. Let $\bar \partial_f$ be the canonical $\bar{\partial}$-operator attached to the holomorphic fibration $f$. Choose and then fix a K\"ahler metric $\omega_Y$. Take a fiberwise K\"ahler metric $\omega_{X/S}$ on $f$ which is modeled on $(Y,\omega_Y)$. Let $\nabla^{1,0}_{\omega_{X/S}}$ be the Chern connection associated to $(\bar \partial_f, \omega_{X/S})$ and a $\mathfrak{k}$-compatible unitary atlas. Then we obtain an almost Higgs field on the complex fiber bundle $(f,T_{X/S})$ by the formula:
$$
\theta_{\omega_{X/S}}=\tfrac{1}{2}(\partial-\partial^{\mathrm{Ch}}_{\omega_{X/S}}),$$
where $\partial$ resp. $\partial^{\mathrm{Ch}}_{\omega_{X/S}}$ are the almost connections (Definition \ref{def:almost_connection}) associated to $\nabla$ resp. $\nabla^{1,0}_{\omega_{X/S}}$, as well as an almost complex structure on $(f,T_{X/S})$ by the formula:
$$
\bar \partial_{\omega_{X/S}}=\bar \partial_f-\bar{\theta}_{\omega_{X/S}},
$$
where $\bar{\theta}_{\omega_{X/S}}$ is the complex conjugate of $\theta_{\omega_{X/S}}$ (Definition \ref{def:comp_conj_higgs}). They form an almost Higgs pair $(\bar \partial_{\omega_{X/S}},\theta_{\omega_{X/S}})$ on $(f,T_{X/S})$. There is a well-defined tensor $G_{D''}$ attached to $D''=\bar{\partial}_{\omega_{X/S}}+\theta_{\omega_{X/S}}$, which is called the pseudo-curvature of the almost Higgs pair. It vanishes if and only if the almost Higgs pair is a genuine Higgs pair. That is, $\bar \partial_{\omega_{X/S}}$ is an integrable complex structure on the complex fiber bundle and hence gives rise to a potentially different holomorphic fibration $f': X'\to S$ with the same underlying differentiable fiber bundle structure as $f: X\to S$ (the identity map on the underlying smooth total space is biholomorphic if and only if $\theta_{\omega_{X/S}}=0$), and $\theta_{\omega_{X/S}}$ is a holomorphic nonlinear Higgs field on $f'$. This process generalizes the classical one for flat vector bundles in \cite{Si1}.
\begin{definition}
Let $(f,\nabla)$ be a flat holomorphic bundle. A fiberwise K\"ahler metric $\omega_{X/S}$ modeled on $(Y,\omega_Y)$ is said to be harmonic if the resulting pseudo-curvature $G_{D''}$ according to the previous process vanishes.
\end{definition}
The existence theorem of Donaldson \cite{donaldson87twisted} and Corlette \cite{corlette88} allows us to obtain the following
\begin{proposition} [Proposition \ref{thm:nonlinear_higgs_correspondence}]\label{Intro_Thm2}
Let $S$ be a compact complex manifold of K\"ahler type and $(f,\nabla)$ be a flat holomorphic bundle reductive of K\"ahler type over $S$. Then there exists a harmonic fiberwise K\"ahler metric on $(f,\nabla)$.
\end{proposition}

\begin{remark}\label{rmk: princiapl bundle approach}
We remind the reader of an alternative approach to the construction of the functor which is more classical. It goes through principal bundles: Let $\rho: \pi_1(S,s_0)\to G\leq \Aut_0(Y)$ be the representation associated to a given flat holomorphic bundle reductive of K\"ahler type. Let $\pi: P=\widetilde S\times_{\rho}G\to S$ be the associated principal $G$-bundle together with the canonical $G$-equivariant flat holomorphic connection $\nabla_P$ on $\pi$. Since $\rho$ is reductive, it follows again from \cite{donaldson87twisted, corlette88} (see also \cite[\S2]{simpson97Hodge}) that there exists an equivariant harmonic map to $G/K$ which is pluriharmonic because of the K\"ahler condition. Thus we obtain a harmonic metric ($K$-reduction) on $\pi$ which yields a principal Higgs bundle structure. From that we may associate a nonlinear Higgs bundle with typical fiber $Y$.
\end{remark}

A variation of nonabelian Hodge structure (or relative de Rham moduli space) admits a natural flat connection, the so-called nonabelian Gauss-Manin connection. Although it is already an issue whether it is reductive of K\"ahler type in general, we know that the complex fiber bundle underlying the associated graded Higgs bundle (the corresponding relative Dolbeault moduli space equipped with the nonabelian Kodaira-Spencer map, see \cite{fu2025}) cannot be isomorphic to the one underlying the variation of nonabelian Hodge structure. Therefore, the previous construction of Higgs bundles does not work for variations of nonabelian Hodge structure, so we have to extend the Simpson mechanism further. To this end, we introduce the twisting maps and the twisted Simpson mechanism (see Section \ref{subsec:twisted_mech}).

To relate the geometries of two different integrable complex structures $J_{X/S}^A$ and $J_{X/S}^B$ on the same real relative tangent bundle $T_{X/S}^\RR$, we define a twisting map $\beta_{X/S}$ as a smooth real vector bundle isomorphism $\beta_{X/S}^\RR: T_{X/S}^\RR \xrightarrow{\cong} T_{X/S}^\RR$ satisfying the intertwining condition $\beta_{X/S}^\RR \comp J_{X/S}^A = J_{X/S}^B \comp \beta_{X/S}^\RR$. Equivalently, it can be viewed as a complex vector bundle isomorphism $\beta_{X/S}: T_{X/S}^A \xrightarrow{\cong} T_{X/S}^B$.

Utilizing this twisting map, the twisted Simpson mechanism relates a flat structure on $X_A = (f, J_{X/S}^A)$ and a Higgs bundle structure on $X_B = (f, J_{X/S}^B)$ as follows. Take a flat connection on $X_A$ (which determines an almost connection $\partial_A$ and a $\dbar$-operator $\dbar_A$) and a fiberwise Riemannian metric $g_{X/S}$ that is K\"ahler with respect to both $J_{X/S}^A$ and $J_{X/S}^B$. Assume that $\beta_{X/S}^{\RR}$ is an isometry for $g_{X/S}$ and that $|J_{X/S}^A-J_{X/S}^B|_{g_{X/S}}$ is a positive smooth function. We define an almost Higgs field $\theta$ on $X_B$ and then a $\dbar$-operator $\dbar_B$ on $X_B$ by
\[
\theta := \tfrac{1}{2}|J_{X/S}^A-J_{X/S}^B|_{g_{X/S}}^{-1} \beta_{X/S} (\partial_A - \partial_{\omega_{X/S}^A}),\quad \dbar_B := \beta_{X/S}(\dbar_A - \dbar_{A,0}) + \dbar_{B,0} - \bar{\theta}_J.
\]
Here $\partial_{\omega_{X/S}^A}$ is a symplectic almost connection associated to $\omega_{X/S}^A$ and $\dbar_A$, $\dbar_{A,0}$ and $\dbar_{B,0}$ are reference $\dbar$-operators on $X_A$ and $X_B$, and $\bar{\theta}_J(\bar{v}):=\mathrm{pr}_{T_{X/S}^B}(\I J_{X/S}^A\widebar{\theta(v)}_{J_{X/S}^B})$. We say $g_{X/S}$ is a $\beta$-twisted harmonic metric if the resulting pair $(\dbar_B, \theta)$ defines a Higgs bundle structure, with the auxiliary choices in the mechanism fixed. Conversely, starting from a Higgs bundle structure $(\dbar_B,\theta)$ on $X_B$, we may construct $\dbar_A$ and then $\partial_A$ on $X_A$ by
\[	\dbar_A := \beta_{X/S}^{-1}(\dbar_B-\dbar_{B,0} + \bar{\theta}_J)+\dbar_{A,0},\quad \partial_A:=\partial_{\omega_{X/S}^A}+2|J_{X/S}^A-J_{X/S}^B|_{g_{X/S}}\beta_{X/S}^{-1}(\theta). \]

In the case of a variation of nonabelian Hodge structure, we find that the fiberwise Hitchin metric is a $(\id+J_{X/S}^AJ_{X/S}^B)/\sqrt{2}$-twisted harmonic metric. Let $S$ be a smooth complex variety and $f: X\to S$ be a smooth projective family of algebraic curves of genus $\geq 2$. For a connected complex reductive group $G$, let $(f_{\dR}: M_{\dR}(X/S,G)\to S, \nabla_{\mathrm{GM}}, F_{\mathrm{Hod}})$ be the variation of nonabelian Hodge structure attached to $f$. Let $(f_{\Dol}: M_{\Dol}(X/S,G)\to S,\theta_{\mathrm{KS}})$ be the associated graded Higgs bundle to the variation.
\begin{theorem}[Theorem \ref{thm: vnhs is harmonic bundle}]
Suppose $G$ is either $\CC^*$ or semisimple. The flat bundle $(f_{\dR}: M_{\dR}(X/S,G)\to S, \nabla_{\mathrm{GM}})$ and the Higgs bundle $(f_{\Dol}: M_{\Dol}(X/S,G)\to S,\theta_{\mathrm{KS}})$ can be reconstructed from each other via the twisting map $(\id+J_{X/S}^AJ_{X/S}^B)/\sqrt{2}$ and the fiberwise Hitchin metric through the twisted Simpson mechanism. Here, in the semisimple case, we restrict to the smooth loci of the moduli spaces.
\end{theorem}

\noindent{\bf Acknowledgements.} The authors would like to thank Tianzhi Hu who drew our attention to the recent work \cite{CTW25}, Junchao Shentu for pointing out an error in the original formulation of Theorem \ref{thm: main result 1}, and Zhaofeng Yu for many valuable discussions about the Simpson mechanism. This work is supported by the Chinese Academy of Sciences Project for Young Scientists in Basic Research (Grant No. YSBR-032). The first-named author is supported by the Shuimu Tsinghua Scholar Program.

\section{Preliminaries}
\subsection{$\dbar$-operators and almost complex structures}\label{sub: complex structure}
Let $f:X\to S$ be a smooth fiber bundle over a complex manifold $S$. Associated to $f$ is the short exact sequence of complex vector bundles over $X$:
\begin{equation}\label{ComplexTangentSeq_General_eq}
0\to T_{X/S}^{\CC}\to TX^{\CC}\xrightarrow{\D f^\CC} f^*TS^{\CC}\to 0,
\end{equation}
where $T_{X/S}^{\CC}$ is the complexification of the real relative tangent bundle $T_{X/S}^{\RR}$ and similarly for the other bundles. An integrable complex structure on $T_{X/S}^{\RR}$ is a choice of a complex subbundle $T_{X/S}\subset T_{X/S}^{\CC}$ satisfying $T_{X/S}^{\CC}=T_{X/S}\oplus \widebar{T_{X/S}}$ and $[T_{X/S},T_{X/S}]\subset T_{X/S}$ (by abuse of notation, $T_{X/S}$ denotes $C^\infty(X,T_{X/S})$).

Alternatively, this structure is defined by a fiberwise endomorphism $J_{X/S} \in C^{\infty}(X, \End(T_{X/S}^\RR))$ satisfying $J_{X/S}^2=-\id$ and the integrability condition that the relative Nijenhuis tensor vanishes. The relationship is given by identifying $T_{X/S}$ with the $+\I$-eigenbundle of $J_{X/S}$ acting on $T_{X/S}^\CC$. Conversely, given the splitting $T_{X/S}^{\CC}=T_{X/S}\oplus \widebar{T_{X/S}}$, the operator $J_{X/S}$ is recovered by defining it to be multiplication by $\I$ on $T_{X/S}$ and by $-\I$ on $\widebar{T_{X/S}}$.

\begin{definition}\label{def: cpx fiber bundle}
A \emph{complex fiber bundle} over $S$ is a pair $(f,T_{X/S})$, where $f: X\to S$ is a smooth fiber bundle over $S$ and $T_{X/S}$ is an integrable complex structure on $T_{X/S}^{\RR}$. We define $f_* T_{X/S}$ to be the sheaf on $S$ whose sections over an open subset $U \subset S$ are the smooth sections of $T_{X/S}|_{f^{-1}(U)}$ that are fiberwise holomorphic.
\end{definition}

We may regard a complex fiber bundle as a differentiable family of complex manifolds.  It naturally generalizes the notion of a complex vector bundle.
\begin{lemma}\label{lem:isotrivial_cx_fib_bun}
Let $f: X \to S$ be a smooth fiber bundle with typical fiber $Y$, where $Y$ is a complex manifold.
Suppose there exists an open covering $\mathcal{U} = \{U_a\}$ of $S$ and smooth trivializations $\Phi_a: f^{-1}(U_a) \to U_a \times Y$ such that the transition functions $g_{ab}(s):=\Phi_{a,s} \comp \Phi_{b,s}^{-1}$
take values in $\mathrm{Aut}(Y)$ (the group of biholomorphisms of $Y$) for all $s\in U_{ab}:=U_a\cap U_b$. Then, $X$ admits a unique global fiberwise complex structure $J_{X/S}$ such that each fiber $X_s$ is a complex manifold biholomorphic to $Y$ via the maps $\Phi_{a,s}$. We call such an atlas a \emph{compatible atlas}.
\end{lemma}

\begin{proof}
Let $J_Y$ be the complex structure on $Y$. On each chart $f^{-1}(U_a)$, we define $J_a|_{X_s} = \Phi_{a,s}^* J_Y$.
On the overlap $U_{ab}$, we have $\Phi_{a,s} = g_{ab}(s) \comp \Phi_{b,s}$. Thus,
\[
J_a|_{X_s}  = (g_{ab}(s) \comp \Phi_{b,s})^* J_Y = \Phi_{b,s}^* g_{ab}(s)^* J_Y=\Phi_{b,s}^* J_Y=J_b|_{X_s}.
\]
These local structures glue to a global fiberwise complex structure $J_{X/S}$.
\end{proof}
The complex fiber bundle in the above lemma is \emph{isotrivial}, which means it is locally trivial as a family of complex manifolds with fibers biholomorphic to $Y$. By definition, each isotrivial complex fiber bundle admits a compatible atlas with transition maps taking values in $\mathrm{Aut}(Y)$. Note that every complex vector bundle is an isotrivial complex fiber bundle. An isotrivial complex fiber bundle is a \emph{holomorphic fiber bundle} if and only if the transition maps $g_{ab}:U_{ab}\times Y\to Y$ are holomorphic for some compatible atlas, and such an atlas is called a \emph{holomorphic atlas}.

It is well known that a holomorphic structure on a complex vector bundle over $S$ is equivalent to an integrable $\dbar$-operator on the bundle. Motivated by this fact, we make the following definition.

\begin{definition}\label{def:dbar_operator}
A \emph{$\dbar$-operator} on the complex fiber bundle $(f,T_{X/S})$ is a smooth bundle morphism
\[
\dbar: f^*\widebar{TS}\to \frac{TX^{\CC}}{\widebar{T_{X/S}}}
\]
whose image under the projection $\frac{TX^{\CC}}{\widebar{T_{X/S}}}\to f^*TS^{\CC}$ is contained in $f^*\widebar{TS}$, and such that its composition with the projection to $f^*\widebar{TS}$ is the identity on $f^*\widebar{TS}$. Note that there is a short exact sequence of complex vector bundles associated to $(f,T_{X/S})$:
\begin{equation}\label{eq:quot_exact_seq}
	0\to T_{X/S}\to \frac{TX^\CC}{\widebar{T_{X/S}}}\stackrel{q}{\longrightarrow} f^*TS^\CC\to 0.
\end{equation}
A $\dbar$-operator is just a smooth splitting of \eqref{eq:quot_exact_seq} restricted to $f^*\widebar{TS}$.
\end{definition}

\begin{lemma}\label{lem:dbar_affine_space}
The space of $\dbar$-operators on $(f,T_{X/S})$ is an affine space modeled on $C^{\infty}(X,f^\ast \widebar{T^* S}\otimes T_{X/S})$.
\end{lemma}

\begin{definition}\label{def:alm_cx_str}
An \emph{almost complex structure} on the complex fiber bundle $(f,T_{X/S})$ is a complex subbundle $TX\subset TX^{\CC}$ which contains $T_{X/S}$ and satisfies
\begin{enumerate}
    \item  $TX\oplus \widebar{TX}= TX^{\CC}$;
\item  The image of $TX$ under the composite $TX\hookrightarrow TX^{\CC} \to f^*TS^{\CC}$ is contained in $f^*TS$.
\end{enumerate}
If $TX$ further satisfies the integrability condition $[TX,TX]\subset TX$, we call it a \emph{complex structure}.
\end{definition}

By the Newlander-Nirenberg theorem, a complex structure on $(f,T_{X/S})$ gives rise to a holomorphic fibration structure on $f$ whose holomorphic relative tangent bundle equals $T_{X/S}$, and vice versa.

\begin{proposition}\label{prop:dbar_complex_structure}
	Let $(f, T_{X/S})$ be a complex fiber bundle over $S$. There is a canonical bijection between almost complex structures on $(f, T_{X/S})$ and $\dbar$-operators. Furthermore, an almost complex structure $TX$ is integrable (i.e., a complex structure) if and only if the corresponding $\dbar$-operator is \emph{integrable}, which means the inverse image of $\dbar(f^*\widebar{TS})\subset \frac{TX^{\CC}}{\widebar{T_{X/S}}}$ in $TX^{\CC}$ is closed under Lie bracket.
\end{proposition}
\begin{proof}
Let $Q = TX^{\CC}/\widebar{T_{X/S}}$. Then we have the natural projection $p: TX^{\CC} \to Q$, such that $\D f^{\CC} = q\comp p$. Let $TX$ be an almost complex structure on $(f, T_{X/S})$. Consider the restriction $\bar{\pi} = \D f^{\CC}|_{\widebar{TX}}: \widebar{TX} \to f^*\widebar{TS}$. Its kernel is $\widebar{TX} \cap T_{X/S}^{\CC}$. We claim this intersection is $\widebar{T_{X/S}}$. Clearly $\widebar{T_{X/S}}$ is contained in the intersection. Conversely, let $v \in \widebar{TX} \cap T_{X/S}^{\CC}$. We decompose $v=v^{1,0}+v^{0,1}$ with $v^{1,0}\in T_{X/S}, v^{0,1}\in \widebar{T_{X/S}}$. Since $v^{0,1}\in \widebar{T_{X/S}}\subset \widebar{TX}$ and $v\in \widebar{TX}$, we have $v^{1,0}=v-v^{0,1} \in \widebar{TX}$. But $v^{1,0}\in T_{X/S}\subset TX$. Thus $v^{1,0}\in TX\cap \widebar{TX}=\{0\}$. So $v\in \widebar{T_{X/S}}$. Therefore, we have a short exact sequence
\begin{equation}\label{eq:barTX_SES}
	0 \to \widebar{T_{X/S}} \to \widebar{TX} \xrightarrow{\bar{\pi}} f^*\widebar{TS} \to 0.
\end{equation}
For $u\in f^*\widebar{TS}$, let $w\in \widebar{TX}$ be any lift such that $\bar{\pi}(w)=u$. Define $\dbar(u) := p(w)$. This is well-defined because if $w'$ is another lift, $w-w'\in \ker(\bar{\pi})=\widebar{T_{X/S}}=\ker(p)$. It satisfies the definition of a $\dbar$-operator since $q(\dbar(u)) = q(p(w)) = \D f^{\CC}(w) = u$.

Conversely, given a $\dbar$-operator, we define $\widebar{TX} = p^{-1}(\im(\dbar))$. Let $TX = \overline{\widebar{TX}}$. $\widebar{T_{X/S}}=\ker(p) \subset \widebar{TX}$. By conjugation, $T_{X/S}\subset TX$. $\D f^{\CC}(\widebar{TX}) = q(p(\widebar{TX})) = q(\im(\dbar))$. Since $\dbar$ is a splitting over $f^*\widebar{TS}$, this is $f^*\widebar{TS}$. By conjugation, $\D f^{\CC}(TX)\subset f^*TS$. Let $v\in TX\cap \widebar{TX}$, $\D f^{\CC}(v) \in f^*TS \cap f^*\widebar{TS} = \{0\}$, so $v\in T_{X/S}^{\CC}$.
Suppose $w\in\widebar{TX}\cap T_{X/S}^{\CC} $. The condition $w\in T_{X/S}^{\CC}$ implies $\D f^{\CC}(w)=0$, while $w\in \widebar{TX}$ means $p(w)\in \im(\dbar)$. Thus $q(p(w))=\D f^{\CC}(w)=0$. Since $q|_{\im(\dbar)}$ is an isomorphism onto $f^*\widebar{TS}$, we must have $p(w)=0$. Thus $w\in \ker(p)=\widebar{T_{X/S}}$.
So $\widebar{TX}\cap T_{X/S}^{\CC} = \widebar{T_{X/S}}$.
Therefore, $v \in T_{X/S}^{\CC} \cap TX \cap \widebar{TX} = T_{X/S} \cap \widebar{T_{X/S}} = \{0\}$. By rank reasons, $TX^{\CC}=TX\oplus \widebar{TX}$.

These constructions are mutually inverse. The final statement regarding integrability follows, as the integrability of $TX$ is precisely the definition of the integrability of the corresponding $\dbar$.
\end{proof}

When $(f, T_{X/S})$ is a complex vector bundle, a $\dbar$-operator in the above sense is much more general than that in the classical sense. For example, while the integrability of a classical $\dbar$-operator can be measured through the vanishing of a curvature tensor in $A^{0,2}(S,f_*T_{X/S})$, it is not the case for an arbitrary $\dbar$-operator. We shall now discuss a condition for $\dbar$ so that its integrability can be measured through the vanishing of a curvature tensor, generalizing the complex vector bundle situation (see Example \ref{ex:vb_conn_dbar}).
\begin{lemma}\label{lem:integrability_dbar}
Let $(f, T_{X/S})$ be a complex fiber bundle on $S$. Consider a $\dbar$-operator on $(f, T_{X/S})$ satisfying the following \emph{lifting condition}: There is a lifting $\nabla^{0,1}: f^*\widebar{TS}\to TX^\CC$ of $\dbar$ such that
\begin{equation}\label{eq:lift_cond}
	[\widebar{T_{X/S}},\im(\nabla^{0,1})]\subset \widebar{TX}.
\end{equation}
Then there is a smooth bundle map $F^{0,2}: f^*\wedge^2 \widebar{TS}\to \frac{TX^\CC}{\widebar {TX}}$ such that $F^{0,2}=0$ if and only if $\dbar$ is integrable. If $\dbar$ is integrable, then the lifting condition is automatically satisfied.
\end{lemma}
Since $[\widebar{T_{X/S}},\widebar{T_{X/S}}]\subset \widebar{T_{X/S}}\subset \widebar{TX}$, the above condition is independent of the choice of such a lifting.
\begin{proof}
By the definition of $\widebar{TX}$, we have the short exact sequence \eqref{eq:barTX_SES}. A lifting of $\dbar$ is a splitting of the projection $\widebar{TX}\to f^*\widebar{TS}$. Write $\widebar{TX}=\widebar{T_{X/S}}\oplus \im(\nabla^{0,1})$. By the lifting condition, $[\widebar{TX},\widebar{TX}]\subset \widebar{TX}$ holds if and only if $[\im(\nabla^{0,1}),\textrm{im}(\nabla^{0,1})]\subset \widebar{TX}$ holds. Equivalently, the composite
\[f^*\wedge^2\widebar{TS}\xrightarrow{\frac12[\nabla^{0,1},\nabla^{0,1}]} TX^\CC\to \frac{TX^\CC}{\widebar{TX}}, \quad a\wedge b \mapsto [[\nabla^{0,1}(a),\nabla^{0,1}(b)]],\]
is zero. Suppose the $\dbar$-operator is integrable. By Proposition \ref{prop:dbar_complex_structure}, this induces a complex structure on $X$ such that $f: X\to S$ is a holomorphic fibration. Any smooth splitting $\nabla^{0,1}$ of
\[0\to \widebar{T_{X/S}}\to \widebar{TX}\to f^\ast \widebar{TS}\to 0\]
is a lifting of $\dbar$. Moreover, $[\widebar{T_{X/S}},\im(\nabla^{0,1})]\subset [\widebar{TX},\widebar{TX}]\subset\widebar{TX}$. The lifting condition is satisfied.
\end{proof}
Any lifting $\nabla^{0,1}$ of $\dbar$ is a $(0,1)$-connection to be defined below.

\subsection{Connections}
Recall that a \emph{connection} $\nabla^\RR$ on a smooth fiber bundle $f:X\to S$ is a smooth splitting of the following exact sequence of real tangent bundles
\begin{equation}\label{RealTangentSeq_eq}
	0\to T_{X/S}^{\RR}\to TX^{\RR}\xrightarrow{\D f} f^\ast TS^{\RR}\to 0.
\end{equation}
In other words, $\nabla^\RR$ is a smooth bundle monomorphism $\nabla^\RR: f^\ast TS^{\RR}\to TX^{\RR}$ such that $\D f\comp \nabla^\RR=\id_{f^\ast TS^{\RR}}$.

If $S$ is a complex manifold, a connection $\nabla^\RR$ induces a splitting of the complexification \eqref{ComplexTangentSeq_General_eq}, denoted by $\nabla$. We have $\nabla=\nabla^{1,0}+\nabla^{0,1}$, where $\nabla^{1,0}:=\nabla|_{f^\ast TS}$ and $\nabla^{0,1}:=\nabla|_{f^\ast \widebar{TS}}$. Then $\nabla^{0,1}=\widebar{\nabla^{1,0}}$. More generally, we consider a \emph{complex connection}, which is a splitting $\nabla=\nabla^{1,0}+\nabla^{0,1}$ of \eqref{ComplexTangentSeq_General_eq}, which may not be the complexification of a real connection $\nabla^\RR$. We call $\nabla^{1,0}$ a $(1,0)$-connection and $\nabla^{0,1}$ a $(0,1)$-connection.

Suppose $S$ is an $n$-dimensional complex manifold and $(f,T_{X/S})$ is a complex fiber bundle whose typical fiber is an $m$-dimensional complex manifold (complex structures may vary). We can choose local coordinates $(s^1,\ldots,s^n,z^1,\ldots,z^m)$ adapted to the complex fiber bundle structure. This means $\{s^i\}$ are holomorphic coordinates on $S$ and $\{z^\alpha\}$ are holomorphic coordinates along the fibers. Locally $f$ is the projection $(s,z)\mapsto s$. In these coordinates, we may express $\nabla^{1,0}$ and $\nabla^{0,1}$ as
\begin{align}
	&\nabla^{1,0}:\partial_{i}\mapsto  H_i = \partial_i+\Gamma^{\alpha}_i\partial_{\alpha}+\Gamma^{\bar{\beta}}_i\partial_{\bar{\beta}},\label{local10conn_eq}\\
	&\nabla^{0,1}:\partial_{\bar{j}}\mapsto H_{\bar{j}} = \partial_{\bar{j}}+\Gamma^{\gamma}_{\bar{j}}\partial_{\gamma}+\Gamma^{\bar{\delta}}_{\bar{j}}\partial_{\bar{\delta}},\label{local01conn_eq}
\end{align}
where $\partial_i=\frac{\partial}{\partial s^i}, \partial_\alpha=\frac{\partial}{\partial z^\alpha}$ and so on, and we used the Einstein summation convention. The coefficients $\Gamma$ are smooth functions in $(s,z)$. Consider another such coordinate system $(s', z')$. The transition functions $s'^i=s'^i(s)$ are holomorphic, $z'^\alpha=z'^\alpha(s,\bar{s},z)$ are smooth in the base coordinates and holomorphic in the fiber coordinates $z$. Then we have the transformation laws (cf. \cite[\S3.2]{sardanashvili2013advanced})
\begin{align}
	&{\Gamma'}_i^\alpha=\frac{\partial s^j}{\partial s'^i}(\partial_j+\Gamma_j^\beta \partial_\beta)z'^\alpha,\quad {\Gamma'}_i^{\bar{\beta}}=\frac{\partial s^j}{\partial s'^i}(\partial_j+\Gamma_j^{\bar{\alpha}}\partial_{\bar{\alpha}}) \bar{z}'^{\beta},\label{TransLaw10_eq}\\
	&	{\Gamma'}_{\bar{j}}^{\bar{\delta}}=\frac{\partial \bar{s}^i}{\partial \bar{s}'^j}(\partial_{\bar{i}}+\Gamma_{\bar{i}}^{\bar{\gamma}}\partial_{\bar{\gamma}}) \bar{z}'^{\delta},\quad{\Gamma'}_{\bar{j}}^\gamma=\frac{\partial \bar{s}^i}{\partial \bar{s}'^j}(\partial_{\bar{i}}+\Gamma_{\bar{i}}^\delta \partial_\delta) z'^\gamma.\label{TransLaw01_eq}
\end{align}

Any $\nabla^{0,1}$ induces a $\dbar$-operator via the composition
\begin{equation}\label{InducedDbar_eq}
	f^\ast \widebar{TS}\xrightarrow{\nabla^{0,1}} TX^{\CC}\to \frac{TX^{\CC}}{\widebar{T_{X/S}}}.
\end{equation}
Locally, the induced $\dbar$-operator is determined by the coefficients $\Gamma_{\bar{j}}^\gamma$. Conversely, any $\dbar$-operator is induced by some $(0,1)$-connection $\nabla^{0,1}$.
\begin{lemma}\label{lem:matrix_form_J}
Let $\nabla^\RR$ be a connection on a complex fiber bundle $(f,T_{X/S})$, which induces a splitting $TX^\RR = T_{X/S}^\RR \oplus H^\RR$ with $H^\RR \cong f^* TS^\RR$. The $(0,1)$-part $\nabla^{0,1}$ of its complexification $\nabla$ induces a $\dbar$-operator $\dbar_\nabla$. Let $\dbar$ be any $\dbar$-operator on $(f, T_{X/S})$, and let $J$ be the corresponding almost complex structure on $X$ (Proposition \ref{prop:dbar_complex_structure}). With respect to the splitting defined by $\nabla^\RR$, $J$ can be written in block matrix form as
\[
J = \begin{pmatrix}
J_{X/S} & \phi \\
0 & f^* J_S
\end{pmatrix},
\]
where $\phi \in C^\infty(X, f^* T^*S^\RR\otimes T_{X/S}^\RR)$. Let $\widebar{\Theta} := \dbar_\nabla - \dbar \in C^\infty(X, f^* \widebar{T^*S}\otimes T_{X/S})$. Then we have
\begin{equation}\label{eq:phi_from_dbar}
\phi(v) = 4 \Re \bigl( \I \widebar{\Theta}(v^{0,1}) \bigr), \quad \text{for all } v \in f^* TS^\RR.
\end{equation}
Here, $v^{0,1} = \frac{1}{2}(v + \I f^*J_S v)$ is the $(0,1)$-component of $v$ in $f^*\widebar{TS}$.
\end{lemma}

\begin{proof}
The block triangular form follows from the compatibility of $J$ with the complex fiber bundle structure. Let $v \in f^* TS^\RR$ and let $H_{v^{0,1}}$ be the horizontal lift of $v^{0,1}$ with respect to $\nabla^{0,1}$. By the definition of $\widebar{\Theta}$, we have
\[J(H_{v^{0,1}}-\widebar{\Theta}(v^{0,1}))=-\I(H_{v^{0,1}}-\widebar{\Theta}(v^{0,1})).\]
Since $J(\widebar{\Theta}(v^{0,1}))=\I \widebar{\Theta}(v^{0,1})$, we have $J(H_{v^{0,1}}) = -\I H_{v^{0,1}} + 2\I \widebar{\Theta}(v^{0,1})$. The real horizontal lift decomposes as $H_v=H_{v^{0,1}}+\widebar{H_{v^{0,1}}}$. Since $J$ is a real operator,
\[
J(H_v) = J(H_{v^{0,1}}) + \widebar{J(H_{v^{0,1}})} = \bigl(-\I H_{v^{0,1}} + 2\I \widebar{\Theta}(v^{0,1})\bigr) + \bigl(\I \widebar{H_{v^{0,1}}} - 2\I \overline{\widebar{\Theta}(v^{0,1})}\bigr).
\]
Therefore, $\phi(v) =2\I \widebar{\Theta}(v^{0,1}) + \overline{2\I \widebar{\Theta}(v^{0,1})} = 4 \Re(\I \widebar{\Theta}(v^{0,1}))$.
\end{proof}

\begin{lemma}\label{lem:LC_coordinates}
A $\dbar$-operator satisfies the lifting condition if and only if for any (equivalently, some) $(0,1)$-connection $\nabla^{0,1}$ inducing it, its coefficients in any adapted local coordinate system satisfy
\begin{equation}\label{eq:LC}
\partial_{\bar{\beta}}\Gamma_{\bar{j}}^\gamma = 0
\end{equation}
for all $\beta, \gamma, j$. This condition is independent of the choice of adapted coordinates. We call such $\nabla^{0,1}$ \emph{mixed relatively holomorphic}.
\end{lemma}
\begin{proof}
 $\dbar$ satisfies the lifting condition if and only if $[\widebar{T_{X/S}},\im(\nabla^{0,1})]\subset \widebar{TX}$. We compute
\[
[\partial_{\bar{\beta}}, H_{\bar{j}}] = [\partial_{\bar{\beta}}, \partial_{\bar{j}}+\Gamma^{\gamma}_{\bar{j}}\partial_{\gamma}+\Gamma^{\bar{\delta}}_{\bar{j}}\partial_{\bar{\delta}}] = (\partial_{\bar{\beta}}\Gamma^{\gamma}_{\bar{j}})\partial_{\gamma} + (\partial_{\bar{\beta}}\Gamma^{\bar{\delta}}_{\bar{j}})\partial_{\bar{\delta}}.
\]
This vector belongs to $\widebar{TX}$ if and only if $\partial_{\bar{\beta}}\Gamma_{\bar{j}}^\gamma=0$.

To show the condition \eqref{eq:LC} is well-defined, we check the transformation law using \eqref{TransLaw01_eq}.
\begin{align*}
\partial_{\bar{\beta}'} {\Gamma'}_{\bar{j}}^{\gamma} &= \frac{\partial \bar{s}^k}{\partial \bar{s}'^j} \partial_{\bar{\beta}'} \Bigl( \frac{\partial z'^\gamma}{\partial \bar{s}^k} + \frac{\partial z'^\gamma}{\partial z^\alpha} \Gamma_{\bar{k}}^\alpha \Bigr)\\
&= \frac{\partial \bar{s}^k}{\partial \bar{s}'^j} \Bigl( 0 + \frac{\partial z'^\gamma}{\partial z^\alpha} (\partial_{\bar{\beta}'} \Gamma_{\bar{k}}^\alpha) + 0 \Bigr)\\
&=\frac{\partial \bar{s}^k}{\partial \bar{s}'^j} \frac{\partial z'^\gamma}{\partial z^\alpha} \frac{\partial \bar{z}^\delta}{\partial \bar{z}'^\beta} (\partial_{\bar{\delta}} \Gamma_{\bar{k}}^\alpha)=0.
\end{align*}
Hence \eqref{eq:LC} is well-defined.
\end{proof}

The curvature of $\nabla^{0,1}$ is defined as $F^{0,2}_{\nabla^{0,1}}:=\frac12[\nabla^{0,1},\nabla^{0,1}]^{\mathrm{vert}}$, i.e., $$
F^{0,2}_{\nabla^{0,1}}(u,v)=[\nabla^{0,1}(u),\nabla^{0,1}(v)]-\nabla^{0,1}([u,v]),$$ where $u,v\in C^{\infty}(S,\widebar{TS})$. We have $F^{0,2}_{\nabla^{0,1}}\in C^{\infty}(X,f^*(\wedge^2\widebar{T^*S})\otimes T_{X/S}^\CC)$. Combining with Lemma \ref{lem:integrability_dbar}, we have the following result.
\begin{corollary}\label{cor:dbarint_via_conn_curv}
Let $\dbar$ be a $\dbar$-operator on a complex fiber bundle $(f, T_{X/S})$. Let $\nabla^{0,1}$ be any $(0,1)$-connection inducing $\dbar$. Then $\dbar$ is integrable if and only if $\nabla^{0,1}$ is mixed relatively holomorphic and the $(1,0)$-vertical part $F^{0,2}_{\dbar}$ (independent of the choice of such mixed relatively holomorphic $\nabla^{0,1}$) of the curvature of $\nabla^{0,1}$ vanishes.
\end{corollary}
The following example illustrates that a flat connection may not induce an integrable $\dbar$-operator.
\begin{example}\label{ex:flat_real_nonintegrable_dbar}
Let $\Delta_s=\{s\in\CC:|s|<1\}$ and equip $f:X=\Delta_s\times\CC_z\to\Delta_s$ with the standard fiberwise complex structure. Consider the smooth bundle diffeomorphism
\[
\Psi:\Delta_s\times\CC_w\longrightarrow X,\qquad \Psi(s,w)=(s,w+\bar s\,\bar w).
\]
Its inverse is given by $w=(z-\bar s\,\bar z)/(1-|s|^2)$. Push the trivial real connection on $\Delta_s\times\CC_w$ forward by $\Psi$. The resulting real connection $\nabla^\RR$ is flat. In the adapted coordinates $(s,z)$, let $B:=(\bar z-sz)/(1-|s|^2)$. The horizontal lifts of its complexification are
\[
H_s=\partial_s+\overline{B}\,\partial_{\bar z},\qquad H_{\bar s}=\partial_{\bar s}+B\,\partial_z.
\]
However, $\partial_{\bar z}B=(1-|s|^2)^{-1}\ne0$. Thus $\nabla^{0,1}$ is not mixed relatively holomorphic, and the induced $\dbar$-operator is not integrable.
\end{example}

\begin{lemma}\label{lem:LC_space}
	Whenever nonempty, the space $\mathcal{A}^{0,1}_{\mathrm{LC}}$ of $\dbar$-operators satisfying the lifting condition is an affine space modeled on $A^{0,1}(S,f_\ast T_{X/S})$.
\end{lemma}
\begin{proof}
Fix $\dbar_0 \in \mathcal{A}^{0,1}_{\mathrm{LC}}$. By Lemma \ref{lem:dbar_affine_space}, any other $\dbar$-operator is of the form $\dbar = \dbar_0 + \Phi$, where $\Phi \in C^{\infty}(X,f^\ast \widebar{T^* S}\otimes T_{X/S})$. Let $\Gamma_{0,\bar{j}}^\gamma$ be the local coefficients of $\dbar_0$, and $\Phi_{\bar{j}}^\gamma$ be the coefficients of $\Phi$. The coefficients of $\dbar$ are $\Gamma_{\bar{j}}^\gamma = \Gamma_{0,\bar{j}}^\gamma + \Phi_{\bar{j}}^\gamma$. By Lemma \ref{lem:LC_coordinates}, we have $\partial_{\bar{\beta}}\Gamma_{0,\bar{j}}^\gamma=0$. Then $\dbar$ satisfies the lifting condition if and only if $\partial_{\bar{\beta}}\Gamma_{\bar{j}}^\gamma=0$, which is equivalent to $\partial_{\bar{\beta}}\Phi_{\bar{j}}^\gamma=0$. This condition means that $\Phi$ is holomorphic along the fibers, and can be identified with an element in $A^{0,1}(S,f_\ast T_{X/S})$.
\end{proof}

\begin{definition}
Let $\dbar$ be a $\dbar$-operator, which induces an almost complex structure $TX$. A $(1,0)$-connection $\nabla^{1,0}$ (resp. $(0,1)$-connection $\nabla^{0,1}$) is called \emph{pure} with respect to $\dbar$ (or simply pure, when $\dbar$ is clear from the context) if it is a splitting of the following exact sequence \eqref{HolTangentSeq_eq} (resp. \eqref{ConjTangentSeq_eq}). Equivalently, $\im(\nabla^{1,0})\subset TX$ (resp. $\im(\nabla^{0,1})\subset \widebar{TX}$). A complex connection $\nabla=\nabla^{1,0}+\nabla^{0,1}$ is called \emph{pure} if both $\nabla^{1,0}$ and $\nabla^{0,1}$ are pure.
\begin{align}
	0&\to  T_{X/S}\to TX\to f^\ast TS\to 0,\label{HolTangentSeq_eq}\\
	0&\to \widebar{T_{X/S}}\to \widebar{TX}\to f^\ast \widebar{TS}\to 0.\label{ConjTangentSeq_eq}
\end{align}
\end{definition}
Note that pure connections always exist. A $(0,1)$-connection $\nabla^{0,1}$ is pure if and only if it induces $\dbar$, and a $(1,0)$-connection $\nabla^{1,0}$ is pure if and only if $\widebar{\nabla^{1,0}}$ induces $\dbar$.

\begin{lemma}\label{lem:new10_conn}
Let $f: X \to S$ be a complex fiber bundle. Let $\nabla^{1,0}$ and $\nabla^{0,1}$ be connections locally given by \eqref{local10conn_eq} and \eqref{local01conn_eq} respectively. Then these connections induce a $(1,0)$-connection $\nabla_1^{1,0}$ given locally by
\[
  \partial_i\mapsto \partial_i + \Gamma^\alpha_i \partial_\alpha + \overline{\Gamma^\beta_{\bar{i}}} \partial_{\bar{\beta}}.
\]
In particular, $\nabla_1^{1,0}$ is pure with respect to the $\dbar$-operator induced by $\nabla^{0,1}$.
\end{lemma}

\begin{proof}
We verify the transformation law for $\nabla_1^{1,0}$. We have ${\Gamma_1}_i^\alpha=\Gamma_i^\alpha$ and ${\Gamma_1}_i^{\bar \beta}=\overline{\Gamma_{\bar i}^\beta}$. Then by \eqref{TransLaw10_eq} and \eqref{TransLaw01_eq},
\begin{align*}
	{\Gamma_1'}_i^\alpha&=\frac{\partial s^j}{\partial s'^i}(\partial_j+\Gamma_j^\beta \partial_\beta)z'^\alpha=\frac{\partial s^j}{\partial s'^i}(\partial_j+{\Gamma_1}_j^\beta \partial_\beta)z'^\alpha,\\
	{\Gamma_1'}_i^{\bar \beta}&=\overline{\frac{\partial \bar{s}^j}{\partial \bar{s}'^i}(\partial_{\bar{j}}+\Gamma_{\bar{j}}^\alpha \partial_\alpha)z'^\beta}=\frac{\partial s^j}{\partial s'^i}(\partial_j+{\Gamma_1}_j^{\bar{\alpha}}\partial_{\bar{\alpha}}) \bar{z}'^{\beta}.
\end{align*}
The connection $\nabla_1^{1,0}$ is globally well-defined.
\end{proof}

Suppose that the $\dbar$-operator $\dbar$ is integrable. Then $f:X\to S$ is a holomorphic fibration. We can choose adapted local holomorphic coordinates such that the transition functions $z'(s,z)$ are holomorphic in $s$. In this case, a connection $\nabla^{1,0}$ (resp. $\nabla^{0,1}$) is pure with respect to $\dbar$ if and only if $\Gamma_i^{\bar{\beta}}=0$ (resp. $\Gamma_{\bar{j}}^\gamma=0$) in these adapted holomorphic coordinates, which are well-defined by \eqref{TransLaw10_eq} and \eqref{TransLaw01_eq}.

\subsection{Curvature classes}
Let $(f, T_{X/S})$ be a complex fiber bundle. Let $\dbar$ be a $\dbar$-operator (not necessarily integrable) and $\nabla^{1,0}$ be a $(1,0)$-connection (not necessarily pure). We consider the operator $D:=\nabla^{1,0}+\dbar$. We define the $(1,1)$-part of its curvature, $F^{1,1}_{D}$, via the composite map
\begin{equation}\label{Curvature11Def_eq}
	f^*(\wedge^{1,1}TS)\xrightarrow{[\nabla^{1,0},\nabla^{0,1}]} TX^\CC \xrightarrow{\mathrm{pr}_{T_{X/S}}} T_{X/S},
\end{equation}
where $\mathrm{pr}_{T_{X/S}}$ is defined using the splitting $TX^\CC=f^*TS^\CC\oplus T_{X/S}^\CC$ given by the connection $\nabla^{1,0}+\nabla^{0,1}$ for any $(0,1)$-connection $\nabla^{0,1}$ lifting $\dbar$, and $T_{X/S}^\CC=T_{X/S}\oplus \widebar{T_{X/S}}$. $\mathrm{pr}_{T_{X/S}}$ is independent of the choice of $\nabla^{0,1}$. $F^{1,1}_{D}$ is well-defined if and only if it is independent of the choice of $\nabla^{0,1}$. Equivalently, $[\im(\nabla^{1,0}), \widebar{T_{X/S}}]\to T_{X/S}$ is zero. In this case, $F^{1,1}_{D}$ is the $(1,0)$-vertical part of the $(1,1)$-part $F_{\nabla}^{1,1}$ of the curvature of $\nabla=\nabla^{1,0}+\nabla^{0,1}$. Therefore $F_D^{1,1}\in C^\infty(X,f^*(\wedge^{1,1}T^*S)\otimes T_{X/S})$. Choose adapted local coordinates $(s^i, z^\alpha)$ such that
\begin{align}
	\nabla^{1,0}(\partial_i) &= H_i = \partial_i+\Gamma^{\alpha}_i\partial_{\alpha}+\Gamma^{\bar{\beta}}_i\partial_{\bar{\beta}},\label{eq:10conn}\\
	\dbar(\partial_{\bar{j}}) &= [\partial_{\bar{j}}+\Gamma^{\gamma}_{\bar{j}}\partial_{\gamma}] \mod \widebar{T_{X/S}}.\label{eq:dbarf}
\end{align}

\begin{lemma}\label{lem:curv_def_rel_hol}
	$F^{1,1}_{D}$ is well-defined if and only if $\nabla^{1,0}$ is \emph{relatively holomorphic}, i.e.,
	\begin{equation}\label{RelHolConn_eq}
		\partial_{\bar{\beta}}\Gamma_i^\alpha=0,
	\end{equation}
	for all $\alpha,\beta,i$ in adapted local coordinates. Locally,
	\begin{equation}\label{Curvature11Local_eq}
		F^{1,1}_{D} = \bigl( \partial_i \Gamma_{\bar{j}}^\alpha - \partial_{\bar{j}} \Gamma_i^\alpha  +  \Gamma_i^\beta \partial_\beta \Gamma_{\bar{j}}^\alpha - \Gamma_{\bar{j}}^\gamma \partial_\gamma \Gamma_i^\alpha  + \Gamma_i^{\bar{\beta}} \partial_{\bar{\beta}} \Gamma_{\bar{j}}^\alpha \bigr) \D s^i \wedge \D\bar{s}^j\otimes \partial_\alpha.
	\end{equation}
\end{lemma}
\begin{proof}
 First we show that the condition \eqref{RelHolConn_eq} is well-defined. In fact, by \eqref{TransLaw10_eq},
\begin{align*}
	\partial_{\bar{\beta}'} {\Gamma'}_{i}^{\alpha} &= \partial_{\bar{\beta}'} \Big( \frac{\partial s^j}{\partial s'^i} \left( \partial_j z'^\alpha + \Gamma_j^\gamma \partial_\gamma z'^\alpha \right) \Big)=\frac{\partial s^j}{\partial s'^i} \partial_{\bar{\beta}'} \left( \partial_j z'^\alpha + \Gamma_j^\gamma \partial_\gamma z'^\alpha \right)\\
	&=\frac{\partial s^j}{\partial s'^i} \big(\partial_{\bar{\beta}'}(\partial_j z'^\alpha) + (\partial_{\bar{\beta}'} \Gamma_j^\gamma) \partial_\gamma z'^\alpha + \Gamma_j^\gamma \partial_{\bar{\beta}'} (\partial_\gamma z'^\alpha) \big)\\
	&=\frac{\partial s^j}{\partial s'^i}(\partial_{\bar{\beta}'} \Gamma_j^\gamma) \partial_\gamma z'^\alpha=0.
\end{align*}

For $H_i=\nabla^{1,0}\partial_i=\partial_i+\Gamma_i^\alpha\partial_\alpha+\Gamma_i^{\bar{\beta}}\partial_{\bar{\beta}}\in\im(\nabla^{1,0})$ and  $\partial_{\bar{\gamma}} \in \widebar{T_{X/S}}$, we have
\[[\partial_{\bar{\gamma}}, H_i] = (\partial_{\bar{\gamma}}\Gamma_i^\alpha) \partial_\alpha+(\partial_{\bar{\gamma}}\Gamma_i^{\bar{\beta}}) \partial_{\bar{\beta}}.\] So $F^{1,1}_{D}$ is well-defined if and only if all $\partial_{\bar{\gamma}}\Gamma_i^\alpha$ vanish. \eqref{Curvature11Local_eq} follows from
\[ [\partial_i+\Gamma_i^\alpha\partial_\alpha+\Gamma_i^{\bar{\beta}}\partial_{\bar{\beta}}, \partial_{\bar{j}}+\Gamma_{\bar{j}}^\gamma\partial_\gamma]=\bigl( \partial_i \Gamma_{\bar{j}}^\alpha - \partial_{\bar{j}} \Gamma_i^\alpha  +  \Gamma_i^\beta \partial_\beta \Gamma_{\bar{j}}^\alpha - \Gamma_{\bar{j}}^\gamma \partial_\gamma \Gamma_i^\alpha  + \Gamma_i^{\bar{\beta}} \partial_{\bar{\beta}} \Gamma_{\bar{j}}^\alpha \bigr) \partial_\alpha\mod \widebar{T_{X/S}}.\qedhere\]
\end{proof}
If $\dbar$ satisfies the lifting condition, then by \eqref{Curvature11Local_eq}, $F_{D}^{1,1}$ is independent of $\Gamma_i^{\bar{\beta}}$. This motivates the following definition.
\begin{definition}\label{def:almost_connection}
	An \emph{almost connection} $\partial$ on a complex fiber bundle $(f,T_{X/S})$ is a smooth splitting of \eqref{eq:quot_exact_seq} restricted to $f^*TS$.
\end{definition}
Similar to the $\dbar$-operator case, any $(1,0)$-connection $\nabla^{1,0}$ induces an almost connection via the composition
\begin{equation}\label{InducedAC_eq}
	f^\ast TS\xrightarrow{\nabla^{1,0}} TX^{\CC}\to \frac{TX^{\CC}}{\widebar{T_{X/S}}}.
\end{equation}
Conversely, any almost connection is induced by a $(1,0)$-connection in this way.
 If $\dbar$ satisfies the lifting condition, then $F_D^{1,1}$ only depends on $\hat{D}:=\partial+\dbar$ and we may write it as $F_{\hat{D}}^{1,1}$. Similar to \eqref{eq:lift_cond}, we can also define the lifting condition for $\partial$ by $[\widebar{T_{X/S}},\im(\nabla^{1,0})]\subset \widebar{T_{X/S}}\oplus \im(\nabla^{1,0})$. Then $\partial$ satisfies the lifting condition if and only if every $\nabla^{1,0}$ inducing $\partial$ is relatively holomorphic. If $\nabla^{1,0}$ is relatively holomorphic, the $(1,0)$-vertical part of the curvature $F_{\nabla^{1,0}}^{2,0}$ of $\nabla^{1,0}$ only depends on $\partial$, and we denote it by $F_{\partial}^{2,0}$. Note that a $(1,0)$-connection $\nabla^{1,0}$ is equivalent to an operator $\hat{D}=\partial+\dbar$ since $\nabla^{1,0}$ induces $\partial$ and $\widebar{\nabla^{1,0}}$ induces $\dbar$, and conversely $\hat{D}=\partial+\dbar$ determines $\nabla^{1,0}$ by Lemma \ref{lem:new10_conn}.

\begin{example}\label{ex:vb_conn_dbar}
Let $f: E\to S$ be a smooth complex vector bundle of rank $m$ over a complex manifold $S$ of dimension $n$. Let $\{e_\alpha\}$ be a local smooth frame for $E$. This induces adapted local coordinates $(s^i, z^\alpha)$ on $E$.

A linear connection $D$ on $E$ determines a connection $\nabla^\RR$, i.e., a splitting of \eqref{RealTangentSeq_eq}. $D$ is locally represented by its connection 1-form matrix $\omega=(\omega_\alpha^\beta)$ relative to the frame $\{e_\alpha\}$, $D e_\alpha = \omega_\alpha^\beta \otimes e_\beta$. We decompose $\omega = \omega^{1,0} + \omega^{0,1}$, and write
\[\omega^{1,0}_\beta{}^\alpha = A_{i\beta}^\alpha(s) \D s^i, \quad
\omega^{0,1}_\beta{}^\alpha = B_{\bar{j}\beta}^\alpha(s) \D \bar{s}^j.\]

Let $\nabla$ be the complexification of $\nabla^{\RR}$. It decomposes as $\nabla=\nabla^{1,0}+\nabla^{0,1}$, with $\nabla^{0,1}=\widebar{\nabla^{1,0}}$. The horizontal lift $H_i = \nabla^{1,0}(\partial_i)$ satisfies $\D f(H_i)=\partial_i$, $(A_{j\beta}^\alpha(s)z^\beta\D s^j+\D z^\alpha)(H_i) = 0$, and $(\overline{B_{\bar{j}\beta}^\alpha(s)}\bar{z}^\beta\D s^j+\D \bar{z}^\alpha)(H_i) = 0$. Then
\begin{align*}
	\nabla^{1,0}(\partial_i) &= \partial_i - (A_{i\beta}^\alpha(s) z^\beta) \partial_\alpha-(\overline{B_{\bar{i}\beta}^\alpha(s)}\bar{z}^\beta)\partial_{\bar{\alpha}},\\
	\nabla^{0,1}(\partial_{\bar{j}}) &= \partial_{\bar{j}} - (B_{\bar{j}\beta}^\alpha(s) z^\beta) \partial_\alpha-(\overline{A_{j\beta}^\alpha(s)} \bar{z}^\beta) \partial_{\bar{\alpha}}.
\end{align*}
$\nabla^{1,0}$ is relatively holomorphic and $\nabla^{0,1}$ is mixed relatively holomorphic, which induce $\partial$ and $\dbar$, both satisfying the lifting condition. Note that although there is a decomposition $D=D^{1,0}+D^{0,1}$, $\nabla^{0,1}$ is not induced by $D^{0,1}$, and similarly for $\nabla^{1,0}$. If $E$ is a holomorphic vector bundle with holomorphic structure determined by $\dbar_E=D^{0,1}$ and $\{e_\alpha\}$ is a local holomorphic frame, then $\omega^{0,1}=0$. $D^{0,1}=\dbar_E$ does not canonically induce $\nabla^{0,1}$ in general. However, $\partial$ is determined by $D^{1,0}$ and $\dbar$ is determined by $D^{0,1}$.
\end{example}

\begin{lemma}\label{lem:RH_space}
Let $\dbar$ be a $\dbar$-operator on $(f,T_{X/S})$. Whenever nonempty, the space $\mathcal{A}_{\mathrm{RH},\dbar}$ of relatively holomorphic pure $(1,0)$-connections $\nabla^{1,0}$ is an affine space modeled on $A^{1,0}(S,f_*T_{X/S})$. $\mathcal{A}_{\mathrm{RH},\dbar}$ is isomorphic to the space $\mathcal{A}_{\mathrm{LC}}^{1,0}$ of almost connections satisfying the lifting condition.
\end{lemma}
\begin{proof}
Let $\nabla_1^{1,0},\nabla_2^{1,0}\in \mathcal{A}_{\mathrm{RH},\dbar}$. In adapted local coordinates, we write
\[\nabla_1^{1,0}(\partial_i)=\partial_i+{\Gamma_1}_i^\alpha\partial_\alpha+{\Gamma_1}_i^{\bar{\beta}}\partial_{\bar{\beta}},\quad \nabla_2^{1,0}(\partial_i)=\partial_i+{\Gamma_2}_i^\alpha\partial_\alpha+{\Gamma_2}_i^{\bar{\beta}}\partial_{\bar{\beta}}.\]
Since their conjugates induce the same $\dbar$, we have ${\Gamma_1}_i^{\bar{\beta}}={\Gamma_2}_i^{\bar{\beta}}$. Then
\begin{equation}\label{TransgressionForm_eq}
G_{1,2}:=\nabla_1^{1,0}-\nabla_2^{1,0}=({\Gamma_1}_i^\alpha-{\Gamma_2}_i^\alpha) \D s^i\otimes\partial_\alpha.
\end{equation}
We will prove that $G_{1,2}$ is a global section of $C^\infty(X, f^* T^*S \otimes T_{X/S})$ by showing it transforms correctly under coordinate changes $s' = s'(s), z' = z'(s,\bar{s},z)$. By \eqref{TransLaw10_eq},
\begin{align*}
({\Gamma'_1}_i^\alpha - {\Gamma'_2}_i^\alpha) \D {s'}^i\otimes \partial_{\alpha'}
&= \frac{\partial s^j}{\partial s'^i} \big( \big( \partial_j z'^\alpha + {\Gamma_1}_j^\beta \partial_\beta z'^\alpha \big) - \big( \partial_j z'^\alpha + {\Gamma_2}_j^\beta \partial_\beta z'^\alpha \big) \big)\D {s'}^i\otimes\partial_{\alpha'}  \\
&= \frac{\partial s^j}{\partial s'^i} \big( {\Gamma_1}_j^\beta - {\Gamma_2}_j^\beta \big) \partial_\beta z'^\alpha \D {s'}^i\otimes \partial_{\alpha'} \\
&= \frac{\partial s^j}{\partial s'^i} \big({\Gamma_1}_j^\beta - {\Gamma_2}_j^\beta\big)   \frac{\partial s'^i}{\partial s^p} \D s^p\otimes \partial_\beta\\
&= ({\Gamma_1}_p^\beta - {\Gamma_2}_p^\beta) \D s^p\otimes \partial_\beta.
\end{align*}
Therefore $G_{1,2}$ is globally defined. Since $\nabla_1^{1,0},\nabla_2^{1,0}$ are both relatively holomorphic, we have $G_{1,2}\in A^{1,0}(S,f_*T_{X/S})$. On the other hand, for any $\nabla_1^{1,0}$ and $G_{1,2}\in A^{1,0}(S,f_*T_{X/S})$, we have $\nabla_1^{1,0}+G_{1,2}\in \mathcal{A}_{\mathrm{RH},\dbar}$.

The map $\mathcal{A}_{\mathrm{RH},\dbar}\to \mathcal{A}_{\mathrm{LC}}^{1,0}$ is provided by \eqref{InducedAC_eq}. Its inverse is provided by Lemma \ref{lem:new10_conn}, which determines $\nabla^{1,0}$ using $\partial$ and $\dbar$.
\end{proof}
Combining Lemma \ref{lem:new10_conn}, Lemma \ref{lem:LC_space}, and Lemma \ref{lem:RH_space}, we obtain the following result.
\begin{corollary}\label{cor:RH_LC_space}
Whenever nonempty, the space $\mathcal{A}^{1,0}_{\mathrm{RH,LC}}$ of relatively holomorphic $(1,0)$-connections $\nabla^{1,0}$ whose conjugates induce $\dbar$-operators satisfying the lifting condition is an affine space modeled on $A^1(S,f_*T_{X/S})$. $\mathcal{A}^{1,0}_{\mathrm{RH,LC}}$ is nonempty if and only if there exists a relatively holomorphic $(1,0)$-connection $\nabla^{1,0}$ and a mixed relatively holomorphic $(0,1)$-connection $\nabla^{0,1}$. $\mathcal{A}^{1,0}_{\mathrm{RH,LC}}$ is isomorphic to the space $\mathcal{A}_{\mathrm{LC}}$ of $\partial+\dbar$, both satisfying the lifting conditions.
\end{corollary}

Under the condition \eqref{RelHolConn_eq}, one may directly use \eqref{Curvature11Local_eq} instead of \eqref{Curvature11Def_eq} to define $F^{1,1}_{D}$, since one may verify that \eqref{Curvature11Local_eq} transforms tensorially under a change of adapted coordinates $(s, z) \to (s', z')$, where $s'(s)$ is holomorphic, and $z'(s, \bar{s}, z)$ is holomorphic in $z$.

From now on we assume that $f:X\to S$ is a holomorphic fibration, $\dbar_f$ is the canonical integrable $\dbar$-operator, and $\nabla^{1,0}$ is relatively holomorphic and pure. The situation simplifies significantly. In adapted holomorphic coordinates, $\Gamma_i^{\bar{\beta}}=0$ in \eqref{eq:10conn} and $\Gamma_{\bar{j}}^\gamma=0$ in \eqref{eq:dbarf}. Then \eqref{Curvature11Local_eq} simplifies to
\begin{equation}\label{Curvature11Local_eq1}
	F^{1,1}_{D}=-\big(\partial_{\bar{j}}\Gamma_i^\alpha\big)  \D s^i\wedge \D \bar{s}^j\otimes \partial_\alpha.
\end{equation}
As suggested above, we can directly check that the local expressions \eqref{Curvature11Local_eq1} glue to a global element as follows. Consider the change of adapted holomorphic coordinates $(s, z) \to (s', z')$, then $s'(s)$ is holomorphic, and $z'(s, z)$ is holomorphic in $s$ and $z$.
\begin{lemma}
	$ \big(\partial_{\bar{j}'} \Gamma'{}_i^\alpha \big) \D s'^i \wedge \D\bar{s}'^j\otimes \partial_{\alpha'} =\big(\partial_{\bar{j}}\Gamma_i^\alpha\big) \D s^i\wedge \D \bar{s}^j\otimes \partial_\alpha$.
\end{lemma}
\begin{proof}
	The antiholomorphic derivative transforms as
	\[
	\partial_{\bar{j}'} = \frac{\partial}{\partial \bar{s}'^j} = \widebar{ \frac{\partial s^k}{\partial s'^j}} \partial_{\bar{k}}- \widebar{ \frac{\partial s^k}{\partial s'^j}} \frac{\partial \bar{z}'^\beta}{\partial \bar{s}^k} \frac{\partial \bar{z}^\gamma}{\partial \bar{z}'^\beta}\partial_{\bar{\gamma}}.
	\]
	Apply $\partial_{\bar{j}'}$ to the transformation formula \eqref{TransLaw10_eq} and use \eqref{RelHolConn_eq},
	\begin{align*}
	\partial_{\bar{j}'} \Gamma'{}_i^\alpha
	&= \partial_{\bar{j}'} \Big( \frac{\partial s^\ell}{\partial s'^i} \big( \partial_\ell z'^\alpha + \Gamma_\ell^\beta \partial_\beta z'^\alpha \big) \Big) \\
	&= \frac{\partial s^\ell}{\partial s'^i} \partial_{\bar{j}'} \big( \Gamma_\ell^\beta \partial_\beta z'^\alpha \big) = \frac{\partial s^\ell}{\partial s'^i} \big( \partial_{\bar{j}'} \Gamma_\ell^\beta \big) \partial_\beta z'^\alpha \\
	&= \frac{\partial s^\ell}{\partial s'^i} \widebar{ \frac{\partial s^k}{\partial s'^j} } \big( \partial_{\bar{k}} \Gamma_\ell^\beta \big) \partial_\beta z'^\alpha.
	\end{align*}
	The relative tangent vectors transform as
	\[
	\partial_{\alpha'} = \frac{\partial z^\gamma}{\partial z'^\alpha} \partial_\gamma, \quad \text{thus} \quad \partial_\beta z'^\alpha \partial_{\alpha'} = \partial_\beta z'^\alpha \frac{\partial z^\gamma}{\partial z'^\alpha} \partial_\gamma = \delta_\beta^\gamma \partial_\gamma = \partial_\beta.
	\]
	Therefore
	\[
	\big( \partial_{\bar{j}'} \Gamma'{}_i^\alpha \big) \partial_{\alpha'} = \frac{\partial s^\ell}{\partial s'^i} \widebar{ \frac{\partial s^k}{\partial s'^j} } \big( \partial_{\bar{k}} \Gamma_\ell^\beta \big) \partial_\beta.
	\]
	The differential form transforms as
	\[
	\D s'^i \wedge \D\bar{s}'^j = \frac{\partial s'^i}{\partial s^p} \widebar{ \frac{\partial s'^j}{\partial s^q} } \D s^p \wedge \D\bar{s}^q.
	\]
	Then we have
	\begin{align*}
	\big(\partial_{\bar{j}'} \Gamma'{}_i^\alpha\big)   \D s'^i \wedge \D\bar{s}'^j\otimes \partial_{\alpha'} &= \Big( \frac{\partial s'^i}{\partial s^p} \widebar{ \frac{\partial s'^j}{\partial s^q} } \D s^p \wedge \D\bar{s}^q \Big)\otimes \Big( \frac{\partial s^\ell}{\partial s'^i} \widebar{ \frac{\partial s^k}{\partial s'^j} } \big( \partial_{\bar{k}} \Gamma_\ell^\beta \big) \partial_\beta \Big) \\
	&=  \frac{\partial s^\ell}{\partial s'^i} \frac{\partial s'^i}{\partial s^p}  \widebar{ \frac{\partial s^k}{\partial s'^j} \frac{\partial s'^j}{\partial s^q} } \big( \partial_{\bar{k}} \Gamma_\ell^\beta \big)  \D s^p \wedge \D\bar{s}^q\otimes \partial_\beta\\
	&=\delta_p^\ell \delta_q^k\big( \partial_{\bar{k}} \Gamma_\ell^\beta \big)   \D s^p \wedge \D\bar{s}^q\otimes \partial_\beta\\
	&=\big( \partial_{\bar{q}} \Gamma_p^\beta \big) \D s^p \wedge \D\bar{s}^q\otimes \partial_\beta. \qedhere
	\end{align*}
\end{proof}

By \eqref{RelHolConn_eq} and \eqref{Curvature11Local_eq1}, $F^{1,1}_{D}$ can be regarded as an element in $A^{1,1}(S,f_\ast T_{X/S})$. We have
\[A^{1,1}(S,f_\ast T_{X/S})\subset C^\infty(X,f^*(\wedge^{1,1}T^*S)\otimes T_{X/S})\subset A^{0,1}(X, f^*T^* S \otimes T_{X/S}).\]
 $F_D^{1,1}$ is a cocycle in $A^{0,1}(X, f^*T^* S \otimes T_{X/S})$, since locally we can write
\begin{equation}\label{CurvLocalExact_eq}
	F^{1,1}_{D}=\dbar_X(\Gamma_i^\alpha \D s^i\otimes \partial_\alpha),
\end{equation}
which means $F^{1,1}_{D}$ is locally exact, hence closed. Therefore we obtain the class \[[F^{1,1}_{D}]\in H^1(X,f^*\Omega_S \otimes T_{X/S}),\] called the \emph{curvature class} of $D$.

\begin{proposition}\label{CurvClassIndep_prop}
	The class $[F^{1,1}_{D}]\in H^1(X,f^*\Omega_S \otimes T_{X/S})$ is independent of the choice of a relatively holomorphic pure $(1,0)$-connection $\nabla^{1,0}$.
\end{proposition}
\begin{proof}
	Let $\nabla_1^{1,0},\nabla_2^{1,0}$ be two relatively holomorphic pure $(1,0)$-connections and $F_1,F_2$ be the corresponding $(1,1)$-type curvatures. By \eqref{CurvLocalExact_eq}, locally we have $F_1-F_2=\dbar_X G_{1,2}$, where $G_{1,2}$ is given by \eqref{TransgressionForm_eq}, and $G_{1,2}\in A^{1,0}(S,f_\ast T_{X/S})\subset C^{\infty}(X,f^*T^*S\otimes T_{X/S})$. Therefore $F_1-F_2$ is $\dbar_X$-exact, and $[F_1]=[F_2]$ in $ H^1(X,f^*\Omega_S \otimes T_{X/S})$.
\end{proof}

In general, there may not exist a relatively holomorphic pure $(1,0)$-connection on $f$. In fact, let \[\rho_{s_0}:T_{s_0}S\to H^1(X_{s_0},TX_{s_0})\] be the Kodaira-Spencer map at $s_0\in S$, then $\rho_{s_0}(\partial_i)$ is represented by (see \cite[\S4]{schumacher12positivity})
\[(\bar{\partial} H_i)|_{X_{s_0}}=(\bar{\partial}( \partial_i+\Gamma_i^\alpha\partial_\alpha))|_{X_{s_0}}=(\partial_{\bar{\beta}}\Gamma_i^\alpha)  \D \bar{z}^\beta\otimes \partial_\alpha. \]
For completeness, we give a proof in the more general setting of a complex fiber bundle.
\begin{proposition}\label{prop:ks_conn_rep}
Let $f: X \to S$ be a complex fiber bundle. Let $\mathcal{U} = \{U_a\}$ be an open cover of $X$ with adapted coordinates $(s, z_a)$ such that the transition functions $z_a = f_{ab}(s,\bar{s},z_b)$ are holomorphic in $z_b$. Consider a $(1,0)$-connection $\nabla^{1,0}$ defined locally by
 $H_i|_{U_a}=\nabla^{1,0}(\partial_i)=\partial_i + \Gamma_{i,a}^\alpha \partial_{\alpha,a} + \Gamma_{i,a}^{\bar{\beta}}\partial_{\bar{\beta},a}$. Then the Kodaira-Spencer class $\rho_{s_0}(\partial_i) \in H^1(X_{s_0}, TX_{s_0})$ is represented in the Dolbeault cohomology by
\[ b_i = (\partial_{\bar{\beta},a} \Gamma_{i,a}^\alpha) \D\bar{z}_a^\beta\otimes \partial_{\alpha,a}\quad\text{ on }U_a|_{X_{s_0}}. \]
Similarly, $\rho_{s_0}(\partial_{\bar i})$ is represented by $b_{\bar{i}}= (\partial_{\bar{\beta},a} \Gamma_{\bar{i},a}^\alpha) \D\bar{z}_a^\beta\otimes \partial_{\alpha,a}$ on $U_a|_{X_{s_0}}$ for a $(0,1)$-connection $\nabla^{0,1}$ defined locally by $H_{\bar i}|_{U_a}=\nabla^{0,1}(\partial_{\bar i})=\partial_{\bar i} + \Gamma_{\bar{i},a}^\alpha \partial_{\alpha,a} + \Gamma_{\bar{i},a}^{\bar{\beta}}\partial_{\bar{\beta},a}$.
\end{proposition}

\begin{proof}
Consider the \v{C}ech-Dolbeault double complex $K^{p,q} = \check{C}^p(\mathcal{U}, \mathcal{A}^{0,q}(TX_{s_0}))$. By the differential geometric definition, the Kodaira-Spencer class $\rho(\partial_i)$ is represented by a \v{C}ech 1-cocycle $\theta \in \check{C}^1(\mathcal{U}, TX_{s_0})$, given by $\theta_{ab} := (\partial_i f_{ab}^\alpha(s,\bar{s}, z_b) )\partial_{\alpha,a}$ on the overlap $U_{ab}|_{X_{s_0}}$. By \eqref{TransLaw10_eq},
\begin{align*}
	\theta_{ab}&=(\partial_i f_{ab}^\alpha) \partial_{\alpha,a}=\Gamma_{i,a}^\alpha\partial_{\alpha,a}-\Gamma_{i,b}^\beta(\partial_{\beta,b}f_{ab}^\alpha) \partial_{\alpha,a}
	\\ &=\Gamma_{i,a}^\alpha\partial_{\alpha,a}-\Gamma_{i,b}^\alpha\partial_{\alpha,b}=:\sigma_a-\sigma_b.
\end{align*}
Therefore, we find a 0-cochain $\sigma = \{\sigma_a\} \in K^{0,0}$ such that $\delta\sigma=-\theta$. On the other hand,
\[\dbar \sigma_a=(\dbar_{X_{s_0}} \Gamma_{i,a}^\alpha)\partial_{\alpha,a}= (\partial_{\bar{\beta},a} \Gamma_{i,a}^\alpha) \D\bar{z}_a^\beta\otimes \partial_{\alpha,a}.\]
This implies $\dbar\sigma=b_i$ and $\D_{\mathrm{tot}}\sigma=-\theta+b_i$, so the Kodaira-Spencer class $\rho(\partial_i)$ is represented by $b_i$. The proof for $\rho_{s_0}(\partial_{\bar i})$ is similar.
\end{proof}

\begin{corollary}\label{cor:rel_hol_ks_vanish}
If the complex fiber bundle $f:X\to S$ admits a relatively holomorphic $(1,0)$-connection and a mixed relatively holomorphic $(0,1)$-connection, then the Kodaira-Spencer map for $f$ vanishes everywhere.
\end{corollary}

\begin{proposition}\label{prop:exist_fiberwise_holo_connection}
Let $f: X \to S$ be a complex fiber bundle with connected base $S$. If $f$ is isotrivial, then it admits a relatively holomorphic $(1,0)$-connection and a mixed relatively holomorphic $(0,1)$-connection. The converse is true if $f$ has compact fibers. If moreover $f$ is a holomorphic fibration, then it is a holomorphic fiber bundle.
\end{proposition}

\begin{proof}
Choose a compatible atlas $\mathcal{U} = \{(U_a, \Phi_a)\}$ with trivializations $\Phi_a: f^{-1}(U_a) \xrightarrow{\cong} U_a \times Y$. On each chart $U_a$, the trivialization induces a canonical flat connection compatible with the product structure. We define the local horizontal lifts of $\partial_i$ and $\partial_{\bar{i}}$ as $H_{i,a} := \Phi_a^* ( \partial_i)$, $H_{\bar{i},a} := \Phi_a^* ( \partial_{\bar i})$. Let $\{\chi_a\}$ be a smooth partition of unity on $S$ subordinate to the cover $\{U_a\}$. We define the global horizontal lifts by
\[
H_i := \sum_a f^*(\chi_a) H_{i,a}, \quad H_{\bar{i}} := \sum_a f^*(\chi_a) H_{\bar{i},a}.
\]
These define the connections $\nabla^{1,0}$ and $\nabla^{0,1}$.

We verify that $\nabla^{1,0}$ is relatively holomorphic. It suffices to check this in an arbitrary chart $(U_b, \Phi_b)$ with coordinates $(s, z_b)$.
The term $H_{i,b}$ is just $\partial_i$. Consider the term $H_{i,a}$ from a different chart.
The coordinate transformation is $z_a = g_{ab}(s,\bar{s},z_b)$. The vector field $H_{i,a}$ expressed in the $b$-coordinates is
\[
H_{i,a} =\partial_i + (\partial_i g_{ba}^\alpha(s,\bar{s},z_a)) \partial_{z_b^\alpha} + (\partial_i \bar{g}_{ba}^\beta(s,\bar{s}, z_a)) \partial_{\bar{z}_b^\beta}.
\]
Note that $\partial_i g_{ba}^\alpha(s,\bar{s},z_a)$ is holomorphic in $z_b$.
The global connection coefficient is a convex combination
\[ \Gamma_i^\alpha(s, z_b) = \sum_a \chi_a(s) (\partial_i g_{ba}^\alpha(s,\bar{s}, z_a)), \]
which is holomorphic in $z_b$.
The proof for $\nabla^{0,1}$ is identical.

Now suppose that $f$ has compact fibers and admits a relatively holomorphic $(1,0)$-connection $\nabla^{1,0}$ and a mixed relatively holomorphic $(0,1)$-connection $\nabla^{0,1}$. Using Lemma \ref{lem:new10_conn}, we obtain a $(1,0)$-connection $\nabla_1^{1,0}$ from $\nabla^{1,0}$ and $\nabla^{0,1}$. $\nabla_1^{1,0}$ is relatively holomorphic and $\widebar{\nabla_1^{1,0}}$ is mixed relatively holomorphic, so by Lemma \ref{lem:hol_parallel_transport} below, the parallel transport maps defined by $\nabla_1^{1,0}$ are biholomorphisms. Then $f$ admits a compatible atlas provided by radial parallel transports as in Lemma \ref{lem:hol_triv_via_conn}, which means $f$ is an isotrivial complex fiber bundle. If moreover $f$ is a holomorphic fibration, since its fibers are compact and biholomorphic, $f$ is a holomorphic fiber bundle by the Fischer–Grauert theorem.
\end{proof}
\begin{remark}\label{rmk:curv_class_S}
	Suppose $f:X\to S$ is a holomorphic fibration admitting a relatively holomorphic pure $(1,0)$-connection $\nabla^{1,0}$, and $f^{\mathrm{hol}}_*T_{X/S}$ is locally free of finite rank, where $f^{\mathrm{hol}}_*T_{X/S}$ is the direct image sheaf, which is a sheaf of $\mathcal{O}_S$-modules. Then $F^{1,1}_D$ defines a class $[F^{1,1}_{D}]_S\in H^{1,1}(S, f^{\mathrm{hol}}_*T_{X/S})\cong H^1(S, \Omega_S \otimes f^{\mathrm{hol}}_*T_{X/S})$ which is independent of $\nabla^{1,0}$. In particular, if $f$ is a proper holomorphic fibration which admits a relatively holomorphic pure $(1,0)$-connection, then its restriction over each connected component of $S$ is a holomorphic fiber bundle by Proposition \ref{prop:exist_fiberwise_holo_connection}, and $f^{\mathrm{hol}}_*T_{X/S}$ is locally free of finite rank.
\end{remark}

\subsection{Decomposition of $\dbar_X$}
This subsection was inspired by \cite[\S3.3]{rollenske20}. Let $\nabla$ be a pure complex connection on a holomorphic fibration $f:X\to S$. Then we have a splitting of the complexified tangent bundle
$$
TX^{\mathbb{C}}= TX\oplus \overline{TX}=T_{X/S}\oplus \overline{T_{X/S}} \oplus H^{1,0}\oplus H^{0,1} = T_{X/S}^{\mathbb{C}}\oplus H^{\CC}
$$
induced by $\nabla$, where $T_{X/S}^{\mathbb{C}}=T_{X/S}\oplus \overline{T_{X/S}}$ is the vertical bundle and $H^\CC=H^{1,0}\oplus H^{0,1}\cong f^\ast TS^\CC$ is the horizontal bundle defined by $\nabla$, with $H^{1,0}=\im(\nabla^{1,0})$ and $H^{0,1}=\im(\nabla^{0,1})$. So we have the decomposition
\[\wedge^{p,q}T^*X= \midoplus_{k+j=p,l+s=q}(\wedge^{k}(H^{1,0})^*)\wedge(\wedge^{l}(H^{0,1})^*)\wedge(\wedge^{j,s}T_{X/S}^*). \]
Correspondingly, we may decompose the space of $(p,q)$-forms on $X$ as
\[
A^{p,q}(X) = \midoplus_{k+j=p,l+s=q} A^{(k,l|j,s)}(X).
\]
A form in $A^{(k,l|j,s)}(X)$ is said to have \emph{bi-degree} $(k,l|j,s)$.

$\dbar_X$ does not generally preserve this bi-degree. Its decomposition reflects how the connection interacts with the complex structure. As in \cite[Lem.~3.11]{rollenske20}, the integrability of $T_{X/S}$ forces the decomposition of $\dbar_X$ to take a special form
\begin{equation}\label{DbarDecomp_eq}
\dbar_X = \dbar_{\mathrm{v}} + \dbar_{\mathrm{h}} + R_{\mathrm{A}_1} + R_{\mathrm{A}_2} + R_{\mathrm{KS}},
\end{equation}
where the components are characterized by their shifts in bi-degree, which are $(0,0|0,1)$, $(0,1|0,0)$, $(1,1|-1,0)$, $(0,2|0,-1)$, and $(1,0|-1,1)$ respectively. It was observed in loc. cit. that $R_{\mathrm{A}} = R_{\mathrm{A}_1} + R_{\mathrm{A}_2}$ is related to the Atiyah class in the situation of a principal bundle, and $R_{KS}$ is related to the Kodaira-Spencer class. We shall prove \eqref{DbarDecomp_eq} and expound these tensors in our context.

In local holomorphic coordinates $(s^i, z^\alpha)$, let $\nabla^{1,0}$ and $\nabla^{0,1}$ have coefficients $\Gamma_i^\alpha$ and $\Gamma_{\bar{j}}^{\bar{\beta}}$, respectively. The horizontal lifts are
\[
H_i = \partial_i + \Gamma_i^\alpha \partial_\alpha, \quad H_{\bar{j}} = \partial_{\bar{j}} + \Gamma_{\bar{j}}^{\bar{\beta}} \partial_{\bar{\beta}}.
\]
We introduce the adapted coframe. The horizontal forms are spanned by $\{\D s^i, \D\bar{s}^j\}$. The vertical forms are spanned by
\begin{equation}\label{eq:vertical_form}
	\phi^\alpha = \D z^\alpha - \Gamma_i^\alpha \D s^i, \quad
	\phi^{\bar{\beta}} = \D\bar{z}^\beta - \Gamma_{\bar{j}}^{\bar{\beta}} \D\bar{s}^j.
\end{equation}

The tensors $R_{\mathrm{KS}}$ and $R_{\mathrm{A}_1}$ are determined by the action of $\dbar_X$ on the vertical $(1,0)$-forms $\phi^\alpha$.
\begin{align*}
\dbar_X(\phi^\alpha) &= \dbar_X(\D z^\alpha - \Gamma_i^\alpha \D s^i) = -(\dbar_X \Gamma_i^\alpha) \wedge \D s^i \\
&= -\bigl( (\partial_{\bar{j}}\Gamma_i^\alpha) \D\bar{s}^j + (\partial_{\bar{\beta}}\Gamma_i^\alpha) \D\bar{z}^\beta \bigr) \wedge \D s^i.
\end{align*}
Substitute $\D\bar{z}^\beta = \phi^{\bar{\beta}} + \Gamma_{\bar{j}}^{\bar{\beta}} \D\bar{s}^j$ to express the result in the adapted coframe as
\begin{align}
\dbar_X(\phi^\alpha) &= (\partial_{\bar{j}}\Gamma_i^\alpha) \D s^i \wedge \D\bar{s}^j + (\partial_{\bar{\beta}}\Gamma_i^\alpha) \D s^i \wedge (\phi^{\bar{\beta}} + \Gamma_{\bar{j}}^{\bar{\beta}} \D\bar{s}^j) \nonumber \\
&= \underbrace{\bigl( \partial_{\bar{j}}\Gamma_i^\alpha + (\partial_{\bar{\beta}}\Gamma_i^\alpha) \Gamma_{\bar{j}}^{\bar{\beta}} \bigr) \D s^i \wedge \D\bar{s}^j}_{\text{Type } (1,1|0,0)} + \underbrace{(\partial_{\bar{\beta}}\Gamma_i^\alpha) \D s^i \wedge \phi^{\bar{\beta}}}_{\text{Type } (1,0|0,1)}. \label{DbarThetaLocal_eq}
\end{align}
By definition, the first term is $R_{\mathrm{A}_1}(\phi^\alpha)$ and the second term is $R_{\mathrm{KS}}(\phi^\alpha)$. Similarly, by computing $\dbar_X(\phi^{\bar{\beta}})$ we obtain $R_{\mathrm{A}_2}$. There are no other possibilities for bi-degree shifts and \eqref{DbarDecomp_eq} follows.

\subsubsection*{Relation to the Kodaira-Spencer map}
The operator $R_{\mathrm{KS}}$ can be identified with a global tensor $\mathcal{R}_{\mathrm{KS}} \in A^{(1,0|0,1)}(X,T_{X/S})$. Using the dual basis $\partial_\alpha$,
\begin{equation}\label{eq:KS_tensor}
	\mathcal{R}_{\mathrm{KS}} = R_{\mathrm{KS}}(\phi^\alpha) \otimes \partial_\alpha = (\partial_{\bar{\beta}}\Gamma_i^\alpha)  \D s^i \wedge \phi^{\bar{\beta}}\otimes \partial_\alpha.
\end{equation}
 As shown in Proposition \ref{prop:ks_conn_rep}, when restricted to $X_{s_0}$ for $s_0\in S$ and contracted by $\partial_i$, this exactly represents the Kodaira-Spencer class $\rho_{s_0}(\partial_i)$.

\subsubsection*{Relation to the curvature}
The operator $R_{\mathrm{A}_1}$ can be identified with a global tensor
\[
\mathcal{R}_{\mathrm{A}_1} = \bigl( \partial_{\bar{j}}\Gamma_i^\alpha + (\partial_{\bar{\beta}}\Gamma_i^\alpha) \Gamma_{\bar{j}}^{\bar{\beta}} \bigr) \D s^i \wedge \D \bar{s}^j \otimes \partial_\alpha \in A^{(1,1|0,0)}(X,T_{X/S}).
\]
This tensor $\mathcal{R}_{\mathrm{A}_1}$ is the negative of the vertical $(1,0)$-component of $F_\nabla^{1,1}$. Similarly, $R_{\mathrm{A}_2}$ is the negative of the vertical $(0,1)$-component of $F_\nabla^{0,2}$.

\subsubsection*{The relatively holomorphic case}
The connection between these general tensors and the specific definitions in the previous subsection becomes clear when we consider a relatively holomorphic pure $(1,0)$-connection $\nabla^{1,0}$. In this case $R_{\mathrm{KS}}=0$ and $R_{\mathrm{A}_1}$ simplifies and becomes independent of the choice of $\nabla^{0,1}$:
    \[
    \mathcal{R}_{\mathrm{A}_1} = (\partial_{\bar{j}}\Gamma_i^\alpha)  \D s^i \wedge \D\bar{s}^j\otimes \partial_\alpha.
    \]
Comparing this with \eqref{Curvature11Local_eq1}, we have $F^{1,1}_{D} = - \mathcal{R}_{\mathrm{A}_1}$.

\begin{remark}
Similarly to \eqref{DbarDecomp_eq}, the decomposition of $\partial_X$ has the form
\begin{equation}
    \partial_X =\partial_{\mathrm{v}} + \partial_{\mathrm{h}} + R_{\mathrm{A}_1'} + R_{\mathrm{A}_2'} + R_{\mathrm{KS}'},
\end{equation}
where the components have bi-degree shifts $(0,0|1,0)$, $(1,0|0,0)$, $(1,1|0,-1)$, $(2,0|-1,0)$, and $(0,1|1,-1)$ respectively. $R_{\mathrm{A}_1'}$ and $R_{\mathrm{A}_2'}$ can be identified with the negatives of the vertical $(0,1)$-component of $F_\nabla^{1,1}$ and the vertical $(1,0)$-component of $F_\nabla^{2,0}$, respectively.
\end{remark}

\begin{proposition}\label{prop:dolb_class}
	The tensor $\mathcal{R}:=-(\mathcal{R}_{\mathrm{KS}}+\mathcal{R}_{\mathrm{A}_1})\in A^{0,1}(X,f^*T^*S\otimes T_{X/S})$ is $\dbar_X$-closed and defines a cohomology class $[\mathcal{R}]\in H^1(X,f^*\Omega_S\otimes T_{X/S})$, which is independent of a pure complex connection $\nabla$ (always exists). When $\nabla^{1,0}$ is relatively holomorphic, $\mathcal{R}=F_D^{1,1}$.
\end{proposition}
\begin{proof}
	Note that locally we have \[\mathcal{R}=\dbar_X(\Gamma_i^\alpha  \D s^i\otimes \partial_\alpha),\]
	so $\mathcal{R}$ is $\dbar_X$-closed. By \eqref{TransgressionForm_eq}, $[\mathcal{R}]$ is independent of $\nabla$. The last statement follows from the above computations.
\end{proof}

\subsection{Associated bundles}\label{subsec:assoc_bundle}
Let $S$ and $Y$ be complex manifolds. Let $G$ be a complex Lie group with a holomorphic action $G\times Y\to Y$. Let $\pi_P:P\to S$ be a smooth principal $G$-bundle. The adjoint bundle $\ad P := P\times_G \mathfrak{g}$ is a smooth complex vector bundle over $S$. We consider the associated smooth fiber bundle
\[
f:X:=P\times_G Y \longrightarrow S,\qquad [p,y]\longmapsto \pi_P(p).
\]
Since the transition maps for $X$ take values in the image of $G\to \Aut(Y)$, by Lemma \ref{lem:isotrivial_cx_fib_bun}, $X$ naturally inherits the structure of an isotrivial complex fiber bundle $(f, T_{X/S})$.

Let $\mathfrak{g}$ be the Lie algebra of $G$. The action induces the infinitesimal action map $\tau_0: \mathfrak{g} \to H^0(Y, TY)$ (the space of holomorphic vector fields on $Y$), defined by
\begin{equation}\label{eq:tau0}
	\tau_0(\xi):=\bigl(\tilde{\tau}_0(\xi)\bigr)^{1,0},\quad \big(\tilde{\tau}_0(\xi)\big)(y) := \frac{\D}{\D t}\Big|_{t=0}\,\exp(-t\xi)\cdot y,
\end{equation}
where $t\in\RR$ and $\exp:\mathfrak{g}\to G$ is the exponential map. $\tau_0$ is a complex-linear Lie algebra homomorphism and is $G$-equivariant with respect to the adjoint action on $\mathfrak g$ and the push-forward action on $H^0(Y,TY)$. We show the $G$-equivariance. For any $g\in G$ and $\xi\in\mathfrak{g}$,
\begin{align*}
\big(\tilde{\tau}_0(\Ad_g\xi)\big)(y)
&= \frac{\D}{\D t}\Big|_{t=0}\,g\cdot\bigl(\exp(-t\xi)\cdot(g^{-1}\cdot y)\bigr)\\
&= (\D L_g)_{g^{-1}\cdot y}\Bigl(\frac{\D}{\D t}\Big|_{t=0}\,\exp(-t\xi)\cdot(g^{-1}\cdot y)\Bigr)\\
&=(\D L_g)_{g^{-1}\cdot y}(\tilde{\tau}_0(\xi)(g^{-1}\cdot y))
\\&= \bigl(g_*\tilde{\tau}_0(\xi)\bigr)(y),
\end{align*}
where $L_g:Y\to Y$ is $y\mapsto g\cdot y$. 

Let $N:=\{g\in G\mid g\cdot y=y\text{ for every }y\in Y\}$. Then $N$ is a closed normal complex Lie subgroup of $G$, and $\mathfrak n:=\mathrm{Lie}(N)=\ker\tau_0$ is a complex $\Ad_G$-invariant ideal of $\mathfrak g$. Set $G_{\mathrm{eff}}:=G/N$ and $P_{\mathrm{eff}}:=P/N$. Then $G_{\mathrm{eff}}$ acts effectively on $Y$, $P_{\mathrm{eff}}$ is a principal $G_{\mathrm{eff}}$-bundle, and there is a canonical isomorphism $P\times_GY\cong P_{\mathrm{eff}}\times_{G_{\mathrm{eff}}}Y$. Moreover, $\tau_0$ induces an injective $G_{\mathrm{eff}}$-equivariant homomorphism $\hat\tau_0:\mathfrak g/\mathfrak n\hookrightarrow H^0(Y,TY)$. Let $\mathfrak n_P:=P\times_G\mathfrak n\subset\ad P$. Then $\ad P_{\mathrm{eff}}\cong(\ad P)/\mathfrak n_P$. Note that $f_*T_{X/S}$ is canonically isomorphic to the sheaf of smooth sections of the bundle $P \times_G H^0(Y, TY)$ (which has infinite rank if $H^0(Y, TY)$ is infinite-dimensional).

\begin{lemma}\label{lem:tau_associated}
There is a natural morphism of sheaves
\begin{equation}\label{eq:morphism_tau}
	\tau:\ad P \longrightarrow f_*T_{X/S}.
\end{equation}
Its kernel is the sheaf of smooth sections of $\mathfrak n_P$. Consequently, $\tau$ induces an injective morphism
\begin{equation}\label{eq:morphism_tau_hat}
\hat\tau:\ad P_{\mathrm{eff}} \cong(\ad P)/\mathfrak n_P \longrightarrow f_*T_{X/S}.
\end{equation}
If $\hat{\tau}_0:\mathfrak g/\mathfrak n\to H^0(Y,TY)$ is an isomorphism, then $\hat\tau$ is an isomorphism of smooth complex vector bundles.
\end{lemma}

\begin{proof}
	A smooth section $A$ of $\ad P$ over $U\subset S$ corresponds to a smooth map $A:\pi_P^{-1}(U)\to\mathfrak{g}$ such that $A(p\cdot g)=\Ad_{g^{-1}}A(p)$.
	We define a vertical vector field $\widetilde V_A$ on $(\pi_P^{-1}(U))\times Y$ by
	\[
	\widetilde V_A\big|_{(p,y)} := \big(0,\tau_0(A(p))|_y\big).
	\]
$\widetilde V_A$ is invariant under the action $r_g(p,y)=(p\cdot g,\,g^{-1}\!\cdot y)$ since
\begin{align*}
	(\D r_g)_{(p,y)}\bigl(\widetilde V_A\big|_{(p,y)}\bigr)
	&= \bigl(0,(\D L_{g^{-1}})_y\bigl(\tau_0(A(p))\bigr)\bigr)\\
	&= \big(0,\tau_0\big(\Ad_{g^{-1}}A(p)\big)\big|_{g^{-1}\cdot y}\big) \quad (\text{by equivariance of } \tau_0)\\
	&= \big(0,\tau_0\big(A(p\cdot g)\big)\big|_{g^{-1}\cdot y}\big) = \widetilde V_A\big|_{r_g(p,y)}.
\end{align*}
Thus $\widetilde V_A$ descends to a smooth vertical vector field $V_A$ on $X|_U$, which is fiberwise holomorphic. The assignment $A\mapsto V_A$ defines a morphism of sheaves $\tau:\ad P\to f_*T_{X/S}$. In a local trivialization, $V_A=0$ if and only if $A$ takes values in $\mathfrak n$. Hence $\ker\tau=\mathfrak n_P$, and $\tau$ induces the injective morphism $\hat\tau$. If $\hat{\tau}_0$ is an isomorphism, then $\dim_\CC H^0(Y, TY)=\dim_\CC \mathfrak{g}-\dim_\CC \mathfrak{n}$, which is finite. The above construction provides a bijection between equivariant maps into $\mathfrak{g}/\mathfrak{n}$ and equivariant maps into $H^0(Y, TY)$, hence $\hat{\tau}$ is an isomorphism of complex vector bundles.
\end{proof}

Let $A$ be a \emph{principal connection} on $P$, which is an element in $A^1(P, \mathfrak{g})$ satisfying
\[\Ad_g (R_g^* A) = A, \quad A(v_\xi) = \xi, \quad \forall \xi \in \mathfrak{g} \text{ and } g\in G,\]
where $v_\xi$ is the fundamental vector field on $P$ associated to $\xi$ and $R_g$ is the right action by $g$. The space of such connections is an affine space modeled on $A^{1}(S, \ad P)$. The connection $A$ defines a $G$-invariant horizontal distribution $H_P^\RR = \ker A \subset TP^\RR$. On $P \times Y$, the distribution $H_{(p,y)}^\RR = H_{P,p}^\RR \oplus 0$ is $G$-invariant and descends via the quotient map $q: P \times Y \to X$ to a distribution $H_X^\RR \subset TX^\RR$, defining a connection $\nabla_A^\RR$ on $X$. Its complexification $\nabla_A$ decomposes into $\nabla_A^{1,0}$ and $\nabla_A^{0,1}=\widebar{\nabla_A^{1,0}}$.
\begin{remark}\label{rmk:principal_dbar_alconn}
When $Y=G$ with the left multiplication action, $f:X\to S$ is canonically isomorphic to $\pi_P:P\to S$ and $H_{X}^\RR\cong H_P^\RR$ under this isomorphism. A \emph{principal $\dbar$-operator} on $\pi_P$ is a $\dbar$-operator (Definition \ref{def:dbar_operator}) induced by $\nabla_A^{0,1}$, where $A$ is a principal connection. Similarly, a \emph{principal almost connection} is an almost connection induced by $\nabla_A^{1,0}$. The space of principal $\dbar$-operators (resp. almost connections) is affine modeled on $A^{0,1}(S,\ad P)$ (resp. $A^{1,0}(S,\ad P)$). A principal connection is equivalent to the sum of a principal $\dbar$-operator and a principal almost connection.

	A principal $\dbar$-operator on $P$ is equivalent to a \emph{principal almost complex structure} on $P$ (\cite[Def.~2.4]{CTW25}), i.e., an almost complex structure on $\pi:P\to S$ (Definition \ref{def:alm_cx_str}) such that the corresponding operator $J_{P}\in C^{\infty}(P,\End(TP^\RR))$ satisfies $J_P\comp (R_g)_*=(R_g)_*\comp J_P$ for any $g\in G$.
\end{remark}

\begin{lemma}\label{lem:b_rel_holo_associated}
For any principal connection $A$ on $P$, the induced $(1,0)$-connection $\nabla_A^{1,0}$ on $f:X\to S$ is relatively holomorphic. Similarly, $\nabla_A^{0,1}$ is mixed relatively holomorphic.
\end{lemma}
\begin{proof}
	Consider the local smooth trivialization $P|_U \cong U \times G$, so $X|_U \cong U \times Y$ via $[(s,g),y] \mapsto (s, g y)$.  Let $(s^i)$ be holomorphic coordinates on $U$ and $(z^\alpha)$ be local holomorphic coordinates on $Y$.  The connection form $A$, pulled back to the identity section $(s,e)$, can be written as
\begin{equation}\label{eq:princ_conn_local}
	A = A_i(s) \D s^i + B_{\bar{j}}(s) \D\bar{s}^j,
\end{equation}
where $A_i, B_{\bar{j}}: U \to \mathfrak{g}$ are smooth functions. Let $s^i=x^i+\I y^i$. At $(s,e) \in P$, under the natural identification $T_{(s,e)} P \cong T_s U \oplus \mathfrak{g}$, the horizontal lift $\partial_{x^i}^\sharp$ of $\partial_{x^i}$ is $\partial_{x^i}^\sharp = \partial_{x^i}-A_i-B_{\bar{i}}$. Similarly, $\partial_{y^i}^\sharp = \partial_{y^i}-\I A_i+\I B_{\bar{i}}$.  Then
\begin{align*}
\D q_{(s,e,z)}(\partial_{x^i}^\sharp, 0) &= (\partial_{x^i}, \tilde{\tau}_0(A_i)(z)+\tilde{\tau}_0(B_{\bar{i}})(z)),\\
\D q_{(s,e,z)}(\partial_{y^i}^\sharp, 0) &= (\partial_{y^i}, \tilde{\tau}_0(\I A_i)(z)-\tilde{\tau}_0(\I B_{\bar{i}})(z)).
\end{align*}
Therefore, we have
\begin{align*}
\nabla_{A}^{1,0}(\partial_i)&=\partial_i+\tfrac{1}{2}\bigl({\tau}_0(A_i)+{\tau}_0(B_{\bar{i}})-\I({\tau}_0(\I A_i)-{\tau}_0(\I B_{\bar{i}}))\bigr)
\\&\qquad\,+\tfrac{1}{2}\bigl(\widebar{{\tau}_0(A_i)}+\overline{{\tau}_0(B_{\bar{i}})}-\I(\widebar{{\tau}_0({\I} A_i)}-\overline{{\tau}_0(\I B_{\bar{i}})})\bigr)\\
&=\partial_i+\tau_0(A_i)+\overline{\tau_0(B_{\bar{i}})}\\
\nabla_{A}^{0,1}(\partial_{\bar{i}})&=\partial_{\bar{i}}+\tau_0(B_{\bar{i}})+\widebar{\tau_0(A_i)}.
\end{align*}
 Since $\im(\tau_0)\subset H^0(Y,TY)$, $\nabla_A^{1,0}$ is relatively holomorphic and  $\nabla_A^{0,1}$ is mixed relatively holomorphic.
\end{proof}

\begin{lemma}\label{lem:affine_map_associated}
There is a natural affine map
\[
\{\text{principal connections on }P\}\longrightarrow\mathcal{A}^{1,0}_{\mathrm{RH,LC}},
\]
given by $A \mapsto \nabla_A^{1,0}$, where $\mathcal{A}^{1,0}_{\mathrm{RH,LC}}$ is defined in Corollary \ref{cor:RH_LC_space}. Two principal connections have the same image if and only if their difference belongs to $A^1(S,\mathfrak n_P)$, and there is a natural affine injection
\[
\{\text{principal connections on }P_{\mathrm{eff}}\}
\hookrightarrow
\mathcal{A}^{1,0}_{\mathrm{RH,LC}}.
\]
If $\hat{\tau}_0$ is an isomorphism, then this map is an affine isomorphism.
\end{lemma}

\begin{proof}
The map $A \mapsto \nabla_A^{1,0}$ is affine, since $\nabla_{A + \alpha}^{1,0}-\nabla_A^{1,0} =  \tau(\alpha^{1,0})+\widebar{\tau(\alpha^{0,1})}$, which corresponds to $\tau(\alpha)\in A^1(S,f_*T_{X/S})$. By Lemma \ref{lem:b_rel_holo_associated}, the image is contained in $\mathcal{A}^{1,0}_{\mathrm{RH,LC}}$. By Lemma \ref{lem:tau_associated}, two principal connections have the same image precisely when their difference belongs to $A^1(S,\mathfrak n_P)$. The quotient therefore gives the asserted affine injection for $P_{\mathrm{eff}}$. If $\hat{\tau}_0$ is an isomorphism, then $\hat\tau$ is an isomorphism, and so is the induced map between the modeling spaces $A^1(S,\ad P_{\mathrm{eff}})\to A^1(S,f_*T_{X/S})$. Hence the affine injection is an affine isomorphism.
\end{proof}

Let $\partial_A$ be the almost connection induced by $\nabla_A^{1,0}$ and $\dbar_A$ be the $\dbar$-operator induced by $\nabla_A^{0,1}$, both satisfying the lifting condition. Let $D_A = \nabla_A^{1,0} + \dbar_A$ and $\hat{D}_A=\partial_A+\dbar_A$. Then the curvatures $F^{1,1}_{D_A}=F^{1,1}_{\hat{D}_A}$ (defined by \eqref{Curvature11Def_eq}), $F^{0,2}_{\dbar_A}$ (defined in Corollary \ref{cor:dbarint_via_conn_curv}), and $F^{2,0}_{\partial_A}$ (defined after Definition \ref{def:almost_connection}) are well-defined tensors.

\begin{lemma}\label{lem:curvature_correspondence}
Let $F_A$ be the curvature of $A$. The curvatures of the induced operators on $f:X\to S$ satisfy
\[F^{1,1}_{D_A} = \tau\bigl( F_A^{1,1} \bigr),\qquad F^{0,2}_{\dbar_A} = \tau\bigl( F_A^{0,2} \bigr), \qquad F^{2,0}_{\partial_A} = \tau\bigl( F_A^{2,0} \bigr).\]
Consequently, the induced $\dbar$-operator $\dbar_A$ is integrable if and only if $F_A^{0,2}\in A^{0,2}(S,\mathfrak n_P)$, or equivalently the connection $\hat{A}$ induced on $P_{\mathrm{eff}}$ satisfies $F_{\hat A}^{0,2}=0$.
\end{lemma}
\begin{proof}
Recall that $F_{D_A}^{1,1}$ is locally given by \eqref{Curvature11Local_eq}. Using the notations in the proof of Lemma \ref{lem:b_rel_holo_associated}, we compute that
\begin{align*}
 \bigl( \partial_i \Gamma_{\bar{j}}^\alpha - \partial_{\bar{j}} \Gamma_i^\alpha \bigr) + \bigl( \Gamma_i^\beta \partial_\beta \Gamma_{\bar{j}}^\alpha - \Gamma_{\bar{j}}^\beta \partial_\beta \Gamma_i^\alpha \bigr) &= \tau_0(\partial_i B_{\bar{j}} - \partial_{\bar{j}} A_i)^\alpha + [\tau_0(A_i), \tau_0(B_{\bar{j}})]^\alpha \\
&= \tau_0(\partial_i B_{\bar{j}} - \partial_{\bar{j}} A_i + [A_i, B_{\bar{j}}])^\alpha.
\end{align*}
The curvature of $A$ is $F_A = \D A + \frac{1}{2}[A, A]$. Its $(1,1)$-part is locally
\[
F_A^{1,1} = (\partial_i B_{\bar{j}} - \partial_{\bar{j}} A_i + [A_i, B_{\bar{j}}]) \D s^i \wedge \D \bar{s}^j.
\]
Thus, $F^{1,1}_{D_A} = \tau(F_A^{1,1})$.  Similarly, $F^{0,2}_{\dbar_A} = \tau( F_A^{0,2})$ and $F^{2,0}_{\partial_A} = \tau(F_A^{2,0})$. Since the lifting condition holds, $\dbar_A$ is integrable if and only if $F^{0,2}_{\dbar_A}=0$.  By Lemma \ref{lem:tau_associated}, this is equivalent to $F_A^{0,2}\in A^{0,2}(S,\mathfrak n_P)$. Since $F_{\hat A}$ is the image of $F_A$ under $\ad P\to\ad P_{\mathrm{eff}}$, this is equivalent to $F_{\hat A}^{0,2}=0$.
\end{proof}
\begin{remark}\label{rmk:curv_coresp}
Let $A$ be a principal connection on $P$ which induces an almost connection $\partial_A$ and a $\dbar$-operator $\dbar_A$ on $\pi_P:P\to S$. Then $F_{\hat{D}_A}^{1,1}=-F_A^{1,1}$, $F_{\dbar_A}^{0,2}=-F_A^{0,2}$, and $F_{\partial_A}^{2,0}=-F_A^{2,0}$.
\end{remark}

Now we assume that $P$ is a holomorphic principal bundle, then $X$ is a holomorphic fiber bundle. In this case, the map $A \mapsto V_A$ constructed in Lemma \ref{lem:tau_associated} preserves the holomorphic structure, so $\tau$ in \eqref{eq:morphism_tau} is a morphism of analytic sheaves (with $f_*T_{X/S}$ replaced by $f^{\mathrm{hol}}_*T_{X/S}$).
\begin{definition}
A connected complex Lie group $G$ is \emph{reductive} if it is of the form $\left(\left(\mathbb{C}^*\right)^r \times G_0\right) / Z$ where $G_0$ is semisimple and $Z$ is a finite central subgroup. Equivalently, $G$ is the complexification $K^\CC$ of a compact real Lie group $K$. A complex Lie group $G$ is \emph{reductive} if $G^{\circ}$ is reductive and $G / G^{\circ}$ is finite, where $G^{\circ}$ is the identity component of $G$.
\end{definition}
\begin{lemma}\label{lem:reductive_finite_split}
	Let $\pi_P:P\to S$ be a holomorphic principal $G$-bundle. Suppose that $G$ is a complex reductive Lie group and $H^0(Y,TY)$ is finite-dimensional. Then $\hat\tau$ in \eqref{eq:morphism_tau_hat} is a split injective morphism of holomorphic vector bundles.
\end{lemma}
\begin{proof}
Since $H^0(Y,TY)$ is a finite-dimensional holomorphic representation of the complex reductive group $G$, the representation is completely reducible. $\tau_0$ is $G$-equivariant, so the image $\im(\tau_0)\cong\mathfrak g/\mathfrak n$ is a $G$-subrepresentation. By complete reducibility, there exists a $G$-invariant subspace $W$ such that $H^0(Y,TY)\cong(\mathfrak g/\mathfrak n)\oplus W$. Consequently,
\[
f^{\mathrm{hol}}_*T_{X/S} \cong P\times_G ((\mathfrak g/\mathfrak n)\oplus W)\cong (\ad P/\mathfrak n_P)\oplus(P\times_G W),
\]
and $\hat\tau$ is the inclusion of the first summand.
\end{proof}

A \emph{principal complex connection} $A$ on $P$ is a principal connection belonging to $A^{1,0}(P,\mathfrak{g})$. The space of such connections is an affine space modeled on $A^{1,0}(S, \ad P)$. Given a principal complex connection $A$ on $P$, we define the induced connection as above. In a local holomorphic trivialization, $B_{\bar{j}}=0$ in \eqref{eq:princ_conn_local}, $\dbar_A$ is the canonical $\dbar$-operator on $f$, and $\nabla_A^{1,0}$ is a relatively holomorphic pure $(1,0)$-connection. By Lemma \ref{lem:affine_map_associated}, we have the following result.

\begin{corollary}\label{cor:affine_map_associated}
There is a natural affine map
\[
\big\{\text{principal complex connections on }P\big\}\longrightarrow \big\{\text{relatively holomorphic pure $(1,0)$-connections on }f\big\},
\]
given by $A \mapsto \nabla_A^{1,0}$. Under the identification \[\big\{\text{principal complex connections on }P\big\}/A^{1,0}(S,\mathfrak n_P)\cong \big\{\text{principal complex connections on }P_{\mathrm{eff}}\big\},\]
it induces an affine injection
\[
\big\{\text{principal complex connections on }P_{\mathrm{eff}}\big\}\hookrightarrow \big\{\text{relatively holomorphic pure $(1,0)$-connections on }f\big\}.
\]
If $\hat{\tau}_0$ is an isomorphism, then this map is an affine isomorphism.
\end{corollary}

\begin{theorem}\label{thm:atiyah_associated}
Let $f:X \to S$ be an associated bundle as above. Suppose $H^0(Y,TY)$ is finite-dimensional. Then for any relatively holomorphic pure $(1,0)$-connection $\nabla^{1,0}$ on $f$, the curvature class $[F_{D}^{1,1}]_S \in H^{1,1}(S, f^{\mathrm{hol}}_*T_{X/S})$ is the negative of the image of the Atiyah class $A(P) \in H^1(S, \Omega_S\otimes \ad P)$ under the Dolbeault isomorphism and the morphism $\tau:\ad P\to f^{\mathrm{hol}}_*T_{X/S}$.
\end{theorem}
\begin{proof}
Let $A$ be a principal complex connection on $P$, then by \cite[Prop.~4]{atiyah57} (with our \v{C}ech--Dolbeault convention), $A(P)=-[F_A^{1,1}]$ under the Dolbeault isomorphism. By Lemma \ref{lem:curvature_correspondence}, we have $\tau([F_A^{1,1}])=[F_{D_A}^{1,1}]_S$. By Lemma \ref{lem:b_rel_holo_associated}, $\nabla^{1,0}_A$ is a relatively holomorphic pure $(1,0)$-connection. Then $[F_{D_A}^{1,1}]_S=[F_{D}^{1,1}]_S$ by Remark \ref{rmk:curv_class_S}. The conclusion then follows.
\end{proof}

\begin{remark}\label{rmk:cpt_fiber}
Let $S$ be connected and let $f: X \to S$ be a proper surjective holomorphic fibration which admits a relatively holomorphic pure $(1,0)$-connection. Then $f$ is a holomorphic fiber bundle by Proposition \ref{prop:exist_fiberwise_holo_connection}, which arises as the associated bundle $X = P \times_G Y\to S$ for a holomorphic principal $G$-bundle $\pi_P: P \to S$ and a compact complex manifold $Y$, where $G =\Aut(Y)$ is a complex Lie group. In this case, $\tau_0$ is an isomorphism and $f^{\mathrm{hol}}_*T_{X/S}$ is locally free, so $\tau$ is an isomorphism by Lemma \ref{lem:tau_associated}.
\end{remark}

\section{Holomorphic connections}\label{sec:hol_conn}
Let $ f: X \to S $ be a holomorphic fibration between (possibly noncompact) complex manifolds of dimensions $m+n$ and $n$. As in \eqref{HolTangentSeq_eq}, there is a short exact sequence of holomorphic vector bundles on $X$:
\begin{equation}\label{HolTangentSeq_eq1}
	0 \to T_{X/S} \to TX \to f^* TS \to 0.
\end{equation}
Recall that a \emph{holomorphic connection} on $f$ is a holomorphic splitting of \eqref{HolTangentSeq_eq1}. This is equivalent to the connection coefficients $\Gamma_i^\alpha(s,z)$ being holomorphic functions in both $s$ and $z$. The following lemma is straightforward.
\begin{lemma}
Whenever nonempty, the space of holomorphic connections on $f$ is an affine space modeled on $H^0(S,\Omega_S\otimes f^{\mathrm{hol}}_*T_{X/S})$.
\end{lemma}
Together with Lemma \ref{lem:affine_map_associated}, this implies the following.
\begin{corollary}\label{cor:affine_map_associated_hol}
Suppose $\pi_P:P\to S$ is a holomorphic principal $G$-bundle, $G$ acts holomorphically on $Y$, and $f:P\times_G Y\to S$ is the associated holomorphic fiber bundle. Then there is a natural affine map
\[\big\{\text{holomorphic connections on }P\big\}\longrightarrow\big\{\text{holomorphic connections on }f\big\},\]
given by $A \mapsto \nabla_A^{1,0}$. Let
\[
N:=\ker\bigl(G\to\Aut(Y)\bigr),\qquad G_{\mathrm{eff}}:=G/N,\qquad P_{\mathrm{eff}}:=P/N.
\]
Applying the associated-bundle construction directly to $P_{\mathrm{eff}}$ gives an affine injection
\[
\big\{\text{holomorphic connections on }P_{\mathrm{eff}}\big\} \hookrightarrow \big\{\text{holomorphic connections on }f\big\}.
\]
If $\hat{\tau}:\ad P_{\mathrm{eff}}\to f^{\mathrm{hol}}_*T_{X/S}$ is an isomorphism, then this map is an affine isomorphism.
\end{corollary}

\begin{lemma}\label{lem:hol_conn_curv}
Let $(f,T_{X/S})$ be a complex fiber bundle equipped with an integrable $\dbar$-operator $\dbar$. Let $\nabla^{1,0}$ be a pure $(1,0)$-connection. Then $\nabla^{1,0}$ is a holomorphic connection if and only if $\nabla^{1,0}$ is relatively holomorphic and $F_D^{1,1}=0$ (well-defined by Lemma \ref{lem:curv_def_rel_hol} since $\nabla^{1,0}$ is relatively holomorphic).
\end{lemma}
\begin{proof}
This is immediate from \eqref{Curvature11Local_eq1}.
\end{proof}
The following example illustrates that a flat connection may not induce a holomorphic connection.
\begin{example}\label{ex:flat_real_nonholomorphic_connection}
Let $\Delta=\{s\in\CC:|s|<1\}$ and let $f:X=\Delta_s\times\CC_z\to\Delta_s$ carry its standard product holomorphic structure. Consider the smooth bundle diffeomorphism
\[
\Phi:\Delta_s\times\CC_w\longrightarrow X,\qquad \Phi(s,w)=(s,w+s\bar w).
\]
Its inverse is given by $w=(z-s\bar z)/(1-|s|^2)$. Push the product real connection forward by $\Phi$. The resulting real connection is flat, and its complexified horizontal lifts are
\[
H_s=\partial_s+A\,\partial_z,\qquad H_{\bar s}=\partial_{\bar s}+\overline{A}\,\partial_{\bar z},\qquad A:=(\bar z-\bar s z)/(1-|s|^2).
\]
Since $H_{\bar s}\equiv\partial_{\bar s}\pmod{\widebar{T_{X/S}}}$, it induces the canonical $\dbar$-operator of the product holomorphic structure. However, $\partial_{\bar z}A=(1-|s|^2)^{-1}\ne0$, so $\nabla^{1,0}$ is not relatively holomorphic and hence is not holomorphic.
\end{example}

\subsection{Extension classes}
The short exact sequence \eqref{HolTangentSeq_eq1} defines an extension class
\[
A(X) \in H^1(X, f^* \Omega_S \otimes T_{X/S}),
\]
which vanishes if and only if there exists a holomorphic connection on $f$. We will give a differential geometric interpretation of $A(X)$.

We first compute a Čech cocycle representing the class $A(X)$. Choose an open cover $ \{U_a\} $ of $ X $ with local coordinates $(z_a^1, \dots, z_a^m, s_a^1, \dots, s_a^n)$, where $ f $ is given by $ (z_a, s_a) \mapsto s_a $ on $U_a$. On overlaps $ U_{ab} = U_a \cap U_b $, the transition functions are $s_b^j = s_b^j(s_a^1, \dots, s_a^n)$, $z_b^\beta= z_b^\beta(z_a^1, \dots, z_a^m, s_a^1, \dots, s_a^n)$. The Jacobian matrix has block structure
\[
J_{ab} = \begin{pmatrix}
\frac{\partial z_b^\beta}{\partial z_a^\alpha} & \frac{\partial z_b^\beta}{\partial s_a^i} \\[1.5mm]
0 & \frac{\partial s_b^j}{\partial s_a^i}
\end{pmatrix} =: \begin{pmatrix}
B_{ab} & C_{ab} \\
0 & A_{ab}
\end{pmatrix}.
\]
On each $ U_a $, define the standard holomorphic splitting $ \sigma_a: f^* TS |_{U_a} \to TX |_{U_a} $ by
\[
\sigma_a(\partial_{s_a^i}) = \partial_{s_a^i}, \quad i = 1, \dots, n.
\]
Then $A(X)$ is represented by the Čech $1$-cocycle $(U_{ab},c_{ab})$, where
\[
c_{ab} = \sigma_b - \sigma_a \in \Gamma(U_{ab}, f^* \Omega_S \otimes T_{X/S}).
\]

We compute
\begin{align*}
	\sigma_b(\partial_{s_a^i}) &= \frac{\partial s_b^j}{\partial s_a^i} \sigma_b(\partial_{s_b^j}) = \frac{\partial s_b^j}{\partial s_a^i} \partial_{s_b^j}\\
	&=(A_{ab})_i^j \big( (A_{ab}^{-1})_j^k \partial_{s_a^k} - (B_{ab}^{-1})_\beta^\alpha (C_{ab} A_{ab}^{-1})_j^\beta \partial_{z_a^\alpha} \big)
	\\&=\partial_{s_a^i} - \Big( \frac{\partial z_b^\beta}{\partial s_a^i} \frac{\partial z_a^\alpha}{\partial z_b^\beta} \Big) \partial_{z_a^\alpha}.
\end{align*}
The difference is
\[
c_{ab}(\partial_{s_a^i}) = \sigma_b(\partial_{s_a^i}) - \sigma_a(\partial_{s_a^i}) = -\Big( \frac{\partial z_b^\beta}{\partial s_a^i} \frac{\partial z_a^\alpha}{\partial z_b^\beta} \Big) \partial_{z_a^\alpha}.
\]
Therefore
\begin{equation}\label{Cech1cocycle_eq}
	c_{ab} = -\Big( \frac{\partial z_b^\beta}{\partial s_a^i} \frac{\partial z_a^\alpha}{\partial z_b^\beta} \Big) \D s_a^i \otimes \partial_{z_a^\alpha}=-\left( B_{ab}^{-1} C_{ab}  \right)^\alpha_i  \D s_a^i \otimes \partial_{z_a^\alpha}.
\end{equation}

\begin{theorem}\label{thm:Ext_class_curv_class}
	Let $\nabla$ be a pure complex connection on a holomorphic fibration $f:X\to S$. Then $A(X)=[\mathcal{R}]$, where $[\mathcal{R}]$ is defined in Proposition \ref{prop:dolb_class}. In particular, if $f$ admits a relatively holomorphic connection $\nabla^{1,0}$, then $A(X)=[F_D^{1,1}]$.
\end{theorem}
\begin{proof}
	Let $\nabla=\nabla^{1,0}+\nabla^{0,1}$ be a pure complex connection. Then $\nabla^{1,0}:f^* TS \to TX$ is a pure $(1,0)$-connection, given locally on $U_a$ by
\[
\nabla^{1,0}(\partial_{s_a^i}) = \partial_{s_a^i} + {\Gamma_a}_i^\alpha \partial_{z_a^\alpha}.
\]
It suffices to show that the Čech $1$-cocycle $c_{ab}$ defined by \eqref{Cech1cocycle_eq} and
	the Dolbeault representative $\mathcal{R}$ locally given by \[\bigl( (\partial_{\bar{s}_a^j} {\Gamma_a}_i^\alpha) \D\bar{s}_a^j+(\partial_{\bar{z}_a^\beta}{\Gamma_a}_i^\alpha)\D\bar{z}_a^\beta \bigr) \otimes \D s_a^i \otimes \partial_{z_a^\alpha}\in A^{0,1}(U_a,f^*T^*S\otimes T_{X/S})\] represent the same cohomology class in $H^1(X, f^* \Omega_S \otimes T_{X/S})$.

On each $U_a$, we have the standard holomorphic splitting $\sigma_a$. Define a Čech $0$-cochain $\gamma = \{\gamma_a\}$ by
\[
\gamma_a = \nabla^{1,0} - \sigma_a \in C^{\infty}(U_a, f^* T^*S \otimes T_{X/S}).
\]
Locally on $U_a$, this is $\gamma_a = {\Gamma_a}_i^\alpha \D s_a^i \otimes \partial_{z_a^\alpha}$.

Now we compute the Dolbeault differential of $\gamma_a$. We have
\[	\dbar_X \gamma_a = \dbar_X \bigl( {\Gamma_a}_i^\alpha \D s_a^i \otimes \partial_{z_a^\alpha} \bigr) = \bigl( (\partial_{\bar{s}_a^j} {\Gamma_a}_i^\alpha) \D\bar{s}_a^j+(\partial_{\bar{z}_a^\beta}{\Gamma_a}_i^\alpha)\D\bar{z}_a^\beta \bigr) \otimes \D s_a^i \otimes \partial_{z_a^\alpha}.\]
This is exactly the local expression for $\mathcal{R}|_{U_a}$.

Next, compute the Čech coboundary $\delta \gamma$. On overlaps $U_{ab} = U_a \cap U_b$,
\[
(\delta \gamma)_{ab} = \gamma_b - \gamma_a = (\nabla^{1,0} - \sigma_b) - (\nabla^{1,0} - \sigma_a) = \sigma_a - \sigma_b = -c_{ab}.
\]
Thus $\delta \gamma = -c$, where $c = \{c_{ab}\}$ is the Čech $1$-cocycle.

In the Čech-Dolbeault double complex for the cover $\mathcal{U}=\{U_a\}$ and the sheaf $f^* \Omega_S \otimes T_{X/S}$, the element $\gamma \in \check{C}^0(\mathcal{U}, \mathcal{A}^{0,0}(f^* T^*S \otimes T_{X/S}))$ satisfies
\[
\D_{\mathrm{tot}} \gamma = \delta \gamma +\dbar_X \gamma = -c + \mathcal{R},
\]
where $\D_{\mathrm{tot}}$ is the total differential. Therefore, $c$ and $\mathcal{R}$ represent the same cohomology class in $H^1(X, f^* \Omega_S \otimes T_{X/S})$.
\end{proof}

 Consider the Leray spectral sequence for the $\mathcal O_X$-module $f^\ast\Omega_S\otimes T_{X/S}$, whose $E_2$-page is
\[
E_2^{p,q}=H^p\big(S,\Omega_S\otimes R^q f^{\mathrm{hol}}_*T_{X/S}\big)\Longrightarrow H^{p+q}\big(X,f^\ast\Omega_S\otimes T_{X/S}\big).
\]
Then we have the five-term exact sequence
\begin{multline*}
	0\to H^1(S,\Omega_S\otimes f^{\mathrm{hol}}_*T_{X/S})\xrightarrow{\iota} H^1(X,f^\ast\Omega_S\otimes T_{X/S})\xrightarrow{\rho} H^0(S,\Omega_S\otimes R^1f^{\mathrm{hol}}_*T_{X/S}) \\ \to H^2(S,\Omega_S\otimes f^{\mathrm{hol}}_*T_{X/S})\to H^2(X,f^\ast\Omega_S\otimes T_{X/S}).
\end{multline*}

The following result is known (e.g. \cite[Lem.~4.6]{rizzi25}). We provide a direct differential geometric proof.

\begin{lemma}\label{lem:rho-kodaira-spencer}
Let $f: X \to S$ be a holomorphic fibration. Then the map
\[
\rho: H^1(X,f^\ast\Omega_S\otimes T_{X/S})\to H^0(S,\Omega_S\otimes R^1f^{\mathrm{hol}}_*T_{X/S})
\]
induced by the Leray spectral sequence sends the extension class $A(X)$ to the global Kodaira-Spencer class $\mathrm{ks}$ of the fibration.
\end{lemma}

	\begin{proof}
	Fix a point $s_0 \in S$ and a tangent vector $v = \frac{\partial}{\partial s_a^i} \big|_{s_0} \in T_{s_0}S$.
		The extension class $A(X)$ is represented by the cocycle $c_{ab}$. Restricting to the fiber $X_{s_0}$, we obtain a Čech 1-cocycle $c_{ab}|_{X_{s_0}}(v)$ with values in $TX_{s_0}$. This represents the extension class of the restricted exact sequence $0 \to TX_{s_0} \to TX|_{X_{s_0}} \to f^* TS|_{X_{s_0}} \to 0$, contracted by $v$, which coincides with $\rho_{s_0}(v)$. For a local holomorphic vector field $v$ on an open set $V\subset S$, the cocycle $\{c_{ab}(f^*v)\}$ on $f^{-1}(V)$ represents the connecting class whose image in $H^0(V,R^1f^{\mathrm{hol}}_*T_{X/S})$ is $\mathrm{ks}(v)$.
By the construction of the Leray edge morphism, the same image is the contraction of $\rho(A(X))$ with $v$.
Therefore $\rho(A(X))=\mathrm{ks}$.
	\end{proof}
\begin{remark}
	If $f$ admits a relatively holomorphic pure $(1,0)$-connection, then $\mathrm{ks}$ vanishes by Corollary \ref{cor:rel_hol_ks_vanish}. If moreover $f^{\mathrm{hol}}_*T_{X/S}$ is locally free of finite rank, then $\iota([F_D^{1,1}]_S)=A(X)$ by Theorem \ref{thm:Ext_class_curv_class}.
\end{remark}
We have a converse of Corollary \ref{cor:affine_map_associated_hol} as follows.
\begin{corollary}\label{cor:assoc_holconn_imply_princ_holconn}
Let $\pi_P:P\to S$ be a holomorphic principal $G$-bundle with $G$
acting holomorphically on $Y$, and let $f:P\times_G Y\to S$ be the associated holomorphic fiber bundle. Suppose that $G_{\mathrm{eff}}$ is reductive and that $H^0(Y,TY)$ is finite-dimensional. If $f$ admits a holomorphic connection, then so does $P_{\mathrm{eff}}$.
\end{corollary}
\begin{proof}
If $f$ admits a holomorphic connection, then $A(X)=0$. By the above remark, $[F_D^{1,1}]_S=0$. By Lemma \ref{lem:reductive_finite_split}, $\hat{\tau}$ induces an injection $H^{1,1}(S,\ad P_{\mathrm{eff}})\to H^{1,1}(S,f^{\mathrm{hol}}_*T_{X/S})$. Then by Theorem \ref{thm:atiyah_associated}, we have $A(P_{\mathrm{eff}})=0$. Therefore, $P_{\mathrm{eff}}$ admits a holomorphic connection.
\end{proof}

\subsection{Nonlinear Riemann-Hilbert correspondence}
In this subsection, we establish the correspondence between representations of the fundamental group of the base manifold into the automorphism group of the fiber and complete flat holomorphic bundles. This is a nonlinear analogue of the classical Riemann-Hilbert correspondence.

Let $S$ be a connected complex manifold and $Y$ be a complex manifold. Let $\nabla^{1,0}$ be a holomorphic connection on a holomorphic fiber bundle $f:X\to S$ with fiber $Y$. $\nabla^{1,0}$ is determined by $H_i :=\nabla^{1,0}(\partial_i)=\partial_i + \Gamma_i^\alpha \partial_\alpha$, where $\Gamma_i^\alpha$ are holomorphic functions in $(s,z)$. These define the holomorphic horizontal distribution $H = \text{span}\{H_i\}$.

The curvature of $\nabla^{1,0}$ is purely of type $(2,0)$, defined by $F^{2,0}_{\nabla^{1,0}} := \frac12[\nabla^{1,0}, \nabla^{1,0}]^{\mathrm{vert}}$. It measures the failure of $H$ to be integrable, which is determined by
\begin{equation}\label{eq:hol_curv_coeff}
	[H_i, H_j]
	= \big( \partial_i\Gamma_j^\beta - \partial_j\Gamma_i^\beta + \Gamma_i^\alpha\partial_\alpha\Gamma_j^\beta - \Gamma_j^\alpha\partial_\alpha\Gamma_i^\beta \big) \partial_\beta=: R_{ij}^\beta \partial_\beta.
\end{equation}
Locally, the curvature is  $F^{2,0}_{\nabla^{1,0}} =\frac12 R_{ij}^\beta \partial_\beta\otimes \D s^i \wedge \D s^j$. The holomorphic connection $\nabla^{1,0}$ is called \emph{flat} if its curvature $F^{2,0}_{\nabla^{1,0}}$ vanishes.
\begin{lemma}\label{lem:flatness_local_trivialization}
A holomorphic connection $\nabla^{1,0}$ on $f: X \to S$ is flat if and only if for every point $x \in X$, there exist local holomorphic coordinates $(s, z)$ around $x$, such that the projection $f$ is given by $(s, z) \mapsto s$, and the connection coefficients $\Gamma_i^\alpha$ vanish identically.
\end{lemma}
\begin{proof}
If such a trivialization exists, then $\Gamma_i^\alpha = 0$. By \eqref{eq:hol_curv_coeff}, $R_{ij}^\beta = 0$, so the connection is flat.

Conversely, suppose the connection is flat. The holomorphic horizontal distribution $H$ is integrable. By the complex Frobenius theorem, $H$ defines a holomorphic foliation $\mathcal{F}$ transverse to the fibers. Locally around any point $x \in X$, we can find adapted holomorphic coordinates $(s, z)$ such that the leaves of the foliation are given by $\{z = \text{constant}\}$ and the projection $f$ is locally given by $(s, z) \mapsto s$. In these coordinates, the horizontal distribution $H$ is spanned by $\partial_i$, which means $\Gamma_i^\alpha$ must be identically zero.
\end{proof}

Let $\gamma: [0, 1] \to S$ be a smooth path. Let $v(t) = \dot{\gamma}(t)$ be the real tangent vector of the path. Decompose $v = v^{1,0} + v^{0,1}$. In local coordinates $s(t)$ representing $\gamma(t)$, the $(1,0)$-component is $v^{1,0}(t) = \frac{\D s^i}{\D t} \partial_i$. The horizontal lift of $v^{1,0}(t)$ is $\frac{\D s^i}{\D t} H_i$.

\begin{definition}
Let $\nabla^{1,0}$ be a pure (1,0)-connection on a holomorphic fibration. A smooth path $\tilde{\gamma}(t)$ in $X$ is defined to be the \emph{horizontal lift} of $\gamma(t)$ with respect to $\nabla^{1,0}$ if $\tilde{v}^{1,0}(t)=\nabla^{1,0}(v^{1,0})$ is the $(1,0)$-part of $\dot{\tilde{\gamma}}(t)$. Using local holomorphic coordinates $(s(t),z(t))$ of $\tilde{\gamma}(t)$, the horizontal lift is determined by
\begin{equation}\label{eq:horizontal_ODE}
\frac{\D z^\alpha}{\D t} = \Gamma_i^\alpha(s(t), z(t)) \frac{\D s^i}{\D t}.
\end{equation}
The \emph{parallel transport} $\tau(\gamma): X_{\gamma(0)} \to X_{\gamma(1)}$ maps an initial point $x_0$ to the endpoint of the unique horizontal lift starting at $x_0$, provided the solution exists for $t \in [0, 1]$.
\end{definition}
\begin{remark}\label{rmk:real_conn}
	Let $\nabla^{1,0}$ be a pure (1,0)-connection, and let $\nabla^{0,1}:=\widebar{\nabla^{1,0}}$. Then $\nabla:=\nabla^{1,0}+\nabla^{0,1}$ is the complexification of a real connection $\nabla^\RR$, i.e., a splitting of \eqref{RealTangentSeq_eq}. The horizontal lift defined above using $\nabla^{1,0}$ is identical to that defined using $\nabla^\RR$. In general, for a $(1,0)$-connection on a complex fiber bundle $(f,T_{X/S})$, we define the horizontal lift with respect to $\nabla^{1,0}$ to be that with respect to $\nabla^\RR$ whose complexification is $\nabla^{1,0}+\widebar{\nabla^{1,0}}$. Locally, it is determined by
	\begin{equation}\label{eq:horizontal_ODE1}
	\frac{\D z^\alpha}{\D t} = \Gamma_i^\alpha(s(t), z(t)) \frac{\D s^i}{\D t} + \overline{\Gamma_i^{\bar{\alpha}}(s(t),z(t))} \frac{\D \bar{s}^i}{\D t}.
	\end{equation}
\end{remark}

\begin{definition}
	A (1,0)-connection $\nabla^{1,0}$ on a complex fiber bundle is called \emph{complete} if the horizontal lift of any smooth path $\gamma: [0, 1] \to S$ starting at any point $x_0 \in X_{\gamma(0)}$ is defined for all $t \in [0, 1]$.
\end{definition}

\begin{remark}\label{rmk:cpt_complete}
	If the fiber $Y$ is compact, any (1,0)-connection is automatically complete. If $Y$ is noncompact, completeness is a nontrivial condition (see Example \ref{ex:Incomplete_Hol_Conn}).
\end{remark}

\begin{lemma}\label{lem:completeness_associated_connection}
Let $G$ be a complex Lie group acting holomorphically on a complex manifold $Y$ (not necessarily compact). Let $\pi_P: P \to S$ be a smooth principal $G$-bundle, and $f: X = P \times_G Y \to S$ be the associated isotrivial complex fiber bundle. If $A$ is a principal connection on $P$, then the induced $(1,0)$-connection $\nabla_A^{1,0}$ on $X$ is complete.
\end{lemma}

\begin{proof}
It is well-known that principal connections are complete (see \cite[Th.~10.4]{Sontz15}). Let $q: P \times Y \to X$ be the quotient map. Recall that the horizontal distribution $H_X^\RR$ of $\nabla_A^\RR$ is the image under $\D q$ of the distribution $H_P^\RR \oplus 0$ on $P \times Y$, where $H_P^\RR$ is the horizontal distribution of $A$.

Let $x_0 \in X_{\gamma(0)}$. We can represent $x_0$ as $[p_0, y_0]$ for some $p_0 \in P_{\gamma(0)}$ and $y_0 \in Y$. The horizontal lift $\tilde{\gamma}_P(t)$ of $\gamma(t)$ in $P$ starting at $p_0$ exists for all $t \in [0, 1]$. Consider the path $\tilde{\gamma}_X(t)$ in $X$ defined by
\[
\tilde{\gamma}_X(t) := q(\tilde{\gamma}_P(t), y_0) = [\tilde{\gamma}_P(t), y_0].
\]
We verify that this is the required horizontal lift. Clearly, $f(\tilde{\gamma}_X(t)) = \pi_P(\tilde{\gamma}_P(t)) = \gamma(t)$, and $\tilde{\gamma}_X(0) = [p_0, y_0] = x_0$. The tangent vector is $\tilde{\gamma}_X'(t) = \D q_{(\tilde{\gamma}_P(t), y_0)} (\tilde{\gamma}_P'(t), 0)$. Since $\tilde{\gamma}_P(t)$ is horizontal in $P$, $\tilde{\gamma}_P'(t) \in H_P^\RR$. Thus, $(\tilde{\gamma}_P'(t), 0) \in H_P^\RR \oplus 0$. By the definition of $H_X^\RR$, $\tilde{\gamma}_X'(t)$ lies in $H_X^\RR$. Since the horizontal lift $\tilde{\gamma}_X(t)$ is defined for all $t \in [0, 1]$, the induced connection $\nabla_A^{1,0}$ is complete.
\end{proof}

\begin{lemma}\label{lem:hol_parallel_transport}
Let $\nabla^{1,0}$ be a complete $(1,0)$-connection on a complex fiber bundle $f: X \to S$. Then $\nabla^{1,0}$ is relatively holomorphic and $\widebar{\nabla^{1,0}}$ is mixed relatively holomorphic if and only if, for every smooth path $\gamma:[0,1]\to S$ and every $t\in[0,1]$, the parallel transport map $\tau_t:X_{\gamma(0)}\to X_{\gamma(t)}$ is a biholomorphism.
\end{lemma}
\begin{proof}
It suffices to work in adapted local coordinates $(s,z)$, where the connection is given by smooth coefficients $\Gamma_i^\alpha(s, z)$ and $\Gamma_i^{\bar{\alpha}}(s,z)$. A path $(s(t), z(t))$ in $X$ is horizontal if \eqref{eq:horizontal_ODE1} is satisfied. Let $z(t; z_0)$ be the solution to \eqref{eq:horizontal_ODE1} with initial condition $z_0$.  Let $w^\alpha_\beta(t) = \frac{\partial z^\alpha}{\partial \bar{z}_0^\beta}(t)$. Then $w^\alpha_\beta(0)=0$. Differentiate \eqref{eq:horizontal_ODE1} with respect to $\bar{z}_0^\beta$,
\begin{equation}\label{eq:dbarz_ODE}
	\frac{\D w^\alpha_\beta}{\D t} = \biggl( \frac{\partial \Gamma_i^\alpha}{\partial z^\gamma} w^\gamma_\beta + \frac{\partial \Gamma_i^\alpha}{\partial \bar{z}^\gamma} \frac{\partial \bar{z}^\gamma}{\partial \bar{z}_0^\beta} \biggr) \frac{\D s^i}{\D t}+\biggl( \frac{\partial \overline{\Gamma_i^{\bar\alpha}}}{\partial z^\gamma} w^\gamma_\beta + \frac{\partial \overline{\Gamma_i^{\bar\alpha}}}{\partial \bar{z}^\gamma} \frac{\partial \bar{z}^\gamma}{\partial \bar{z}_0^\beta} \biggr) \frac{\D \bar{s}^i}{\D t}.
\end{equation}

If $\nabla^{1,0}$ is relatively holomorphic and $\widebar{\nabla^{1,0}}$ is mixed relatively holomorphic, then \eqref{eq:dbarz_ODE} simplifies to
\[ \frac{\D w^\alpha_\beta}{\D t} = \biggl(\frac{\partial \Gamma_i^\alpha}{\partial z^\gamma} \frac{\D s^i}{\D t}+\frac{\partial \overline{\Gamma_i^{\bar\alpha}}}{\partial z^\gamma} \frac{\D \bar{s}^i}{\D t}\biggr) w^\gamma_\beta . \]
By uniqueness of solutions, $w^\alpha_\beta(t) = 0$ for all $t$. Thus, $\tau_t$ is holomorphic. Since $\tau_t$ is bijective, it is a biholomorphism.

Conversely, suppose that for any smooth path $\gamma(t)$ in $S$, the parallel transport map $\tau_t: z_0 \mapsto z(t)$ is holomorphic. This implies that  $w^\alpha_\beta(t) = 0$, and then $\frac{\D w^\alpha_\beta}{\D t}(0) = 0$. By \eqref{eq:dbarz_ODE},
\[
0 = \frac{\partial \Gamma_i^\alpha}{\partial \bar{z}^\beta}(s(0), z_0) \frac{\D s^i}{\D t}(0)+\frac{\partial \overline{\Gamma_i^{\bar \alpha}}}{\partial \bar{z}^\beta}(s(0), z_0) \frac{\D \bar{s}^i}{\D t}(0).
\]
Since we can choose a path passing through any point $(s(0), z_0)$ with any velocity, we must have
\[
\frac{\partial \Gamma_i^\alpha}{\partial \bar{z}^\beta} = \frac{\partial \overline{\Gamma_i^{\bar \alpha}}}{\partial \bar{z}^\beta} = 0. \qedhere
\]
\end{proof}

\begin{lemma}\label{lem:hol_triv_via_conn}
Let $\nabla^{1,0}$ be a complete holomorphic connection on a surjective holomorphic fibration $f: X \to S$. For any $s_0 \in S$, there exists a holomorphic coordinate ball $U$ around $s_0$ such that the bundle $f^{-1}(U) \to U$ can be holomorphically trivialized via parallel transports along radial paths in $U$.
\end{lemma}
\begin{proof}
	Let $U$ be a holomorphic coordinate ball centered at $s_0$. Let $Y=X_{s_0}$. Define $\Phi: U \times Y \to f^{-1}(U)$ by $\Phi(s, z_0) = \tau(\gamma_s)(z_0)$, where $\gamma_s(t)=ts$ is the radial path. The ODE \eqref{eq:horizontal_ODE} becomes
	\begin{equation}\label{eq:ODE_with_parameters}
	\frac{\D z^\alpha}{\D t} = \Gamma_i^\alpha(ts, z(t)) s^i.
	\end{equation}
	Since $\Gamma_i^\alpha$ are holomorphic functions, by holomorphic dependence on parameters, the solution $z(t; s, z_0)$ is holomorphic in $(s, z_0)$. Thus $\Phi$ is a holomorphic map. By Lemma \ref{lem:hol_parallel_transport}, its restriction to $\{s\} \times Y$ is a biholomorphism onto $X_s$. So $\Phi$ is a bijection, and is a biholomorphism as desired.
\end{proof}

\begin{lemma}\label{lem:Path_Independence_Flat}
Let $\nabla^{1,0}$ be a complete, flat holomorphic connection on a holomorphic fiber bundle $f: X \to S$. Let $\gamma_0, \gamma_1: [0, 1] \to S$ be two smooth paths that are path-homotopic. Then $\tau(\gamma_0) = \tau(\gamma_1)$.
\end{lemma}
\begin{proof}
Let $\mathcal{F}$ be the holomorphic transversal foliation determined by $\nabla^{1,0}$. Let $H(t, u): [0, 1] \times [0, 1] \to S$ be a smooth homotopy between $\gamma_0$ and $\gamma_1$. Fix $x_0 \in X_{s_0}$. Let $\widetilde{H}(t, u)$ be the unique horizontal lift of $\gamma_u(t):=H(t,u)$ starting at $x_0$. By the smooth dependence of solutions of ODEs on parameters, the map $\widetilde{H}: [0, 1] \times [0, 1] \to X$ is smooth.

Let $L_{x_0}$ be the unique leaf of the foliation $\mathcal{F}$ passing through $x_0$. By construction, $$\widetilde{H}([0, 1] \times [0, 1]) \subset L_{x_0}.$$  Define the path $\beta: [0, 1] \to X$ by $\beta(u) = \widetilde{H}(1, u)$. We have $$f(\beta(u)) = f(\widetilde{H}(1, u)) = H(1, u) = s_1.$$ Thus, $\beta(u)\in L_{x_0} \cap X_{s_1}$. For every $u$, $\dot\beta(u)\in (T_{\beta(u)}L_{x_0})\cap (T_{\beta(u)}X_{s_1})=\{0\}$, since the foliation is transverse to the fibers. Therefore, $\beta$ is constant, and $\beta(0) = \beta(1)$. Since $x_0$ is arbitrary, $\tau(\gamma_0) = \tau(\gamma_1)$.
\end{proof}

\begin{remark}
	Alternatively, by Remark \ref{rmk:real_conn}, it suffices to show the associated real connection $\nabla^\RR$ is flat. Since $\nabla^{1,0}$ is flat, we have $F^{2,0} = F^{0,2} = 0$. We compute
\[ [H_i, H_{\bar{j}}] = [\partial_i + \Gamma_i^\alpha \partial_\alpha, \partial_{\bar{j}} + \overline{\Gamma_j^\beta} \partial_{\bar{\beta}}]= \bigl( \partial_i \overline{\Gamma_j^\beta} + \Gamma_i^\alpha \partial_\alpha \overline{\Gamma_j^\beta} \bigr) \partial_{\bar{\beta}} - \bigl( \partial_{\bar{j}} \Gamma_i^\alpha + \overline{\Gamma_j^\beta} \partial_{\bar{\beta}} \Gamma_i^\alpha \bigr) \partial_\alpha.\]
If $\nabla^{1,0}$ is a holomorphic connection, then $[H_i, H_{\bar{j}}]=0$. Thus, for a holomorphic connection, $F^{1,1}=0$ automatically. Therefore, $\nabla^{\RR}$ is flat as wanted.

If $\nabla^{1,0}$ is merely relatively holomorphic, then $\partial_{\bar{\beta}} \Gamma_i^\alpha = 0$, $\partial_\alpha \overline{\Gamma_j^\beta} = 0$. The formula simplifies to
\[
[H_i, H_{\bar{j}}] = (\partial_i \overline{\Gamma_j^\beta}) \partial_{\bar{\beta}} - (\partial_{\bar{j}} \Gamma_i^\alpha) \partial_\alpha.
\]
This is generally nonzero.
\end{remark}

\begin{proposition}\label{prop:constant_transition_functions}
If a holomorphic connection $\nabla^{1,0}$ on $f: X \to S$ is complete and flat, then the fiber bundle $X$ admits a holomorphic atlas such that the transition functions are locally constant maps to $\Aut(Y)$.
\end{proposition}
\begin{proof}
 Let $\{U_a\}$ be a cover of $S$ by coordinate balls. Let $s_a \in U_a$ be the center and identify $Y = X_{s_a}$. Since $\nabla^{1,0}$ is complete and holomorphic, we obtain a local holomorphic trivialization $\Phi_a: U_a \times Y \to f^{-1}(U_a)$ by Lemma \ref{lem:hol_triv_via_conn}. This collection of charts forms a holomorphic atlas for $X$.

 If we fix $z_0 \in Y$, the map $s \mapsto \Phi_a(s, z_0)$ defines a local section. Due to path independence, this section traces out the unique leaf of the foliation $\mathcal{F}$ passing through $z_0$ at $s_a$. By definition, this is a horizontal section. Therefore, in the coordinates induced by $\Phi_a$, the horizontal sections are precisely the constant sections $\{z_a = \text{constant}\}$. A trivialization with this property is called a \emph{flat trivialization}.

Let $U_{ab} = U_a \cap U_b$, $g_{ab}: U_{ab}\to \Aut(Y)$ be the transition map, given by $$g_{ab}(s) = \Phi_a(s, \cdot)^{-1} \comp \Phi_b(s, \cdot).$$ Fix $z_b \in Y$. The map $s \mapsto (s, z_b)$ represents a horizontal section in the $\Phi_b$-trivialization. In the $\Phi_a$-trivialization, this horizontal section is represented by $(s,g_{ab}(s)(z_b))$. Hence $g_{ab}(s)$ is locally constant.
\end{proof}
\begin{lemma}\label{lem:fiberwise_holomorphic_flat_morphism}
Let $(f_i:X_i\to S,\nabla_i^{1,0})$, $i=1,2$, be holomorphic fiber bundles equipped with holomorphic connections, and let $H_i^\RR$ be the corresponding real horizontal distributions. For a smooth morphism $F$, i.e., a smooth map $F:X_1\to X_2$ satisfying $f_2\comp F=f_1$, the following conditions are equivalent:
\begin{enumerate}[label=(\roman*)]
\item $F$ is a fiberwise holomorphic flat morphism, i.e., each restriction $F_s:X_{1,s}\to X_{2,s}$ is holomorphic and $\D F(H_1^\RR)\subset H_2^\RR$;
\item $F$ is holomorphic and preserves the holomorphic horizontal distributions.
\end{enumerate}
\end{lemma}

\begin{proof}
Fiberwise holomorphicity implies that $\D F$ is complex-linear on the vertical tangent bundle. Since $f_2\comp F=f_1$ and $F$ preserves the real horizontal distributions, it maps the horizontal lift of a real tangent vector $v\in TS^\RR$ to the horizontal lift of the same vector $v$. The holomorphicity of the connections implies that this restriction of $\D F$ is also complex-linear. Thus $\D F$ is complex-linear on the whole tangent bundle, and $F$ is holomorphic. The converse is immediate.
\end{proof}

\begin{theorem}\label{thm:HolomorphicRH}
Let $S$ be a connected complex manifold. Let $\mathbf{CFB}(S)$ be the category whose objects are holomorphic fiber bundles $f:X\to S$ equipped with complete flat holomorphic connections $\nabla^{1,0}$. Its morphisms are fiberwise holomorphic flat morphisms. Denote $\Gamma:=\pi_1(S,s_0)$. Let $\mathbf{REP}(\Gamma)$ be the category whose objects are pairs $(Y, \rho)$, where $Y$ is a complex manifold and $\rho: \Gamma \to \Aut(Y)$ is a representation, and whose morphisms are $\Gamma$-equivariant holomorphic maps $\psi:Y_1\to Y_2$ (i.e., $\psi \comp \rho_1(\gamma) = \rho_2(\gamma) \comp \psi$ for all $\gamma \in \Gamma$). Then the functor
\[\mathsf{F}: \mathbf{CFB}(S) \xrightarrow{\sim} \mathbf{REP}(\Gamma)\]
defined on objects by $\mathsf{F}(X, \nabla^{1,0}) = (X_{s_0}, \rho)$, where $\rho$ is the monodromy representation, and on morphisms by $\mathsf F(F)=F|_{X_{1,s_0}}$, is an equivalence of categories. 
\end{theorem}

\begin{proof}
Consider a representation $\rho: \Gamma \to G=\Aut(Y)$. Let $\pi: \widetilde{S} \to S$ be the universal covering of $S$. Define an action of $\Gamma$ on the product $\widetilde{S} \times Y$ by
\[
\gamma\cdot (\tilde{s}, y)  = (\gamma\cdot \tilde{s}, \rho(\gamma)y), \quad \gamma \in \Gamma.
\]
Since the actions on $\widetilde{S}$ and $Y$ are holomorphic, this diagonal action is holomorphic, free, and properly discontinuous. The quotient space $X := (\widetilde{S} \times Y)/\Gamma$ (denoted $\widetilde{S} \times_\rho Y$) is a complex manifold. The projection $f: X \to S$ induced by $\pi$ is a holomorphic fiber bundle with fiber $Y$.

We define a flat holomorphic connection on $X$. Consider the trivial connection $\nabla_0$ on the projection $\widetilde{S} \times Y \to \widetilde{S}$. Its horizontal distribution is $H_0 = T\widetilde{S} \oplus 0 \subset T(\widetilde{S} \times Y)$, which is a holomorphic subbundle. Let $L_\gamma$ denote the action of $\gamma$. The differential acts on a horizontal vector $(v, 0) \in H_0(\tilde{s}, y)$ as
\[
\D L_\gamma (v, 0) = (\D( \gamma \cdot ) v, \D(\rho(\gamma)) 0) = (\D(\gamma\cdot) v, 0).
\]
This vector is again horizontal. Thus, $H_0$ descends to a holomorphic distribution $H$ on $X$, defining a holomorphic connection $\nabla^{1,0}$ on $f: X \to S$. Since the curvature of the trivial connection $\nabla_0$ is zero, and the curvature is a tensor that descends to the quotient, the curvature of $\nabla^{1,0}$ is also zero. Moreover, $\nabla^{1,0}$ is complete. Let $\gamma: [0, 1] \to S$ be a path and $x_0 \in X_{\gamma(0)}$. Lift $\gamma$ to a path $\tilde{\gamma}$ in the universal cover $\widetilde{S}$. We can represent $x_0$ as $[(\tilde{\gamma}(0), y_0)]$ for some $y_0 \in Y$. The horizontal lift of $\gamma$ starting at $x_0$ is the projection of the trivial horizontal lift in $\widetilde{S} \times Y$, given explicitly by $t \mapsto [(\tilde{\gamma}(t), y_0)]$. Since $\tilde{\gamma}(t)$ is defined for all $t \in [0, 1]$, the horizontal lift exists for the entire interval, proving that $\nabla^{1,0}$ is complete.

Conversely, let $(X, \nabla^{1,0})$ be a complete flat holomorphic fiber bundle over $S$ with fiber $Y$. Fix a base point $s_0 \in S$ and identify the fiber $X_{s_0}$ with $Y$. For any loop $\gamma$ based at $s_0$, the parallel transport $\tau(\gamma): Y \to Y$ is defined by lifting $\gamma$ horizontally. By Lemma \ref{lem:hol_parallel_transport}, $\tau(\gamma)$ is a biholomorphism, i.e., $\tau(\gamma) \in G$. Since $\nabla^{1,0}$ is flat, $\tau(\gamma)$ depends only on the homotopy class of $\gamma$. This defines the monodromy representation
\[
\rho: \pi_1(S, s_0) \to G, \quad \rho([\gamma]) = \tau(\gamma)^{-1}.
\]

We verify that these constructions are inverse to each other up to isomorphism. Suppose we start with a representation $\rho$ and construct the complete flat holomorphic bundle $(X, \nabla^{1,0}) = \widetilde{S} \times_\rho Y$. Fix $s_0 \in S$ and a lift $\tilde{s}_0 \in \widetilde{S}$. Identify $Y \cong X_{s_0}$ via $y \mapsto [(\tilde{s}_0, y)]$. A loop $\gamma$ lifts to a path from $\tilde{s}_0$ to $[\gamma]\cdot\tilde{s}_0$. The parallel transport of $y$ along $\gamma$ is $t \mapsto [(\tilde{\gamma}(t), y)]$. The endpoint is
$[([\gamma]\cdot\tilde{s}_0, y)] = [(\tilde{s}_0, \rho([\gamma])^{-1}y)]$.
Thus the holonomy of $(X, \nabla^{1,0})$ recovers $\rho$.

Suppose we start with a complete flat holomorphic bundle $(X, \nabla^{1,0})$, and let $\rho$ be its monodromy representation. We construct an isomorphism $\Psi: \widetilde{S} \times_\rho Y \to X$. Let $\alpha$ be a path in $\widetilde{S}$ from $\tilde{s}_0$ to $\tilde{s}$, and $\pi(\alpha)$ its projection in $S$. Define $\Phi: \widetilde{S} \times Y \to X$ by $\Phi(\tilde{s}, y) = \tau(\pi(\alpha))(y)$, using the identification $Y\cong X_{s_0}$. By Lemma \ref{lem:Path_Independence_Flat}, $\Phi$ is well-defined. Since $\nabla^{1,0}$ is holomorphic, by Lemma \ref{lem:hol_triv_via_conn}, $\Phi$ is holomorphic. $\Phi$ maps the trivial horizontal distribution on $\widetilde{S} \times Y$ to the horizontal distribution of $\nabla^{1,0}$. One verifies that $\Phi$ respects the action of $\Gamma$ using the definition of $\rho$, so it descends to an isomorphism $\Psi$.

It remains to show that for any two complete flat bundles $(X_1, \nabla_1^{1,0})$ and $(X_2, \nabla_2^{1,0})$ with fibers $Y_1, Y_2$ over $s_0$ and corresponding monodromy representations $\rho_1, \rho_2$, the map
	\[
	\mathsf{F}_{1,2}: \Hom_{\mathbf{CFB}}((X_1, \nabla_1^{1,0}), (X_2, \nabla_2^{1,0})) \to \Hom_{\mathbf{REP}}((Y_1, \rho_1), (Y_2, \rho_2))
	\]
	given by restricting a bundle morphism to the fiber at $s_0$ ($F \mapsto F|_{X_{1, s_0}}$) is a bijection. First, $\mathsf F_{1,2}$ is well-defined. Indeed, preservation of
horizontal lifts gives $F_{s_0}\comp\tau_1(\gamma^{-1})=\tau_2(\gamma^{-1})\comp F_{s_0}$ for every loop $\gamma$ based at $s_0$. Since $\rho_i([\gamma])=\tau_i(\gamma)^{-1}=\tau_i(\gamma^{-1})$, we obtain
\[
F_{s_0}\comp\rho_1([\gamma])=\rho_2([\gamma])\comp F_{s_0}.
\]
Thus $F_{s_0}$ is an equivariant holomorphic map and hence a
morphism in $\mathbf{REP}(\Gamma)$.

Injectivity: Suppose $F, G: X_1 \to X_2$ are two morphisms such that $F|_{X_{1, s_0}} = G|_{X_{1, s_0}}$. Let $x \in X_1$ be any point. Since $S$ is connected and the connection is complete, there exists a horizontal path $\tilde{\gamma}$ in $X_1$ connecting some $y \in X_{1, s_0}$ to $x$. Let $\gamma$ be the projection of this path on $S$. Since morphisms preserve horizontal distributions, $F(\tilde{\gamma})$ and $G(\tilde{\gamma})$ must be the unique horizontal lifts of $\gamma$ in $X_2$ starting at $F(y)$ and $G(y)$ respectively. Since $F(y) = G(y)$, uniqueness implies $F(x) = G(x)$. Thus $F=G$.

Surjectivity: Let $\psi: Y_1 \to Y_2$ be a morphism in $\mathbf{REP}(\Gamma)$. We construct a bundle morphism $F$. Recall that $X_i \cong (\widetilde{S} \times Y_i)/\Gamma$. Define
	\[
	\widetilde{F}: \widetilde{S} \times Y_1 \to \widetilde{S} \times Y_2, \quad (\tilde{s}, y) \mapsto (\tilde{s}, \psi(y)).
	\]
	$\widetilde{F}$ is holomorphic and preserves the trivial horizontal distribution (as it acts as the identity on $\widetilde{S}$). We verify it descends to the quotients. For any $\gamma \in \Gamma$,
	\[
	\widetilde{F}(\gamma \cdot (\tilde{s}, y)) = \widetilde{F}(\gamma\tilde{s}, \rho_1(\gamma)y) = (\gamma\tilde{s}, \psi(\rho_1(\gamma)y))=\gamma \cdot (\tilde{s}, \psi(y)) = \gamma \cdot \widetilde{F}(\tilde{s}, y),
	\]
	by the $\Gamma$-equivariance of $\psi$. Thus $\widetilde{F}$ descends to a holomorphic bundle morphism $F: X_1 \to X_2$. Since $\widetilde{F}$ preserves the trivial horizontal distributions, $F$ preserves the induced horizontal distributions. By construction, $F|_{X_{1, s_0}} = \psi$.
\end{proof}
Notice that for a holomorphic fiber bundle $f: X\to S$, the proof of Proposition \ref{prop:exist_fiberwise_holo_connection} shows that there exists a relatively holomorphic pure $(1,0)$-connection $\nabla^{1,0}$ on $f$. By Proposition \ref{CurvClassIndep_prop}, there is a well-defined curvature class $[F^{1,1}_{D}]$.
\begin{corollary}\label{cor:RH_RiemannSurface}
Let $S$ be a connected Riemann surface and $Y$ be a compact complex manifold. Let $f: X \to S$ be a holomorphic fiber bundle with typical fiber $Y$. If $[F^{1,1}_{D}]$ vanishes, then this bundle arises from a representation $\rho: \Gamma \to \Aut(Y)$.
\end{corollary}
\begin{proof}
By Theorem \ref{thm:Ext_class_curv_class}, $f: X \to S$ admits a holomorphic connection, which is complete as $Y$ is compact. Let $\nabla_{\mathrm{hol}}$ be such a connection. Its curvature $F_{\nabla_{\mathrm{hol}}}$ is of type $(2,0)$ which must vanish since $S$ is a Riemann surface. Then the statement follows from Theorem \ref{thm:HolomorphicRH}.
\end{proof}

\begin{lemma}\label{lem:flat_char_torsion}
Let $S$ be a compact connected Riemann surface and $G$ be a connected complex reductive group. Let $P$ be a flat principal $G$-bundle over $S$. Then the characteristic class $c(P) \in \pi_1(G)$ (\cite[Prop.~5.1]{ramanathan75stable}) is a torsion element.
\end{lemma}

\begin{proof}
    The flat $G$-bundle $P$ over $S$ is determined (up to isomorphism) by a homomorphism $\rho: \pi_1(S) \to G$. Let $g$ be the genus of $S$. The fundamental group of $S$ is generated by $a_1, b_1, \dots, a_g, b_g$ subject to the relation $\prod_{i=1}^g [a_i, b_i] = e$, where $[x, y] = xyx^{-1}y^{-1}$ denotes the commutator. Since $\rho$ is a homomorphism, the images $A_i = \rho(a_i)$ and $B_i = \rho(b_i)$ in $G$ satisfy the relation $\prod_{i=1}^g [A_i, B_i] = e_G$. Let $\widetilde{G}$ be the universal cover of $G$. By \cite[Prop.~6.1]{ramanathan75stable}, \[c(P) = \prod_{i=1}^g [\widetilde{A}_i, \widetilde{B}_i],\] where $\widetilde{A}_i, \widetilde{B}_i \in \widetilde{G}$ are arbitrary lifts of $A_i, B_i$. The characteristic element is given by the product of their commutators in $\widetilde{G}$, which is well-defined since the kernel $\pi_1(G)$ of $\widetilde{G}\to G$ is central.

    Since $G$ is a connected complex reductive group, $\widetilde{G}=\widetilde{Z}_0\times \widetilde{G}'$, where $\widetilde{G}'$ is the universal cover of the semisimple part $G'=[G,G]$, $\widetilde{Z}_0$ is the universal cover of the identity component $Z_0$ of the center of $G$. We have
    $c(P) \in \{e\} \times \widetilde{G}'$. Since $\pi_1(G')$ is finite \cite[Th.~7.1]{theodor_dieck85}, $c(P)$ is a torsion element in $\pi_1(G)$.
\end{proof}

\begin{corollary}\label{cor:Weil}
Let $S$ be a compact connected Riemann surface and $G$ be a connected complex reductive group. Let $f:X\to S$ be a holomorphic fiber bundle with structure group $G$, i.e., it is the associated bundle of a holomorphic principal $G$-bundle $\pi_P:P\to S$. Then $f$ admits a holomorphic connection if each summand in the Remak decomposition of $\ad P$ has degree zero and $c(P)\in \pi_1(G)$ is torsion. The converse holds if $H^0(Y,TY)$ is finite-dimensional and $G\le \Aut(Y)$.
\end{corollary}
\begin{proof}
	By Corollary \ref{cor:affine_map_associated_hol}, $f$ admits a holomorphic connection if $\pi_P$ does. Let $\chi: G \to \mathbb{C}^*$ be a holomorphic character. It induces a homomorphism $\chi_*: \pi_1(G) \to \pi_1(\mathbb{C}^*) \cong \mathbb{Z}$. By \cite[Rmk.~5.1]{ramanathan75stable},
    $\deg(L_\chi) = \chi_*(c(P))$, where $L_\chi = P \times_\chi \CC$. Assume $c(P)$ is a torsion element in $\pi_1(G)$. There exists a positive integer $k$ such that $k \cdot c(P) = 0$. For any $\chi$, we have
    \[
    k \cdot \deg(L_\chi) = k \cdot \chi_*(c(P)) = \chi_*(k \cdot c(P)) = \chi_*(0) = 0.
    \]
Then $\deg(L_\chi) = 0$. Note that by Weil's theorem, $\ad P$ admits a holomorphic connection if and only if each summand in the Remak decomposition of $\ad P$ has degree zero. If moreover $\ad P$ admits a holomorphic connection, then so does $\pi_P$ by \cite[Th.~3.1]{hassan_biswas03}.

Suppose $H^0(Y,TY)$ is finite-dimensional, $G\le \Aut(Y)$, and $f$ admits a holomorphic connection. By Corollary \ref{cor:assoc_holconn_imply_princ_holconn}, $P$ also admits a holomorphic connection which is automatically flat on a Riemann surface. By Lemma \ref{lem:flat_char_torsion}, $c(P)$ is torsion. By \cite[Th.~3.1]{hassan_biswas03}, $\ad P$ also admits a holomorphic connection.
\end{proof}

\begin{example}\label{ex:Incomplete_Hol_Conn}
Let the base be $S = \mathbb{C}$ (with coordinate $s$) and the fiber be $Y = \mathbb{C}$ (with coordinate $z$). Consider the trivial holomorphic bundle $f: X = S \times Y \to S$. We define a connection $\nabla^{1,0}$ on $X$ by
\[
\nabla^{1,0}(\partial_s) = \partial_s + z^2 \partial_z.
\]
The connection coefficient is $\Gamma(s, z) = z^2$, which is holomorphic, so $\nabla^{1,0}$ is a holomorphic connection. Furthermore, since the base $S$ is 1-dimensional, its curvature vanishes automatically. Thus, $(X, \nabla^{1,0})$ is a flat holomorphic bundle.

 A curve $(s(t), z(t))$ in $X$ is horizontal if
\begin{equation}\label{eq:parallel_transport_noncompact}
\frac{\D z}{\D t} = z^2 \frac{\D s}{\D t}.
\end{equation}
Consider the path $\gamma: [0, 1] \to S$ defined by $\gamma(t) = t$. Here $s(t)=t$ and $\frac{\D s}{\D t}=1$. The ODE \eqref{eq:parallel_transport_noncompact} becomes $\frac{\D z}{\D t} = z^2$.
The solution with initial condition $z(0)=z_0\in X_0$ is $z(t;z_0)=z_0/(1-tz_0)$ on its interval of existence. It exists for all $t\in[0,1]$ precisely when $z_0\notin[1,\infty)$, and on this domain the parallel transport is $\tau(\gamma)(z_0)=z_0/(1-z_0)$. The solution starting at $z_0=1$ is $z(t; 1) = \frac{1}{1-t}$, which blows up as $t \to 1^-$. The connection $\nabla^{1,0}$ is incomplete.

We can further demonstrate that $\nabla^{1,0}$ does not come from any representation $\Gamma\to \Aut(Y)$. The base $S=\mathbb{C}$ is simply connected, so $\Gamma=\{e\}$. The only representation is the trivial one, $\rho_{triv}$. The flat holomorphic bundle corresponding to $\rho_{triv}$ is the trivial bundle $X$ equipped with the trivial connection $\nabla_{triv}$, whose horizontal distribution is spanned by $\partial_s$. We show that $(X, \nabla^{1,0})$ cannot be isomorphic to $(X, \nabla_{triv})$.

Suppose there exists an isomorphism of flat bundles $\Psi: (X, \nabla_{triv}) \to (X, \nabla^{1,0})$. It must be a fiber-preserving biholomorphism $\Psi(s, w) = (s, \psi(s, w))$ that intertwines the connections. This requires
\[
\D\Psi(\partial_s) = \partial_s + \frac{\partial \psi}{\partial s} \partial_z=\partial_s + \psi(s, w)^2 \partial_z,
\]
which is equivalent to $\frac{\partial \psi}{\partial s} = \psi^2$. The solution is $\psi(s, w) = \frac{g(w)}{1-s g(w)}$, where $g(w) = \psi(0, w)$. Since $\Psi$ must be a global isomorphism, $g(w)$ must be in $\Aut(\mathbb{C})$, i.e., $g(w)=aw+b$ with $a\neq 0$. For any $s\neq 0$, there exists a $w_0 \in \mathbb{C}$ such that $g(w_0) = 1/s$. At the point $(s, w_0)$, the function $\psi(s, w)$ is singular. Thus, $\Psi$ is not globally defined on $X=\mathbb{C}^2$.
\end{example}

\section{Relatively Kähler fibrations}
The classical nonabelian Hodge correspondence establishes a profound connection between flat bundles and Higgs bundles. The existence of a canonical metric, the so-called harmonic metric, is the crucial analytical ingredient establishing the correspondence. Motivated by the generalization of this correspondence to nonlinear settings, we find it essential to investigate metrics on fiber bundles.

\subsection{Fiberwise Kähler metrics}
\begin{definition}\label{def:fib_Riem}
	A \emph{fiberwise Riemannian metric} $g_{X/S}$ on a smooth fiber bundle $f:X\to S$ is an element in $C^\infty (X,\mathrm{Sym}^2((T_{X/S}^\mathbb{R})^*))$, such that for each $x \in X$, $g_x$ is a positive definite inner product on $T_{X/S, x}^\mathbb{R}$.
\end{definition}
\begin{definition}
Let $(f:X\to S,T_{X/S})$ be a complex fiber bundle.	A \emph{fiberwise Kähler metric} $\omega_{X/S}$ on $(f,T_{X/S})$ is a smooth section of $\wedge^{1,1}T^*_{X/S}$, such that for every $s \in S$, the restriction $\omega_s := \omega_{X/S}|_{X_s}$ is a Kähler metric on the fiber $X_s$.
\end{definition}
A fiberwise Kähler metric determines a fiberwise Riemannian metric $g_{X/S}(\sbullet,\sbullet)=\omega_{X/S}(\sbullet,J_{X/S}\sbullet)$ in the sense of Definition \ref{def:fib_Riem}.  Similar to Lemma \ref{lem:isotrivial_cx_fib_bun}, we have the following.
\begin{lemma}\label{lem:unitary atlas}
Let $f: X \to S$ be an isotrivial complex fiber bundle with a compatible atlas $\{(U_a, \Phi_a)\}$. Let $\omega_Y$ be a fixed Kähler form on $Y$.
If the transition functions $g_{ab}$ actually take values in the subgroup of holomorphic isometries $\Aut(Y, \omega_Y)$, then there exists a unique fiberwise Kähler metric $\omega_{X/S}$ on $X$ such that $\omega_s = \Phi_{a,s}^* \omega_Y$ for $s\in U_a$. We call such an atlas a \emph{unitary atlas}.
\end{lemma}

\begin{definition}
	A fiberwise Kähler metric $\omega_{X/S}$ on an isotrivial complex fiber bundle $f:X\to S$ is said to be \emph{modeled on $(Y, \omega_Y)$} if there exists a unitary atlas for $(f,\omega_{X/S})$.
\end{definition}

\begin{definition}\label{def:Kah_conn}
	Let $(f,T_{X/S})$ be a complex fiber bundle with a $\dbar$-operator $\dbar_f$ and a fiberwise Kähler metric $\omega_{X/S}$. A \emph{symplectic connection} associated to $\dbar_f$ and $\omega_{X/S}$ is a $(1,0)$-connection which is pure with respect to $\dbar_f$ and preserves $\omega_{X/S}$, i.e., $\mathcal{L}_{H_v} \omega_{X/S} = 0$ for the horizontal lift $H_v$ of any real vector field $v$ on $S$. The almost connection induced by a symplectic connection is called a \emph{symplectic almost connection}. If it further preserves $J_{X/S}$ (by Lemma \ref{lem:hol_parallel_transport}, this is equivalent to the condition that it is relatively holomorphic and $\dbar_f$ satisfies the lifting condition), or equivalently, is compatible with $g_{X/S}$, it is called a \emph{Kähler connection}. It induces a \emph{Kähler almost connection}.
\end{definition}

Given a real $2$-form $\omega$ on $X$ whose restriction to fibers is $\omega_{X/S}$, by the fiberwise nondegeneracy, it induces the horizontal distribution
\begin{equation}
	H_\omega=\{ v\in TX^\RR \,|\, (\iota_v \omega)|_{T_{X/S}^\RR}=0\}.
\end{equation}
This defines a real connection $\nabla_\omega^\RR$. Its complexification $\nabla_\omega$ decomposes as $\nabla_\omega^{1,0}+\nabla_\omega^{0,1}$. Conversely, given a real connection $\nabla$, we can define a $2$-form $\omega_\nabla$ by
\begin{equation}\label{eq:omega_conn}
	\omega_\nabla(u,v)=\omega_{X/S}(\mathrm{pr}_{\nabla}^{\mathrm{v}} u,\mathrm{pr}_{\nabla}^{\mathrm{v}} v),
\end{equation}
where $\mathrm{pr}_{\nabla}^{\mathrm{v}}: TX^\RR\to T_{X/S}^\RR$ is the projection induced by $\nabla$. The following result is straightforward.

\begin{lemma}\label{lem:omega_compatible_dbar}
Let $(f,T_{X/S})$ be a complex fiber bundle equipped with a $\bar{\partial}$-operator $\dbar_f$ and a fiberwise Kähler metric $\omega_{X/S}$. Let $\omega$ be a real 2-form on $X$ which restricts to $\omega_{X/S}$ on fibers. $\nabla_\omega^{0,1}$ induces $\dbar_f$ iff
\begin{equation}\label{eq:nabla_omega_pure_condition}
    \omega(u, v) = 0 \quad \text{for all } u \in \widebar{TX} \text{ and } v \in \overline{T_{X/S}}.
\end{equation}
 Equivalently, $\iota_v (\omega^{0,2})= 0$ for any $v \in \widebar{T_{X/S}}$, where $\omega^{0,2}$ is the $(0,2)$-part of $\omega$. Here the type decomposition is given by the almost complex structure associated to $\dbar_f$. Conversely, if $\nabla$ is a real connection, and the $(0,1)$-part of its complexification induces $\dbar_f$, then $\omega_\nabla$ is a real $(1,1)$-form, locally given by
 \begin{equation}\label{eq:omega_conn_loc}
 \omega_\nabla=\I g_{\alpha\bar{\beta}}(\D z^\alpha-\Gamma_i^\alpha\D s^i-\Gamma_{\bar j}^\alpha\D\bar s^j)\wedge(\D\bar z^\beta-\Gamma_i^{\bar\beta}\D s^i-\Gamma_{\bar j}^{\bar\beta}\D\bar s^j),
\end{equation}
where the connection coefficients are as in \eqref{local10conn_eq} and \eqref{local01conn_eq}, and $\omega_{X/S}=\I g_{\alpha\bar{\beta}}\D z^\alpha\wedge\D\bar z^\beta$.
\end{lemma}
\begin{remark}
	Note that $\omega^{2,0}=\widebar{\omega^{0,2}}$ since $\omega$ is real. In particular, \eqref{eq:nabla_omega_pure_condition} holds if $\omega$ is a real $(1,1)$-form on $X$. On the other hand, if $\omega$ satisfies  \eqref{eq:nabla_omega_pure_condition}, then the horizontal distribution $H_{\omega}$ defined by $\omega$ is identical to the horizontal distribution $H_{\omega^{1,1}}$ defined by its $(1,1)$-component $\omega^{1,1}$.
\end{remark}

\begin{lemma}[{\cite[Lem.~6.3.5]{mcduff17introduction}}]\label{lem:verti_closed}
	Let $\omega$ be a real $2$-form on $X$ which restricts to a fiberwise symplectic form $\omega_{X/S}$. Then $\nabla_\omega^\RR$ is symplectic, i.e., preserves $\omega_{X/S}$, if and only if $\omega$ is \emph{vertically closed} in the sense that $\D \omega(v_1,v_2,\cdot)=0$ for all vertical tangent vectors $v_1,v_2\in T_{X/S}^\RR$.
\end{lemma}

\begin{lemma}\label{lem:space_Kah_conn}
	Whenever nonempty, the space of Kähler connections on a fiberwise Kähler complex fiber bundle $(f:X\to S,\omega_{X/S})$ associated to $\omega_{X/S}$ and a $\dbar$-operator $\dbar_f$ satisfying the lifting condition is an affine space modeled on $A^{1,0}(S,\mathfrak{a}_{X/S})$, which consists of $(1,0)$-forms on $S$ taking values in vertical parallel $(1,0)$-vector fields.
\end{lemma}
\begin{proof}
This follows from Lemma \ref{lem:RH_space} and the fact that if $u$ and $Ju$ are infinitesimal automorphisms of a Kähler manifold, then $u$ is parallel \cite[Lem.~ III.1]{kobayashi72transformation}.
\end{proof}

Let $Y$ be a complex manifold. Given a holomorphic vector field $v\in H^0(Y,TY)$, let $v_\RR:=v+\bar{v}$ be its real part. Let $\omega_Y$ be a Kähler metric. We define
\begin{align*}
	\mathfrak{aut}(Y,\omega_Y)&:=\{ v\in H^0(Y,TY)\,\mid\, v_\RR \text{ is Killing}  \},\\
	\mathfrak{aut}(Y,\omega_Y)_\CC&:=\{ v+\I w \,\mid\, v,w\in \mathfrak{aut}(Y,\omega_Y)\}=\mathfrak{aut}(Y,\omega_Y)+\I\,\mathfrak{aut}(Y,\omega_Y).
\end{align*}
\begin{proposition}\label{prop: kahler connection}
Let $S$ be connected and let $f: X \to S$ be a complex fiber bundle with a fiberwise Kähler metric $\omega_{X/S}$ and a $\dbar$-operator $\dbar_f$ satisfying the lifting condition. If it admits a unitary atlas $\{(U_a,\Phi_a)\}$ such that on each $U_a$, $\dbar_f-\dbar_a\in A^{0,1}(U_a,\mathfrak{aut}(Y,\omega_Y)_\CC)$, where $\dbar_a$ is the trivial $\dbar$-operator on the unitary trivialization $(f^{-1}(U_a),\omega_{X/S})\cong U_a\times (Y,\omega_Y)$, then there exists a Kähler connection $\nabla^{1,0}$ on $X$ associated to $\dbar_f$ and $\omega_{X/S}$. The converse holds when the Kähler connection $\nabla^{1,0}$ is complete.
\end{proposition}
\begin{proof}
	By Lemma \ref{lem:space_Kah_conn}, it suffices to work locally on a unitary trivialization. By assumption we may write
	\[\dbar_f-\dbar_a=v+\I w,\]
	where $v,w\in A^{0,1}(U_a,\mathfrak{aut}(Y,\omega_Y))$ (possibly in a non-unique way). We define an almost connection $\partial$ by
	\begin{equation}\label{eq:kah_al_conn_loc}
		(\partial-\partial_a)(\partial_i)=v(\partial_{\bar{i}})-\I w(\partial_{\bar{i}}),
	\end{equation}
where $\partial_a$ is the trivial almost connection determined by the unitary trivialization. Then $\hat{D}=\partial+\dbar_f$ corresponds to a Kähler connection on $U_a$. Using a partition of unity subordinate to $\{U_a\}$, we obtain a global Kähler connection $\nabla^{1,0}$.

Conversely, let $\nabla^{1,0}$ be a complete Kähler connection, which is equivalent to $\partial+\dbar_f$. Using the radial parallel transports as in Lemma \ref{lem:hol_triv_via_conn}, we obtain a unitary atlas. In each unitary chart $U_a$, we have $(\partial-\partial_a)(\partial_i)+(\dbar_f-\dbar_a)(\partial_{\bar{i}})\in \mathfrak{aut}(Y,\omega_Y)$ and $(\partial-\partial_a)(\partial_i)-(\dbar_f-\dbar_a)(\partial_{\bar{i}})\in\I\, \mathfrak{aut}(Y,\omega_Y)$. Thus, $\dbar_f-\dbar_a\in A^{0,1}(U_a,\mathfrak{aut}(Y,\omega_Y)_\CC)$.
\end{proof}

\begin{definition}\label{def:k_compatible_unitary_atlas}
Let $\mathfrak{k}\subset\mathfrak{aut}(Y,\omega_Y)$ be a real subspace satisfying $\mathfrak{k}\cap\I\mathfrak{k}=\{0\}$, and write $\mathfrak{k}^\CC=\mathfrak{k}\oplus\I\mathfrak{k}$. For a unitary atlas $\mathcal U=\{(U_a,\Phi_a)\}$, write $g_{ab}(s)=\Phi_{a,s}\comp\Phi_{b,s}^{-1}$. For $s\in U_{ab}$ and $v\in T_sS^\RR$, define
\[
\mathsf{g}_{ab,s}(v):=\Bigl(\frac{\D}{\D t}\Big|_{t=0}g_{ab}(s(t))\comp g_{ab}(s)^{-1}\Bigr)^{1,0},
\]
where $s(0)=s$ and $\dot s(0)=v$. Then $\mathsf{g}_{ab,s}(v)$ is a holomorphic vector field on $Y$. We call $\mathcal U$ \emph{$\mathfrak{k}$-compatible} if $(g_{ab}(s))_*\mathfrak{k}=\mathfrak{k}$ and $\mathsf{g}_{ab,s}(v)\in\mathfrak{k}$ for every $U_{ab}$, every $s\in U_{ab}$, and every $v\in T_sS^\RR$. Two $\mathfrak{k}$-compatible unitary atlases are called \emph{equivalent} if all transition maps between their charts, restricted to the corresponding overlaps, satisfy these same conditions.
\end{definition}

\begin{proposition}\label{prop:Chern_conn}
Given the setup of Proposition \ref{prop: kahler connection}, suppose further that $\{(U_a,\Phi_a)\}$ is a $\mathfrak{k}$-compatible atlas and that $B_a:=\dbar_f-\dbar_a\in A^{0,1}(U_a,\mathfrak{k}^\CC)$. Then this atlas determines a Kähler connection, called the \emph{Chern connection}, which is independent of the choice of equivalent $\mathfrak{k}$-compatible unitary atlases.
\end{proposition}
\begin{proof}
Define $\phi_{\mathfrak{k}}:\mathfrak{k}^\CC\to\mathfrak{k}^\CC$ by $\phi_{\mathfrak{k}}(p+\I q)=p-\I q$ for $p,q\in\mathfrak{k}$. For $B\in A^{0,1}(S,\mathfrak{k}^\CC)$, define $\phi B\in A^{1,0}(S,\mathfrak{k}^\CC)$ by $(\phi B)(u)=\phi_{\mathfrak{k}}(B(\bar u))$ for $u\in TS$. Set $A_a:=\phi B_a$ and define the local almost connection by $\partial=\partial_a+A_a$.

Extend $\mathsf{g}_{ab}$ complex-linearly in the base tangent variable and decompose it into types. The condition $\mathsf{g}_{ab,s}(v)\in\mathfrak{k}$ for real $v$ gives $\mathsf{g}_{ab}^{1,0}=\phi(\mathsf{g}_{ab}^{0,1})$. Moreover, $(g_{ab})_*\mathfrak{k}=\mathfrak{k}$ implies that push-forward commutes with $\phi_{\mathfrak{k}}$. The transformation law for $\dbar_f$ is $B_a=(g_{ab})_*B_b+\mathsf{g}_{ab}^{0,1}$. Applying $\phi$ therefore yields $A_a=(g_{ab})_*A_b+\mathsf{g}_{ab}^{1,0}$. This is precisely the transformation law for the coefficients of an almost connection, so the local almost connections glue. Together with $\dbar_f$, they determine a global pure $(1,0)$-connection $\nabla^{1,0}$ and its underlying real connection. For $v\in T_sS^\RR$ and $s\in U_a$, the vertical $(1,0)$-component of its horizontal lift is
\[
A_a(v^{1,0})+B_a(v^{0,1})=\phi_{\mathfrak{k}}\bigl(B_a(v^{0,1})\bigr)+B_a(v^{0,1})\in\mathfrak{k}.
\]
Therefore, $\nabla^{1,0}$ is a Kähler connection. The same transformation calculation applies to transitions between equivalent atlases, proving the independence.
\end{proof}

\begin{example}\label{ex:Chern_conn_atlas_dependence}
Let $S=\CC_s$ and let $f:X=S\times\CC_z\to S$ carry the standard product holomorphic structure, its canonical $\dbar$-operator $\dbar_f$, and the fiberwise Kähler metric $\omega_s=\I\,\D z\wedge\D\bar z$. Take the model fiber $(Y,\omega_Y)=(\CC_\zeta,\I\,\D\zeta\wedge\D\bar\zeta)$ and the real subspace $\mathfrak{k}=\RR\partial_\zeta\subset\mathfrak{aut}(Y,\omega_Y)$, which satisfies $\mathfrak{k}\cap\I\mathfrak{k}=\{0\}$. Consider the two global trivializations $\Phi_z(s,z)=(s,z)$, $\Phi_w(s,z)=(s,z+s)$, and the corresponding single-chart atlases $\mathcal U_z=\{(S,\Phi_z)\}$ and $\mathcal U_w=\{(S,\Phi_w)\}$, which are $\mathfrak{k}$-compatible. Moreover, the product $\dbar$-operators $\dbar_z$ and $\dbar_w$ determined by these trivializations both coincide with $\dbar_f$, so $\dbar_f-\dbar_z=\dbar_f-\dbar_w=0$. However, these atlases are not equivalent in the sense of Definition \ref{def:k_compatible_unitary_atlas}. Indeed, their transition map is $g_{wz}(s)(\zeta)=\zeta+s$. Although $(g_{wz}(s))_*\mathfrak{k}=\mathfrak{k}$, writing $s=x+\I y$ gives $\mathsf{g}_{wz}(\partial_x)=\partial_\zeta$ and $\mathsf{g}_{wz}(\partial_y)=\I\partial_\zeta\notin\mathfrak{k}$. Proposition \ref{prop:Chern_conn} gives the product connection in each respective trivialization. They are distinct K\"ahler connections.
\end{example}

\begin{proposition}
    Let $f:X\to S$ be a complex fiber bundle with a fiberwise Kähler metric modeled on $(Y,\omega_Y)$ and a $\dbar$-operator $\dbar_f$ satisfying the lifting condition. If it admits a Kähler connection $\nabla^{1,0}$, then in any local unitary trivialization there exist $P_{ij},Q_{ij}\in\mathfrak{aut}(Y,\omega_Y)$ such that the curvature of the induced Kähler almost connection $\partial$ satisfies
\[F_\partial^{2,0}(\partial_i,\partial_j)=P_{ij}-\I Q_{ij},\qquad F_{\dbar_f}^{0,2}(\partial_{\bar{i}},\partial_{\bar{j}})=P_{ij}+\I Q_{ij}.\]
If, moreover, for every $s\in S$ the real subspace $\mathfrak c_s:=\operatorname{span}_\RR\{P_{ij}(s),Q_{ij}(s):i<j\}\subset H^0(Y,TY)$ satisfies $\mathfrak c_s\cap\I\mathfrak c_s=\{0\}$, then $F_{\partial}^{2,0}=0$ if and only if $F_{\dbar_f}^{0,2}=0$. In particular, if $\nabla^{1,0}$ is the Chern connection determined by a $\mathfrak{k}$-compatible unitary atlas, and $\mathfrak k_{(2)}\cap i\mathfrak k_{(2)}=\{0\}$, where $\mathfrak k_{(2)}:=\mathfrak k+\operatorname{span}_{\mathbb R}\{[u,v]:u,v\in\mathfrak k\}$, then $F_{\partial}^{2,0}=0$ if and only if $F_{\dbar_f}^{0,2}=0$.
\end{proposition}
\begin{proof}
    We work in a local unitary trivialization. Write $\partial(\partial_i)=\partial_i+A_i\mod\widebar{T_{X/S}}$ and $\dbar_f(\partial_{\bar{i}})=\partial_{\bar{i}}+B_i\mod\widebar{T_{X/S}}$. The vertical parts of the real horizontal lifts of $\partial_{x^i}$ and $\partial_{y^i}$ are respectively $2(v_i+\bar v_i)$ and $2(w_i+\bar w_i)$, where $s^i=x^i+\I y^i$, $v_i:=\tfrac12(A_i+B_i)$ and $w_i:=\tfrac{1}{2\I}(B_i-A_i)$. Since $\nabla^{1,0}$ is Kähler, $v_i,w_i\in\mathfrak{aut}(Y,\omega_Y)$. Thus $\dbar_f$ is given by $\partial_{\bar{i}}\mapsto\partial_{\bar{i}}+(v_i+\I w_i)\mod\widebar{T_{X/S}}$, and $\partial$ is given by $\partial_i\mapsto\partial_i+(v_i-\I w_i)\mod\widebar{T_{X/S}}$. We compute
    \begin{align*}
        F_{\dbar_f}^{0,2}(\partial_{\bar{i}},\partial_{\bar{j}})&=[\partial_{\bar{i}}+v_i+\I w_i,\partial_{\bar{j}}+v_j+\I w_j]\\
        &=\partial_{\bar{i}}(v_j+\I w_j)+[v_i+\I w_i,v_j+\I w_j]-\partial_{\bar{j}}(v_i+\I w_i)\\
        &=\tfrac{1}{2}(\partial_{x^i}v_j-\partial_{y^i}w_j)+[v_i,v_j]-[w_i,w_j]-\tfrac{1}{2}(\partial_{x^j}v_i-\partial_{y^j}w_i)\\
        &\quad +\I\bigl(\tfrac{1}{2}(\partial_{y^i}v_j+\partial_{x^i}w_j)+[w_i,v_j]+[v_i,w_j]-\tfrac{1}{2}(\partial_{y^j}v_i+\partial_{x^j}w_i)\bigr).
    \end{align*}
On the other hand,
    \begin{align*}
         F_{\partial}^{2,0}(\partial_{i},\partial_{j})&=[\partial_{i}+v_i-\I w_i,\partial_{j}+v_j-\I w_j]\\
        &=\partial_{i}(v_j-\I w_j)+[v_i-\I w_i,v_j-\I w_j]-\partial_j(v_i-\I w_i)\\
        &=\tfrac{1}{2}(\partial_{x^i}v_j-\partial_{y^i}w_j)+[v_i,v_j]-[w_i,w_j]-\tfrac{1}{2}(\partial_{x^j}v_i-\partial_{y^j}w_i)\\
        &\quad -\I\bigl(\tfrac{1}{2}(\partial_{y^i}v_j+\partial_{x^i}w_j)+[w_i,v_j]+[v_i,w_j]-\tfrac{1}{2}(\partial_{y^j}v_i+\partial_{x^j}w_i)\bigr).
    \end{align*}
 The result follows.
\end{proof}

\subsection{Relatively Kähler fibrations}
Let $f:X\to S$ be a holomorphic fibration as in Section \ref{sec:hol_conn}. We call it a \emph{relatively Kähler fibration} if there is a $\D$-closed real $(1,1)$-form $\omega$ on $X$ whose restriction to each fiber $X_s$ is Kähler, i.e., it gives rise to a fiberwise Kähler metric $\omega_{X/S}$. The form $\omega$ determines a smooth splitting of \eqref{HolTangentSeq_eq1}, i.e. a pure $(1,0)$-connection $\nabla_\omega^{1,0}$, given by
\begin{align}
	\partial_i&\mapsto \partial_i+\Gamma_i^\alpha\partial_\alpha=:H_i,\quad \Gamma_i^\alpha=-g^{\bar{\beta}\alpha}g_{i\bar{\beta}},\quad \text{where} \label{eq:rel_kah_lift}\\
	\omega&=\I (g_{\alpha\bar{\beta}}\,\D z^\alpha \wedge \D \bar{z}^\beta+g_{\alpha \bar{j}}\,\D z^\alpha \wedge \D\bar{s}^j+g_{i\bar{\beta}}\,\D s^i\wedge\D\bar{z}^\beta+g_{i\bar{j}}\,\D s^i\wedge\D\bar{s}^j),\label{eq:rel_kah_coeff}
\end{align}
and $(g^{\bar{\beta}\alpha})$ denotes the inverse matrix of $(g_{\alpha\bar{\beta}})$. The horizontal lift $H_i$ of $\partial_i$ with respect to $\omega$ is orthogonal to $T_{X/S}^\CC$, with respect to the Hermitian form
\[\langle v_1,v_2\rangle_\omega:=\omega(v_1,J\widebar{v_2}). \]
By Lemma \ref{lem:omega_compatible_dbar} and Lemma \ref{lem:verti_closed}, $\nabla_\omega^{1,0}$ is a symplectic connection associated to $\omega_{X/S}$ and the canonical $\dbar$-operator $\dbar_f$. We call $\omega$ \emph{relatively holomorphic} if the induced connection $\nabla_{\omega}^{1,0}$ is relatively holomorphic in the sense of \eqref{RelHolConn_eq}. In this case $\nabla_{\omega}^{1,0}$ is a Kähler connection. This is a nontrivial condition, as the following example illustrates.
\begin{example}
	Consider the trivial disk bundle $f:\CC\times \Delta\to \CC$, $(s,z)\mapsto s$, where $\Delta:=\{z\in\CC\mid  |z|< 1\}$. Let
	\[\phi=|z|^2+2|s|^2+\tfrac{1}{2}(\bar{s}z^2+s\bar{z}^2). \]
	Then $\omega=\I\partial\dbar\phi$ is a Kähler form on $X$. By \eqref{eq:rel_kah_lift}, $\Gamma_s^z=-\bar{z}$, so $\omega$ is not relatively holomorphic. Consider a path $\gamma(t)=t$ in $\CC$, then its horizontal lift is determined by the ODE $\frac{\D z}{\D t}=-\bar{z}$. The resulting local parallel transport is given on its domain of definition by $\tau_t(x+\I y)=\E^{-t}x+\I\E^t y$, which is not holomorphic for $t\ne0$.
\end{example}

By Proposition \ref{prop:exist_fiberwise_holo_connection}, if $f$ is proper and $\omega$ is relatively holomorphic, then the restriction of $f$ over each connected component of $S$ is a holomorphic fiber bundle. In the following we give some examples of relatively Kähler fiber bundles where $\omega$ is relatively holomorphic.

\begin{example}
If $f:(X,\omega_X)\to (S,\omega_S)$ is a relatively Kähler fibration, which is a complex nilpotent fibration in the sense of \cite[Definition 5.4]{rollenske20}, then $\omega_X$ is relatively holomorphic. Indeed, by \cite[Definition 3.5]{rollenske20}, there exists an orthonormal local frame $\{V_1, \dots, V_{m+n}\}$ of $TX$ such that $\{V_j\}_{j\le m}$ spans $T_{X/S}$ and $\{V_k\}_{k>m}$ spans the image of $\nabla_{\omega_X}^{1,0}$, which is spanned by $\{H_i\}$ in the above notation, and
\begin{equation}\label{CNCSBracket_eq}
[V_j, V_k] = 0 \quad \text{and} \quad [V_j, \widebar{V}_k] = 0, \quad \text{for all } 1\le j\le m \text{ and } 1\le k\le m+n.
\end{equation}
By Lemma \ref{lem:curv_def_rel_hol}, it suffices to show that for any vertical $(0,1)$-vector field $\widebar{V}$ and any horizontal $(1,0)$-vector field $W$, the Lie bracket $[\widebar{V}, W]$ has no vertical $(1,0)$-component. We write $\widebar{V} = \sum_{j\le m} a_j \widebar{V}_j$ and $W = \sum_{k>m} b_k V_k$, where $a_j, b_k$ are smooth local functions. Then
\[
[\widebar{V}, W] = \sum_{j\le m, k>m} [a_j \widebar{V}_j, b_k V_k] = \sum_{j\le m, k>m} \left( a_j \widebar{V}_j(b_k) V_k - b_k V_k(a_j) \widebar{V}_j + a_j b_k [\widebar{V}_j, V_k] \right).
\]
The nilpotent condition \eqref{CNCSBracket_eq} states that $[V_j, \widebar{V}_k]=0$ for $j\le m$ and all $k$. Taking the complex conjugate yields $[\bar{V}_j, V_k]=0$. Substituting this back,
\[
[\widebar{V}, W] = \sum_{k>m} \widebar{V}(b_k) V_k - \sum_{j\le m} W(a_j) \widebar{V}_j,
\]
which has no vertical $(1,0)$-component. The claim is proved.
\end{example}

\begin{example}\label{ex:chern_conn}
	Let $f:E\to S$ be a holomorphic vector bundle with a Hermitian metric $h$. Let $\phi\in C^\infty(E)$ be defined by $\phi(v):=|v|_h^2$. Let $\omega:=\I \partial \bar{\partial} \phi$, then $\omega$ is relatively Kähler and relatively holomorphic. In a local holomorphic trivialization $E|_U\cong U\times \CC^m$, we have $\phi(s, z) = h_{\alpha\bar{\beta}}(s) z^\alpha \bar{z}^\beta$, and
\[
    \omega = \I \bigl( h_{\alpha\bar{\beta}} \D z^\alpha\wedge \D \bar{z}^\beta +  (\bar{z}^\beta \partial_{\bar{j}} h_{\alpha\bar{\beta}} ) \D z^\alpha \wedge \D \bar{s}^j + (z^\alpha \partial_i h_{\alpha\bar{\beta}} )\D s^i \wedge \D \bar{z}^\beta + (z^\alpha \bar{z}^\beta \partial_i \partial_{\bar{j}} h_{\alpha\bar{\beta}}) \D s^i \wedge \D \bar{s}^j \bigr).
\]
This implies
	\[\Gamma_i^\alpha=-h^{\bar{\beta}\alpha}z^\gamma\partial_i h_{\gamma\bar{\beta}}, \]
where $h_{\alpha\bar{\beta}}$ only depends on $(s,\bar s)$. In fact, $\nabla_{\omega}^{1,0}$ is exactly induced by the Chern connection associated to $h$ and the holomorphic structure of $E$. The Chern connection has connection $1$-form $h^{\bar{\beta}\alpha}\partial_i h_{\gamma\bar{\beta}}\D s^i$, then by Example \ref{ex:vb_conn_dbar}, the coefficients $\Gamma_i^\alpha$ of the induced $(1,0)$-connection have the above form.

Let $\nabla_{\omega}^\RR$ be the real connection determined by $\nabla_{\omega}^{1,0}$ and $\omega_{\nabla}$ be the corresponding $2$-form given by \eqref{eq:omega_conn}. By \eqref{eq:omega_conn_loc}, we have
\begin{align*}
    \omega_\nabla &= \I h_{\alpha\bar{\beta}} \bigl( \D z^\alpha \wedge \D \bar{z}^\beta - \D z^\alpha \wedge (\widebar{\Gamma_j^\beta} \D \bar{s}^j) - (\Gamma_i^\alpha \D s^i) \wedge \D \bar{z}^\beta + (\Gamma_i^\alpha \D s^i) \wedge (\widebar{\Gamma_j^\beta} \D \bar{s}^j) \bigr)\\
		&=\I \bigl( h_{\alpha\bar{\beta}} \D z^\alpha\wedge \D \bar{z}^\beta +  (\bar{z}^\beta \partial_{\bar{j}} h_{\alpha\bar{\beta}} ) \D z^\alpha \wedge \D \bar{s}^j + (z^\alpha \partial_i h_{\alpha\bar{\beta}} )\D s^i \wedge \D \bar{z}^\beta \\
        &\phantom{=\I \bigl( h_{\alpha\bar{\beta}} \D z^\alpha\wedge \D \bar{z}^\beta +  (\bar{z}^\beta \partial_{\bar{j}} h_{\alpha\bar{\beta}} ) \D z^\alpha \wedge \D \bar{s}^j } + (h^{\bar{\delta}\gamma} (\partial_i h_{\alpha\bar{\delta}}) (\partial_{\bar{j}} h_{\gamma\bar{\beta}}) z^\alpha \bar{z}^\beta) \D s^i \wedge \D \bar{s}^j \bigr).
\end{align*}
Then
\[  \omega - \omega_\nabla = \I \bigl( \partial_i \partial_{\bar{j}} h_{\alpha\bar{\beta}} - h^{\bar{\delta}\gamma} (\partial_i h_{\alpha\bar{\delta}}) (\partial_{\bar{j}} h_{\gamma\bar{\beta}}) \bigr) z^\alpha \bar{z}^\beta \D s^i \wedge \D \bar{s}^j=-\I \langle F_h z, z \rangle_h,\]
where $F_h\in A^{1,1}(S,\End E)$ is the curvature of the Chern connection.
\end{example}

\subsection{Associated relatively Kähler bundles}\label{subsec:assoc_rel_kahler}
We now investigate the construction of relatively Kähler metrics on associated bundles using the symplectic coupling mechanism. This construction applies whenever the action of the compact subgroup used for the reduction is Hamiltonian. We show that such induced metrics are naturally relatively holomorphic.

Let $(Y, \omega_Y)$ be a Kähler manifold (not necessarily compact) and $S$ be a complex manifold. Let $G$ be a complex Lie group acting holomorphically on $Y$. Let $\pi_P:P\to S$ be a smooth principal $G$-bundle. We consider the associated isotrivial complex fiber bundle $f: X = P\times_G Y\to S$.

A point $p \in P_s$ defines a biholomorphism $\phi_p: Y \to X_s$ by $\phi_p(y) = [p, y]$. Note that $\phi_{p\cdot g} = \phi_p \comp g$.
\begin{definition}\label{def:fib_kahler_model}
	A fiberwise Kähler metric $\omega_{X/S}$ is said to be \textbf{$G$-modeled on $(Y, \omega_Y)$} if for every $p \in P$, $\phi_p^* \omega_s \in G\cdot \omega_Y := \{ (g^{-1})^* \omega_Y \mid g \in G\}$. In particular, $(X_s,\omega_s)$ is biholomorphically isometric to $(Y,\omega_Y)$ via some $\phi_p$. $\omega_{X/S}$ is said to be \textbf{$G$-modeled on $(Y, [\omega_Y])$} if for every $p \in P$, $\phi_p^* \omega_s \in [\omega_Y]$.
\end{definition}
The following result is analogous to the lifting theorem of \cite[\S3]{CH2000}. 
\begin{lemma}\label{lem:smooth_orbit_lifting}
Let a Lie group $G$ act smoothly on a Hausdorff Fréchet space $E$, and let $\eta:U\to E$ be a smooth map from a smooth manifold whose image is contained in one $G$-orbit. For every $s_0\in U$, after shrinking $U$ around $s_0$, there is a smooth map $g:U\to G$ such that $g(s_0)=e$ and $\eta(s)=g(s)\cdot\eta(s_0)$.
\end{lemma}
\begin{proof}
Set $\eta_0=\eta(s_0)$ and $H=G_{\eta_0}$. The stabilizer $H$ is closed, and the orbit map $j:G/H\to E$ is an injective immersion. Since $G/H$ is finite-dimensional, finitely many continuous linear functionals on $E$ give local coordinates for $j$ by the inverse function theorem. In particular, $j$ is locally a smooth embedding.

For $\zeta$ in the orbit, let $V_\zeta\subset E$ be the image of the infinitesimal action at $\zeta$. We first show that every smooth curve $c:I\to E$ with image in the orbit satisfies $c'(t)\in V_{c(t)}$. Choose compact subsets $C_\nu$ covering $G/H$. The closed sets $c^{-1}(j(C_\nu))$ cover $I$, and the union of their interiors is dense by the Baire category theorem. On each such interior, $j^{-1}\comp c$ is continuous, since $j|_{C_\nu}$ is a homeomorphism onto its image, and is then smooth by the local coordinate description of $j$. Thus $c'(t)\in V_{c(t)}$ on a dense open subset. Near any fixed parameter $t_0\in I$, choose infinitesimal generators whose values form a basis of $V_{c(t)}$. Their linear independence persists locally, and they continue to span $V_{c(t)}$ because its dimension is constant along the orbit. Continuous linear functionals on $E$ give a continuously varying projection onto this span. The tangency condition therefore extends to every parameter by continuity.

Applying this to curves in $U$ shows that $\D\eta_s(T_sU^\RR)\subset V_{\eta(s)}$. The same choice of generators and linear functionals gives, locally, a smooth $\mathfrak g$-valued $1$-form $b$ such that $(b_s(v))_E(\eta(s))=\D\eta_s(v)$, where $\xi_E(\zeta):=\frac{\D}{\D t}\big|_{t=0}\exp(t\xi)\cdot\zeta$. Take real coordinates centered at $s_0=0$ and shrink the coordinate neighborhood to a ball. For $s$ sufficiently small, solve
\[
\frac{\D g_s(t)}{\D t}=(\D R_{g_s(t)})_e\,b_{ts}(s),\qquad g_s(0)=e,\qquad 0\leq t\leq1,
\]
where $R_g$ denotes right multiplication on $G$. These finite-dimensional ordinary differential equations have solutions on $[0,1]$ for small $s$, depending smoothly on $(s,t)$. Differentiating $g_s(t)^{-1}\cdot\eta(ts)$ shows that it is constant in $t$. Hence $g(s):=g_s(1)$ has the required properties.
\end{proof}

\begin{lemma}\label{lem:smooth_G_unitary_atlas}
A smooth fiberwise Kähler metric $\omega_{X/S}$ is $G$-modeled on $(Y,\omega_Y)$ if and only if there are local smooth sections $\sigma_a:U_a\to P$ satisfying $\phi_{\sigma_a(s)}^*\omega_s=\omega_Y$, i.e., it admits a unitary atlas induced by local smooth sections of $P$.
\end{lemma}
\begin{proof}
Suppose that the unitary trivialization $\Phi_a$ is induced by a local section $\sigma_a: U_a \to P$. We have $\Phi_{a, s}^{-1}(y) = [\sigma_a(s), y] = \phi_{\sigma_a(s)}(y)$.
The unitary condition means $\omega_s = \Phi_{a, s}^* \omega_Y$, which is equivalent to $\phi_{\sigma_a(s)}^* (\omega_s) = \omega_Y$.
Now let $p \in P_s$ be arbitrary, we can write $p = \sigma_a(s) \cdot g$ for some $g\in G$. Then
\[
\phi_p^* (\omega_s) = (\phi_{\sigma_a(s)} \comp g)^* (\omega_s) = g^*  \phi_{\sigma_a(s)}^* (\omega_s)= g^*\omega_Y\in G\cdot \omega_Y.
\]

Conversely, suppose that $\omega_{X/S}$ is $G$-modeled on $(Y,\omega_Y)$. Fix $s_0\in S$ and choose a smooth local section $\sigma:U\to P$, adjusted by a constant element of $G$ so that $\phi_{\sigma(s_0)}^*\omega_{s_0}=\omega_Y$. The forms $\eta_s:=\phi_{\sigma(s)}^*\omega_s$ define a smooth map into the Fréchet space of smooth real two-forms on $Y$, with its compact-open $C^\infty$ topology. Their values lie in the orbit of $\omega_Y$ for the smooth action $g\cdot\eta=(g^{-1})^*\eta$. By Lemma \ref{lem:smooth_orbit_lifting}, after shrinking $U$, there is a smooth map $g:U\to G$ such that $\eta_s=(g(s)^{-1})^*\omega_Y$. The section $\tilde{\sigma}(s):=\sigma(s)g(s)$ then satisfies $\phi_{\tilde{\sigma}(s)}^*\omega_s=g(s)^*\eta_s=\omega_Y$.
\end{proof}

Further assume that $G$ is reductive, with a compact real form $K$ preserving $\omega_Y$. Let $\mathfrak{k}$, $\mathfrak{g}$ be the Lie algebras of $K$, $G$. Then $\mathfrak g=\mathfrak k\oplus\I\mathfrak k$. $K$ meets every connected component of $G$.
\begin{lemma}\label{lem:fib_kahler_reduction}
	If $P$ admits a reduction of the structure group to $K$, given by a principal $K$-subbundle $P_K \subset P$, then it induces a fiberwise Kähler metric $\omega_{X/S}$ on $X$ which is $G$-modeled on $(Y, \omega_Y)$.
	Conversely, if $X$ admits a fiberwise Kähler metric $\omega_{X/S}$ $G$-modeled on $(Y, \omega_Y)$, then it induces a reduction of the structure group of $P$ to $H:=\Stab_G(\omega_Y):=\{g\in G \,|\, (g^{-1})^*\omega_Y=\omega_Y\}$, given by a principal $H$-subbundle $P_H \subset P$. Furthermore, the $K$-reductions of $P$ inducing the prescribed metric $\omega_{X/S}$ are precisely the $K$-reductions contained in $P_H$ (exist if and only if $P_H/K\to S$ admits a smooth section).
\end{lemma}

\begin{proof}
$\Rightarrow$: Let $\omega_s :=(\phi_p^{-1})^*\omega_Y$, for $p\in (P_K)_s$. We verify this is well-defined. Let $p, p' \in (P_K)_s$. There is a unique $k\in K$ such that $p' = p\cdot k$. We have $\phi_{p'} = \phi_p \comp k$. Then
\[
\phi_{p'}^* \omega_s = (\phi_p \comp k)^* \omega_s = k^* (\phi_p^* \omega_s) = k^* \omega_Y=\omega_Y.\]
The metric $\omega_s$ is well-defined on each fiber and is $G$-modeled on $(Y, \omega_Y)$. Since $P_K$ is a smooth subbundle, this construction yields a smooth fiberwise Kähler metric $\omega_{X/S}$.

$\Leftarrow$:
Suppose $\omega_{X/S}$ is a fiberwise Kähler metric $G$-modeled on $(Y, \omega_Y)$. Define
\[ P_H := \{p \in P \mid \phi_p^* \omega_s = \omega_Y, \text{ where } s=\pi_P(p) \}. \]
 Let $p \in P_H$ and $g \in G$, then
$\phi_{p\cdot g}^* \omega_s = g^* (\phi_p^* \omega_s) = g^* \omega_Y$. Thus, $p\cdot g \in P_H$ if and only if $g^* \omega_Y = \omega_Y$, which is equivalent to $g \in H$. Therefore, the right action of $H$ on each fiber of $P_H$ is free and transitive. The stabilizer $H$ is a closed Lie subgroup of $G$. By Lemma \ref{lem:smooth_G_unitary_atlas}, there are local smooth sections $\sigma_a:U_a\to P$ satisfying $\phi_{\sigma_a(s)}^*\omega_s=\omega_Y$, and $P_H|_{U_a}=\sigma_a(U_a)H$. Hence $P_H$ is a smooth principal $H$-subbundle of $P$.

When $P_K\subset P_H$ is a $K$-reduction, let $\omega'_{X/S}$ be the fiberwise metric it induces. For every $p\in(P_K)_s$, we have $\phi_p^*\omega_s=\omega_Y=\phi_p^*\omega'_s$, so $\omega'_{X/S}=\omega_{X/S}$. Conversely, if a $K$-reduction $P_K\subset P$ induces $\omega_{X/S}$, then $\phi_p^*\omega_s=\omega_Y$ for every $p\in P_K$, and therefore $P_K\subset P_H$.
\end{proof}

\begin{remark}
Since $G$ is reductive and $K$ is a maximal compact subgroup, $G/K$ is contractible. Hence $P$ admits a smooth $K$-reduction, and the above lemma shows that $X$ always admits a smooth fiberwise Kähler metric $G$-modeled on $(Y,\omega_Y)$. This does not imply that an arbitrarily prescribed $G$-modeled metric is induced by a $K$-reduction.

If $H$ is compact, since $K\subset H$ and $K$ is a maximal compact subgroup of $G$, we have $H=K$ and the correspondence is bijective between smooth $K$-reductions of $P$ and smooth fiberwise Kähler metrics $G$-modeled on $(Y,\omega_Y)$. In particular, if $Y$ is compact and the action homomorphism $G\to\Aut(Y)$ is proper, then $H$ is compact. This applies when $G\le \Aut(Y)$ is a closed complex Lie subgroup.
\end{remark}

Now we further assume that $P$ is a holomorphic principal $G$-bundle. Then the associated bundle $f:X\to S$ is a holomorphic fiber bundle. Suppose $P$ admits a reduction of the structure group to $K$, called a \emph{Hermitian structure}. Let $P_K$ be the corresponding principal $K$-bundle. By \cite{singer59}, there is a unique complex unitary connection (called the \emph{(principal) Chern connection}) $A\in A^{1,0}(P,\mathfrak g)$ on $P$ with respect to the Hermitian structure, i.e., $A$ is induced by a principal connection $A_K$ on $P_K$. The construction also applies when $G$ and $K$ are disconnected.

\begin{lemma}\label{lem:assoc_Chern_conn}
	Let $\pi_P:P\to S$ be a holomorphic principal $G$-bundle, where $G=K^\CC$ is reductive with $K$ acting by Kähler isometries on a Kähler manifold $(Y,\omega_Y)$. Let $f:X=P\times_G Y\to S$ be the associated bundle. Then for any Hermitian structure $P_K$, the Chern connection $A$ on $P$ induces a complete Kähler connection $\nabla_A^{1,0}$ on $f$, compatible with the fiberwise metric $\omega_{X/S}$ induced by $P_K$ and the canonical $\dbar$-operator $\dbar_f$. We call it the \emph{associated Chern connection} determined by $(P,P_K)$. If $\mathfrak k_Y:=\tau_0(\mathfrak k)$ satisfies $\mathfrak k_Y\cap\I\mathfrak k_Y=\{0\}$, this is the Chern connection of Proposition \ref{prop:Chern_conn} relative to the $\mathfrak k_Y$-compatible atlas class induced by local smooth sections of $P_K$. In particular, this holds if $\tau_0$ is injective. In all cases, $F_{\partial_A}^{2,0}=0$.
\end{lemma}
\begin{proof}
The parallel transport maps of the induced connection are Kähler isometries, since they are induced by $A_K$. Compatibility of $A$ with the holomorphic structure of $P$ implies that $\nabla_A^{1,0}$ is pure with respect to $\dbar_f$, and it is relatively holomorphic by Lemma \ref{lem:b_rel_holo_associated}. Hence $\nabla_A^{1,0}$ is a Kähler connection. Completeness follows from Lemma \ref{lem:completeness_associated_connection}.

Suppose that $\mathfrak k_Y\cap\I\mathfrak k_Y=\{0\}$. Choose local smooth sections $\sigma_a:U_a\to P_K$, with induced unitary trivializations $\Phi_{a,s}=\phi_{\sigma_a(s)}^{-1}$, and write $\sigma_b=\sigma_a k_{ab}$ for $k_{ab}:U_{ab}\to K$.
Then $g_{ab}(s)$ is the action of $k_{ab}(s)$ on $Y$, and the $G$-equivariance of $\tau_0$ gives $(g_{ab}(s))_*\mathfrak k_Y=\tau_0(\Ad_{k_{ab}(s)}\mathfrak k)=\mathfrak k_Y$. For $v\in T_sS^\RR$, choose $s(t)$ with $s(0)=s$ and $\dot s(0)=v$ and put $\xi_{ab,s}(v):=\frac{\D}{\D t}\big|_{t=0}\bigl(k_{ab}(s(t))k_{ab}(s)^{-1}\bigr)\in\mathfrak k$. Then we have $\mathsf{g}_{ab,s}(v)=-\tau_0(\xi_{ab,s}(v))\in\mathfrak k_Y$. The same calculation applies to transitions between any two systems of local smooth sections of $P_K$, so these unitary atlases form a single $\mathfrak k_Y$-compatible atlas class. Let $\phi_Y$ be the conjugation on $\mathfrak k_Y^\CC$ with fixed-point set $\mathfrak k_Y$. Then $\phi_Y\comp\tau_0=\tau_0\comp\phi_K$. In a smooth $P_K$-trivialization, write the principal Chern connection form as $b+\phi_K(b)$, where $b$ has type $(0,1)$. The associated coefficients are $\tau_0(b)$ and $\tau_0(\phi_K(b))=\phi_Y(\tau_0(b))$, exactly as in Proposition \ref{prop:Chern_conn}.

Finally, the curvature $F_A$ of the Chern connection $A$ has type $(1,1)$. Lemma \ref{lem:curvature_correspondence} therefore gives $F_{\partial_A}^{2,0}=\tau(F_A^{2,0})=0$.
\end{proof}

\begin{proposition}\label{prop:induced_rel_kahler}
Given the setup above, assume the action of $K$ on $(Y,\omega_Y)$ is Hamiltonian, with a fixed $K$-equivariant moment map $\mu:Y\to\mathfrak k^*$ satisfying $\D\langle\mu,\xi\rangle=-\iota_{\tilde{\tau}_0(\xi)}\omega_Y$ for every $\xi\in\mathfrak k$. Then the Chern connection $A$ on $P$ induces a relatively Kähler form $\omega$ on $X$. Furthermore, the restriction of $\omega$ to the fibers $X_s$ coincides with the fiberwise Kähler metric $\omega_{X/S}$ induced by the reduction $P_K$ as described in Lemma \ref{lem:fib_kahler_reduction}.
\end{proposition}
\begin{proof}
This essentially follows from \cite[Th.~6.3.3]{mcduff17introduction} and \cite[Prop.~3.3]{mccarthy2022canonical}. We only recall the construction of $\omega$. On $P_K\times Y$, define the 2-form
\[
\widehat\omega := \omega_Y - \D\langle\mu,A_K\rangle.
\]
This form is closed. One can show that
\begin{equation}\label{eq:symplectic_coupling}
\widehat\omega = \pi_{A_K}^\ast\omega_Y - \langle \mu, F_{A_K}\rangle,
\end{equation}
where $\pi_{A_K}: TP_K^{\RR}\times TY^{\RR} \to TY^{\RR}$ is given by $\pi_{A_K}(v,\hat y)=\hat y-\tilde{\tau}_0(A_K(v))(y)$, and $F_{A_K}$ is the curvature of $A_K$. Then $\widehat\omega$ is basic and descends to a smooth closed $(1,1)$-form $\omega$ on $X=P_K\times_K Y$. For each $s\in S$, the restriction $\omega|_{X_s}$ can be identified with $\omega_Y$, hence $\omega$ is a relatively Kähler form on $X$.

Next, we verify that for any $s\in S$ and $p \in (P_K)_s$, $\phi_p^* (\omega|_{X_s}) = \omega_Y$. We have $\phi_p = q \comp i_p$, where $i_p: Y \to P_K\times Y$ is the inclusion $i_p(y) = (p,y)$, and $q: P_K\times Y \to X$ is the quotient map. By construction, $q^* \omega = \widehat\omega$, so $\phi_p^* \omega = i_p^* \widehat\omega=\omega_Y$.
\end{proof}
\begin{remark}\label{rmk:ham_total_real}
Under the Hamiltonian hypothesis of Proposition \ref{prop:induced_rel_kahler}, the subspace $\mathfrak k_Y=\tau_0(\mathfrak k)$ automatically satisfies $\mathfrak k_Y\cap\I\mathfrak k_Y=\{0\}$. Indeed, suppose that $\tau_0(\xi)=\I\tau_0(\eta)$ for $\xi,\eta\in\mathfrak k$. Set $V=\tilde{\tau}_0(\xi)$ and $W=\tilde{\tau}_0(\eta)$, so that $V=J_YW$. Along the integral curve $\gamma(t)=\exp(-t\xi)\cdot y$ of $V$, the Hamiltonian identity gives
\[
\frac{\D}{\D t}\langle\mu(\gamma(t)),\eta\rangle=-\omega_Y(W,V)|_{\gamma(t)}=-|W(y)|^2.
\]
The last equality holds because $V=J_YW$ and the norm of the Killing field $V$ is constant along its own flow. The left-hand side is the derivative of a bounded function of $t$, since $\gamma(t)$ lies in the compact orbit $K\cdot y$. Thus $W(y)=0$. Since $y$ is arbitrary, $\tau_0(\eta)=\tau_0(\xi)=0$, proving $\mathfrak k_Y\cap\I\mathfrak k_Y=\{0\}$.
\end{remark}

We now relate the connection $\nabla_\omega^{1,0}$ induced by the relatively Kähler form $\omega$ (defined by \eqref{eq:rel_kah_lift}) and the associated connection $\nabla_A^{1,0}$ induced by the complex connection $A$ (defined in Section \ref{subsec:assoc_bundle}).

\begin{lemma}\label{lem:connection_coincidence}
Let $A$ be a complex unitary connection on $P$ and $\omega$ be the induced relatively Kähler form on $X$. Then the pure $(1,0)$-connection $\nabla_\omega^{1,0}$ coincides with the associated connection $\nabla_A^{1,0}$.
\end{lemma}
\begin{proof}
Let $q:P_K\times Y\to X$ be the quotient map. If $u\in T_pP_K^\RR$ is $A_K$-horizontal and $w\in T_yY^\RR$, then \eqref{eq:symplectic_coupling} gives $\widehat\omega((u,0),(0,w))=0$. Thus $\D q(\ker A_K\oplus0)$ is the horizontal distribution determined by $\omega$, since it is a horizontal complement of the correct rank and the restriction of $\omega$ to the vertical tangent bundle is nondegenerate. This is precisely the horizontal distribution of the associated connection, so $\nabla_\omega^{1,0}=\nabla_A^{1,0}$.
\end{proof}

\begin{corollary}\label{cor:assoc_rel_kahler_fiberwise_holo}
The relatively Kähler form $\omega$ induced by a complex unitary connection $A$ on $P$ is relatively holomorphic.
\end{corollary}
\begin{proof}
This follows immediately from Lemma \ref{lem:connection_coincidence} and Lemma \ref{lem:b_rel_holo_associated}.
\end{proof}

\begin{example}[Relatively cscK fibrations]\label{ex:rel_csck_fib}
Let $(Y,H_Y)$ be a polarized compact complex manifold with a cscK metric $\omega_Y\in c_1(H_Y)$. Let $G=\mathrm{Aut}(Y,H_Y)$ be the group of holomorphic automorphisms of $Y$ which lift to $H_Y$, and set $K=\Stab_G(\omega_Y)$. By the Matsushima–Lichnerowicz theorem (see e.g. \cite[Th.~2.2]{hallam23}), $G$ is reductive and $K$ is a maximal compact subgroup. The $K$-action is automatically Hamiltonian.

If $P\to S$ is a holomorphic principal $G$-bundle with a $K$-reduction, the above construction applies, yielding a relatively cscK fibration which is relatively holomorphic by Corollary \ref{cor:assoc_rel_kahler_fiberwise_holo}. Conversely, by \cite[Lem.~3.2]{mccarthy2022canonical}, any smooth polarized isotrivial relatively cscK fibration arises as an associated bundle $X=P\times_G Y=P_K\times_K Y$.
\end{example}

\section{Nonlinear harmonic metrics}\label{sec:nonlin_harm}
The primary goal of this section is to establish a framework for the correspondence between flat holomorphic bundles and nonlinear Higgs bundles, generalizing the classical non-abelian Hodge correspondence. Using the classical result of Donaldson and Corlette, we realize one direction by constructing a harmonic metric on the flat holomorphic bundle $(f,\nabla)$ arising from a representation $\rho: \Gamma:=\pi_1(S,s_0)\to G$, which yields a nonlinear Higgs bundle structure. We also discuss the harmonic maps to the space of all Kähler metrics $\mathcal{H}_0$ on a complex manifold $Y$ as a broader context, viewing it as the natural generalization for flat bundles which do not come from a representation of $\Gamma$ to a finite-dimensional Lie group $G$.

\subsection{Nonlinear Higgs bundles}
\begin{definition}
An \emph{almost Higgs field} $\theta$ on a complex fiber bundle $(f:X\to S,T_{X/S})$ is an element of $C^{\infty}(X,f^* T^*S\otimes  T_{X/S})$. $\theta$ is called \emph{relatively holomorphic} if $\theta\in A^{1,0}(S,f_* T_{X/S})$.
\end{definition}

\begin{definition}\label{def:holomorphic_Higgs_field1}
Let $f: X\to S$ be a holomorphic fibration. A \emph{holomorphic Higgs field} on $f$ is an $\mathcal{O}_S$-linear morphism $\theta: TS\to f^{\mathrm{hol}}_*T_{X/S}$ satisfying the integrability condition $[\theta,\theta]=0$. The pair $(f,\theta)$ is called a \emph{(nonlinear) Higgs bundle}.
\end{definition}

Let $(f, T_{X/S})$ be a complex fiber bundle. Let $\dbar_f$ be a $\dbar$-operator satisfying the lifting condition \eqref{eq:lift_cond}, and let $\theta$ be an almost Higgs field. Let $D'':=\dbar_f+\theta$. We define $G^{1,1}_{D''}$ by
\begin{equation}\label{eq:pseudo_curv11}
	G^{1,1}_{D''}(v,\bar w):=\operatorname{pr}_{T_{X/S}}([\theta(v),\nabla^{0,1}(\bar w)]-\theta([v,\bar w]^{1,0})),
\end{equation}
where $v$ is a local $(1,0)$-vector field and $\bar{w}$ is a local $(0,1)$-vector field on $S$, $\nabla^{0,1}$ is a lifting of $\dbar_f$, $[v,\bar w]^{1,0}$ denotes the $(1,0)$-part of the Lie bracket on $S$, and $\operatorname{pr}_{T_{X/S}}:T_{X/S}^\CC\to T_{X/S}$ is the canonical projection. Here we note that $\D f^\CC( [\theta(v),\nabla^{0,1}(\bar{w})])=[0,\bar{w}]=0$, so $[\theta(v),\nabla^{0,1}(\bar{w})]$ is a vertical vector field. \eqref{eq:pseudo_curv11} is $C^\infty(S)$-linear in both $v$ and $\bar w$, hence defines an element of $C^{\infty}(X,f^*(\wedge^{1,1}T^*S)\otimes T_{X/S})$. \eqref{eq:pseudo_curv11} is independent of the choice of $\nabla^{0,1}$ if and only if $\theta$ is relatively holomorphic. In this case, locally we have
\begin{equation}\label{eq:pseudo_curv11_loc}
		G_{D''}^{1,1} = \bigl( -\partial_{\bar{j}} \theta_i^\alpha + \theta_i^\gamma \partial_\gamma \Gamma_{\bar{j}}^\alpha- \Gamma_{\bar{j}}^\gamma \partial_\gamma \theta_i^\alpha  \bigr)  \D s^i \wedge \D \bar{s}^j\otimes \partial_\alpha,
\end{equation}
where locally $\theta= \theta_i^\alpha \D s^i\otimes \partial_\alpha$ and $\dbar_f(\partial_{\bar{j}})=[\partial_{\bar{j}}+\Gamma_{\bar{j}}^\gamma\partial_\gamma]\mod \widebar{T_{X/S}}$. Next we define $G_{\theta}^{2,0}\in C^{\infty}(X,f^*(\wedge^2 T^*S)\otimes T_{X/S})$ by $G_{\theta}^{2,0}:=\frac12[\theta,\theta]$, so that $G_{\theta}^{2,0}(u,v)=[\theta(u),\theta(v)]$. Assume from now on that $\theta$ is relatively holomorphic. Then the \emph{pseudo-curvature} $G_{D''}$ of $(\dbar_f,\theta)$ is defined as
\[G_{D''}:=G_{\dbar_f}^{0,2}+G_{D''}^{1,1}+G_{\theta}^{2,0}\in C^{\infty}(X,f^*(\wedge^2T^*S^\CC)\otimes T_{X/S}),\]
where $G_{\dbar_f}^{0,2}:=F_{\dbar_f}^{0,2}$, which is defined in Corollary \ref{cor:dbarint_via_conn_curv}. Since $\dbar_f$ satisfies the lifting condition and $\theta$ is relatively holomorphic, we have $G_{D''}\in
A^2(S,f_*T_{X/S})$.

If $\dbar_f$ is integrable, then in adapted holomorphic coordinates, \eqref{eq:pseudo_curv11_loc} simplifies to
\begin{equation}\label{eq:pseudo_curv11_loc1}
		G_{D''}^{1,1} = -\bigl( \partial_{\bar{j}} \theta_i^\alpha \bigr)  \D s^i \wedge \D \bar{s}^j\otimes \partial_\alpha.
\end{equation}
Then $G_{D''}^{1,1}=0$ if and only if $\theta\in H^0(X,f^* \Omega_S\otimes  T_{X/S})$. Therefore $(\dbar_f,\theta)$ is a Higgs bundle if and only if $G_{D''}=0$.

\begin{definition}
	Let $G$ be a complex Lie group and $\pi:P\to S$ be a smooth principal $G$-bundle. An \emph{almost Higgs pair} is $(\dbar_\pi,\theta)$, where $\dbar_\pi$ is a principal $\dbar$-operator (Remark \ref{rmk:principal_dbar_alconn}) and $\theta$ is an element of $A^{1,0}(S,\ad P)$. The \emph{pseudo-curvature} of $(\dbar_\pi,\theta)$ is
	\[G_{(\dbar_\pi,\theta)}:=-F_{\dbar_\pi}^{0,2}+\dbar_\pi\theta+\tfrac12[\theta,\theta]\in A^2(P,\mathfrak{g}),\]
	where $\dbar_\pi:A^{1,0}(S,\ad P)\to A^{1,1}(S,\ad P)$ is the Dolbeault operator (\cite[\S2.2.2]{CTW25}) determined by $\dbar_\pi$. For any principal connection $A$ inducing $\dbar_\pi$, $\dbar_\pi\theta:=(\D_A\theta)^{1,1}$, which is independent of the choice of $A$, since for any other $A'$ inducing $\dbar_\pi$, $A'=A+a$ for some $a\in A^{1,0}(S,\ad P)$ and $([a,\theta])^{1,1}=0$. By Remark \ref{rmk:curv_coresp}, $-F_{\dbar_\pi}^{0,2}=F_A^{0,2}$.
\end{definition}
By the $G$-equivariance and the horizontality, $G_{(\dbar_\pi,\theta)}\in A^2(S,\ad P)$. A \emph{$G$-Higgs bundle} is a holomorphic principal $G$-bundle with a holomorphic section $\theta$ of $\ad(P)\otimes\Omega_S$ satisfying $[\theta,\theta]=0$. $(\dbar_\pi,\theta)$ is a $G$-Higgs bundle if and only if $G_{(\dbar_\pi,\theta)}=0$. An almost Higgs pair $(\dbar_\pi,\theta)$ on $P$ induces a $\dbar$-operator $\dbar_f$ satisfying the lifting condition and a relatively holomorphic almost Higgs field $\tau(\theta)$ on the associated bundle $f:X=P\times_G Y\to S$, where $\dbar_f$ is induced by $\nabla_A^{0,1}$ for any principal connection $A$ inducing $\dbar_\pi$ ($\dbar_f$ is independent of the choice of $A$) and $\tau:A^{1,0}(S,\ad P)\to A^{1,0}(S,f_*T_{X/S})$ is induced by the morphism $\tau$ in Lemma \ref{lem:tau_associated}. Similarly to Lemma \ref{lem:curvature_correspondence}, we have the following.
\begin{lemma}\label{lem:pseudo_curvature_correspondence}
	The pseudo-curvature of the induced operator $D''=\dbar_f+\tau(\theta)$ satisfies
	\[G_{D''}=\tau(G_{(\dbar_\pi,\theta)}).\]
\end{lemma}
\begin{proof}
	By Lemma \ref{lem:curvature_correspondence}, $G_{D''}^{0,2}=\tau(G_{(\dbar_\pi,\theta)}^{0,2})$. Let $u, v \in TS$, we have
	\[ G_{D''}^{2,0}(u, v) = [\tau(\theta(u)), \tau(\theta(v))] = \tau([\theta(u),\theta(v)])=\tau\!\left(\tfrac12[\theta,\theta](u,v)\right)=\tau(G_{(\dbar_\pi,\theta)}^{2,0}(u,v)). \]
	$G_{D''}^{1,1}$ is locally given by \eqref{eq:pseudo_curv11_loc}, we compute
	\begin{align*}
		 -\partial_{\bar{j}} \tau_0(\theta_i)^\alpha + \tau_0(\theta_i)^\gamma \partial_\gamma \Gamma_{\bar{j}}^\alpha- \Gamma_{\bar{j}}^\gamma \partial_\gamma \tau_0(\theta_i)^\alpha&=\tau_0(-\partial_{\bar{j}}\theta_i)^\alpha-[\tau_0(B_{\bar{j}}), \tau_0(\theta_i)]^\alpha\\
		 &=-\tau_0(\partial_{\bar{j}}\theta_i+[B_{\bar{j}},\theta_i])^\alpha.
	\end{align*}
On the other hand,
\[G_{(\dbar_\pi,\theta)}^{1,1}=(\D_A\theta)^{1,1}=\dbar\theta+[B_{\bar{j}},\theta_i]\D \bar{s}^j\wedge \D s^i=-(\partial_{\bar{j}}\theta_i+[B_{\bar{j}},\theta_i])\D s^i\wedge\D \bar{s}^j.\]
Therefore, $G_{D''}^{1,1}=\tau(G_{(\dbar_\pi,\theta)}^{1,1})$.
\end{proof}

We shall define the complex conjugation of a relatively holomorphic almost Higgs field on a complex fiber bundle. Let $\theta: TS\to f_*T_{X/S}$ be such a Higgs field on a complex fiber bundle $f$. Then for $s\in S$, $\theta$ defines a $\CC$-linear map $\theta_s: T_s S\to H^0(X_s,TX_s)$, where $T_s S$ has constant finite dimension $n$.

\begin{definition}\label{def:metric_adapted}
 A fiberwise Kähler metric $\omega_{X/S}$ on $f$ is said to be \emph{aut-finite}, if for each $s\in S$, $\Aut(X_s,\omega_s)$ is a real Lie group of finite dimension independent of $s$. It is \emph{$\theta$-adapted}, if $\theta\in A^{1,0}(S,\mathfrak{k}_{X/S}^\CC)$, where $\mathfrak{k}_{X/S}$ is a smooth vector bundle on $S$, with fibers $\mathfrak{k}_s\subset \mathfrak{aut}(X_s,\omega_s)$ satisfying $\mathfrak{k}_s\cap \I\mathfrak{k}_s=\{0\}$, $\mathfrak{k}_s^\CC=\mathfrak{k}_s\oplus \I\mathfrak{k}_s$.
\end{definition}

\begin{definition}\label{def:comp_conj_higgs}
	Let $\theta$ be a relatively holomorphic almost Higgs field on $f$, and $\omega_{X/S}$ be a $\theta$-adapted fiberwise Kähler metric on $f$. The \emph{complex conjugate} $\bar \theta_{\omega_{X/S}}$ of $\theta$ with respect to $\omega_{X/S}$ (and $\mathfrak{k}_{X/S}$) is the element of $A^{0,1}(S,\mathfrak{k}_{X/S}^\CC)$ whose value at $s\in S$ is defined by sending $\bar{v}\in \widebar{T_s S}$ to $-\phi_s(\theta_s(v))$ where $\phi_s$ is the involution given by the real form $\mathfrak{k}_s$ of $\mathfrak{k}_s^\CC$.
\end{definition}
\begin{remark}
If there are no parallel vector fields on $(X_s,\omega_s)$, then $\mathfrak{aut}(X_s,\omega_s)\cap\I\mathfrak{aut}(X_s,\omega_s)=\{0\}$. We may choose $\mathfrak{k}_s=\mathfrak{aut}(X_s,\omega_s)$ whenever these spaces form a smooth vector bundle of finite rank over $S$.
\end{remark}
\begin{lemma}\label{lem:thetabar_assoc}
	Let $(\pi_P:P\to S,\theta_P)$ be a $G$-Higgs bundle, where $G=K^\CC$ is reductive with $K$ acting by Kähler isometries on a Kähler manifold $(Y,\omega_Y)$. Let $(f:X=P\times_G Y\to S,\theta=\tau(\theta_P))$ be the associated Higgs bundle as in Lemma \ref{lem:pseudo_curvature_correspondence}. Let $P_K$ be a Hermitian structure on $P$, which induces a fiberwise Kähler metric $\omega_{X/S}$ by Lemma \ref{lem:fib_kahler_reduction}. Assume that $\mathfrak k_Y:=\tau_0(\mathfrak k)$ satisfies $\mathfrak k_Y\cap\I\mathfrak k_Y=\{0\}$, and use $\mathfrak k_{X/S}:=P_K\times_K\mathfrak k_Y$ in Definition \ref{def:metric_adapted}. Then $\omega_{X/S}$ is $\theta$-adapted and $\bar{\theta}_{\omega_{X/S}}=\tau(\bar{\theta}_{P,K})$, where $\bar{\theta}_{P,K}\in A^{0,1}(S,\ad P)$ is the conjugate of $\theta_P$ with respect to the Hermitian structure $P_K$.
\end{lemma}
\begin{proof}
The bundle $\mathfrak k_{X/S}=P_K\times_K\mathfrak k_Y$ is smooth, and $\theta$ takes values in its complexification. The conjugations satisfy $-\phi_Y\comp\tau_0=-\tau_0\comp\phi_K$, as in the proof of Lemma \ref{lem:assoc_Chern_conn}. Applying this identity to the coefficients of $\theta_P$ proves the assertion.
\end{proof}

\begin{remark}
A harmonic reduction of a flat principal $G$-bundle over a compact K\"ahler base induces the nonlinear Simpson structure relative to its associated atlas class whenever $\mathfrak k_Y\cap\I\mathfrak k_Y=\{0\}$. Let $N$ be the kernel of the action. If $N$ is an algebraic normal subgroup, then $G_{\mathrm{eff}}=G/N$ is reductive, $K_{\mathrm{eff}}$ is the image of $K$, and the quotient homomorphism induces a totally geodesic map $G/K\to G_{\mathrm{eff}}/K_{\mathrm{eff}}$. Composing with it preserves equivariant harmonicity. These conditions are not automatic for an arbitrary holomorphic action. However, if $\Stab_G(\omega_Y)=K$, then $N\subset K$ is a compact complex Lie subgroup with zero Lie algebra, hence finite, and $G/K\cong(G/N)/(K/N)$.
\end{remark}

\subsection{Simpson mechanism}\label{sec:simpson_mechanism}
In this subsection, we reformulate the construction of Simpson \cite[\S1]{Si1} in a way that generalizes to the fiber bundle setting. Throughout this subsection, all Chern connections are taken relative to a fixed $\mathfrak k$-compatible unitary atlas class as in Proposition \ref{prop:Chern_conn}, and the same real bundle $\mathfrak k_{X/S}$ is used to define complex conjugation.

\vspace{1em}
\noindent{\itshape From Higgs bundles to flat bundles:} Let $(f,\theta)$ be a Higgs bundle, $\dbar_f$ be the canonical $\dbar$-operator. Let $\omega_{X/S}$ be a fiberwise Kähler metric satisfying the conditions in Definition \ref{def:comp_conj_higgs}, so that the complex conjugate $\bar{\theta}_{\omega_{X/S}}$ is well-defined. We also regard $\bar{\theta}_{\omega_{X/S}}$ as a bundle map $f^*\widebar{TS} \to T_{X/S} \to TX^\CC/ \widebar{T_{X/S}}$.

Now we define a smooth bundle morphism $f^*\widebar{TS}\to TX^\CC/ \widebar{T_{X/S}}$ by
\begin{equation}
		\dbar_{\omega_{X/S}} := \dbar_f + \bar{\theta}_{\omega_{X/S}}. \label{eq:simpson_dbar_A}
\end{equation}
Since its composite with the projection $TX^\CC/ \widebar{T_{X/S}}\to f^*\widebar{TS}$ is the identity, it defines a $\dbar$-operator on $(f,T_{X/S})$. Next, we define an almost connection by
\begin{equation}\label{eq:simpson_partial_A}
	\partial_{\omega_{X/S}}:=\partial_{\omega_{X/S}}^{\mathrm{Ch}}+2\theta,
\end{equation}
where $\partial_{\omega_{X/S}}^{\mathrm{Ch}}$ is the Chern almost connection (assuming it exists) associated to $\dbar_{\omega_{X/S}}$ and $\omega_{X/S}$. Clearly, $\dbar_{\omega_{X/S}}$ and $\partial_{\omega_{X/S}}$ both satisfy the lifting conditions. Then the curvatures are well-defined, and $(\partial_{\omega_{X/S}},\dbar_{\omega_{X/S}})$ defines a flat holomorphic bundle structure $(f,\nabla_{\omega_{X/S}}^{1,0})$, where $\nabla_{\omega_{X/S}}^{1,0}$ is the $(1,0)$-connection determined by $\hat{D}_{\omega_{X/S}}:=\partial_{\omega_{X/S}}+\dbar_{\omega_{X/S}}$, if and only if \[F_{\nabla_{\omega_{X/S}}^{1,0}}:=F_{\partial_{\omega_{X/S}}}^{2,0}+F_{\hat{D}_{\omega_{X/S}}}^{1,1}+F_{\dbar_{\omega_{X/S}}}^{0,2}=0\in A^2(S,f_*T_{X/S}).\]

\vspace{0.5em}
\noindent{\itshape From flat bundles to Higgs bundles:}
Let $(f,\nabla^{1,0})$ be a flat holomorphic fiber bundle, $\dbar_\nabla$ be the canonical $\dbar$-operator and $\partial_{\nabla}$ be the almost connection determined by $\nabla^{1,0}$. Let $\omega_{X/S}$ be a fiberwise Kähler metric on $f$ and $\partial_{\omega_{X/S}}^{\mathrm{Ch}}$ be the Chern almost connection (assume it exists) associated to $\dbar_\nabla$ and $\omega_{X/S}$. We define an almost Higgs field by
\begin{equation}
		\theta_{\omega_{X/S}} := \tfrac{1}{2} (\partial_\nabla - \partial_{\omega_{X/S}}^{\mathrm{Ch}}), \label{eq:simpson_theta_B}
\end{equation}
which is clearly relatively holomorphic. Assume that $\omega_{X/S}$ satisfies the conditions in Definition \ref{def:comp_conj_higgs}. Next, we define a $\dbar$-operator on $(f,T_{X/S})$ by
\begin{equation}\label{eq:simpson_dbar_B}
		\dbar_{\omega_{X/S}} := \dbar_\nabla - \bar{\theta}_{\omega_{X/S}},
\end{equation}
	where $\bar{\theta}_{\omega_{X/S}}$ is the complex conjugate of $\theta_{\omega_{X/S}}$ with respect to $\omega_{X/S}$. Let $D''_{\omega_{X/S}}:=\dbar_{\omega_{X/S}}+\theta_{\omega_{X/S}}$, then $(\dbar_{\omega_{X/S}},\theta_{\omega_{X/S}})$ defines a Higgs bundle structure if and only if $G_{D''_{\omega_{X/S}}}=0$.

\begin{remark}
It is clear from the construction that the transformations \eqref{eq:simpson_dbar_A}-\eqref{eq:simpson_partial_A} and \eqref{eq:simpson_theta_B}-\eqref{eq:simpson_dbar_B} are inverse to each other.
\end{remark}

\subsection{$G$-harmonic metrics and Higgs bundles}
Let $(S,\omega_S)$ be a connected compact Kähler manifold and $(Y,\omega_Y)$ be a Kähler manifold. Let $G\leq \Aut_0(Y)$ be a connected complex Lie subgroup. The space of Kähler potentials relative to $\omega_Y$ is
\begin{equation}\label{eq:space_kah_poten}
	\calH_{\omega_Y}:=\big\{\phi\in C^\infty(Y,\RR)\, \big|\, \omega_\phi:=\omega_Y+\sqrt{-1}\,\partial\bar\partial\phi>0\big\}.
\end{equation}
If $Y$ is compact, this is a Fr\'echet manifold, being an open subset of $C^\infty(Y,\RR)$. The \emph{space of Kähler metrics} in the cohomology class $[\omega_Y]$ is
\begin{equation}\label{eq:space_kah_metric}
\calH_{\omega_Y,0}:=\{\omega\mid \omega\text{ is a K\"ahler form on }Y,\ [\omega]=[\omega_Y]\}.
\end{equation}
For notational simplicity, we will omit the subscript $\omega_Y$ from now on. If $Y$ is compact, the $\partial\bar\partial$-lemma gives $\calH_0=\{\omega_\phi\mid\phi\in\calH\}$ and $T_\phi\calH\cong C^\infty(Y,\RR)$. Smooth maps into $\calH_0$ are understood as smooth families of differential forms on $Y$.

\begin{lemma}\label{lem:flat_bundle_isomorphism}
	Let $\rho: \Gamma:=\pi_1(S,s_0) \to G$ be a representation. Let $X = \widetilde{S} \times_\rho Y$ be the flat fiber bundle and $P = \widetilde{S} \times_\rho G$ be the flat principal $G$-bundle corresponding to $\rho$, where $\pi: \widetilde{S}\to S$ is the universal cover. Then there is a canonical isomorphism $\Phi: P\times_G Y  \to X$ satisfying $\Phi \comp \phi_p = \varphi_{\tilde{s}} \comp g$, where $\varphi_{\tilde{s}}: Y \xrightarrow{\cong} X_s$ is defined by $\varphi_{\tilde{s}}(y) = [\tilde{s}, y]_\rho$ for $\tilde{s}\in \pi^{-1}(s)$, and $\phi_p: Y \xrightarrow{\cong} (P \times_G Y)_s$ is defined by $\phi_p(y) = [p, y]_G$ for $p=[\tilde s,g]_\rho\in P_s$.
\end{lemma}
	\begin{proof}
		By the definition of the quotient $X$, $[\gamma\cdot\tilde{s}, y]_\rho = [\tilde{s}, \rho(\gamma)^{-1}y]_\rho$. Thus, $\varphi_{\gamma\cdot\tilde{s}} = \varphi_{\tilde{s}}\comp \rho(\gamma)^{-1}$. We define the map $\Phi: P \times_G Y \to X$ by \[ \Phi([[\tilde{s}, g]_\rho, y]_G) := [\tilde{s}, gy]_\rho. \]We verify this is well-defined. (i) Independence of the representative for $p\in P$: Let $[\tilde{s}', g']_\rho = [\tilde{s}, g]_\rho$. This means there exists $\gamma \in \Gamma$ such that $\tilde{s}' = \gamma \cdot \tilde{s}$ and $g' = \rho(\gamma) g$. Then \[ [\tilde{s}', g'y]_\rho = [\gamma\cdot\tilde{s}, (\rho(\gamma)g)y]_\rho. \]By the definition of the quotient $X$, this equals $[\tilde{s}, \rho(\gamma)^{-1}((\rho(\gamma)g)y)]_\rho = [\tilde{s}, gy]_\rho$.

		(ii) Independence of the representative for $[p, y]_G$: We use the standard associated bundle equivalence $[p\cdot h, h^{-1}y]_G = [p, y]_G$. If $p=[\tilde{s}, g]_\rho$, then $p\cdot h = [\tilde{s}, gh]_\rho$. \[ \Phi([p\cdot h, h^{-1}y]_G) = \Phi([[\tilde{s}, gh]_\rho, h^{-1}y]_G) = [\tilde{s}, (gh)(h^{-1}y)]_\rho = [\tilde{s}, gy]_\rho. \]

		The map $\Phi$ is a smooth bundle isomorphism over $S$. Let $p=[\tilde{s}, g]_\rho$. We have $(\Phi \comp \phi_p)(y) = \Phi([p, y]_G) = [\tilde{s}, gy]_\rho$. On the other hand, $(\varphi_{\tilde{s}} \comp g)(y) = \varphi_{\tilde{s}}(gy) = [\tilde{s}, gy]_\rho$. Thus, $\Phi \comp \phi_p = \varphi_{\tilde{s}} \comp g$.
	\end{proof}

\begin{lemma}\label{lem:fib_Kah_equivalence}
Let $f: X = \widetilde{S} \times_\rho Y \to S$ be a flat fiber bundle defined by $\rho: \Gamma \to G$, identified with the associated bundle of $P = \widetilde{S} \times_\rho G$ via Lemma \ref{lem:flat_bundle_isomorphism}. Then there is a natural bijection between the following two collections of objects:
\begin{enumerate}
    \item Fiberwise Kähler metrics $\omega_{X/S}$ on $X$ which are $G$-modeled on $(Y, [\omega_Y])$ (in the sense of Definition \ref{def:fib_kahler_model}).
    \item Smooth $\rho$-equivariant maps $h: \widetilde{S} \to \calH_0$, where the $\rho$-equivariance condition is
		\begin{equation}\label{eq:rho-equivariance}
		h(\gamma\cdot \tilde{s}) = {(\rho(\gamma)^{-1})}^\ast h(\tilde{s}),\qquad \forall\,\gamma\in\Gamma, \tilde{s}\in\widetilde S.
		\end{equation}
\end{enumerate}
\end{lemma}
\begin{proof}
	(2) $\Rightarrow$ (1): Let $h: \widetilde{S} \to \calH_0$ be a $\rho$-equivariant map. We define $\omega_s:=(\varphi_{\tilde{s}}^{-1})^*  h(\tilde{s})$, where $\tilde{s}\in \pi^{-1}(s)$. Let $\tilde{s}'=\gamma\cdot\tilde{s}$, then
	$\varphi_{\tilde{s}'}^\ast \omega_s = (\varphi_{\tilde{s}}\comp \rho(\gamma)^{-1})^\ast \omega_s = {(\rho(\gamma)^{-1})}^\ast (\varphi_{\tilde{s}}^\ast \omega_s)= {(\rho(\gamma)^{-1})}^\ast h(\tilde{s})$. By the equivariance of $h$, this equals $h(\tilde{s}')$. Thus we obtain a well-defined fiberwise Kähler metric $\omega_{X/S}$. Let $p=[\tilde{s}, g]_\rho \in P_s$. Using $\phi_p = \varphi_{\tilde{s}} \comp g$ (by Lemma \ref{lem:flat_bundle_isomorphism}, where we omitted $\Phi$), we have $\phi_p^* \omega_s = (\varphi_{\tilde{s}} \comp g)^* \omega_s = g^* (\varphi_{\tilde{s}}^* \omega_s) = g^* h(\tilde{s})$. Since $h(\tilde{s}) \in \calH_0$ and $G$ preserves the class, $g^* h(\tilde{s}) \in \calH_0$.

	(1) $\Rightarrow$ (2): Let $\omega_{X/S}$ be $G$-modeled on $(Y, [\omega_Y])$. We define $h(\tilde{s}) := \varphi_{\tilde{s}}^\ast \omega_s$. Consider $p_e = [\tilde{s}, e]_\rho \in P_s$. Then $\phi_{p_e} = \varphi_{\tilde{s}} \comp e = \varphi_{\tilde{s}}$. By Definition \ref{def:fib_kahler_model}, $\phi_{p_e}^* \omega_s \in [\omega_Y]$. Therefore, $h(\tilde{s}) \in \calH_0$. $h$ is $\rho$-equivariant since
	\[h(\gamma\cdot \tilde{s}) = \varphi_{\gamma\cdot \tilde{s}}^\ast \omega_s = (\varphi_{\tilde{s}}\comp \rho(\gamma)^{-1})^\ast \omega_s = {(\rho(\gamma)^{-1})}^\ast (\varphi_{\tilde{s}}^\ast \omega_s) = {(\rho(\gamma)^{-1})}^\ast h(\tilde{s}). \qedhere\]
\end{proof}
\begin{corollary}\label{cor:fib_Kah_orbit_equivalence}
Under the assumptions of Lemma \ref{lem:fib_Kah_equivalence}, there is a natural bijection between the following two sets:
\begin{enumerate}
    \item Fiberwise Kähler metrics $\omega_{X/S}$ on $X$ which are $G$-modeled on $(Y, \omega_Y)$ (in the sense of Definition \ref{def:fib_kahler_model}).
    \item Smooth $\rho$-equivariant maps $h: \widetilde{S} \to G\cdot\omega_Y\cong G/K$, where $K:=\Stab_G(\omega_Y)$.
\end{enumerate}
\end{corollary}

\begin{proof}
We utilize the correspondence established in Lemma \ref{lem:fib_Kah_equivalence} and check that the condition of being $G$-modeled on $(Y, \omega_Y)$ is equivalent to the image of the corresponding map $h$ being contained in $G\cdot\omega_Y \subset \mathcal{H}_0$.

(2) $\Rightarrow$ (1): Let $p=[\tilde{s}, g]_\rho \in P_s$, $\phi_p^* \omega_s = g^* h(\tilde{s})\in G\cdot\omega_Y$.

(1) $\Rightarrow$ (2): Let $p_e = [\tilde{s}, e]_\rho \in P_s$, $\phi_{p_e}^* \omega_s \in G\cdot\omega_Y$. Since $\varphi_{\tilde{s}} = \phi_{p_e}$, $h(\tilde{s})=\phi_{p_e}^* \omega_s \in G\cdot\omega_Y$.
\end{proof}

From now on, we make the following assumption.
\begin{assumption}\label{ass:finite_dim}
$G$ is a connected complex reductive Lie subgroup of $\Aut_0(Y)$. $K = \Stab_G(\omega_Y)$ is a compact real form of $G$.
\end{assumption}
\begin{definition}
A fiberwise Kähler metric	$\omega_{X/S}$ on $f$ is called \emph{$G$-harmonic} if it corresponds via Corollary \ref{cor:fib_Kah_orbit_equivalence} to a $\rho$-equivariant harmonic map $h:\widetilde{S}\to G\cdot \omega_Y$, where $G\cdot \omega_Y\cong G/K$ is a Riemannian symmetric space.
\end{definition}

\begin{theorem}[{\cite{donaldson87twisted, corlette88}}]\label{thm:harm-DC}
Let $\rho:\Gamma\to G$ be a reductive representation. Then there exists a $\rho$-equivariant harmonic map $h: \widetilde{S} \to G/K$. Consequently, the flat bundle $f$ admits a $G$-harmonic fiberwise Kähler metric $\omega_{X/S}$.
\end{theorem}
\begin{remark}\label{rmk:unique_equiv_harm}
	The $\rho$-equivariant harmonic map $h$ is unique up to post-composition with elements of the centralizer of $\rho(\Gamma)<G$ (\cite[Rmk.~6.7]{loustau2020harmonic}).
\end{remark}
\begin{remark}\label{rmk:global_form}
Assume further that the $K$-action is Hamiltonian. Then $\omega_{X/S}$ is induced by a relatively holomorphic relatively Kähler form $\omega$ by Lemma \ref{lem:fib_kahler_reduction}, Proposition \ref{prop:induced_rel_kahler}, and Corollary \ref{cor:assoc_rel_kahler_fiberwise_holo}.
\end{remark}

\begin{proposition}\label{thm:nonlinear_higgs_correspondence}
Let $(S,\omega_S)$ be a connected compact Kähler manifold and $(Y,\omega_Y)$ be a Kähler manifold. Let $G \leq \Aut_0(Y)$ be a connected complex reductive subgroup satisfying Assumption \ref{ass:finite_dim}. Let $\rho: \Gamma \to G$ be a reductive representation, and let $(f: X \to S, \nabla^{1,0})$ be the corresponding flat holomorphic fiber bundle. Let $\omega_{X/S}$ be a $G$-harmonic fiberwise Kähler metric on $f$ (which exists by Theorem \ref{thm:harm-DC}). The Higgs pair $(\dbar_{\omega_{X/S}}, \theta_{\omega_{X/S}})$ obtained from $(f, \nabla^{1,0})$ and $\omega_{X/S}$ via the Simpson mechanism (Section \ref{sec:simpson_mechanism}) is a Higgs bundle (pseudo-curvature vanishes) and is the one associated to the $G$-Higgs bundle $(\mathcal{E}_G, \varphi)$ constructed from $\rho$ via the classical nonabelian Hodge correspondence.

Furthermore, the resulting Higgs bundle $(f, \dbar_{\omega_{X/S}}, \theta_{\omega_{X/S}})$ is independent of the choices of the Kähler metric $\omega_Y$ on $Y$ such that $\Stab_G(\omega_Y)$ is a compact real form of $G$ and the Kähler metric $\omega_S$ on $S$. The Higgs bundle structure is also independent of the $G$-harmonic metric $\omega_{X/S}$ on the fixed flat bundle.
\end{proposition}
\begin{proof}
Let $P = \widetilde{S} \times_\rho G$ be the flat principal $G$-bundle associated to $\rho$, equipped with the flat connection $A$. By Lemma \ref{lem:flat_bundle_isomorphism}, we identify $X$ with the associated bundle $P \times_G Y$. Under this identification, the flat holomorphic connection $\nabla^{1,0}_A$ on $X$ is induced by $A$. By Theorem \ref{thm:harm-DC}, there exists a $\rho$-equivariant harmonic map $h: \widetilde{S} \to G/K$, where $K=\Stab_G(\omega_Y)$. This map determines a reduction of the structure group of $P$ to $K$, denoted by $P_K$. By Lemma \ref{lem:fib_kahler_reduction}, this reduction induces the $G$-harmonic fiberwise Kähler metric $\omega_{X/S}$ on $X$.

With respect to the reduction $\iota_K:P_K\to P$, the flat connection $A$ decomposes as $\iota_K^* A = A_K + \psi$, where $A_K$ is a connection on $P_K$ and $\psi \in A^1(S, P_K\times_K \I\mathfrak{k})$. The $G$-Higgs bundle $(\mathcal{E}_G, \varphi)$ is defined by the principal $\dbar$-operator determined by $A_K$ (extended to $P$) and the Higgs field $\varphi = \psi^{1,0}$ (see \cite[Lem.~1.1]{Si1} for $G=\GL(n,\CC)$ and \cite[\S6.2.3]{loustau2020harmonic} for general reductive groups). We now compute the operators defined by the Simpson mechanism in Section \ref{sec:simpson_mechanism}. It suffices to work locally as in \eqref{eq:princ_conn_local}, where we may write $A=A^{1,0}+A^{0,1}$ with $A^{1,0}\in A^{1,0}(U,\mathfrak{g})$, $A^{0,1}\in A^{0,1}(U,\mathfrak{g})$. Let $A_{\mathrm{Ch}}$ be the Chern connection on $P$ determined by the holomorphic structure associated to $A$ and the Hermitian structure $P_K$. By Lemma \ref{lem:assoc_Chern_conn}, $\partial_{\omega_{X/S}}^{\mathrm{Ch}}$ is induced by $A_{\mathrm{Ch}}$. Then $\theta_{\omega_{X/S}}$ is induced by $(A-A_{\mathrm{Ch}})^{1,0}/2$. Locally,
\[\psi^{1,0}(\partial_i)=\varphi(\partial_i)=(\psi(\partial_{x^i})-\I\psi(\partial_{y^i}))/2,\quad \psi^{0,1}(\partial_{\bar{i}})=(\psi(\partial_{x^i})+\I\psi(\partial_{y^i}))/2=\bar{\varphi}_K(\partial_{\bar{i}}). \]
Note that $\psi(\partial_{x^i}),\psi(\partial_{y^i})\in \I \mathfrak{k}$. This implies $A_{\mathrm{Ch}}^{1,0}=A_K^{1,0}-\varphi$, and then $(A-A_{\mathrm{Ch}})^{1,0}/2=\varphi$. By Lemma \ref{lem:thetabar_assoc}, the conjugate $\bar{\theta}_{\omega_{X/S}}$ is induced by $\bar{\varphi}_K=\psi^{0,1}$. Then $\dbar_{\omega_{X/S}}=\dbar_\nabla-\bar{\theta}_{\omega_{X/S}}$ is induced by $A^{0,1}-\psi^{0,1}=A_K^{0,1}$. Therefore $(\dbar_{\omega_{X/S}},\theta_{\omega_{X/S}})$ is induced by $(\mathcal{E}_G, \varphi)$, which is a nonlinear Higgs bundle. Here, all Chern connections and conjugations on $f$ are taken relative to the unitary atlas induced by local smooth sections of $P_K$.

Let $\omega_{Y}'$ be another Kähler metric on $Y$ such that $K'= \Stab_G(\omega_{Y}')$ is also a compact real form of $G$. There exists $u \in G$ such that $K' = u K u^{-1}$. This conjugation induces a $G$-equivariant isometry of symmetric spaces $\iota: G/K \to G/K'$ by $x K \mapsto x u^{-1} K'$. Then $h' := \iota \comp h$ is a $\rho$-equivariant harmonic map $\widetilde{S}\to G/K'$. Let $\sigma$ be a local frame corresponding to the Hermitian structure induced by $h$. Then $\sigma' = \sigma u^{-1}$ corresponds to $h'$. The flat connection in these frames is expressed by $A\in A^1(U,\mathfrak{g})$ and $A'= u A u^{-1}$ respectively. Correspondingly, we have the splitting
\[A'=A_K'+\psi', \text{ where } A_K'=u A_K u^{-1},~\psi'=u \psi u^{-1}.\]
Therefore, the Higgs bundles induced by $h$ and $h'$ are identical.

The Higgs bundle structure $(f, \dbar_{\omega_{X/S}}, \theta_{\omega_{X/S}})$ is independent of $\omega_S$ since the harmonicity implies the pluriharmonicity by the Siu–Sampson theorem \cite[Th.~4.1]{loustau2020harmonic} and the pluriharmonicity is independent of $\omega_S$.

To compare two harmonic $K$-reductions on the fixed flat principal bundle, choose a faithful representation $\iota:G\hookrightarrow\GL(V)$ and a $K$-invariant Hermitian form on $V$. The induced symmetric-space map is totally geodesic, so the reductions induce harmonic metrics on the same flat vector bundle $E=P\times_GV$. The identity is a flat section of the harmonic Hom bundle formed from these two harmonic structures. By \cite[Lem.~1.2]{Si1}, it is also a Higgs section, so the two vector-bundle Dolbeault operators and Higgs fields agree. Injectivity of $\D\iota$ implies equality of the principal operators, and applying $\tau$ gives equality of the nonlinear operators.
\end{proof}

\begin{definition}\label{def:admissible_flat_morphism}
Let $S$ be a connected compact K\"ahler manifold and set $\Gamma:=\pi_1(S,s_0)$. For $i=1,2$, let $X_i=\widetilde S\times_{\rho_i}Y_i$ be flat bundles with reductive representations $\rho_i:\Gamma\to G_i$, where $G_i\leq\Aut_0(Y_i)$ satisfy Assumption \ref{ass:finite_dim}. We call such flat bundles \emph{reductive of K\"ahler type}. A fiberwise holomorphic flat morphism $F:X_1\to X_2$, written as $F([\tilde s,y]_{\rho_1})=[\tilde s,\psi(y)]_{\rho_2}$, is called \emph{admissible} if there exist a (possibly disconnected) complex reductive Lie group $H$, a reductive representation $\sigma:\Gamma\to H$, and algebraic homomorphisms (with respect to the natural algebraic structures on the reductive groups) $p_i:H\to G_i$ such that
\[
p_i\comp\sigma=\rho_i,\qquad \psi(p_1(h)\cdot y)=p_2(h)\cdot\psi(y)\quad(h\in H,\ y\in Y_1).
\]
The auxiliary data are not part of the morphism.
\end{definition}
\begin{remark}
All flat complex-linear maps between semisimple flat vector bundles are admissible for their $\GL$-presentations. Indeed, the joint Zariski closure of the two monodromy representations is reductive, and the linear intertwining equations extend to this closure. See Proposition \ref{prop:flat_is_monodromy_associated} below for more general cases.
\end{remark}

\begin{proposition}\label{prop:faithful_functor}
 Let $S$ be a connected compact complex manifold admitting a Kähler metric. Let $\mathbf{RFB}(S)$ be the category whose objects are $(f:X\to S,\nabla^{1,0})$, each of which is a flat holomorphic bundle arising from a reductive representation $\rho:\Gamma\to G$ with $G$ satisfying  Assumption \ref{ass:finite_dim} for some Kähler form $\omega_Y$ on $Y$. For each object, fix such a presentation $X=\widetilde S\times_\rho Y$. The morphisms are the admissible morphisms $F: X_1 \to X_2$ of Definition \ref{def:admissible_flat_morphism}.

 Let $\mathbf{NHIG}(S)$ be the category whose objects are nonlinear Higgs bundles $(f:X\to S,\dbar_f,\theta)$, and whose morphisms, called \emph{Higgs morphisms}, are holomorphic bundle morphisms $F: X_1 \to X_2$ such that $\D F_x(\theta_1(v)_x)=\theta_2(v)_{F(x)}$ for every $s\in S$, $x\in(X_1)_s$, and $v\in T_sS$.

	There exists a faithful functor
	\[
	\mathsf{H}: \mathbf{RFB}(S) \longrightarrow \mathbf{NHIG}(S),
	\]
	which assigns to a reductive flat bundle the nonlinear Higgs bundle determined by the Simpson mechanism using a $G$-harmonic metric. On morphisms, $\mathsf H(F)=F$.
\end{proposition}
\begin{proof}
We first check that admissible maps are closed under composition. Replace each auxiliary $H_{ij}$ by the reductive Zariski closure $R_{ij}$ of its monodromy. Let
\[
L:=\overline{(\sigma_{12},\sigma_{23})(\Gamma)}^{\mathrm{Zar}}\subset R_{12}\times R_{23}.
\]
Both projections are surjective. The unipotent radical of $L^\circ$ maps to connected normal unipotent subgroups of both reductive factors, which are trivial. Therefore, $L$ is reductive. The two homomorphisms from $L$ to the middle group $G_2$ (note that the auxiliary data provide algebraic homomorphisms $H_{12}\to G_2$ and $H_{23}\to G_2$) agree on the dense monodromy and hence agree everywhere. The composite fiber map is equivariant for $L$, which supplies auxiliary data for the composite. Identity maps are admissible for the auxiliary data $H=G$, $\sigma=\rho$, $p_1=p_2=\mathrm{id}$.

Let $F$ be admissible for $(H,\sigma,p_1,p_2)$, and replace $H$ by $\overline{\sigma(\Gamma)}^{\mathrm{Zar}}$. Pass to the connected finite unramified holomorphic cover $q_S:S^\circ\to S$ corresponding to $\sigma^{-1}(H^\circ)$. The restricted representation is Zariski dense in the connected reductive group $H^\circ$. By Theorem \ref{thm:harm-DC}, it admits a harmonic $K_H$-reduction, where $K_H\subset H^\circ$ is a maximal compact subgroup.

For each $i$, choose $c_i\in G_i$ such that $c_i p_i(K_H)c_i^{-1}\subset K_i$. The map
\[
H^\circ/K_H\longrightarrow G_i/K_i,\qquad hK_H\longmapsto p_i(h)c_i^{-1}K_i,
\]
is equivariant and totally geodesic, so the common harmonic reduction induces a harmonic reduction of each endpoint principal $G_i$-bundle over $S^\circ$. Differentiated equivariance of $\psi$ makes $q_S^*F$ preserve the Higgs data associated to this principal harmonic reduction. More explicitly, after the frame changes by $c_i^{-1}$, the fiber map is $\psi_c(y)=c_2\psi(c_1^{-1}y)$, which is equivariant for the conjugated homomorphisms and induces the same original bundle map $F$.

The harmonic-reduction independence proved in Proposition \ref{thm:nonlinear_higgs_correspondence} identifies these endpoint Higgs operators with the pullbacks of the independently chosen endpoint operators. Thus $q_S^*F$ is holomorphic and intertwines the Higgs fields for those operators. These equations descend through the surjective local biholomorphism $q_S$, so $F$ itself is a Higgs-bundle morphism. The assignment $\mathsf H(F)=F$ preserves identities and composition and is faithful.
\end{proof}

Theorem \ref{thm:HolomorphicRH} identifies fiberwise holomorphic flat morphisms with monodromy-equivariant holomorphic maps between the reference fibers.
The following proposition determines when such a morphism is admissible relative to the fixed reductive representations.

\begin{proposition}\label{prop:flat_is_monodromy_associated}
Let $S$ be a connected compact K\"ahler manifold, and let $X_i=\widetilde S\times_{\rho_i}Y_i$, $i=1,2$, be two flat bundles reductive of K\"ahler type as in Definition \ref{def:admissible_flat_morphism}.
Let $F:X_1\to X_2$ be a fiberwise holomorphic flat morphism, and let $\psi:Y_1\to Y_2$ be the corresponding holomorphic map, so that
\[
\psi\comp\rho_1(\gamma)=\rho_2(\gamma)\comp\psi,\qquad F([\tilde s,y]_{\rho_1})=[\tilde s,\psi(y)]_{\rho_2}.
\]
Set
\[
\mathbf G:=\overline{(\rho_1,\rho_2)(\Gamma)}^{\mathrm{Zar}}\subseteq G_1\times G_2,\qquad \hat\rho:=(\rho_1,\rho_2):\Gamma\to\mathbf G,\qquad p_i:=\operatorname{pr}_i|_{\mathbf G}.
\]
Then the following statements hold.
\begin{enumerate}[label=(\roman*)]
\item The (possibly disconnected) algebraic group $\mathbf G$ is reductive, and $F$ is admissible if and only if $\psi$ is $\mathbf G$-equivariant, i.e., $\psi(p_1(g)y)=p_2(g)\psi(y)$ for all $g\in\mathbf G$, $y\in Y_1$. When this holds, $(\mathbf G,\hat\rho,p_1,p_2)$ provides auxiliary data for the admissibility of $F$.
\item Suppose that $Y_i$ are smooth quasi-projective varieties, that the $G_i$-actions $G_i\times Y_i\to Y_i$ are algebraic, and that $\psi$ is algebraic.
Then $F$ is admissible.
In particular, if the $Y_i$ are projective and the $G_i$-actions are algebraic, every fiberwise holomorphic flat morphism $X_1\to X_2$ is admissible.
\item Suppose that $Y_i=V_i$ are finite-dimensional complex vector spaces and $G_i=\GL(V_i)$ acts by the standard linear action.
Then every fiberwise holomorphic flat morphism $X_1\to X_2$ is admissible.
\end{enumerate}
\end{proposition}

\begin{proof}
(i) The description of $F$ in terms of $\psi$ is provided by Theorem \ref{thm:HolomorphicRH}. Let $R_i:=\overline{\rho_i(\Gamma)}^{\mathrm{Zar}}\subseteq G_i$. Each $R_i$ is reductive since $\rho_i$ is reductive. The image of an algebraic-group homomorphism is Zariski closed, so $p_i(\mathbf G)$ is a closed subgroup containing $\rho_i(\Gamma)$. Since $\mathbf G\subseteq R_1\times R_2$, it follows that $p_i(\mathbf G)=R_i$. Let $U=R_u(\mathbf G^\circ)$ be the unipotent radical of the identity component. It is normal in $\mathbf G$. Then $p_i(U)$ is a connected normal unipotent algebraic subgroup of $R_i$, so $p_i(U)=\{e\}$ for both $i$ and $U=\{e\}$. Therefore, $\mathbf G$ is reductive, and its Zariski-dense representation $\hat\rho$ is reductive.

If $\psi$ is $\mathbf G$-equivariant, the data $(\mathbf G,\hat\rho,p_1,p_2)$ satisfy the conditions of Definition \ref{def:admissible_flat_morphism}, proving admissibility. Conversely, suppose that $F$ is admissible for the auxiliary data consisting of a reductive algebraic group $H$, a reductive representation $\sigma:\Gamma\to H$, and algebraic homomorphisms $a_i:H\to G_i$. The image of $a:=(a_1,a_2):H\to G_1\times G_2$ is Zariski closed and contains $(\rho_1,\rho_2)(\Gamma)$ because $a_i\comp\sigma=\rho_i$. Hence $\mathbf G\subseteq a(H)$.
The defining $H$-equivariance of $\psi$ implies its equivariance under every pair in $a(H)$, and therefore under $\mathbf G$.

(ii) Fix $y\in Y_1$. The maps
\[
G_1\times G_2\longrightarrow Y_2,\qquad (g_1,g_2)\longmapsto\psi(g_1\cdot y),\qquad (g_1,g_2)\longmapsto g_2\cdot\psi(y)
\]
are algebraic. For each $y\in Y_1$, let $Z_y:=\{(g_1,g_2)\in G_1\times G_2\,|\, \psi(g_1\cdot y)=g_2\cdot\psi(y)\}$. $Z_y$ is Zariski closed, since it is the inverse image of the closed diagonal in $Y_2\times Y_2$ under the algebraic map $(g_1,g_2)\mapsto(\psi(g_1\cdot y),g_2\cdot\psi(y))$. Monodromy equivariance gives $(\rho_1,\rho_2)(\Gamma)\subseteq Z_y$, hence $\mathbf G\subseteq Z_y$. Since this holds for every $y$, the map $\psi$ is $\mathbf G$-equivariant, and \textup{(i)} applies. When both fibers are projective, the holomorphic map $\psi$ is algebraic, which proves the stated special case.

(iii) Expand the entire holomorphic map $\psi:V_1\to V_2$ at the origin as $\psi(v)=\sum_{d=0}^{\infty}P_d(v)$, where each $P_d:V_1\to V_2$ is a homogeneous polynomial of degree $d$, and the series converges uniformly on compact subsets of $V_1$. For $\gamma\in\Gamma$ and $v\in V_1$, apply monodromy equivariance to $tv$, $t\in\CC$, and compare Taylor coefficients in $t$.
Linearity of the monodromy actions gives
\[
P_d(\rho_1(\gamma)v)=\rho_2(\gamma)P_d(v)\qquad(d\geq0).
\]
The argument in \textup{(ii)}, applied to $P_d$ and the standard linear actions, yields
\[
P_d(p_1(g)v)=p_2(g)P_d(v)\qquad(g\in\mathbf G,\ d\geq0).
\]
Summing the convergent series and using the linearity of $p_2(g)$, we obtain
\[
\psi(p_1(g)v)=\sum_{d=0}^{\infty}P_d(p_1(g)v)=\sum_{d=0}^{\infty}p_2(g)P_d(v)=p_2(g)\psi(v).
\]
(i) now proves (iii).
\end{proof}

\begin{example}\label{ex:veronese_nonlinear_morphism}
Let $E=\widetilde S\times_\rho V$ be a semisimple flat complex vector bundle over $S$, where $\dim_{\CC}V=m\geq1$ and $\rho:\Gamma\to\GL(V)$ is reductive.
Fix an integer $d\geq2$ and put $V_d=\operatorname{Sym}^dV$, $N=\dim_{\CC}V_d=\binom{m+d-1}{d}$, and $\rho_d=\operatorname{Sym}^d\rho$.
The representation $\rho_d$ is reductive, since it is obtained by applying an algebraic representation to the reductive Zariski closure of $\rho(\Gamma)$.
Thus
\[
E_d:=\operatorname{Sym}^dE=\widetilde S\times_{\rho_d}V_d
\]
is again a semisimple flat vector bundle. The map
\[
v_d:V\longrightarrow V_d,\qquad v_d(v)=v^{\otimes d} \qquad (m=2,\ d=2:\ (v_0,v_1)\mapsto(v_0^2,\sqrt2\,v_0v_1,v_1^2)\bigr),
\]
satisfies $v_d(gv)=\operatorname{Sym}^d(g)v_d(v)$ for every $g\in\GL(V)$.
It induces a fiberwise holomorphic flat morphism
\[
F_d:E\longrightarrow E_d,\qquad F_d([\tilde s,v]_\rho)=[\tilde s,v_d(v)]_{\rho_d}.
\]
This map is admissible by Proposition \ref{prop:flat_is_monodromy_associated}. Explicitly, one may choose the auxiliary data
\[
H=\GL(V),\qquad \sigma=\rho,\qquad p_1=\mathrm{id}_{\GL(V)},\qquad p_2=\operatorname{Sym}^d:\GL(V)\to\GL(V_d).
\]

Since $v_d(tv)=t^dv_d(v)$ and $d\geq2$, while $v_d(v)\neq0$ for $v\neq0$, the map $F_d$ is not fiberwise complex-linear.
Therefore, admissible flat morphisms between semisimple flat vector bundles strictly generalize the usual complex-linear flat morphisms.
\end{example}

\subsection{Harmonic maps to $(\calH_0,g_{L^2})$}
When $Y$ is noncompact, its holomorphic automorphism group can be infinite-dimensional. A complete flat holomorphic bundle with a typical fiber $Y$ may not come from a representation of $\Gamma:=\pi_1(S,s_0)$ to a complex Lie subgroup of $\Aut(Y)$.

The following example gives a flat holomorphic bundle whose monodromy cannot be realized by an algebraic action of a finite-dimensional algebraic group on its reference fiber.
\begin{example}
Fix a set of generators of $\pi_1(\CC\setminus\{1,2\},0)$. Consider the representation $\rho: \pi_1(\CC\setminus\{1,2\},0)\to \Aut(\CC^2)$ given by
\[
(z_1,z_2)\mapsto a(z_1,z_2):=(z_2,z_1),\quad (z_1,z_2)\mapsto b(z_1,z_2):=(z_1,z_1^2+z_2).
\]
Then $\deg((a\comp b)^n)=2^n$. For an algebraic action of a finite-dimensional algebraic group on $\CC^2$, the degrees of all action maps are uniformly bounded. Consequently, $\rho$ cannot factor through such an algebraic action.
\end{example}
Eventually, we have to consider equivariant harmonic maps to an infinite-dimensional space, the space of Kähler metrics $\calH_0$. In the remainder of this subsection, assume that $Y$ is compact and connected. We show that $G$-harmonicity is equivalent to harmonicity with respect to the Mabuchi metric under Assumption \ref{ass:finite_dim}, provided that the $K$-action is Hamiltonian.

Recall that $\calH$ and $\calH_0$ are defined by \eqref{eq:space_kah_poten} and \eqref{eq:space_kah_metric}. Let
$V:=\int_{Y}\omega_\phi^m$
be the total volume, which is independent of $\phi\in \calH$. The Mabuchi $L^2$-metric on $\calH$ is defined by
\begin{equation}\label{eq:Mabuchi-metric}
g_{L^2}(\xi,\eta)_\phi:=\frac{1}{V}\int_{Y}\xi\eta\,\omega_\phi^m,\qquad \xi,\eta\in T_\phi\calH.
\end{equation}
The Levi-Civita connection $D$ of $g_{L^2}$ has the Christoffel symbol \cite[Lem.~3.1]{rubinstein10bergman}
\begin{equation}\label{eq:Christoffel-symbol}
\Gamma_\phi(\xi,\eta) = -\tfrac{1}{2}\,\langle \nabla_{g_\phi}\xi,\nabla_{g_\phi}\eta\rangle_{g_\phi},
\end{equation}
where $g_\phi(u,v)=\omega_\phi (u , J_Y v)$ is the Riemannian metric on $Y$ corresponding to $\omega_\phi$.

 The map $\Pi: \calH \to \calH_0$ given by $\phi \mapsto \omega_\phi$ is surjective, with fibers corresponding to the addition of constants.  $\mathcal{H}_0$ can be identified with a totally geodesic subspace of $\mathcal{H}$. This gives a Riemannian structure on $\mathcal{H}_0$. Moreover, $\calH$ is isometric to the Riemannian product $\calH_0\times\RR$, see \cite[\S2.4]{chen00}. In fact, $\calH_0\cong I^{-1}(0)$, the space of normalized Kähler potentials, where $I:\calH\to\RR$ is the Monge--Amp{\`e}re energy: \[I(\phi) = \frac{1}{(m+1)V}\sum_{j=0}^m \int_Y \phi \omega_Y^j \wedge \omega_\phi^{m-j}.\] Since $I(\phi+c)=I(\phi)+c$, for any K\"ahler form $\omega'\in[\omega_Y]$, there exists a unique $\phi\in\calH_0$ such that $\omega'=\omega_\phi$.

A path $\phi(t)$ in $\calH$ is a geodesic if it satisfies the geodesic equation $D_{\dot{\phi}}\dot{\phi}=0$, which expands using \eqref{eq:Christoffel-symbol} to
\begin{equation}\label{eq:geodesic-eq}
\ddot{\phi} - \tfrac{1}{2} |\nabla_{g_\phi} \dot{\phi}|_{g_\phi}^2 = 0.
\end{equation}

Let $G=\Aut_0(Y)$ be the identity component of the holomorphic automorphism group of $Y$, which acts (on the right) on $\calH_0$ via pullbacks. Let $g\in G$, then $[g^\ast\omega_Y]=[\omega_Y]$. Using the identification
$\calH_0\cong I^{-1}(0)$, it induces a map $R_g:I^{-1}(0)\to I^{-1}(0)$ by $\omega_{R_g(\phi)}:=g^\ast \omega_\phi$. Then $R_g(\phi)=R_g(0)+\phi\comp g $. Moreover, $R_g$ extends to a map $R_g:\calH\to\calH$ via
\[R_g(\phi)=R_g(\phi-I(\phi))+I(\phi),\quad \phi\in\calH.\]
It is known that $R_g$ is a differentiable $L^2$ isometry of $\calH$ \cite[\S2.3]{darvas21}.

By \cite{guedj14metric, darvas17}, the metric completion
$(\widebar{\calH},d_{g_{L^2}})$ of $(\calH,d_{g_{L^2}})$
is a Hadamard space, i.e. a complete geodesic metric space which has nonpositive curvature in the sense of Alexandrov. The isometry $\calH\cong \calH_0\times \RR$ extends to $\widebar{\calH}\cong \overline{\calH_0}\times \RR$, where $G$ also acts by isometries and restricts to a left action on $\overline{\calH_0}$, given by
\[
g\cdot \omega_\phi :=\lim_{i\to\infty} {(g^{-1})}^\ast\omega_{\phi_i},\qquad g\in G,
\]
where $\{\omega_{\phi_i}\}$ is a Cauchy sequence in $\calH_0$ converging to $\omega_\phi$. $\overline{\calH_0}$ is also a Hadamard space.

Harmonic maps with metric space targets have been studied in \cite{ks93, ks97}. When the target is $\calH_0$, we may describe the harmonic map more explicitly \cite{rubinstein10bergman}. Let $(N,g)$ be a compact oriented smooth Riemannian manifold (possibly with smooth boundary), with local coordinates $(y^1,\dots,y^n)$ and metric components $g_{ab}$ and inverse $g^{ab}$. A smooth map
\[
\phi: (N,g) \longrightarrow (\calH_0,g_{L^2}),\qquad y \longmapsto \phi(y,\cdot)\in\calH_0,
\]
has energy
\begin{equation}\label{eq:energy-functional}
E(\phi):=\int_N |\D\phi|^2\,\D V_{N,g} = \frac{1}{V}\int_{N\times Y} g^{ab}(y)\,\frac{\partial\phi}{\partial y^a}(y,x)\,\frac{\partial\phi}{\partial y^b}(y,x)\, \omega_{\phi}^m\wedge \D V_{N,g}.
\end{equation}
A harmonic map is a critical point of $E(\phi)$ under variations fixing the boundary values, equivalently it satisfies the weak Euler--Lagrange equation
\[
\D_D^\ast \D\phi = 0,
\]
where $\D_D$ is the exterior covariant derivative associated to the pullback of the Levi-Civita connection $D$ on $(\calH_{\omega_Y},g_{L^2})$ via $\phi$ and $\D_D^\ast$ is its formal adjoint. Locally, using \eqref{eq:Christoffel-symbol}, the Euler--Lagrange equation reads
\[
\Delta_N\phi - \tfrac{1}{2}\,g^{ab}\,\langle \nabla_{g_\phi}\partial_{y^a}\phi,\ \nabla_{g_\phi}\partial_{y^b}\phi\rangle_{g_\phi} = 0,
\]
where $\Delta_N=g^{ab}(\partial_{y^a}\partial_{y^b}-\Gamma_{ab}^c\partial_{y^c})$.

Let $h:\widetilde{S}\to \calH_0$ be $\rho$-equivariant. The energy density $|\D h|^2$ on $\widetilde{S}$ is $\Gamma$-invariant and descends to $S$. Define the energy of $h$ by
\begin{equation}\label{eq:equiv-energy-functional}
E(h):=\int_S |\D h|^2\,\D V_{S,g}.
\end{equation}
We call the fiberwise Kähler metric $\omega_{X/S}$ determined by $h$ via Lemma \ref{lem:fib_Kah_equivalence} \emph{harmonic} with respect to the Mabuchi metric if $h$ is an equivariant harmonic map, which means that $h$ is a critical point of \eqref{eq:equiv-energy-functional} and satisfies \eqref{eq:rho-equivariance}.
\begin{lemma}
The orbit $G\cdot\omega_Y$ is an embedded closed smooth submanifold of $\calH_0$.
\end{lemma}
\begin{proof}
It suffices to show the action $G \curvearrowright \calH_0$ is proper. It is known that the action $\Diff(Y) \curvearrowright \mathcal{R}(Y)$ is proper, where $\mathcal{R}(Y)$ is the space of Riemannian metrics (see e.g. \cite[Th.~2.23]{corro21}). By \cite[Prop.~2.2]{corro21}, we only need to show that for any $\omega\in \mathcal{H}_0$, if a sequence $\{g_n\}$ in $G$ is such that $\omega_n = g_n^*\omega$ converges smoothly in $\calH_0$ to $\omega$, then $\{g_n\}$ has a convergent subsequence in $G$. By \cite[Lem.~2.25]{corro21}, there is a subsequence (still denoted by $\{g_n\}$) converging in $C^\infty$ to $g\in \Diff(Y)$. The maps $g_n\in \Aut_0(Y)$ are holomorphic, so the limit $g$ is also holomorphic. Since $g\in \Diff(Y)$, $g^{-1}$ is also holomorphic. Therefore, $g\in \Aut_0(Y)$. Since $G$ is a closed subgroup of $\Aut_0(Y)$, and $g_n \in G$ converges to $g$ in $\Aut_0(Y)$, the limit $g$ must belong to $G$.
\end{proof}

\begin{proposition}\label{prop:orbit_tot_geodesic}
Let $G$ be a connected complex reductive closed Lie subgroup of $\mathrm{Aut}_0(Y)$. If the stabilizer $K=\Stab_G(\omega_Y)$ is a compact real form of $G$ and the action of $K$ on $(Y,\omega_Y)$ is Hamiltonian, then the orbit $G\cdot\omega_Y$ is a totally geodesic submanifold of $\mathcal{H}_0$. Therefore, a $G$-harmonic metric is equivalent to a metric $G$-modeled on $(Y,\omega_Y)$ that is harmonic with respect to the Mabuchi metric.
\end{proposition}
\begin{proof}
The induced Mabuchi metric on $G/K$ is $G$-invariant. Each geodesic through $\omega_Y$ is generated by an element of $\I\mathfrak k$ and, by the Hamiltonian hypothesis and \cite[Th.~3.5]{mabuchi87}, is also a Mabuchi geodesic. Translating by $G$ proves that $G\cdot\omega_Y$ is totally geodesic. The vanishing of the second fundamental form identifies the two harmonic-map equations.
\end{proof}

\begin{proposition}
Let $G$ be a connected complex reductive closed Lie subgroup of $\mathrm{Aut}_0(Y)$. If $G\cdot\omega_Y$ is a totally geodesic submanifold of $\mathcal{H}_0$, then $K = \mathrm{Stab}_G(\omega_Y)$ is a compact real form of $G$.
\end{proposition}

\begin{proof}
Since $Y$ is compact and $G$ is closed in $\mathrm{Aut}_0(Y)$, $K$ is a closed subgroup of the compact group of holomorphic isometries of $(Y,\omega_Y)$, hence compact. The induced Mabuchi metric makes $G\cdot\omega_Y\cong G/K$ a finite-dimensional homogeneous Riemannian manifold, hence complete. By total geodesy, its minimizing geodesics are geodesics in the Hadamard completion $\overline{\calH_0}$. Uniqueness of geodesics in that completion makes the orbit geodesically convex and therefore contractible. Let $L\subset G$ be a maximal compact subgroup containing $K$. The Cartan decomposition $G=\exp(\I\mathfrak l)L$ gives a deformation retraction of $G/K$ onto $L/K$. Hence $L/K$ is contractible. A compact manifold without boundary is contractible only when it is a point, by its top-dimensional mod-two homology. Thus $K=L$, which is a compact real form of $G$.
\end{proof}

\begin{example}\label{ex: mabuchi's result}
	Let $(Y,H_Y,\omega_Y)$ be a polarized compact cscK manifold as in Example \ref{ex:rel_csck_fib}. Then $G=\Aut_0(Y,H_Y)$ and $K=\mathrm{Isom}_0(Y,H_Y,\omega_Y)$ satisfy Assumption \ref{ass:finite_dim}. According to the Yau-Tian-Donaldson conjecture (which has been resolved in the Kähler-Einstein case), the existence of $\omega_Y$ would be guaranteed by the $K$-polystability of $(Y,H_Y)$. More generally, one can consider a compact cscK manifold $(Y,\omega_Y)$ and let $G:=\ker(\Aut_0(Y)\to \Alb(Y))^\circ$, which is reductive with compact real form $K=\Stab_G(\omega_Y)$ \cite{Fu78, mabuchi87}. Let $\mathcal{C}_0 \subset \calH_0$ be the subset of normalized potentials corresponding to cscK metrics. Then by \cite[Th.~6.3]{mabuchi87} and \cite[Th.~1.3]{berman_berndtsson17convexity} (see also {\cite[Prop.~ 2.1]{hallam23}}), $\mathcal{C}_0=G\cdot\omega_Y\cong G/K$ is a connected, finite-dimensional, totally geodesic smooth submanifold of $\calH_0$. By Theorem \ref{thm:harm-DC} and Remark \ref{rmk:global_form}, a reductive representation $\rho:\Gamma\to G$ gives rise to a relatively holomorphic relatively cscK form $\omega$ on a flat holomorphic bundle, whose restriction to fibers $\omega_{X/S}$ is harmonic.
\end{example}

\section{Variation of nonabelian Hodge structure}
In the final section, we introduce a new factor into the Simpson mechanism, the twisting maps. This leads to the twisted Simpson mechanism, and a broader notion of nonlinear harmonic bundles. Then we prove in two important special cases that variations of nonabelian Hodge structure are nonlinear harmonic bundles.

\subsection{Twisted Simpson mechanism}\label{subsec:twisted_mech}
Let $f: X \to S$ be a smooth fiber bundle on a complex manifold $S$. Suppose the relative real tangent bundle $T_{X/S}^\RR$ is equipped with two integrable fiberwise complex structures $T_{X/S}^A$ and $T_{X/S}^B$. We denote the resulting complex fiber bundles by $X_A = (f, T_{X/S}^A)$ and $X_B = (f, T_{X/S}^B)$.
We denote the quotient bundles by
\[
Q_A := \frac{TX^\CC}{\widebar{T_{X/S}^A}} \quad \text{and} \quad Q_B := \frac{TX^\CC}{\widebar{T_{X/S}^B}}.
\]

\begin{lemma}\label{lem:real_beta}
Let $Y$ be a smooth manifold equipped with two complex structures $J_A$ and $J_B$, and let $TY_A$ and $TY_B$ be the corresponding holomorphic tangent bundles. There exists a natural bijection between complex vector bundle isomorphisms $\beta_Y: TY_A \to TY_B$ and real vector bundle isomorphisms $\beta_Y^\RR: TY^\RR \to TY^\RR$ satisfying
\[\beta_Y^\RR \comp J_A = J_B \comp \beta_Y^\RR.\]
\end{lemma}
\begin{proof}
Suppose we are given a complex bundle isomorphism $\beta_Y: TY_A \to TY_B$. We define $\beta_Y^\RR: TY^\RR \to TY^\RR$ by
$\beta_Y^\RR(u) := \beta_Y(v) + \widebar{\beta_Y(v)}$, where $u=v+\bar{v}$  with $v\in TY_A$. Then we have
\begin{align*}
    \beta_Y^\RR (J_A u) &= \beta_Y(\I v) + \widebar{\beta_Y({\I} v)}= \I\beta_Y(v) - \I\overline{\beta_Y(v)}\\
		&=J_B(\beta_Y(v) + \widebar{\beta_Y(v)})=J_B(\beta_Y^\RR(u)).
\end{align*}
Conversely, let $\beta_Y^\RR: TY^\RR \to TY^\RR$ be a real isomorphism such that $\beta_Y^\RR \comp J_A = J_B \comp \beta_Y^\RR$. We extend $\beta_Y^\RR$ complex-linearly to $\beta_Y^\CC: TY^\CC \to TY^\CC$.
For any $v \in TY_A$,
\[
    J_B(\beta_Y^\CC (v)) = \beta_Y^\CC(J_A v) = \beta_Y^\CC(\I v) = \I \beta_Y^\CC(v).
\]
This implies that $\beta_Y :=\beta_Y^{\mathbb{C}} \big|_{TY_A} : TY_A \to TY_B$ is a complex bundle isomorphism.
\end{proof}
To relate geometries of $X_A$ and $X_B$, we demand an identification map.
\begin{definition}
A \emph{twisting map} $\beta_{X/S}$ is a smooth complex vector bundle isomorphism $\beta_{X/S}: T_{X/S}^A \xrightarrow{\cong} T_{X/S}^B$, equivalently a real bundle isomorphism $\beta_{X/S}^\RR:T_{X/S}^\RR\to T_{X/S}^\RR$ satisfying $\beta_{X/S}^\RR\comp J_{X/S}^A=J_{X/S}^B\comp \beta_{X/S}^\RR$. It is said to be \emph{effective} if $\beta_{X/S}=\id$ when $T^{A}_{X/S}=T^B_{X/S}$, or if $J^{A}_{X/S}-J^B_{X/S}$ is nowhere vanishing when $T^{A}_{X/S}\neq T^B_{X/S}$.
\end{definition}

Now we consider a twisted version of the Simpson mechanism in Section \ref{sec:simpson_mechanism}.
\begin{definition}[Twisted  Simpson mechanism]
	Fix two reference $\dbar$-operators $\dbar_{A,0}:f^*\widebar{TS}\to Q_A$ and $\dbar_{B,0}:f^*\widebar{TS}\to Q_B$ on $X_A$ and $X_B$ respectively. Suppose $J_{X/S}^A-J_{X/S}^B$ is nowhere vanishing.
	\begin{enumerate}
		\item {\itshape From almost Higgs fields to almost connections:}  Let $\dbar_B:f^*\widebar{TS}\to Q_B$ be a $\dbar$-operator on $X_B$, and let $\theta$ be an almost Higgs field on $X_B$. Let $g_{X/S}$ be a fiberwise Riemannian metric on $f$ such that $\omega_{X/S}^B(\sbullet,\sbullet):=g_{X/S}(J_{X/S}^B\sbullet,\sbullet)$ is a fiberwise Kähler metric on $X_B$. Define $\bar{\theta}_J$ by
		\begin{equation}\label{eq:theta_barJ}
			\bar{\theta}_J(\bar{v}):=\mathrm{pr}_{T_{X/S}^B}(\I J_{X/S}^A\widebar{\theta(v)}_{J_{X/S}^B}),
		\end{equation}
where $v\in TS$, $\widebar{(\cdot)}_{J_{X/S}^B}: T_{X/S}^B\to \widebar{T_{X/S}^B}$ is the conjugation, and $\mathrm{pr}_{T_{X/S}^B}:T_{X/S}^\CC\to T_{X/S}^B$ is the projection.

Now we define a smooth bundle morphism $f^*\widebar{TS}\to Q_A$ by
\begin{equation}
		\dbar_A := \beta_{X/S}^{-1}(\dbar_B-\dbar_{B,0} + \bar{\theta}_J)+\dbar_{A,0}, \label{eq:twist_simpson_dbar_A}
\end{equation}
which is a $\dbar$-operator on $X_A$. We further assume that $\beta_{X/S}^\RR$ is an isometry with respect to $g_{X/S}$, and $g_{X/S}$ corresponds to a fiberwise Kähler metric $\omega_{X/S}^A$ on $X_A$. Next, we define an almost connection on $X_A$ by
\begin{equation}\label{eq:twist_simpson_partial_A}
	\partial_A:=\partial_{\omega_{X/S}^A}+2|J_{X/S}^A-J_{X/S}^B|_{g_{X/S}}\beta_{X/S}^{-1}(\theta),
\end{equation}
where $\partial_{\omega_{X/S}^A}$ is a symplectic almost connection associated to $\dbar_A$ and $\omega_{X/S}^A$ and $|\cdot|_{g_{X/S}}$ denotes the pointwise operator norm of the endomorphism with respect to the metric $g_{X/S}$. We assume that the positive function $|J_{X/S}^A-J_{X/S}^B|_{g_{X/S}}$ is smooth.
		\item {\itshape From connections to almost Higgs fields:} Suppose $X_A$ is equipped with a $\dbar$-operator $\dbar_A$ and an almost connection $\partial_A$.
Let $g_{X/S}$ be a fiberwise Riemannian metric on $f$ which corresponds to a fiberwise Kähler metric $\omega_{X/S}^A$ on $X_A$. Let $\partial_{\omega_{X/S}^A}$ be a symplectic almost connection associated to $\dbar_A$ and $\omega_{X/S}^A$. We define an almost Higgs field on $X_B$ by
\begin{equation}
		\theta := \tfrac{1}{2}|J_{X/S}^A-J_{X/S}^B|_{g_{X/S}}^{-1} \beta_{X/S} (\partial_A - \partial_{\omega_{X/S}^A}). \label{eq:twist_simpson_theta_B}
\end{equation}
We further assume that $\beta_{X/S}^\RR$ is an isometry with respect to $g_{X/S}$, and $g_{X/S}$ corresponds to a fiberwise Kähler metric $\omega_{X/S}^B$ on $X_B$. Next, we define a $\dbar$-operator on $X_B$ by
\begin{equation}\label{eq:twist_simpson_dbar_B}
		\dbar_B := \beta_{X/S} (\dbar_A-\dbar_{A,0})+\dbar_{B,0} - \bar{\theta}_J,
\end{equation}
	where $\bar{\theta}_J$ is defined by \eqref{eq:theta_barJ}.
	\end{enumerate}
\end{definition}
\begin{remark}
	When $J_{X/S}^A=J_{X/S}^B$, the mechanism degenerates and is not compatible with the Simpson mechanism in Section \ref{sec:simpson_mechanism}. By removing $|J_{X/S}^A-J_{X/S}^B|_{g_{X/S}}$ (or its inverse) and replacing $\bar{\theta}_J$ by $\bar{\theta}_{\omega_{X/S}^B}$ in \eqref{eq:twist_simpson_dbar_A}-\eqref{eq:twist_simpson_dbar_B}, we obtain a version compatible with Section \ref{sec:simpson_mechanism} ($\beta_{X/S}=\id$ and $\dbar_{B,0}=\dbar_{A,0}$) without assuming $J_{X/S}^A\neq J_{X/S}^B$. However, $\bar{\theta}_J$ is always well-defined while Definition \ref{def:comp_conj_higgs} of $\bar{\theta}_{\omega_{X/S}^B}$ requires extra conditions.

	In the direction $(1)$, we may unify the two approaches by
	\begin{align*}
			\dbar_A &:= \beta_{1,X/S}^{-1}(\dbar_B-\dbar_{B,0} + \bar{\theta}_{1,J})+\beta_{2,X/S}^{-1}(\bar{\theta}_{2,\omega_{X/S}^B})+\dbar_{A,0},\\
			\partial_A &:=\partial_{\omega_{X/S}^A}+2|J_{X/S}^A-J_{X/S}^B|_{g_{X/S}}\beta_{1,X/S}^{-1}(\theta_1)+2\beta_{2,X/S}^{-1}(\theta_2),
	\end{align*}
	where $\beta_{1,X/S}$ and $\beta_{2,X/S}$ are two twisting maps and $\theta_1$ and $\theta_2$ are almost Higgs fields, where $\theta_2$ satisfies the extra conditions such that $\bar{\theta}_{2,\omega_{X/S}^B}$ is well-defined.
\end{remark}
\begin{remark}
	 $J_{X/S}^A\widebar{\theta(v)}_{J_{X/S}^B}\in T_{X/S}^B$ if and only if  $\theta(v)\in \ker (J_{X/S}^AJ_{X/S}^B+J_{X/S}^BJ_{X/S}^A)$. In this case, $\bar{\theta}_J(\bar{v})=\I J_{X/S}^A\widebar{\theta(v)}_{J_{X/S}^B}$. If this holds for all $\theta$, then $J_{X/S}^AJ_{X/S}^B=-J_{X/S}^BJ_{X/S}^A$ and $J_{X/S}^A,J_{X/S}^B,J_{X/S}^C:=J_{X/S}^AJ_{X/S}^B$ form a hypercomplex structure on each fiber. In this case, if $g_{X/S}$ is Kähler with respect to $J_{X/S}^A$ and $J_{X/S}^B$, then it must be Kähler with respect to $J_{X/S}^C$, so that $(g_{X/S},J_{X/S}^A,J_{X/S}^B,J_{X/S}^C)$ is a hyperkähler structure, and $|J_{X/S}^A-J_{X/S}^B|_{g_{X/S}}=\sqrt{2}$.
\end{remark}
We demonstrate the analogy between $\bar{\theta}_J$ and $\bar{\theta}_{\omega_{X/S}}$ in the following lemma.
\begin{lemma}
 Let $(Y,\omega_Y)$ be a Kähler manifold. Let $v\in \mathfrak{k}^\CC$ be a holomorphic vector field, where $\mathfrak{k}$ is a subspace of $\mathfrak{aut}(Y,\omega_Y):=\{v\in H^0(Y,TY)\,|\, v_\RR \text{ is Killing}\}$ such that $\mathfrak{k}\cap \I \mathfrak{k}=\{0\}$ and that vector fields in $\mathfrak{k}$ are Hamiltonian. Then we have $\iota_v\omega_Y=\dbar f_v$ for some function $f_v$ and $\iota_{\bar{v}_{\mathfrak{k}}}\omega_Y=\dbar \widebar{f_v}$, where $\bar{v}_{\mathfrak{k}}\in \mathfrak{k}^\CC$ is the conjugate of $v$ determined by the real structure $\mathfrak{k}$.

	Let $(Y, g_Y, J_A, J_B, J_C)$ be a hyperkähler manifold. Let $\Omega_{J_B} = \omega_{J_A} - \I\omega_{J_C}$ be the holomorphic symplectic form for $J_B$. Let $v$ be a $J_B$-holomorphic vector field that is Hamiltonian with respect to $\Omega_{J_B}$, i.e., $\iota_v \Omega_{J_B} = \partial f_v$ for some holomorphic function $f_v$. Define the conjugate vector field $\tilde{v}$ as the Hamiltonian vector field of the function $ -\frac{1}{2}\overline{f_v}$ with respect to $\omega_{J_B}$, i.e., $\iota_{\tilde{v}} \omega_{J_B} = -\D\left( \frac{1}{2} \widebar{f_v} \right)$.
Then $\tilde{v} = \I J_A(\bar{v})$.
\end{lemma}
\begin{proof}
  We prove the first statement. Since $v\in \mathfrak{k}^\CC$, we have $v=u+\I w$ and $\bar{v}_{\mathfrak{k}}=u-\I w$ for $u,w\in \mathfrak{k}$. Since $u,w$ are Hamiltonian, we have
	\[\iota_v\omega_Y=\iota_{u+\I w}\omega_Y=\dbar (f_u+\I f_w)=:\dbar (f_v),\]
	for some real functions $f_u$ and $f_w$. Then we have $\iota_{\bar{v}_{\mathfrak{k}}}\omega_Y=\dbar\widebar{f_v}$.

	Now we prove the second statement. We have $\iota_v (\omega_{J_A} - \I\omega_{J_C}) = \partial f_v$. Taking the complex conjugate,
	\[ \iota_{\bar{v}} (\omega_{J_A} + \I\omega_{J_C}) = \bar{\partial} \widebar{f_v} = \D\widebar{f_v}.\]
	Since $\Omega_{J_{B}}$ has type $(2,0)$ and $\bar{v}$ has type $(0,1)$ with respect to $J_B$, we have $\iota_{\bar{v}}\Omega_{J_B}=0$, implying $\iota_{\bar{v}} \omega_{J_A} = \I \iota_{\bar{v}} \omega_{J_C}$. Then $\iota_{\bar{v}} \omega_{J_C} = -\frac{\I}{2} \D\widebar{f_v}$. On the other hand, \[\iota_{\I J_A\bar{v}} \omega_{J_B} (\cdot) = \I g_Y(J_B(J_A\bar{v}), \cdot)=-\I\iota_{\bar{v}}\omega_{J_C}=-\tfrac{1}{2} \D\widebar{f_v}.\]
	Therefore, $\tilde{v}=\I J_A(\bar{v})$.
\end{proof}
\begin{definition}
Let $(\dbar_B,\theta)$ be a Higgs bundle structure on $X_B$, i.e., the pseudo-curvature $G_{D_B''}=0$ for $D_B''=\dbar_B+\theta$. Let $g_{X/S}$ be a fiberwise Riemannian metric which is Kähler with respect to $J_{X/S}^A$ and $J_{X/S}^B$. With the auxiliary choices in the twisted Simpson mechanism fixed, including an isometry $\beta_{X/S}^\RR$, we call $g_{X/S}$ a \emph{$\beta$-twisted harmonic metric} if the pair $(\dbar_A,\partial_A)$ obtained from \eqref{eq:twist_simpson_dbar_A} and \eqref{eq:twist_simpson_partial_A} defines a flat holomorphic bundle structure on $X_A$. Conversely, starting from a flat holomorphic connection on $X_A$ and making the same auxiliary choices, we call $g_{X/S}$ a \emph{$\beta$-twisted harmonic metric} if the pair $(\dbar_B,\theta)$ obtained from \eqref{eq:twist_simpson_theta_B} and \eqref{eq:twist_simpson_dbar_B} defines a Higgs bundle structure on $X_B$.
\end{definition}
The following three results generalize Lemma \ref{lem:space_Kah_conn}, Proposition \ref{prop: kahler connection}, and Proposition \ref{prop:Chern_conn} (where $\tilde{\theta}=0$), and the proofs are completely analogous.
\begin{lemma}\label{lem:twisted_affine_space}
Let $(f, T_{X/S}^A)$ be a complex fiber bundle with a fiberwise Kähler metric $\omega_{X/S}^A$ and a $\dbar$-operator $\dbar_A$ satisfying the lifting condition. Let $\tilde{\theta} \in C^\infty(X, f^*T^*S\otimes T_{X/S}^A)$. Whenever nonempty, the space of symplectic almost connections $\partial_{\omega_{X/S}^A}$ associated to $\dbar_A$ and $\omega_{X/S}^A$ such that $\partial_{\omega_{X/S}^A} + \tilde{\theta}$ satisfies the lifting condition is affine modeled on $A^{1,0}(S, \mathfrak{a}_{X/S})$.
\end{lemma}

\begin{proposition}\label{prop:twisted_existence}
In the setup of Lemma \ref{lem:twisted_affine_space},  suppose that $f$ admits a \emph{symplectic atlas} $\mathcal U=\{(U_a,\Phi_a)\}$ (i.e. $\omega_s^A = \Phi_{a,s}^* \omega_Y$ for $s\in U_a$) such that $\partial_a+\tilde\theta$ satisfies the lifting condition and that, in local holomorphic base coordinates, every coefficient of $B_a:=\dbar_A-\dbar_a$ can be written smoothly as $u+\I v$, where $u|_{X_s},v|_{X_s}\in\mathfrak{aut}(X_s,\omega_s^A)$ for all $s\in U_a$. Then there exists a symplectic connection $\nabla_{\omega_{X/S}^A}^{1,0}$ associated to $\dbar_A$ and $\omega_{X/S}^A$ such that $\partial_{\omega_{X/S}^A} + \tilde{\theta}$ satisfies the lifting condition.
\end{proposition}

\begin{lemma}\label{lem:twisted_Chern_conn}
Suppose the conditions of Proposition \ref{prop:twisted_existence} hold. Fix a smooth real vector bundle $\mathfrak k_{X/S}$ of finite rank over $S$, smoothly realized as families of fiberwise holomorphic vector fields, with $\mathfrak k_s\subset\mathfrak{aut}(X_s,\omega_s^A)$ and $\mathfrak k_s\cap\I\mathfrak k_s=\{0\}$.
Write $\mathfrak k_{X/S}^\CC=\mathfrak k_{X/S}\oplus\I\mathfrak k_{X/S}$, let $\phi_s$ be its conjugation, and set $(\phi B)(u):=\phi_s(B(\bar u))$ for $u\in T_sS$.
Assume that $B_a:=\dbar_A-\dbar_a\in A^{0,1}(U_a,\mathfrak k_{X/S}^\CC)$ and that the  compatibility condition $\partial_a-\partial_b=\phi(\dbar_a-\dbar_b)$ holds on overlaps. Then the local operators $\partial_a+\phi B_a$ glue to a symplectic almost connection associated to $\dbar_A$ and $\omega_{X/S}^A$, whose sum with $\tilde\theta$ satisfies the lifting condition.
The corresponding symplectic connection is called the \emph{$\tilde\theta$-twisted Chern connection} determined by this compatible atlas. It is independent of the choice of an equivalent atlas, where two such atlases are equivalent if the same compatibility condition holds between all their charts on overlaps, with the same $\mathfrak k_{X/S}$.
\end{lemma}

We end this subsection by discussing some properties of the twisting map.

\begin{lemma}
Let $(Y, g_Y)$ be a Riemannian manifold such that $g_Y$ is Kähler with respect to two complex structures $J_A$ and $J_B$. Let $\beta_Y^\RR$ be as above. Suppose $\nabla_{g_Y} \beta_Y^\RR = 0$, i.e., $\beta_Y^\RR$ is parallel with respect to the Levi-Civita connection. Assume that $(J_A - J_B)$ is invertible (which holds if $J_A \neq J_B$ are part of a hyperkähler structure).
If $v$ is a real $J_A$-holomorphic vector field such that $w = \beta_Y^\RR(v)$ is a real $J_B$-holomorphic vector field, then $v$ is parallel.
\end{lemma}

\begin{proof}
Let $M = \nabla_{g_Y} v$. The $J_A$-holomorphicity of $v$ implies $M J_A = J_A M$. The $J_B$-holomorphicity of $w$ implies $(\nabla_{g_Y} w) J_B = J_B (\nabla_{g_Y} w)$. Since $\beta_Y^\RR$ is parallel, $\nabla_{g_Y} w = \nabla_{g_Y}(\beta_Y^\RR v) = \beta_Y^\RR M$. We have
\[ (\beta_Y^\RR M) J_B = J_B (\beta_Y^\RR M)=\beta_Y^\RR J_A M. \]
Thus $M J_B = J_A M=M J_A$, and then $M=\nabla_{g_Y} v = 0$.
\end{proof}

\begin{lemma}
Let $(Y, g_Y, J_A, J_B, J_C)$ be a hyperkähler manifold (where $J_C = J_A J_B$). Let $\beta_{Y}^\RR$ be an endomorphism of $TY^{\RR}$ which belongs to $\mathrm{span}_\RR\{\id, J_A, J_B, J_C\}$. If $\beta_Y^\RR \comp J_A = J_B \comp \beta_Y^\RR$, then $\beta_Y^\RR$ has the form
\[ \beta_Y^\RR = c_1 (\id + J_C) + c_2 (J_A + J_B), \quad c_1, c_2 \in \RR. \]
If $\beta_Y^\RR$ is furthermore an isometry with respect to $g_Y$, then $2(c_1^2 + c_2^2) = 1$. This allows for a parameterization by $\theta \in [0, 2\pi)$ as
\[ \beta_Y^\RR(\theta) = (1/\sqrt{2}) \bigl( \cos\theta\, (\id + J_C) + \sin\theta\, (J_A + J_B) \bigr). \]
\end{lemma}
\begin{proof}
 Let $\beta_Y^\RR = a_0 \id + a_1 J_A + a_2 J_B + a_3 J_C$. Imposing $\beta_Y^\RR\comp J_A = J_B\comp \beta_Y^\RR$ leads to the constraints $a_0 = a_3$ and $a_1 = a_2$. Next, we compute
\begin{align*}
	(\beta_Y^\RR)^* \beta_Y^\RR &= \bigl(c_1(\id - J_C) - c_2(J_A + J_B)\bigr) \bigl(c_1(\id + J_C) + c_2(J_A + J_B)\bigr)\\
	&=2(c_1^2 + c_2^2)\id.
\end{align*}
Setting this to $\id$ requires $2(c_1^2 + c_2^2) = 1$.
\end{proof}
\begin{remark}
	Let $(Y, g_Y, J_A, J_B, J_C)$ be an irreducible connected hyperkähler manifold of dimension $4m$, i.e. $\mathrm{Hol}(g_Y)=\mathrm{Sp}(m)$. If $\beta_Y^\RR$ is a parallel endomorphism of $TY^\RR$, then by the holonomy principle, $\beta_Y^\RR$ lies in the commutant of $\mathrm{Sp}(m)$ in $\End(\RR^{4m})$ at each point, which is spanned by $\{\id,J_A,J_B,J_C\}$. Since $\beta_Y^\RR$ is parallel, we have $\beta_Y^\RR\in \mathrm{span}_\RR\{\id, J_A, J_B, J_C\}$.
\end{remark}

\subsection{Cotangent bundles of abelian varieties}\label{subsec:cotangent_ab_var}

Let $S := \{ \Pi \in M_k(\mathbb{C}) \mid \Pi = \Pi^{\text{T}}, \Pi_y:=\Im(\Pi) > 0 \}$ be the Siegel upper half-space of degree $k$. We consider the universal family $f: X \to S$, where the fiber over $\Pi \in S$ is the cotangent bundle of the corresponding abelian variety $A_\Pi = \mathbb{C}^k / (\mathbb{Z}^k + \Pi \mathbb{Z}^k)$, i.e., $X_\Pi = T^* A_\Pi=A_\Pi\times \CC^k$. In this subsection, we will verify that $f: X \to S$ admits a $\beta$-twisted harmonic metric.

$f$ is a trivial smooth fiber bundle, with fiberwise complex structure $T_{X/S}^B$ determined by the complex structure on $A_\Pi$. Denote $X_B:=(f,T_{X/S}^B)$, which is a holomorphic fibration with adapted local holomorphic coordinates $(\Pi_{ij},q^a,p^\alpha)$, where $\Pi_{ij}=(\Pi_{x})_{ij}+\I(\Pi_{y})_{ij}$ ($1 \le i\le j \le k$) are the entries of the period matrix $\Pi$, $q^a=q_x^a+\I q_y^a$ ($a=1,\dots,k$) are (periodic) coordinates on the abelian variety, $p^\alpha=p_x^\alpha+\I p_y^\alpha$ ($\alpha=1,\dots,k$) are linear coordinates on the cotangent fibers. Let $\partial_{ij},\partial_{a},\partial_{\alpha}$ be the corresponding basis of $TX_B$, and $\dbar_B$ be the canonical $\dbar$-operator on $X_B$.

The fiber $X_\Pi$ carries a hyperkähler structure $(g_\Pi,J_\Pi^A,J_\Pi^B,J_\Pi^C)$ given by
\begin{align}
    g_\Pi &= (\Pi_y^{-1})^{ab} (\D q_x^a \otimes \D q_x^b + \D q_y^a \otimes \D q_y^b) + (\Pi_y)_{\alpha\beta} (\D p_x^\alpha \otimes \D p_x^\beta + \D p_y^\alpha \otimes \D p_y^\beta),\label{eq:cotangent_hyperkah_eq1}\\
		J_{\Pi}^B (\partial_{q_x}) &= \partial_{q_y}, \  J_{\Pi}^B (\partial_{q_y}) = -\partial_{q_x}, \  J_{\Pi}^B (\partial_{p_x}) = \partial_{p_y},\  J_{\Pi}^B (\partial_{p_y}) = -\partial_{p_x},\label{eq:cotangent_hyperkah_eq2}\\
		 J_{\Pi}^A (\partial_{q_x}) &= -\Pi_y^{-1} \partial_{p_y}, \  J_{\Pi}^A (\partial_{q_y}) = -\Pi_y^{-1} \partial_{p_x},\  J_{\Pi}^A (\partial_{p_x}) = \Pi_y \partial_{q_y}, \  J_{\Pi}^A (\partial_{p_y}) = \Pi_y \partial_{q_x},\label{eq:cotangent_hyperkah_eq3}\\
		 J_{\Pi}^C (\partial_{q_x}) &= -\Pi_y^{-1} \partial_{p_x}, \  J_{\Pi}^C (\partial_{q_y}) = \Pi_y^{-1} \partial_{p_y},\
	J_{\Pi}^C (\partial_{p_x}) = \Pi_y \partial_{q_x},\ J_{\Pi}^C (\partial_{p_y}) = -\Pi_y \partial_{q_y}.\label{eq:cotangent_hyperkah_eq4}
\end{align}
The corresponding Kähler forms are
\begin{align*}
	 \omega_{J_{\Pi}^B} &= (\Pi_y^{-1})^{ab} \D q_x^a \wedge \D q_y^b + (\Pi_y)_{\alpha\beta} \D p_x^\alpha \wedge \D p_y^\beta,\\
	  \omega_{J_{\Pi}^A} &=  \D p_x^\alpha \wedge \D q_y^\alpha - \D q_x^\alpha \wedge \D p_y^\alpha = \Im(\D p^\alpha \wedge \D q^\alpha),\\
		 \omega_{J_{\Pi}^C} &= \D p_x^\alpha \wedge \D q_x^\alpha - \D p_y^\alpha \wedge \D q_y^\alpha = \Re(\D p^\alpha \wedge \D q^\alpha).
\end{align*}
We define the local $\CC^k$-valued functions $\xi$ and $\eta$ on $X$ by
\begin{align*}
    \xi = \Pi_y^{-1} q_y - \I p_x, \quad \eta &= q_x - \Pi_x \Pi_y^{-1} q_y + \I ( \Pi_x p_x - \Pi_y p_y ).
\end{align*}
These coordinates define the biholomorphism
\[  \rho: (X_\Pi, J_{\Pi}^A) \xrightarrow{\cong} (\CC^*)^{2k},\quad
    (q, p) \mapsto \bigl( \E^{2\pi \I \xi^1}, \dots, \E^{2\pi \I \xi^k}, \E^{2\pi \I \eta^1}, \dots, \E^{2\pi \I \eta^k} \bigr).\]
In fact, this map is well-defined since it is invariant under the translations $q_x \mapsto q_x + m + \Pi_x n$ and $q_y \mapsto q_y + \Pi_y n$ for $(m, n) \in \ZZ^k \times \ZZ^k$. $(q,p)$ can be determined by $(\xi,\eta)$ via
\begin{equation}\label{eq:qp_inverse_rel}
	q = \text{Re}(\eta) + \Pi \text{Re}(\xi),\quad p = -\I \Pi_y^{-1} \bigl( \text{Im}(\eta) + \widebar{\Pi} \text{Im}(\xi) \bigr).
\end{equation}
\begin{lemma}
The local functions $\xi^\alpha$ and $\eta^a$ are holomorphic with respect to $J_{\Pi}^A$.
\end{lemma}
\begin{proof}
    A function $f$ is holomorphic if ${J_{\Pi}^A}^*\D f = \I \D f$. By the definition, we have ${J_\Pi^A}^* \D q_x = \Pi_y \D p_y$, ${J_\Pi^A}^* \D p_y = -\Pi_y^{-1} \D q_x$, ${J_\Pi^A}^* \D q_y = \Pi_y \D p_x$, ${J_\Pi^A}^* \D p_x = -\Pi_y^{-1} \D q_y$. We have $\D \xi = \Pi_y^{-1} \D q_y - \I \D p_x$, and
		\[{J_\Pi^A}^*\D \xi=\Pi_y^{-1} \Pi_y \D p_x+\I\Pi_y^{-1}\D q_y=\I\D\xi. \]
   Thus $\xi$ is holomorphic. For $\eta$, we have
    \begin{align*}
        {J_\Pi^A}^* \D \Re(\eta) &= {J_\Pi^A}^* (\D q_x - \Pi_x \Pi_y^{-1} \D q_y)= \Pi_y \D p_y - \Pi_x \Pi_y^{-1} (\Pi_y \D p_x) = \Pi_y \D p_y - \Pi_x \D p_x. \\
        \I {J_\Pi^A}^*(\D \Im(\eta)) &= \I {J_\Pi^A}^*(\Pi_x \D p_x - \Pi_y \D p_y)=\I  (\Pi_x (-\Pi_y^{-1} \D q_y) - \Pi_y (-\Pi_y^{-1} \D q_x)) =\I (\D q_x-\Pi_x \Pi_y^{-1} \D q_y).
    \end{align*}
    Therefore, ${J_\Pi^A}^* (\D \Re(\eta) + \I \D \Im(\eta)) = -\D \Im(\eta) + \I \D \Re(\eta) = \I (\D \Re(\eta) + \I \D \Im(\eta))$.
\end{proof}
The hyperkähler structure on $X_\Pi$ varies smoothly in $\Pi$, which defines a fiberwise hyperkähler structure $(g_{X/S},J_{X/S}^A,J_{X/S}^B,J_{X/S}^C)$. Then $f:X\to S$ has another complex fiber bundle structure $X_A:=(f,T_{X/S}^A)$. Using the coordinates $(\Pi,\hat{\xi},\hat{\eta})$, where $(\hat{\xi},\hat{\eta})=(\exp(2\pi \I\xi),\exp(2\pi \I\eta))$, we identify $X_A$ with the trivial flat holomorphic fiber bundle $S\times (\CC^*)^{2k}$. We call the trivial flat connection $\nabla^{\RR}$ on $X_A$ the \emph{Gauss-Manin connection}. Its complexification decomposes as $\nabla=\nabla^{1,0}+\nabla^{0,1}$. $\nabla^{0,1}$ induces two $\dbar$-operators $\dbar_{A,0}$ and $\dbar_{B,0}$ on $X_A$ and $X_B$ respectively. $\nabla^{1,0}$ induces an almost connection $\partial_A$ on $X_A$. Let $\dbar_A=\dbar_{A,0}$.

\begin{proposition} \label{prop:GM_lift}
Let $\partial_{ij} = \frac{\partial}{\partial \Pi_{ij}}$ ($1 \le i \le j \le k$) be a basis vector for $TS$. Let $V_{ij} \in M_k(\mathbb{C})$ be the corresponding symmetric matrix variation defined by $(V_{ij})_{\mu\nu} = (\delta_{i\mu}\delta_{j\nu} + \delta_{j\mu}\delta_{i\nu})/(1+\delta_{ij})$ (or simply $V_{ij} = \dot{\Pi}$). The Gauss-Manin $(1,0)$-connection $\nabla^{1,0}$ on $X_B$ is locally given by
    \begin{equation}\label{eq:GM_10}
      \partial_{ij} \mapsto \partial_{ij} +V_{ij} \Pi_y^{-1} q_y\cdot\partial_q + \tfrac{\I}{2} \Pi_y^{-1} V_{ij} p\cdot \partial_p -\frac{\I}{2} \Pi_y^{-1} V_{ij} p\cdot\partial_{\bar{p}}.
    \end{equation}
Consequently, $\dbar_{B,0}-\dbar_B=\tfrac{\I}{2}(\Pi_y^{-1}\D\widebar{\Pi}) \bar{p}\cdot\partial_p\in C^{\infty}(X,f^*\widebar{T^*S}\otimes T_{X/S}^B)$.
\end{proposition}
\begin{proof}
	A vector field is horizontal with respect to the Gauss-Manin connection if $\xi$ and $\eta$ are constant along its integral curve.
	We compute the induced variations $\dot{q}$ and $\dot{p}$ by differentiating \eqref{eq:qp_inverse_rel} with respect to $\Pi$ in the direction $V_{ij}$. Note that $\dot{\Pi} = V_{ij}$ and $\dot{\overline{\Pi}} = 0$. We have
	\[
	\dot{q} = \overset{\cdot}{\wideparen{\Re(\eta)}} + \dot{\Pi} \Re(\xi) + \Pi  \overset{\cdot}{\wideparen{\Re(\xi)}} = \dot{\Pi} \Re(\xi)= V_{ij} \Pi_y^{-1} q_y.
	\]
	For the conjugate coordinate $\bar{q} = \Re(\eta) + \widebar{\Pi} \Re(\xi)$, the derivative is
	\[
	\dot{\bar{q}} = \dot{\overline{\Pi}} \Re(\xi) = 0.
	\]
	 By \eqref{eq:qp_inverse_rel}, the derivative of $\Pi_y p$ vanishes, so
	\[
	\dot{\Pi}_y p + \Pi_y \dot{p} = 0.
	\]
	Note that $\dot{\Pi}_y = \frac{1}{2\I} (\dot{\Pi}- \dot{\overline{\Pi}}) = \frac{1}{2\I} V_{ij}$. Then
	\[\dot{p} = - \Pi_y^{-1} \bigl( \tfrac{1}{2\I} V_{ij} p \bigr) = \tfrac{\I}{2} \Pi_y^{-1} V_{ij} p.\]
	Since $p + \bar{p} = -2\Im(\xi)$ is constant, we have
	\[\dot{\bar{p}} = -\dot{p} = - \tfrac{\I}{2} \Pi_y^{-1} V_{ij} p.
	\]
	Therefore \eqref{eq:GM_10} follows. By taking its conjugate we obtain $\nabla^{0,1}$, which induces $\dbar_{B,0}$. The conjugate of the last term in \eqref{eq:GM_10} is $(\dbar_{B,0}-\dbar_B)(\partial_{\widebar{ij}})$, and the second statement follows.
\end{proof}
Define an almost Higgs field $\theta$ on $X_B$ by $\theta=-\frac{\I}{2}(\D \Pi)p\cdot\partial_q$, i.e.,
\begin{equation}\label{eq:cotang_ab_theta}
	\theta(\partial_{ij})=-\tfrac{\I}{2} V_{ij} p \cdot \partial_q=-\tfrac{\I}{2} (V_{ij})_{ab} p^b \partial_{q^a}.
\end{equation}
The coefficients are holomorphic on $X_B$ and $[\theta,\theta]=0$ by the following lemma, so $(X_B,\dbar_B,\theta)$ is a nonlinear Higgs bundle.
\begin{lemma}\label{lem:rk1_theta_int}
Let $\partial_{i_1 j_1},\partial_{i_2 j_2}\in T_{\Pi}S$. Then $[\theta(\partial_{i_1 j_1}), \theta(\partial_{i_2 j_2})] = 0$.
\end{lemma}
\begin{proof}
The two holomorphic vector fields on $X_\Pi$ are
\[
    \theta(\partial_{i_1 j_1}) = f^a(p) \partial_{q^a} \quad \text{and}\quad \theta(\partial_{i_2 j_2}) =  g^c(p) \partial_{q^c},
\]
where $f^a(p) = -\frac{\I}{2} (V_{i_1 j_1})_{ab} p^b$ and $g^c(p) = -\frac{\I}{2} (V_{i_2 j_2})_{cd} p^d$. Then
\[
    [\theta(\partial_{i_1 j_1}), \theta(\partial_{i_2 j_2})] =  (f^a(p) \partial_{q^a} g^c(p) - g^a(p) \partial_{q^a} f^c(p))\partial_{q^c}=0,
\]
since $f^a(p)$ and $g^c(p)$ are independent of $q$.
\end{proof}

$\theta$ is equivalent to the Kodaira-Spencer map as follows. By Proposition \ref{prop:ks_conn_rep} and Proposition \ref{prop:GM_lift}, the Kodaira-Spencer map $\kappa: T_\Pi S \to H^1(A_\Pi, T A_\Pi)$ is given by
\[\kappa(\partial_{ij})=\dbar_{X_\Pi}(V_{ij}\Pi_y^{-1}q_y\partial_q)=\tfrac{\I}{2}(V_{ij}\Pi_y^{-1})_{ba}\D \bar{q}^a\otimes \partial_{q^b}.\]
For the principally polarized abelian variety $A_\Pi$, we have the canonical identifications $T A_\Pi \cong \mathcal{O}^{\oplus k}$ and
\[ H^1(A_\Pi, T A_\Pi) \cong H^1(A_\Pi, \mathcal{O}) \otimes H^0(A_\Pi, T A_\Pi) \cong \widebar{H^0(A_\Pi, \Omega^1)} \otimes H^0(A_\Pi, \Omega^1)^*. \]
Using the Hermitian metric on $A_\Pi$, we have $\widebar{H^0(A_\Pi, \Omega^1)}\cong H^0(A_\Pi, \Omega^1)^*$. The fiber coordinates $p^\alpha$ of the cotangent bundle $X_\Pi \cong T^* A_\Pi$ can be identified with coordinates on $H^0(A_\Pi, \Omega^1)$. Then $\kappa(\partial_{ij})$ can be identified with a function
\[H_{ij}(q, p) = \tfrac{\I}{4} (V_{ij})_{ab} p^a p^b.\]
\begin{lemma}\label{lem:cotang_theta_ham}
  The vector field $\theta(\partial_{ij})$ is the Hamiltonian vector field of $H_{ij}$ with respect to the holomorphic symplectic form $\Omega_{X/S}^B:=\D p^\alpha\wedge \D q^\alpha$, i.e., $\iota_{\theta(\partial_{ij})} \Omega_{X/S}^B = \D H_{ij}$.
\end{lemma}
\begin{proof}
  By the symmetry of $V_{ij}$, we have
    \[\D H_{ij}=\tfrac{\I}{2} (V_{ij} p) \cdot \D p.\]
On the other hand,
    \begin{align*}
        \iota_{\theta(\partial_{ij})} (\D p^\alpha \wedge \D q^\alpha) &= (\iota_{\theta(\partial_{ij})} \D p^\alpha) \wedge \D q^\alpha - \D p^\alpha \wedge (\iota_{\theta(\partial_{ij})} \D q^\alpha) \\
        &= 0 - \D p^\alpha \wedge \bigl( -\tfrac{\I}{2} (V_{ij})_{ab} p^b \delta_{\alpha a} \bigr) \\
        &= \tfrac{\I}{2} (V_{ij})_{a b} p^b \D p^a=\D H_{ij}. \qedhere
    \end{align*}
\end{proof}

\begin{proposition}
  $\bar{\theta}_J=\dbar_{B,0}-\dbar_B$, where $\bar{\theta}_J$ is defined in \eqref{eq:theta_barJ}.
\end{proposition}
\begin{proof}
Let $v$ be the vector field in \eqref{eq:cotang_ab_theta}. It suffices to show \[\I J_\Pi^A(\bar{v}_{J_\Pi^B})=\tfrac{\I}{2}\Pi_y^{-1}V_{ij}\bar{p}\cdot\partial_{p}.\]
First we have  \[\bar{v}_{J_\Pi^B} = \overline{-\tfrac{\I}{2} (V_{ij} p)^a} \partial_{\bar{q}^a} = \tfrac{\I}{2} (V_{ij} \bar{p})^a \partial_{\bar{q}^a}.\]
Using $J_\Pi^A (\partial_{q_x}) = -\Pi_y^{-1} \partial_{p_y}$ and $J_\Pi^A (\partial_{q_y}) = -\Pi_y^{-1} \partial_{p_x}$, we have
    \begin{align*}
        J_\Pi^A (\partial_{\bar{q}^a}) &= \tfrac{1}{2} \bigl( J_\Pi^A(\partial_{q_x^a}) + \I J_\Pi^A(\partial_{q_y^a}) \bigr) \\
        &= \tfrac{1}{2} \bigl( -(\Pi_y^{-1})^{ab} \partial_{p_y^b} - \I (\Pi_y^{-1})^{ab} \partial_{p_x^b} \bigr) \\
				&=-\I (\Pi_y^{-1})^{ab} \partial_{p^b}.
    \end{align*}
Then we have
    \begin{align*}
        \I J_\Pi^A (\bar{v}_{J_\Pi^B}) &= \I \bigl( \tfrac{\I}{2} (V_{ij} \bar{p})^a \bigr) J_\Pi^A (\partial_{\bar{q}^a}) \\
        &= -\tfrac{1}{2}(V_{ij} \bar{p})^a \bigl( -\I (\Pi_y^{-1})^{ab} \partial_{p^b} \bigr) \\
        &= \tfrac{\I}{2} (\Pi_y^{-1})^{ba} (V_{ij} \bar{p})^a \partial_{p^b}=\tfrac{\I}{2} \Pi_y^{-1} V_{ij} \bar{p}\cdot \partial_p,
    \end{align*}
    where we used the symmetry $(\Pi_y^{-1})^{ba} = (\Pi_y^{-1})^{ab}$.
\end{proof}
This proposition verifies \eqref{eq:twist_simpson_dbar_A}. We proceed to verify \eqref{eq:twist_simpson_partial_A}. Consider the closed $(1,1)$-form on $X_A$ defined by $\omega=\I \partial_{X_A}\dbar_{X_A}\phi$, where $\phi\in C^\infty(X)$ is given by the Hodge norm squared of the cotangent vector $p$:
\begin{equation} \label{eq:hodge_norm_kah_pot}
    \phi(\Pi,q,p) = (\Pi_y)_{\alpha\beta} p^\alpha \bar{p}^\beta.
\end{equation}
\begin{lemma}
	$\omega$ restricts to the fiberwise K\"ahler metric $\omega_{X/S}^A$.
\end{lemma}
\begin{proof}
	It suffices to show that $\I\partial_{X_\Pi^A}\dbar_{X_\Pi^A}\phi=\omega_{J_\Pi^A}$. In the following, we omit the subscript $X_\Pi^A$. It is equivalent to show $\frac{1}{2} \D \D^c \phi = \omega_{J_\Pi^A}$, since $\D \D^c = 2 \I \partial\dbar$ for $\D^c = -{J_\Pi^A}^* \D$. Let $p^\alpha = p_x^\alpha + \I p_y^\alpha$, then
	\[		\phi = (\Pi_y)_{\alpha\beta} (p_x^\alpha + \I p_y^\alpha) (p_x^\beta - \I p_y^\beta) = (\Pi_y)_{\alpha\beta} (p_x^\alpha p_x^\beta + p_y^\alpha p_y^\beta),\]
	where we used the symmetry of $\Pi_y$. We compute
	\begin{align*}
		\D^c \phi &= -{J_\Pi^A}^* \bigl(2 (\Pi_y)_{\alpha\beta} p_x^\alpha \D p_x^\beta + 2 (\Pi_y)_{\alpha\beta} p_y^\alpha \D p_y^\beta\bigr)\\
		&=- \bigl( 2 (\Pi_y)_{\alpha\beta} p_x^\alpha ( -(\Pi_y^{-1})^{\beta a} \D q_y^a ) + 2 (\Pi_y)_{\alpha\beta} p_y^\alpha ( -(\Pi_y^{-1})^{\beta a} \D q_x^a) \bigr)\\
		&=2 p_x^a \D q_y^a + 2 p_y^a \D q_x^a.
	\end{align*}
	Taking the exterior derivative,
	\[\D (\D^c \phi) = 2 \D p_x^a \wedge \D q_y^a + 2 \D p_y^a \wedge \D q_x^a=2\omega_{J_\Pi^A}.\]
	The conclusion follows.
\end{proof}
$\omega$ is a relatively Kähler form on the holomorphic fiber bundle $X_A$, so it induces a $(1,0)$-connection $\nabla_{\omega_{X/S}^A}^{1,0}$ by \eqref{eq:rel_kah_lift}, which is a symplectic connection associated to $\dbar_A$ and $\omega_{X/S}^A$. Let $\partial_{\omega_{X/S}^A}$ be the corresponding symplectic almost connection.
\begin{proposition}
$\nabla_{\omega_{X/S}^A}^{1,0}$ on $X_A$ is given in $(\Pi,q,p)$ by
\begin{equation}\label{eq:symp_A}
\partial_{ij}\mapsto\nabla^{1,0}(\partial_{ij})+\I V_{ij}p\cdot\partial_q+\I\Pi_y^{-1}V_{ij}p\cdot\partial_{\bar p}.
\end{equation}
Consequently, $\partial_A-\partial_{\omega_{X/S}^A}=2\sqrt{2}\beta_{X/S}^{-1}(\theta)$, where $\beta_{X/S}$ corresponds via Lemma \ref{lem:real_beta} to an isometry $\beta_{X/S}^\RR$ given by
    \begin{equation}\label{eq:beta_canon}
        (\beta_{X/S}^\RR)^{-1}= \sqrt{1/2}(\id - J_{X/S}^C).
    \end{equation}
Moreover, if $(\beta_{X/S}^\RR)^{-1}$ has the form $a_0 \id + a_1 J_{X/S}^A + a_2 J_{X/S}^B + a_3 J_{X/S}^C$ for $a_0,a_1,a_2,a_3\in C^\infty(X)$ and satisfies $\partial_A-\partial_{\omega_{X/S}^A}=2\sqrt{2}\beta_{X/S}^{-1}(\theta)$, then it must be the one in \eqref{eq:beta_canon}.
\end{proposition}

\begin{proof}
	We work in the holomorphic coordinates $(\Pi_{ij},\xi^\alpha, \eta^a)$, where the potential $\phi$ depends only on the imaginary parts $\xi_y = \Im(\xi)$ and $\eta_y = \Im(\eta)$. The connection coefficients $\Gamma^{\mathsf{a}}_{ij}$ (where $\mathsf{a}=a$ or $\alpha$) are determined by
  \[  g_{\mathsf{a}\bar{\mathsf{b}}} \Gamma_{ij}^{\mathsf{a}} = - g_{ij,\bar{\mathsf{b}}}, \quad \text{where } g_{\mathsf{a}\bar{\mathsf{b}}} = \partial_{\mathsf{a}} \partial_{\bar{\mathsf{b}}} \phi.\]
Recall that $\phi = p^\dagger \Pi_y p$. Using \eqref{eq:qp_inverse_rel}, we have $p_x = -\xi_y$ and $p_y = -\Pi_y^{-1}(\eta_y + \Pi_x \xi_y)$, so
\begin{equation}\label{eq:phi_xi_eta}
	 \phi(\Pi, \xi_y, \eta_y) = \xi_y^{\text{T}} \Pi_y \xi_y + (\eta_y + \Pi_x \xi_y)^{\text{T}} \Pi_y^{-1} (\eta_y + \Pi_x \xi_y).
\end{equation}
We have $\frac{\partial \phi}{\partial \bar{\xi}}=\frac{\I}{2}\frac{\partial \phi}{\partial \xi_y}$, $\frac{\partial \phi}{\partial \xi}=-\frac{\I}{2}\frac{\partial \phi}{\partial \xi_y}$, and similarly for $\eta$. The fiber coefficient matrix of $\omega$ is
\[ (g_{\mathsf{a}\bar{\mathsf{b}}}) = \tfrac{1}{4} \operatorname{Hess}_{\xi_y,\eta_y}(\phi)= \tfrac{1}{2} \begin{pmatrix} \Pi_y + \Pi_x^{\text{T}} \Pi_y^{-1} \Pi_x & \Pi_x^{\text{T}} \Pi_y^{-1} \\ \Pi_y^{-1} \Pi_x & \Pi_y^{-1} \end{pmatrix}. \]
Its inverse is
\begin{equation} \label{eq:G_inv_final}
  (g^{\bar{\mathsf{b}}\mathsf{a}}) = 2 \begin{pmatrix} \Pi_y^{-1} & -\Pi_y^{-1} \Pi_x \\ -\Pi_x \Pi_y^{-1} & \Pi_y + \Pi_x \Pi_y^{-1} \Pi_x \end{pmatrix}.
\end{equation}

We compute $g_{ij,\bar{\mathsf{b}}} = \partial_{ij} (\partial_{\bar{\mathsf{b}}} \phi)$. Recall $V_{ij} = \dot{\Pi}$, then $\dot{\Pi}_x = \frac{1}{2}(\dot{\Pi}+\dot{\overline{\Pi}}) =\frac{1}{2}V_{ij}$ and $\dot{\Pi}_y = -\frac{\I}{2}V_{ij}$. Note that $\Pi_y$ is symmetric, $\partial_{\eta_y} \phi=(\partial p_y/\partial \eta_y)\partial_{p_y} \phi = -\Pi_y^{-1}(2\Pi_y p_y)=-2p_y$. Then
\begin{align*}
	g_{ij,\bar{\eta}}&= \tfrac{\I}{2}\partial_{ij}(-2p_y)=\I\partial_{ij}\bigl(\Pi_y^{-1}(\eta_y + \Pi_x \xi_y)\bigr)\\
	&=\I \bigl(\tfrac{\I}{2} \Pi_y^{-1} V_{ij} \Pi_y^{-1} (\eta_y + \Pi_x \xi_y) + \Pi_y^{-1} \big(\tfrac{1}{2} V_{ij} \xi_y\big) \bigr)\\
	&=\I\bigl(\tfrac{\I}{2} \Pi_y^{-1} V_{ij} (-p_y) + \tfrac{1}{2} \Pi_y^{-1} V_{ij} (-p_x) \bigr)\\
	&=-\tfrac{\I}{2} \Pi_y^{-1} V_{ij} p.
\end{align*}
Similarly,  $\partial_{\xi_y} \phi = 2 \Pi_y \xi_y - 2 \Pi_x p_y$, and using $\partial_{ij} p_y=(1/2)\Pi_y^{-1}V_{ij}p$ computed above,
\begin{align*}
    g_{ij,\bar{\xi}}&=\partial_{ij} \partial_{\bar{\xi}}\phi =\I \bigl(-\tfrac{\I}{2} V_{ij} \xi_y - \tfrac{1}{2} V_{ij} p_y - \Pi_x (\partial_{ij} p_y)\bigr) \\
    &=\tfrac{\I}{2}( V_{ij} (\I p_x - p_y) - \Pi_x \Pi_y^{-1} V_{ij} p) \\
    &= \tfrac{\I}{2} (\I V_{ij} p - \Pi_x \Pi_y^{-1} V_{ij} p).
\end{align*}

Therefore, the coefficients of the symplectic connection are
\begin{align*}
    \begin{pmatrix} \Gamma_{ij}^\xi \\ \Gamma_{ij}^\eta \end{pmatrix} &= -2 \begin{pmatrix} \Pi_y^{-1} & -\Pi_y^{-1} \Pi_x \\ -\Pi_x \Pi_y^{-1} & \Pi_y + \Pi_x \Pi_y^{-1} \Pi_x \end{pmatrix} \begin{pmatrix} \frac{\I}{2}(\I V_{ij} p - \Pi_x \Pi_y^{-1} V_{ij} p) \\ -\frac{\I}{2} \Pi_y^{-1} V_{ij} p \end{pmatrix}\\
			&=	\begin{pmatrix}	\Pi_y^{-1} V_{ij} p \\ -\Pi_x \Pi_y^{-1} V_{ij} p + \I V_{ij} p		\end{pmatrix}.
\end{align*}
The symplectic connection $\nabla_{\omega_{X/S}^A}^{1,0}$ is given in $(\Pi,\xi,\eta)$ by
\[\partial_{ij}\mapsto \partial_{ij}+\mathcal{V},\text{ where }   \mathcal{V} = \Pi_y^{-1} V_{ij} p \cdot \partial_\xi + (-\Pi_x \Pi_y^{-1} V_{ij} p + \I V_{ij} p)\cdot \partial_\eta.\]
Recall that
\[q= \tfrac{1}{2}(\eta + \bar{\eta}) + \tfrac{1}{2}\Pi(\xi + \bar{\xi}),\quad p=-\tfrac{1}{2} \Pi_y^{-1} \bigl( (\eta - \bar{\eta}) + \widebar{\Pi}(\xi - \bar{\xi}) \bigr).\]
Then we have
\begin{align*}
		\mathcal{V}(q) &= \tfrac{1}{2} \mathcal{V}(\eta) + \tfrac{1}{2} \Pi \mathcal{V}(\xi) = \tfrac{1}{2} \Gamma_{ij}^\eta + \tfrac{1}{2} \Pi \Gamma_{ij}^\xi \\
		&= \tfrac{1}{2} \bigl( (-\Pi_x \Pi_y^{-1} V_{ij} p + \I V_{ij} p) + (\Pi_x + \I \Pi_y)(\Pi_y^{-1} V_{ij} p) \bigr) \\
		&= \tfrac{1}{2} \bigl( -\Pi_x \Pi_y^{-1} V_{ij} p + \I V_{ij} p + \Pi_x \Pi_y^{-1} V_{ij} p + \I \Pi_y \Pi_y^{-1} V_{ij} p \bigr) \\
		&= \tfrac{1}{2} ( \I V_{ij} p + \I V_{ij} p ) = \I V_{ij} p.
\end{align*}
Similarly we have
\begin{align*}
		\mathcal{V}(\bar{q}) &= \tfrac{1}{2} \mathcal{V}(\eta) + \tfrac{1}{2} \widebar{\Pi} \mathcal{V}(\xi) = \tfrac{1}{2} \Gamma_{ij}^\eta + \tfrac{1}{2} \widebar{\Pi} \Gamma_{ij}^\xi \\
		&= \tfrac{1}{2} \bigl( (-\Pi_x \Pi_y^{-1} V_{ij} p + \I V_{ij} p) + (\Pi_x - \I \Pi_y)(\Pi_y^{-1} V_{ij} p) \bigr) \\
		&= \tfrac{1}{2} \bigl( -\Pi_x \Pi_y^{-1} V_{ij} p + \I V_{ij} p + \Pi_x \Pi_y^{-1} V_{ij} p - \I V_{ij} p \bigr) \\
		&= 0.
\end{align*}
Now applying the above two equalities $\mathcal{V}(q)=\I V_{ij}p$ and 	$\mathcal{V}(\bar{q})=0$, we have
\begin{align*}
	\mathcal{V}(p) &= -\tfrac{1}{2} \Pi_y^{-1} \bigl( \mathcal{V}(\eta) + \widebar{\Pi} \mathcal{V}(\xi) \bigr) = -\Pi_y^{-1} \mathcal{V}(\bar{q})=0,\\
	\mathcal{V}(\bar{p}) &= \tfrac{1}{2} \Pi_y^{-1} \bigl( \mathcal{V}(\eta) + \Pi \mathcal{V}(\xi) \bigr) = \Pi_y^{-1} \mathcal{V}(q)=\I \Pi_y^{-1} V_{ij} p.
\end{align*}
Combining these together, we obtain
\[\mathcal{V}=\mathcal{V}(q)\cdot\partial_q+\mathcal{V}(\bar{p})\cdot\partial_{\bar{p}}=\I V_{ij} p\cdot \partial_q +\I \Pi_y^{-1} V_{ij} p\cdot \partial_{\bar{p}}.\]
Therefore, $\nabla_{\omega_{X/S}^A}^{1,0}$ is given by \eqref{eq:symp_A} in $(\Pi,q,p)$.

To show $\partial_A-\partial_{\omega_{X/S}^A}=2\sqrt{2}\beta_{X/S}^{-1}(\theta)$, it remains to check that
\[2(\id-J_\Pi^C)\bigl(-\tfrac{\I}{2} V_{ij} p \cdot \partial_q\bigr)=-(\I V_{ij} p\cdot \partial_q +\I \Pi_y^{-1} V_{ij} p\cdot \partial_{\bar{p}}).\]
This follows since
\[\I J_\Pi^C(V_{ij}p\cdot\partial_q)=-\tfrac{\I}{2}\Pi_y^{-1}V_{ij}p\cdot(\partial_{p_x}+\I\partial_{p_y})=-\I\Pi_y^{-1}V_{ij}p\cdot\partial_{\bar p}.\]
 We prove the last statement. Since $\theta(\partial_{ij})$ is $J_\Pi^B$-holomorphic, we have $J_\Pi^B (\theta(\partial_{ij})) = \I (\theta(\partial_{ij}))$ and
  \[J_\Pi^A (\theta(\partial_{ij}))=-J_\Pi^C (\I\theta(\partial_{ij}))=\tfrac{1}{2}\Pi_y^{-1} V_{ij} p\cdot \partial_{\bar{p}}.\]
Then we have
\[2(a_0\id+a_1J_\Pi^A+a_2J_\Pi^B+a_3J_\Pi^C)(\theta(\partial_{ij}))=(a_2-\I a_0)V_{ij}p\cdot\partial_q+(a_1+\I a_3)\Pi_y^{-1}V_{ij}p\cdot\partial_{\bar{p}}.\]
By comparing the coefficients, we see that $(\beta_{X/S}^\RR)^{-1}$ must be the one in \eqref{eq:beta_canon}.
\end{proof}
\begin{lemma}
	Let $\beta_{X/S}^{\RR}:T_{X/S}^\RR\to T_{X/S}^\RR$ be an isometry with respect to $g_{X/S}$, which satisfies $\beta_{X/S}^{\RR}\comp J_{X/S}^A=J_{X/S}^B\comp \beta_{X/S}^{\RR}$ and $\partial_A-\partial_{\omega_{X/S}^A}=2\sqrt{2}\beta_{X/S}^{-1}(\theta)$. Using the splitting $T_{X/S}^\RR=E_{\mathrm{ab}}\oplus E_{\mathrm{fib}}$ induced by $X_\Pi\cong A_\Pi\times \CC^k$, we have
	\[  (\beta_{X/S}^\RR)^{-1} = (\beta_{X/S,\mathrm{can}}^{\RR})^{-1} \comp \begin{pmatrix} \id_{E_{\mathrm{ab}}} & 0 \\ 0 & U \end{pmatrix},\]
	where $(\beta_{X/S,\mathrm{can}}^{\RR})^{-1}=\sqrt{1/2}(\id-J_{X/S}^C)$ and $U\in C^\infty(X, U(E_{\mathrm{fib}},g_{X/S},J_{X/S}^B))$, i.e., $U$ is a unitary automorphism of $E_{\mathrm{fib}}$.
\end{lemma}
\begin{proof}
   We define the gauge transformation $\Phi := \beta_{X/S,\mathrm{can}}^{\RR} \comp (\beta_{X/S}^\RR)^{-1}$, which is clearly a unitary automorphism of $(T_{X/S}^\RR,g_{X/S},J_{X/S}^B)$. The vectors $\{\theta(\partial_{ij})\}$ span $E_{\mathrm{ab}}^{1,0}$ over $\CC$ at every point where $p\ne0$, so the condition $\beta_{X/S}^{-1}(\theta) = (\beta_{X/S,\mathrm{can}})^{-1}(\theta)$ implies $\Phi v = v$ for all $v\in E_{\mathrm{ab}}$, also at $p=0$ by continuity. Since the splitting $T_{X/S}^\RR=E_{\mathrm{ab}}\oplus E_{\mathrm{fib}}$ is orthogonal with respect to $g_{X/S}$ and is preserved by $J_{X/S}^B$, $\Phi$ has the block diagonal form as above.
\end{proof}
Let $(\Sigma,j)$ be a compact Riemann surface, where $j\in C^\infty(\Sigma,\End T\Sigma^\RR)$ is a complex structure. The moduli space of rank one Higgs bundles $M_{\mathrm{Dol},j}(\CC^*)$ on $(\Sigma,j)$ can be identified with the cotangent bundle of the Jacobian variety $T^*\Jac(\Sigma,j)$, which admits a hyperkähler structure as in \eqref{eq:cotangent_hyperkah_eq1}-\eqref{eq:cotangent_hyperkah_eq4} (see \cite[Ch.~7]{GX08}). As $j\in \mathbf{T}(\Sigma)$ varies, where $\mathbf{T}(\Sigma)$ is the Teichmüller space, we obtain a holomorphic fibration $M_{\mathrm{Dol}}(\CC^*)\to \mathbf{T}(\Sigma)$. Similarly to the above, this holomorphic fibration admits a $\beta$-twisted harmonic metric which is the fiberwise Hitchin metric. In the next subsection, we will consider the fibration whose fibers are moduli spaces of $G$-Higgs bundles.

\newcommand{\bfM}{\mathbf{M}}
\newcommand{\bfT}{\mathbf{T}}
\newcommand{\bfX}{\mathbf{X}}
\subsection{Joint moduli spaces of stable Higgs bundles}\label{subsec:joint_moduli}
In this subsection, we consider the joint moduli space of $G$-Higgs bundles $\bfM(G)$ defined in \cite{CTW25}, where $G$ is a connected complex semisimple Lie group with Lie algebra $\mathfrak{g}$ and Killing form $\kappa_{\mathfrak{g}}$. $\bfM(G)$ consists of equivalence classes of stable $G$-Higgs bundles $(J,\Phi)$, where $J$ is a principal complex structure on a fixed smooth principal $G$-bundle $\pi_P:P\to \Sigma$ over a fixed connected closed oriented surface of genus $g(\Sigma)\geq 2$, and $\Phi\in A_b^1(P,\mathfrak{g})$ satisfies $\Phi\comp J=\I\Phi$ and $\dbar_J\Phi=0$. By \cite[Th.~A]{CTW25}, there is a holomorphic fibration $\pi: \bfM(G)\to \bfT(\Sigma)$, where $\bfT(\Sigma)$ is the Teichmüller space of complex structures on $\Sigma$ and each fiber $\pi^{-1}(j)$ is isomorphic to the moduli space of stable $G$-Higgs bundles on the Riemann surface $(\Sigma,j)$. Moreover, there is a holomorphic Higgs field $\Theta$ on $\pi$. We will verify that the fiberwise Hitchin metric is a $\beta$-twisted harmonic metric. $\bfM(G)$ may have orbifold singularities, but we will restrict to the smooth locus.

Denote the underlying smooth fiber bundle by $f:X\to S$. Let $\mathrm{I}$ be the complex structure on $\bfM(G)$ and $J_{X/S}^B$ be its restriction to the fibers. Let $\dbar_{B}$ be the canonical $\dbar$-operator on $\pi$. Choose a representative $(J,\Phi)$ solving the Hitchin equations with respect to the fixed reduction $P_K$. The tangent space $T_{(J,\Phi)}\bfM(G)^\RR$ consists of equivalence classes of $(\mu,\beta,\psi)$, where $\mu\in A^{0,1}(\Sigma,T_j\Sigma)$ is a Beltrami differential, $\beta\in A_b^{0,1}(P,\mathfrak{g})$, and $\psi\in A_b^{1,0}(P,\mathfrak{g})$, which is \emph{semiharmonic} in the sense that
\[D''(\beta,\psi)+D'\bigl(\tfrac{1}{2\I}\Phi\mu\bigr)=0\quad\text{and}\quad D'(\beta,\psi)=0,\]
where $D''=\dbar_{A_J}+\frac{1}{2\I}\Phi$ and $D'=\partial_{A_J}+(\frac{1}{2\I}\Phi)^*$, $A_J$ is the Chern connection determined by $J$ and a fixed reduction $P_K\subset P$ of the structure group to a fixed compact real form $K$ of $G$. The adjoint $*$ is defined by taking the conjugation on the form part and taking the negative Cartan involution on the Lie algebra part. The complex structure $\mathrm{I}$ is given by
\[\mathrm{I}[(\mu,\beta,\psi)]=[(\I\mu,\I\beta,\I\psi)].\]

 When $\mu=0$, the semiharmonic vector $(0,\beta,\psi)$ is harmonic in the classical sense, and represents a vertical real tangent vector $[(\beta,\psi)]$. By the nonabelian Hodge correspondence, each fiber $\pi^{-1}(j)$ admits a hyperkähler structure $(g_{X/S},J_{X/S}^A,J_{X/S}^B,J_{X/S}^C)$, given by
 \begin{gather*}
 	g_{X/S}([(\beta_1,\psi_1)],[(\beta_2,\psi_2)])=\Re(\langle \beta_1,\beta_2\rangle +\langle \psi_1,\psi_2\rangle),\\
	J_{X/S}^A([(\beta,\psi)])=[(\I\psi^*,-\I\beta^*)],\quad J_{X/S}^B([(\beta,\psi)])=[(\I\beta,\I\psi)],\quad J_{X/S}^C([(\beta,\psi)])=[(\psi^*,-\beta^*)],
\end{gather*}
where $\langle \beta_1,\beta_2\rangle=-\I\int_\Sigma \kappa_{\mathfrak{g}}(\beta_1\wedge \beta_2^*)$ and $\langle \psi_1,\psi_2\rangle=\I\int_\Sigma \kappa_{\mathfrak{g}}(\psi_1\wedge \psi_2^*)$.

  When equipped with the fiberwise complex structure $J_{X/S}^A$, $f$ becomes a trivial bundle $S\times\bfX(G)$, where $\bfX(G)$ is the (smooth locus of the) $G$-character variety. This global trivialization gives rise to the trivial flat real connection $\nabla^\RR$, which defines the \emph{isomonodromic foliation}. Its complexification $\nabla$ induces a $\dbar$-operator $\dbar_{B,0}$ on $X_B:=(f,J_{X/S}^B)$, a $\dbar$-operator $\dbar_A=\dbar_{A,0}$ on $X_A:=(f,J_{X/S}^A)$, and an almost connection $\partial_A$ on $X_A$.

By \cite[Th.~A.(3)]{CTW25}, there is a relatively Kähler form $\omega_0$ on $\bfM(G)$. By Lemma \ref{lem:omega_compatible_dbar}, the symplectic connection $\nabla_{\omega_0}^\RR$ induces $\dbar_B$. Let $[\mu]\in T_jS^\RR$, represented by $\mu\in A^{0,1}(\Sigma,T_j\Sigma)$. We have $J_S([\mu])=[\I\mu]$, where $J_S$ is the complex structure on $S$. Denote $w_{[\mu]}:=\nabla_{\omega_0}^\RR([\mu])$ and $\ell_{[\mu]}:=\nabla^\RR([\mu])$. By \cite[Prop.~5.3.(2)]{CTW25}, $w_{[\mu]}=\frac{1}{2}(\ell_{[\mu]}-\mathrm{I}\ell_{[\I \mu]})$.

\begin{lemma}
	Let $f:X\to S$ be a holomorphic fibration, with the canonical $\dbar$-operator $\dbar_f$ and the complex structure $J$ on $X$. Let $\nabla_1^\RR$ and $\nabla_2^\RR$ be two real connections on $f$, with complexifications $\nabla_1$ and $\nabla_2$. Suppose $\nabla_1^\RR(v)=\frac{1}{2}(\nabla_2^\RR(v)-J\nabla_2^\RR(J_S v))$ for all $v\in TS^\RR$, where $J_S$ is the complex structure on $S$. Then $\nabla_1^{0,1}$ induces $\dbar_f$, $\nabla_1^{1,0}$ and $\nabla_2^{1,0}$ induce the same almost connection. Moreover,
	\begin{equation}\label{eq:diff_dbar_joint_moduli}
		(\dbar_2-\dbar_f)(v^{0,1})=\tfrac{1}{2}(\nabla_2^\RR(v)+J\nabla_2^\RR(J_S v))^{1,0}
	\end{equation}
	for all $v\in TS^\RR$, where $\dbar_2$ is the $\dbar$-operator induced by $\nabla_2^{0,1}$, $v^{0,1}=\frac{1}{2}(v+\I J_S v)$ is the $J_S$-$(0,1)$-vector corresponding to $v$, and the superscript $(1,0)$ on the right hand side has a similar meaning.
\end{lemma}
\begin{proof}
In adapted local holomorphic coordinates, we may write
\[\nabla_2^\RR(\partial_{x^i})= \partial_{x^i}+\Gamma_{x^i}^{u^\alpha}\partial_{u^\alpha}+\Gamma_{x^i}^{w^\alpha}\partial_{w^\alpha},\quad \nabla_2^\RR(\partial_{y^i})= \partial_{y^i}+\Gamma_{y^i}^{u^\alpha}\partial_{u^\alpha}+\Gamma_{y^i}^{w^\alpha}\partial_{w^\alpha},\]
where $s^i=x^i+\I y^i$, $z^\alpha=u^\alpha+\I w^\alpha$. Then we have
\begin{align*}
	\nabla_2^{1,0}(\partial_i)&=\partial_i+\tfrac{1}{2}(\Gamma_{x^i}^{u^\alpha}-\I\Gamma_{y^i}^{u^\alpha})(\partial_\alpha+\partial_{\bar{\alpha}})+\tfrac{1}{2}(\Gamma_{x^i}^{w^\alpha}-\I\Gamma_{y^i}^{w^\alpha})\I(\partial_\alpha-\partial_{\bar{\alpha}})\\
	&=\partial_i+\tfrac{1}{2}(\Gamma_{x^i}^{u^\alpha}-\I\Gamma_{y^i}^{u^\alpha}+\I\Gamma_{x^i}^{w^\alpha}+\Gamma_{y^i}^{w^\alpha})\partial_{\alpha}+\tfrac{1}{2}(\Gamma_{x^i}^{u^\alpha}-\I\Gamma_{y^i}^{u^\alpha}-\I\Gamma_{x^i}^{w^\alpha}-\Gamma_{y^i}^{w^\alpha})\partial_{\bar{\alpha}}.
\end{align*}
By the condition $\nabla_1^\RR(v)=\frac{1}{2}(\nabla_2^\RR(v)-J\nabla_2^\RR(J_S v))$, we have
\begin{align*}
	\nabla_1^{\RR}(\partial_{x^i})&=\tfrac{1}{2}(\partial_{x^i}+\Gamma_{x^i}^{u^\alpha}\partial_{u^\alpha}+\Gamma_{x^i}^{w^\alpha}\partial_{w^\alpha})-\tfrac{1}{2}J(\partial_{y^i}+\Gamma_{y^i}^{u^\alpha}\partial_{u^\alpha}+\Gamma_{y^i}^{w^\alpha}\partial_{w^\alpha})\\
	&=\tfrac{1}{2}(\partial_{x^i}+\Gamma_{x^i}^{u^\alpha}\partial_{u^\alpha}+\Gamma_{x^i}^{w^\alpha}\partial_{w^\alpha})-\tfrac{1}{2}(-\partial_{x^i}+\Gamma_{y^i}^{u^\alpha}\partial_{w^\alpha}-\Gamma_{y^i}^{w^\alpha}\partial_{u^\alpha})\\
	&=\partial_{x^i}+\tfrac{1}{2}(\Gamma_{x^i}^{u^\alpha}+\Gamma_{y^i}^{w^\alpha})\partial_{u^\alpha}+\tfrac{1}{2}(\Gamma_{x^i}^{w^\alpha}-\Gamma_{y^i}^{u^\alpha})\partial_{w^\alpha}.
\end{align*}
Similarly,
\[\nabla_1^{\RR}(\partial_{y^i})=\partial_{y^i}+\tfrac{1}{2}(\Gamma_{y^i}^{u^\alpha}-\Gamma_{x^i}^{w^\alpha})\partial_{u^\alpha}+\tfrac{1}{2}(\Gamma_{x^i}^{u^\alpha}+\Gamma_{y^i}^{w^\alpha})\partial_{w^\alpha}.\]
Then we have
\[\nabla_1^{1,0}(\partial_i)=\partial_i+\tfrac{1}{2}(\Gamma_{x^i}^{u^\alpha}+\Gamma_{y^i}^{w^\alpha}+\I(\Gamma_{x^i}^{w^\alpha}-\Gamma_{y^i}^{u^\alpha}))\partial_{\alpha}.\]
By taking the conjugate, we see that $\nabla_1^{0,1}$ induces $\dbar_f$. By comparing the above formulas for $\nabla_1^{1,0}$ and $\nabla_2^{1,0}$, we see that they induce the same almost connection.

For the last statement, we have
\[(\dbar_2-\dbar_f)(\partial_{\bar{i}})=\widebar{\nabla_2^{1,0}(\partial_i)-\nabla_1^{1,0}(\partial_i)}=\tfrac{1}{2}(\Gamma_{x^i}^{u^\alpha}+\I\Gamma_{y^i}^{u^\alpha}+\I\Gamma_{x^i}^{w^\alpha}-\Gamma_{y^i}^{w^\alpha})\partial_{\alpha}.\]
On the other hand,
\begin{align*}
	\tfrac{1}{2}(\nabla_2^\RR(\partial_{x^i})+J\nabla_2^\RR(\partial_{y^i}))^{1,0}&=\tfrac{1}{2}(\Gamma_{x^i}^{u^\alpha}\partial_{u^\alpha}+\Gamma_{x^i}^{w^\alpha}\partial_{w^\alpha}+\Gamma_{y^i}^{u^\alpha}\partial_{w^\alpha}-\Gamma_{y^i}^{w^\alpha}\partial_{u^\alpha})^{1,0}\\
	&=\tfrac{1}{2}(\Gamma_{x^i}^{u^\alpha}+\I\Gamma_{y^i}^{u^\alpha}+\I\Gamma_{x^i}^{w^\alpha}-\Gamma_{y^i}^{w^\alpha})\partial_{\alpha}.
\end{align*}
Therefore \eqref{eq:diff_dbar_joint_moduli} is proved.
\end{proof}
\begin{lemma}\label{lem:barJ_real}
Let $(Y, J_A, J_B, J_C)$ ($J_C=J_AJ_B$) be a hypercomplex manifold. Let $v,w\in TY^\RR$, with $v_B^{1,0} = \frac{1}{2}(v - \I J_B v)$ and $w_B^{1,0} = \frac{1}{2}(w - \I J_B w)$ being the corresponding $J_B$-$(1,0)$-tangent vectors. Then,
\[
v = J_A w \iff v_B^{1,0} = J_A\bigl(\widebar{w_B^{1,0}}\bigr),
\]
where the conjugation is taken with respect to $J_B$.
\end{lemma}
\begin{proof}
We compute
\[
J_A\bigl(\widebar{w_B^{1,0}}\bigr) = \tfrac{1}{2}J_A (w + \I J_B w) = \tfrac{1}{2} (J_A w + \I J_C w).
\]
Now the equation $v_B^{1,0} = J_A\bigl(\widebar{w_B^{1,0}}\bigr)$ is equivalent to
\[
\tfrac{1}{2}(v - \I J_B v) = \tfrac{1}{2} (J_A w + \I J_C w).
\]
Equating the real and imaginary parts, we obtain $v = J_A w$ and $-J_B v = J_C w$. It remains to check the first equation implies the second. This is true since $-J_B v = -J_B (J_A w) = J_C w$.
\end{proof}
By \cite[Th.~A.(4)]{CTW25}, there is a holomorphic section $\Theta$ of $\Hom(f^*TS,T_{X/S}^B)$, given by $\Theta([\mu]):=[(\frac{1}{2\I}\Phi\mu,0)]$ for $[\mu]\in T_jS^\RR$, where we identified a real tangent vector with its $(1,0)$-part. Similarly to Lemma \ref{lem:cotang_theta_ham}, we have the following.
\begin{lemma}[{\cite[Prop.~7.4]{CTW25}}]\label{lem:Theta_Hamiltonian}
Fix $[\mu]\in T_jS^\RR$ and a stable $G$-Higgs bundle $(J,\Phi)$ with $\pi(J)=j$, representing a point of $f^{-1}(j)$. Endow $f^{-1}(j)$ with the holomorphic symplectic form of $J_{X/S}^B$,
\[
\Omega_{X/S}^B\big((\beta_1,\psi_1),(\beta_2,\psi_2)\big):=\I\int_\Sigma\kappa_{\mathfrak g}\big(\psi_2\wedge\beta_1-\psi_1\wedge\beta_2\big),
\]
and define the function
\[
\varphi_\mu([J,\Phi]):=\tfrac14\int_\Sigma\kappa_{\mathfrak g}(\Phi\otimes\Phi)\,\mu.
\]
Then $\Theta([\mu])$ is the Hamiltonian vector field of $\varphi_\mu$ with respect to $\Omega_{X/S}^B$, i.e., $\iota_{\Theta([\mu])}\Omega_{X/S}^B=\D \varphi_{\mu}$.
\end{lemma}
 Note that $\Theta([\mu])$ is represented by a vector field on the configuration space depending linearly on $\Phi$ whose flow varies $J$ while leaving $\Phi$ invariant. Then similarly to Lemma \ref{lem:rk1_theta_int}, we have $[\Theta,\Theta]=0$. Therefore, $\Theta$ is a Higgs field on $X_B$.
\begin{proposition}
	$\dbar_{B,0}-\dbar_B=\widebar{\Theta}_{J}$.
\end{proposition}
\begin{proof}
	By \eqref{eq:diff_dbar_joint_moduli} and Lemma \ref{lem:barJ_real}, we only need to compute $\frac{1}{2}(\ell_{[\mu]}+\mathrm{I}\ell_{[\I\mu]})$. Let $(\mu, \beta_1, \psi_1)$ be the semiharmonic representative of $\ell_{[\mu]}$ and $(\I\mu, \beta_2, \psi_2)$ be that of $\ell_{[\I\mu]}$.
	Then $\mathrm{I} \ell_{[\I\mu]}$ is represented by $(-\mu, \I\beta_2, \I\psi_2)$, and $\frac{1}{2}(\ell_{[\mu]}+\mathrm{I}\ell_{[\I\mu]})$ is represented by $\frac{1}{2}(0,\beta_1+\I\beta_2,\psi_1+\I\psi_2)$. We have
	\[\tfrac{1}{2}J_{X/S}^A(\ell_{[\mu]}+\mathrm{I}\ell_{[\I\mu]})=\tfrac{1}{2}[(\I\psi_1^*+\psi_2^*,-\I\beta_1^*-\beta_2^*)].\]
	By \cite[Lem.~5.12]{CTW25}, there exist unique Hermitian sections $\zeta_1, \zeta_2\in C^\infty(\Sigma,\ad P)$ such that
\[	D''\zeta_1 = \tfrac{1}{2\I} \bigl(\beta_1^* - \psi_1^* + \tfrac{1}{2\I}\Phi\mu\bigr)\quad\text{and}\quad
			D''\zeta_2 = \tfrac{1}{2\I} \bigl(\beta_2^* - \psi_2^* + \tfrac{1}{2\I}\Phi\I\mu\bigr).\]
Therefore, we have
\[D''(\zeta_1-\I\zeta_2)=\tfrac{1}{2}(\I\psi_1^*+\psi_2^*-\I\beta_1^*-\beta_2^*)-\I \tfrac{1}{2\I}\Phi\mu.\]
This implies
\[\tfrac{1}{2}[(\I\psi_1^*+\psi_2^*,-\I\beta_1^*-\beta_2^*)]=\bigl[\bigl(\tfrac{\I}{2\I}\Phi\mu,0\bigr)\bigr]=\mathrm{I}\Theta([\mu]).\]
	By \eqref{eq:diff_dbar_joint_moduli} and Lemma \ref{lem:barJ_real}, we have
	\[(\dbar_{B,0}-\dbar_B)([\mu]^{0,1})=\tfrac{1}{2}(\ell_{[\mu]}+\mathrm{I}\ell_{[\I\mu]})_B^{1,0}=J_{X/S}^A(\widebar{(-{\mathrm{I}}\Theta([\mu]))_B^{1,0}})=\I J_{X/S}^A(\widebar{\Theta([\mu])_B^{1,0}}).\]
	Therefore, $\dbar_{B,0}-\dbar_B=\widebar{\Theta}_{J}$.
\end{proof}
\begin{remark}
When $G=\GL(n,\CC)$, this result is essentially \cite[Th.~A]{HSZ25}. In \cite[\S3.3]{HSZ25}, a deformation-theoretic interpretation of the conjugation $\bar{(\cdot)}_J$ is given.
\end{remark}
This proposition verifies \eqref{eq:twist_simpson_dbar_A}. We proceed to verify \eqref{eq:twist_simpson_partial_A}. Let $\mathrm{E}:\bfM(G)\to\RR$ be the \emph{energy function} (see \cite[\S5.4]{CTW25}) defined by
\[\mathrm{E}([(J,\Phi)])=\|\Phi\|^2=-2\int_\Sigma \kappa_{\mathfrak{g}}(\Psi\wedge\Psi\comp j),\]
where $\Psi=\frac{1}{2\I}(\Phi-\Phi^*)$ and $j=\pi(J)$ is the induced complex structure on $\Sigma$. By \cite[Eq.~(9.10)]{hitchin87selfdual}, $\mathrm{E}$ is a Kähler potential for $\omega_{X/S}^A$ on each fiber. Let $\omega:=\I\partial_{X_A}\dbar_{X_A}\mathrm{E}$, which is a relatively Kähler form on $X_A$, and it induces a symplectic connection (resp. almost connection) $\nabla_{\omega_{X/S}^A}^{1,0}$ (resp. $\partial_{\omega_{X/S}^A}$) associated to $\dbar_A$ and $\omega_{X/S}^A$.
\begin{lemma}\label{lem:theta_grad_energy}
Let $(Y, g_Y, J_A, J_B, J_C)$ ($J_A J_B = J_C$) be a hyperkähler manifold. Let $f: Y \to \mathbb{C}$ be a smooth function. Suppose $u$ is a $J_A$-$(1,0)$-vector field such that
\[
\omega_{J_A}(u, v) = \I \dbar_{J_A} f (v)
\]
for any $J_A$-$(0,1)$-vector field $v$. Let $w$ be a $J_B$-$(1,0)$-vector field. Then $u = 2(\id - J_C)(w)$ iff
\[
w_\RR = \tfrac{1}{4} \left( \grad \Re(f) + J_C \grad  \Re(f) + J_A \grad  \Im(f) + J_B \grad \Im(f) \right).
\]
\end{lemma}

\begin{proof}
Let $f = \phi + \I\psi$, where $\phi = \Re(f)$ and $\psi = \Im(f)$. The condition $\omega_{J_A}(u, v) = \I \bar{\partial}_{J_A} f (v)$ implies $g_Y(J_A u, v) = \I v(f)=\I g_Y(u, v)$. Note that $v=\frac{1}{2}(v_\RR+\I J_Av_\RR)$. Taking the real part of $g_Y(u, v) = v(f)$ yields
\begin{multline*}
	 g_Y(u_\RR, v_\RR) = v_\RR(\phi) - (J_A v_\RR)(\psi)\\=g_Y(\grad \phi,v_\RR)-g_Y(\grad\psi,J_A v_\RR)=g_Y(\grad \phi + J_A \grad \psi, v_\RR).
\end{multline*}
 Since $v_\RR$ is arbitrary, $u_\RR=\grad \phi + J_A \grad \psi$.

The equation $u = 2(\id - J_C)(w)$ is equivalent to
\begin{align*}
(u_\RR - \I J_A u_\RR) &= 2(\id - J_C)(w_\RR - \I J_B w_\RR) \\
&= 2w_\RR - 2J_C w_\RR - 2\I (J_B w_\RR +J_A w_\RR).
\end{align*}
Comparing the real and imaginary parts, we obtain $u_\RR = 2(\id - J_C) w_\RR$ and $J_Au_\RR = 2 (J_A+J_B) w_\RR$. The first equation implies the second, and is equivalent to
\begin{align*}
	w_\RR &= \tfrac{1}{2} (\id - J_C)^{-1} u_\RR = \tfrac{1}{4} (\id + J_C) u_\RR\\
	&=\tfrac{1}{4} (\id + J_C)(\grad \phi + J_A \grad \psi) \\
	&=\tfrac{1}{4} \left( \grad \Re(f) + J_C \grad  \Re(f) + J_A \grad  \Im(f) + J_B \grad \Im(f) \right). \qedhere
\end{align*}
\end{proof}
\begin{lemma}[{\cite[Lem.~5.19]{CTW25}}]\label{lem:energy_var}
	$\D \mathrm{E}(\ell_{[\mu]})=2\Re\bigl\langle \Phi, \bigl(\tfrac{1}{2\I}\Phi\mu\bigr)^*\bigr\rangle=4\Re \varphi_\mu$.
\end{lemma}
\begin{proposition}
$\partial_A-\partial_{\omega_{X/S}^A}=2(\id-J_{X/S}^C)(\Theta)$.
\end{proposition}
\begin{proof}
Denote $\mathcal{V}:=(\partial_A-\partial_{\omega_{X/S}^A})([\mu]^{1,0})$. By the orthogonality condition satisfied by the symplectic connection, we have $\omega(\nabla_{\omega_{X/S}^A}^{1,0}([\mu]^{1,0}),v)=0$ for any vertical $J_{X/S}^{A}$-$(0,1)$-vector field $v$. This implies
\[\omega_{X/S}^A(\mathcal{V},v)=\omega(\nabla^{1,0}([\mu]^{1,0}),v)=\I\dbar_{J_{X/S}^A}\bigl(\tfrac{1}{2}(\D \mathrm{E}(\ell_{[\mu]})-\I\D \mathrm{E}(\ell_{[\I\mu]}))\bigr)(v),\]
where the second equality follows from the definition $\omega=\I\partial_{X_A}\dbar_{X_A}\mathrm{E}$. By Lemma \ref{lem:theta_grad_energy}, it suffices to prove
\[\Theta([\mu])=\tfrac{1}{8}(\grad \D \mathrm{E}(\ell_{[\mu]}) + J_{X/S}^C \grad \D \mathrm{E}(\ell_{[\mu]})-J_{X/S}^A\grad \D \mathrm{E}(\ell_{[\I\mu]})-J_{X/S}^B\grad \D \mathrm{E}(\ell_{[\I\mu]})).\]
By Lemma \ref{lem:energy_var}, $\D \mathrm{E}(\ell_{[\mu]})=2\Re\bigl\langle \Phi, \bigl(\tfrac{1}{2\I}\Phi\mu\bigr)^*\bigr\rangle$. Let $w=(\beta,\psi)$ be a harmonic vertical tangent vector. Then
\[  \D (\D \mathrm{E}(\ell_{[\mu]}))(w) = 2\Re\bigl( \bigl\langle \psi, \bigl(\tfrac{1}{2\I}\Phi\mu\bigr)^* \bigr\rangle + \bigl\langle \Phi, \bigl(\tfrac{1}{2\I}\psi\mu\bigr)^* \bigr\rangle \bigr) = 4\Re\bigl\langle \psi, \bigl(\tfrac{1}{2\I}\Phi\mu\bigr)^* \bigr\rangle,\]
and $\grad \D \mathrm{E}(\ell_{[\mu]})=4[(0,(\frac{1}{2\I}\Phi\mu)^*)]$. Replacing $\mu$ by $\I\mu$, we have $\grad \D \mathrm{E}(\ell_{[\I \mu]})=4[(0,(\frac{\I}{2\I}\Phi\mu)^*)]$. Consequently,
\begin{multline*}
	\tfrac{1}{8}(\grad \D \mathrm{E}(\ell_{[\mu]}) + J_{X/S}^C \grad \D \mathrm{E}(\ell_{[\mu]})-J_{X/S}^A\grad \D \mathrm{E}(\ell_{[\I\mu]})-J_{X/S}^B\grad \D \mathrm{E}(\ell_{[\I\mu]}))\\
	=\tfrac{1}{2}\bigl[\bigl(0,\bigl(\tfrac{1}{2\I}\Phi\mu\bigr)^*\bigr)+\bigl(\tfrac{1}{2\I}\Phi\mu,0\bigr)+\bigl(\tfrac{1}{2\I}\Phi\mu,0\bigr)-\bigl(0,\bigl(\tfrac{1}{2\I}\Phi\mu\bigr)^*\bigr)\bigr]=\bigl[\bigl(\tfrac{1}{2\I}\Phi\mu,0\bigr)\bigr]=\Theta([\mu]).
\end{multline*}
The proposition is proved.
\end{proof}
\begin{remark}\label{rmk:hitchin}
	When $G=\SL(n,\CC)$, the functions $\mathrm{E}$, $\D \mathrm{E}(\ell_{[\mu]})$, and $\varphi_\mu$ are exactly (up to scaling constants) $f$, $\dot{f}$, $\varphi$ (evaluated on $\mu$) considered in \cite{hitchiin26}. When $G=\GL(n,\CC)$, these functions satisfy the same properties as in Lemmas \ref{lem:Theta_Hamiltonian} and \ref{lem:energy_var} when replacing $\kappa_{\mathfrak{g}}(A_1,A_2)$ by $\mathrm{tr}(A_1A_2)$ (cf. \cite[\S2.2-2.3]{hitchiin26}).
\end{remark}

\subsection{Nonlinear harmonic bundles}
Let $S$ be a smooth complex algebraic variety and $f:X\to S$ be a smooth projective family of algebraic curves of genus $\geq 2$. For each point $s\in S$, we denote by $X_s$ the fiber over $s$. Let $G$ be a connected complex reductive group. In nonabelian Hodge theory, there are three moduli spaces: $M_{\B}(X_s,G)$ is the Betti moduli space, parametrizing the isomorphism classes of irreducible representations of $\pi_1(X_s)$ to $G$; $M_{\dR}(X_s,G)$ is the de Rham moduli space, parametrizing the isomorphism classes of irreducible flat $G$-connections on $X_s$; $M_{\Dol}(X_s,G)$ is the Dolbeault moduli space, parametrizing the isomorphism classes of stable $G$-Higgs bundles on $X_s$ of vanishing rational characteristic classes. It is known that $M_{\B}(X_s,G)$ and $M_{\dR}(X_s,G)$ are complex analytically isomorphic, and $M_{\dR}(X_s,G)$ and $M_{\Dol}(X_s,G)$ are real analytically, but not complex analytically isomorphic. When $s$ varies in $S$, $M_{\B}(X_s,G)$ (resp. $M_{\dR}(X_s,G)$ and $M_{\Dol}(X_s,G)$) form a holomorphic fibration $f_{\B}:M_{\B}(X/S,G)\to S$ (resp. $f_{\dR}:M_{\dR}(X/S,G)\to S$ and $f_{\Dol}:M_{\Dol}(X/S,G)\to S$). We henceforth restrict to the corresponding smooth loci of these moduli spaces, retaining the same notation. For $G=\CC^*$, no restriction is necessary. It follows from \cite[Th.~ 9.11]{simpson94moduli} that $f_{\B}$ and $f_{\dR}$ are isomorphic as holomorphic fiber bundles. Then by \cite[Lem.~ 9.14]{simpson94moduli}, $f_{\dR}$ is isomorphic to the relative Dolbeault moduli space $f_{\Dol}: M_{\Dol}(X/S,G)\to S$ as topological fiber bundles. Based on the recent work \cite{CTW25}, we have the following result.
\begin{lemma}
Notation as above. Suppose that $G$ is semisimple. Then $f_{\dR}$ and $f_{\Dol}$ are isomorphic as smooth fiber bundles.
\end{lemma}
\begin{proof}
    This is a local problem, so we may assume that $f_{\dR}$ is a trivial bundle $S\times M_{\dR}(X_{s_0},G)\to S$. Consider the map
    \[F:M_{\Dol}(X/S,G)\to S\times M_{\dR}(X_{s_0},G),\quad x\mapsto (f_{\Dol}(x),\mathrm{H}(x)),\]
    where $\mathrm{H}$ is the nonabelian Hodge map. $F$ is a homeomorphism. By \cite[Th.~4.23]{CTW25}, $F$ is real analytic. If $\D F_x(v)=(0,0)$, then $\D f_{\Dol,x}(v)=0$, which means $v$ is a vertical tangent vector. Since the restriction of $\mathrm H$ to each fiber is a real analytic diffeomorphism, we have $v=0$. Therefore, $\D F$ is invertible. By the real analytic inverse function theorem, $F^{-1}$ is also real analytic. Hence, $f_{\dR}$ and $f_{\Dol}$ are isomorphic smooth fiber bundles.
\end{proof}
Simpson (\cite[{\S8}]{simpson94moduli}) constructed the Gauss-Manin connection $\nabla_{\mathrm{GM}}$ on $f_{\dR}$, which comes from the isomonodromy deformation of a flat connection. He further constructed in \cite{simpson97Hodge} the nonabelian Hodge filtration $F_{\mathrm{Hod}}$ on $f_{\dR}$. Indeed, via the Rees construction, the natural $\G_m$-action on the relative Hodge moduli $M_{\mathrm{Hod}}(X/S,G)\to \mathbb{A}^1_S$ is interpreted as a filtration on $M_{\dR}(X/S,G)$. It satisfies a nonabelian analogue of the Griffiths transversality \cite[{\S7-8}]{simpson97Hodge}: the $\G_m$-equivariant extension $\widetilde \nabla_{\mathrm{GM}}$ of $\nabla_{\mathrm{GM}}$ over $M_{\mathrm{Hod}}(X/S,G)$ vanishes with order one along $M_{\Dol}(X/S,G)$. On the other hand, an explicit Higgs field $\theta_{\mathrm{KS}}$, the so-called nonabelian Kodaira-Spencer map on $f_{\Dol}$ has recently been constructed in \cite{fu2025}, and it coincides with the residual action of $\widetilde \nabla_{\mathrm{GM}}$ (\cite[Th.~1.2]{fu2025}, see also \cite{Ch12}).

\begin{theorem}\label{thm: vnhs is harmonic bundle}
Notation as above. Suppose that $G$ is semisimple or $\CC^*$. Then the flat bundle $(f_{\dR}: M_{\dR}(X/S,G)\to S,\nabla_{\mathrm{GM}})$ and the Higgs bundle $(f_{\Dol}: M_{\Dol}(X/S,G)\to S,\theta_{\mathrm{KS}})$ can be related to each other via the twisted Simpson mechanism.
\end{theorem}
\begin{proof}
The problem is local. For a given point $s\in S$, take an open neighborhood $s\in U\subset S$ so that the family $X\to S$ over $U$ admits a smooth trivialization $f^{-1}(U)\cong U\times \Sigma$ over $U$. Then in the semisimple case, there is a unique holomorphic map $\phi: U\to \bfT(\Sigma)$ such that the pullback family $\bfM(G)\times_{\bfT(\Sigma)}U\to U$ is isomorphic to $M_{\Dol}(X/S,G)|_U\to U$. Note that the (twisted) Simpson mechanism is compatible with pullback. The result follows from the universal case as carried out in Section \ref{subsec:joint_moduli}.

For the rank one case, i.e., $G=\CC^*$, we consider the associated weight one $\ZZ$-PVHS $V_{\ZZ}$ to $X\to S$. The theory of variation of Hodge structure applied to $V_{\ZZ}$ gives rise to a linear flat bundle $(V,\nabla_{\mathrm{GM}})$ and a linear Higgs bundle $(E=E^{1,0}\oplus E^{0,1},\theta_{\mathrm{KS}})$. They correspond to each other under the nonabelian Hodge correspondence. This holds true for a general weight one $\ZZ$-PVHS, equivalently the polarized abelian family $A:=E^{0,1}/\operatorname{pr}^{0,1}(V_{\ZZ})\to S$, where $\operatorname{pr}^{0,1}:V\to V/F^1V=E^{0,1}$ is the natural projection (in our case, $A=\textrm{Pic}^0(X/S)$). As $V_{\ZZ}\subset V$ and $\nabla_{\mathrm{GM}}$ is linear, we obtain a flat holomorphic bundle by taking the quotient
$(V/V_{\ZZ}\to S, \overline \nabla_{\mathrm{GM}})$, which is nothing but the rank one relative de Rham moduli equipped with the nonabelian Gauss-Manin connection. We know that the relative Dolbeault moduli is the relative cotangent bundle of the relative Picard variety $\textrm{Pic}^0(X/S)$. Analytically, it is isomorphic to the quotient
$E^{1,0}\times_S (E^{0,1}/\operatorname{pr}^{0,1}(V_{\ZZ}))\to S$. The polarization identifies $E^{1,0}$ with $(\operatorname{Lie}(A/S))^*$, so this is the relative cotangent bundle $T^*_{A/S}\to S$.
 Since $\theta_{\mathrm{KS}}$ vanishes on $E^{0,1}$, it descends to a nonlinear Higgs field on the relative cotangent bundle, which in the case of rank one relative Dolbeault moduli is nothing but the nonabelian Kodaira-Spencer map. Now the result follows from the universal one verified in Section \ref{subsec:cotangent_ab_var}.
\end{proof}
The above result leads to the following notion.
\begin{definition}\label{def: nonlinear harmonic bundle}
Let $S$ be a complex manifold. A \emph{nonlinear harmonic bundle} over $S$ is a smooth fiber bundle
$f: X\to S$, equipped with a flat holomorphic bundle structure $(f_A: X_A\to S,\nabla)$ and simultaneously a holomorphic Higgs bundle structure $(f_B: X_B\to S,\dbar_B,\theta)$ (which means $f_A$ and $f_B$ have the same underlying smooth fiber bundle $f$), such that they are related to each other via some effective twisting map $\beta: T_{X/S}^A\to T_{X/S}^B$ (both are regarded as complex subbundles of $T^{\CC}_{X/S}$) and a ($\beta$-twisted) harmonic metric on $f$ through the (twisted) Simpson mechanism. A choice of a twisting map and a ($\beta$-twisted) harmonic metric is not part of the data.
\end{definition}

Naturally, we expect that $(M_{\dR}(X/S,G),\nabla_{\mathrm{GM}}; M_{\Dol}(X/S,G),\theta_{\mathrm{KS}})$ is also a nonlinear harmonic bundle for arbitrary reductive $G$. A similar result should also hold true when $X\to S$ is an arbitrary smooth projective morphism. However, one has to either restrict to the smooth locus (it seems unclear whether $M_{\dR}(X/S,G)$ and $M_{\Dol}(X/S,G)$ have the same smooth locus. Instead, one may consider an even smaller smooth locus, namely the Zariski-dense locus), or consider everything with suitable singularities. \\

Motivated by the half nonlinear Hodge correspondence established in Section \ref{sec:nonlin_harm}, it is a natural subsequent goal to characterize the image of the functor, and then construct the other half correspondence. It involves introducing an appropriate notion of stability for nonlinear Higgs bundles, and a suitable nonlinear analogue of the Hermitian-Yang-Mills equation. This problem becomes increasingly challenging when the corresponding monodromy representations do not factor through a complex Lie subgroup (which one might call infinite-dimensional), when the effective twisting map is non-identity, and when the base manifold $S$ is noncompact.

\bibliographystyle{amsalpha}
\bibliography{hol}

\end{document}